\documentclass[11pt,letterpaper]{amsart}
\usepackage[margin=0.9in]{geometry}
\usepackage[english]{babel}
\usepackage[utf8]{inputenc}
\usepackage{amsmath,amssymb,amsthm,mathtools}
\usepackage{hyperref,mathrsfs}
\usepackage{xcolor}

\usepackage{hyperref}

\def\mX{\mathcal{X}}

\newcommand{\Yn}{Y} 
\newcommand{\dYn}{\dot{\Yn}} 

\newcommand{\fp}[1]{\color{black} {#1} \color{black}  }

\newcommand{\dg}[1]{\color{black} { #1}\color{black}}

\newcommand{\R}{\mathbb{R}}

\newcommand{\n}{n} 
\newcommand{\Z}{\mathbb{Z}}

\newcommand{\E}{\mathcal{E}}

\newcommand{\X}{\mathcal{X}}
\newcommand{\W}{\mathcal{W}}

\newcommand{\uS}{\underline{S}} 
\newcommand{\uO}{\Omega} 

\newcommand{\D}{\mathcal{D}}

\newcommand{\pa}{\partial}
\newcommand{\ve}{\varepsilon}

\def\wt{\widetilde}
\def\jnab{\langle \nabla \rangle}

\renewcommand{\Im}{\operatorname{Im}}

\renewcommand{\Im}{\operatorname{Im}}

\def\C{{\mathbb C}}
\def\Z{{\mathbb Z}}

\def\R{{\mathbb R}}
\def\e{\varepsilon}
\def\jt{\langle t \rangle}
\def\js{\langle s \rangle}

\def\eps{\epsilon}
\def\S{\mathcal{S}}
\def\Dt{\partial \mathcal{D}_t}

\def\what{\widehat}

\newcommand{\comment}[1]{\vskip.3cm\fbox{\parbox{0.93\linewidth}{\footnotesize #1}}\vskip.3cm}

\def\tofill{\vskip30pt $\cdots$ To fill in $\cdots$ \vskip30pt}

\let\div\relax 

\DeclareMathOperator{\curl}{curl}
\DeclareMathOperator{\div}{div}

\numberwithin{equation}{section}

\newtheorem{theorem}{Theorem}[section]
\newtheorem{prop}[theorem]{Proposition}
\newtheorem{cor}[theorem]{Corollary}
\newtheorem{lemma}[theorem]{Lemma}
\newtheorem{remark}[theorem]{Remark}
\newtheorem{defi}[theorem]{Definition}

\begin{document}

\title[Water waves with vorticity]{Long time regularity for $3$d gravity waves \\ with vorticity}

\author{Daniel Ginsberg}
\address{Department of Mathematics, Brooklyn College (CUNY), 2900 Bedford Ave, Brooklyn, NY 11210, USA
(Corresponding Author)}
\email{daniel.ginsberg@brooklyn.cuny.edu}
\author{Fabio Pusateri}
\address{Department of Mathematics, University of Toronto, 40 St. George
street, Toronto, M5S 2E4, Ontario, Canada}
\email{fabiop@mail.math.toronto.edu}

\begin{abstract}
We consider the Cauchy problem for the full free boundary Euler equations 
in $3$d with an initial small velocity of size $O(\e_0)$, 
in a moving domain which is initially an $O(\e_0)$ perturbation of a flat interface.
We assume that the initial vorticity is of size $O(\e_1)$ and prove
a regularity result up to times of the order $\e_1^{-1+}$, independent of $\e_0$.

A key part of our proof is a normal form type argument for the vorticity equation;
this needs to be performed in the full three dimensional domain 
and is necessary to effectively remove the irrotational components from the quadratic stretching terms
and uniformly control the vorticity.
Another difficulty is to obtain sharp decay for the irrotational component of the velocity
and the interface; 
to do this we perform a dispersive analysis 
on the boundary equations, which are forced by a singular 
contribution from the rotational component of the velocity.

As a corollary of our result, when $\e_1$ goes to zero
we recover the celebrated global regularity results of Wu (Invent. Math. 2012) 
and Germain, Masmoudi and Shatah (Ann. of Math. 2013) 
in the irrotational case.
\end{abstract}

\maketitle

\setcounter{tocdepth}{1}
\begin{quote}
\tableofcontents
\end{quote}

\bigskip
\section{Introduction}\label{secintro}
We consider the classical free boundary Euler equations with gravity in three space dimensions:
\begin{subequations}\label{freebdy}
\begin{alignat}{2}
\label{freebdyeul}
(\pa_t + v^k\pa_k) v_i &= -\pa_i p - ge_3, &&\quad \text{ in } \D_t,
\\
\label{freebdyinc}
\div v &= 0, &&\quad \text{ in } \D_t,
\\
p &=0, &&\quad \text{ on } \pa \D_t,
\\
\label{freebdybc}
(1, v) &\text{ is tangent to } \D= \cup_{t \geq 0} \{t\}\times \D_t.
\end{alignat}
\end{subequations}
We are adopting the usual convention of summing over repeated upper and lower indexes.
In what follows we set $g = 1$.
We assume that the boundary of the moving domain, denoted $\pa \D_t$, 
is given by the graph of a function $h$:
\begin{align}
\D_t := \{ (x,y) \in \R^2\times \R \, : \, y \leq h(t,x)\}.
\end{align}
This problem, and closely related models, have been studied extensively.
We will recall the local and global well-posendess theory and other results in the literature below in Subsection \ref{seclit}.

For the moment we point out that in the irrotational case ($\omega := \curl v = 0$) 
one can construct classes of global solutions
close to a flat and still interface; see 
Wu \cite{Wu2} and Germain-Masmoudi-Shatah \cite{GMS2} for the problem \eqref{freebdy},
and \cite{GMSC,DIPP} and the other references given below for the case of other $3$d and $2$d models. 
These are essentially the only known classes of global solutions for the initial value problem.
In this paper we are interested in the regularity question for the Cauchy problem 
for general solutions with rotation, $\omega \neq 0$.

The first natural question to ask is: 
given an initial (divergence free) velocity field and an initial perturbation of a flat interface
of size $\e_0$ (typically measured in a weighted Sobolev space), and an initial vorticity of size $\e_1$, 
what is the maximal time of existence and regularity of solutions?
Our main result shows that 
the above problem admits a solution at least until times that are (almost) of the order of $1/\e_1$,
uniformly in the size $\e_0$ of the irrotational components of the solution. 
This is the natural time scale for the evolution of the vorticity,
which, in three dimensions, is a transport equation with quadratic terms.
By sending $\e_1$ to zero one then also recovers the celebrated results of \cite{GMS2} and \cite{Wu2},
including control on high order energies, and sharp pointwise decay of solutions.

We first give here an informal statement, and will give a more precise one in Theorem \ref{maintheoprecise}:

\begin{theorem}\label{maintheo}
Assume that the initial height $h(0,x):\R^2 \rightarrow \R$ 
and the initial (divergence-free) velocity $v(0,x,y)$
defined on $\D_0 := \{ (x,y) \in \R^2\times \R \, : \, y \leq h(0,x)\}$,
are of size $\e_0$ in sufficiently regular weighted Sobolev spaces.
Assume that $\curl v(0,x,y)$ 
is of size $\e_1$ in a sufficiently regular weighted Sobolev space,
and that $\curl v(0,x,h(x)) = 0$. 


Then, for any fixed $\delta >0$ there exists $\bar{\e}_0$ and $\bar{c}$ sufficiently small, independent of $\e_1$,
such that, for any $\e_1 \leq \e_0 \leq \bar{\e}_0$, the system \eqref{freebdy}
has a unique classical solution $(v,h)$ with the above given initial data $(v(0),h(0))$, 
on the time interval $[-T_{\e_1},T_{\e_1}]$ with
\begin{equation}\label{maintheoT}
T_{\e_1} := \frac{\bar{c}}{\ve_1^{1-\delta}}.
\end{equation}
\end{theorem}


\smallskip
\subsection{Previous results}\label{seclit}
Studies on the free boundary Euler equations go back at least to Cauchy, Laplace and Lagrange \cite{Craik},
and the analysis of \eqref{freebdy}, and several of its variants, 
has been a very active research area in the last few decades. 
We will not try to give a complete list of references here, but only mention 
those results that are most relevant to the present work.
We direct the reader to the extensive lists of references in some of the cited works, 
and to the survey \cite{IoPuRev} for more background.

\smallskip
{\it Local well-posedness}.
The local well-posedness theory of the free boundary Euler equations and several of its variants
is well-understood in a variety of different scenarios, due to the contributions of many authors.
Without being exhaustive we mention \cite{CraigLim,Wu1,Wu2,CL,HL,LL, Lannes,CS2,ShZ2,ShZ3,ABZ2,CHS,IPTT,WZZZ} 
and refer the reader to \cite[Section $2$]{IoPuRev} and to the book of Lannes \cite[Chapter 4]{LannesBook}.
In short, for sufficiently regular Sobolev initial data, classical smooth solutions exist
on a (small) time interval $[-T,T]$ where $T$ is approximately 
the minimum between the inverse of the size of the initial velocity (in a Sobolev space)
and some quantities that depend on the geometry of the interface (e.g.
the so-called `arc-chord constant'). 

We remark that among the cited works only \cite{CL,CS2,ShZ3,HL,LL, IPTT,WZZZ}
treat the full problem with rotation;
for the case of constant vorticity, the paper \cite{ITv} proves an extended life span,
and the recent work of Wang \cite{Wanglow}
establishes low regularity local wellposedness.
All the other works only consider the irrotational case,
customarily referred to as the `water waves' problem.
The main advantage in considering the irrotational problem, as far as local existence is concerned,
is that the equations of motion
can be reduced to equations on the interface for suitable unknowns;
this reduction can be done both in Eulerian or Lagrangian coordinates.

\smallskip
{\it Global irrotational solutions and related results}.
In the irrotational case one can construct global solutions to the water waves problem
in the vicinity of a flat and still interface.
More precisely, for localized initial data in a weighted Sobolev space, 
one can rely on dispersion and pointwise decay 
to prove scattering (and modified scattering) results.
We refer the reader to \cite{WuAG,IoPu2,ADb,IT,IoPu3,IoPu4,IT2,Wa1} for the case of $1$d interfaces, 
and to \cite{Wu3DWW,GMS2,DIPP,Wa2} for $2$d interfaces;
see also \cite[Section 3]{IoPuRev} and \cite{DelortICM} for an overview of these results. 
As a corollary of our main result, when $\e_1$ goes to zero
we recover the  global regularity results for the irrotational problem with gravity 
of Wu \cite{Wu3DWW} and Germain, Masmoudi and Shatah \cite{GMS2}.

Our work is related to the work by Ionescu and Lie \cite{IoLie} 
where the authors prove a similar result for
the one-fluid Euler-Maxwell system 
in 3d, that is, existence of small solutions up to times
of $O(\e_1^{-1})$ where $\e_1$ is the size of the initial `vorticity' $B-\curl v$, where $B$ is the magnetic field,
and decay of the irrotational components.
One major difference in the case of \cite{IoLie} is that the linear decay of the irrotational solutions is 
integrable-in-time, unlike the case of irrotational gravity waves, which decay
at the rate of $t^{-1}$ in $L^\infty_x$. This fact has a major impact on the arguments, 
as we will explain below in Section \eqref{ssecStra} (see for example Step 4).
We also mention that Sun \cite{Sun} proves a similar $O(\e_1^{-1})$ existence result for 
the two-fluid Euler-Maxwell system, and for the Euler-Korteweg system,
by viewing the rotational problem as a perturbation of the irrotational problem,
for which global bounds and integrable-in-time decay are known;
assuming decay for the rotational components, an elegant argument based on energy estimates and `gauge' techniques
provides the claimed long-time existence result, but only obtains weak (exponentially growing) bounds 
on high order energies. Note that also in the case of \cite{Sun} the integrable decay
of irrotational small solutions seems crucial.

\smallskip
{\it The water waves problem with vorticity}.
The question of long-term regularity for water waves with vorticity
is much more delicate than in the irrotational case.
This is due to the fact that the vorticity satisfies a transport equation with a quadratic nonlinear 
(stretching) term.
Moreover, in the free boundary problem, the presence of non-trivial vorticity prevents 
the reduction of the equations solely to the boundary.\footnote{
We also note that the Taylor sign condition $-\nabla_N p|_{\pa \D_t} > 0$, which is 
needed for local well-posedness, holds automatically in the irrotational case but can fail
if there is nonzero vorticity (see, for example, \cite{Su2,WunonC1}), 
though it holds automatically in the small data regime we are working in here.}
So far, to our knowledge, the only available results on extended lifespans are those of the first author
\cite{Gin1}, the work \cite{ITv} proving a time of existence of $\e_0^{-2}$ in the case of constant vorticity in 
the $2d$ case, and \cite{Su1} proving an $\e_0^{-2}$ existence result in the case of point vortices.
Concerning the problem of finding other types of solutions with vorticity,
we mention the recent work of Ehnstrom, Walsh and Zheng \cite{Zhengetal} on stationary solutions. 
Finally, we also mention Castro-Lannes \cite{CaLa}
who proved a well-posedness results with a new Hamiltonian formulation for shallow water waves with vorticity,
Berti-Franzoi-Maspero \cite{BeFrMa} 
who construct quasi-periodic in time solutions with constant vorticity,
and \cite{BeMaMu} who prove an almost global existence result with constant vorticity on the torus.


\smallskip
{\it Further references}.
For further references we refer the reader to the following:
the review \cite{IoPuRev} for more background on the construction
of long-time and global solutions; 
\cite{BeFePu,DeIoPu1} for more literature on spatially periodic solutions;
and to the review \cite{StraussONEPAS} for more on traveling and stationary waves (including the case with vorticity).


\noindent
{\bf Funding Declaration}.
D.G. is supported in part by a start-up grant from Brooklyn College. Part of this work
was completed while D.G. was supported by the Simons Center for Hidden Symmetries and Fusion Energy.
F.P. is supported in part by a start-up grant from the University of Toronto, and NSERC grant RGPIN-2018-06487.

{\bf Acknowledgements}.
The authors thank Alexandru Ionescu and Chongchun Zeng for helpful discussions about the problem.

\bigskip
\section{Strategy and main propositions}\label{secStra}

\subsection{General set-up and some ideas}
We being by decomposing the divergence free vector field $v$ into its rotational and irrotational parts in $\D_t$,
\begin{equation}\label{setup1}
v = \nabla \psi  + v_\omega, \qquad \Delta \psi = 0, \quad v_\omega \cdot n = 0,
\end{equation}
and we denote the vorticity by $\omega = \curl v$.
The moving boundary condition reads 
$$\partial_t h = \nabla \psi \cdot (-\nabla h, 1).$$
We let $\varphi(t,x) := \psi|_{\pa \D_t} = \psi(t,x,h(t,x))$ be the trace of the velocity potential;
one can reconstruct $\psi$ from $\varphi$ solving a standard elliptic problem.
We also define the main dispersive variable,
\begin{equation}\label{uintro}
u = h + i\Lambda^{1/2} \varphi, \qquad \Lambda := |\nabla|.
\end{equation}

The proof of our main result will be based on several interconnected bootstrap arguments for the quantities 
$\nabla \psi, h, v_\omega, \omega$ and $u$, for the vector potential $\beta$ associated to $v_\omega$
(i.e. $-\curl\beta = v_\omega$), and/or their counterparts in the flattened domain
obtained by mapping $y \rightarrow z := y - h(t,x)$.
A high level description of the proof is the following:

\setlength{\leftmargini}{1.5em}
\begin{itemize}

\item {\it High order energy and decay}. 
The basic starting point of our proof is weighted energy estimates for $v$, $h$ and $\omega$.
The weighted $L^2$-based Sobolev norms that we use
are based on the vector fields generated by the invariance of the equation:
(3d) translation and scaling\footnote{Technically these are only approximate invariances since the domain is
not translation or scaling invariant in the vertical direction.} and 2d rotations.
The energy estimate guarantees that top-order energy norms of $v$, $h$ and $\omega$ remain of size $\e_0 \jt^{p_0}$,
with $p_0$ a small constant, as long as we can prove time-decay at a rate of $\jt^{-1}$
for a lower order weighted norm of $v$ and $h$ in $L^\infty$.
See Proposition \ref{propEv} for a precise statement of the energy inequality. 
The main efforts then go into proving the necessary sharp decay in time.
To prove this, we use two separate arguments, one for $v_\omega$,
and one for $u$. 
For these arguments we also need high order bounds on the velocity potential on the interface,
which do not follow immediately from the $L^2$-orthogonality of $\nabla \psi$ and $v_\omega$;
we give the additional arguments needed in Section \ref{secvelpot}.

\item {\it Estimates on $v_\omega$ from the vorticity}.  
Since we work with times $|t| \leq \e_1^{-1+}$, 
proving the needed decay for $v_\omega$ amounts to bounding it (almost) uniformly-in-time by $\e_1$.
Note that the basic energy estimates only guarantees bounds of $O(\e_0)$ for the vorticity.

Naturally, $v_\omega$ can be estimated in terms of $\omega$ through a $\mathrm{div}$-$\mathrm{curl}$ system.
In practice, we relate $\omega$ and $v_\omega$ by introducing the vector potential
$\beta$ such that $-\curl v_\omega = \beta$.
The vector potential satisfies an elliptic system with mixed Dirichlet and Neumann boundary conditions
in the unbounded fluid domain.
When trying to obtain estimates through this elliptic system,
the limited (weighted) regularity and decay available on the geometry 
need to be carefully taken into account. 
It turns out that, all along the argument
we need to 
allow small growth for the highest norms of $v_\omega,\beta,\omega$,
while trying to control uniformly-in-time some lower order norms.
The necessity of letting the highest norms grow slightly in time
is essentially due to the critical nature of the problem, relative to time-decay.
This is also the technical reason why we allow for the presence of a small $\delta>0$
in \eqref{maintheoT} for our maximal time of existence.\footnote{While
this is most likely a technical issue, to avoid this small loss one may need
to make several adjustments to our arguments, or use substantially different arguments
based, for example, on a suitable paradifferential formulation of the problem.
Of course, this loss would not be present
if one were to let $\e_0 = \e_1$, and the existence time would be $\e_1^{-1}$ in this case,
consistently with the local-in-time theory.}

Flattening the domain to a half-space, and using bounds in weighted Lebesgue spaces for the Poisson kernel
we can obtain sufficiently strong bounds for $v_\omega$,
provided certain weighted Lebesgue norms of $\omega$ 
are controlled. See Section \ref{secVP}.

\item To bound the needed weighted Lebesgue norms of $\omega$
we use the vorticity transport equation.
Here one needs to deal with the slowly decaying contributions 
from the stretching terms, which are coming from the non-integrable slow decay 
of the irrotational components of the solution.
To overcome this, we use a normal form type argument on the 
vorticity transport equation in the full three dimensional domain.
This procedure renormalizes the vorticity equation allowing us to propagate the desired 
control on $\omega$. See Section \ref{secVorticity}.
These bounds on $\omega$ imply decay for $v_\omega$.

\item Finally, we need to prove decay for the irrotational components
of the solution $\nabla\psi$ and $h$; this amounts to proving decay for $u$ as in \eqref{uintro}.
We start by deriving boundary equations for $u$ that extend the well-known Zakharov-Craig-Schanz-Sulem
Hamiltonian formulation \cite{Zak0,CraigSS}; see \eqref{der30} and the simplified version in \eqref{der30intro}.
In the general case with rotation, the dispersive-type evolution equation for $u$ is 
`singularly' forced by the restriction to the boundary of 
$v_\omega$. 

To obtain decay for $u$ 
we use weighted $L^2$-$L^\infty$ estimates, and Poincar\'e normal forms
to remove the purely irrotational quadratic components.
To deal with the forcing and the other rotational components we use the estimates 
previously established on $v_\omega$.
Here we need to require more (weighted) regularity for the rotational components,
compared to the regularity of the irrotational components in the $L^\infty_x$-space where we establish time decay.
Moreover, we need to pay particular attention to small frequencies due to the singular nature of 
the forcing.

\end{itemize}

We will describe the above steps and the main bootstrap propositions 
more precisely in Subsection \ref{ssecStra}
after introducing all the necessary notation and parameters.

\subsection{Vector fields and function spaces}
In $\D_t$ we use $x=(x_1,x_2)$ to denote the horizontal variables and $-\infty< y < h(t,x)$ for the vertical one.
For several arguments we will find it convenient to flatten $\partial\D_t$
with the mapping $y \rightarrow z := y - h(t,x)$, which transforms $\D_t$ into the lower-half plane 
$\R^2_x \times \{z<0\}.$

We denote the standard $2$d vector fields
\begin{align}\label{2dvf}
\nabla_x := (\partial_{x_1},\partial_{x_2}), \quad S := \frac{1}{2}t\partial_t + x\cdot \nabla_x
  ,  \quad \Omega := x \wedge \nabla_x;
\end{align}
we will drop the index $x$ for the gradient when there is no risk of confusion.
We denote the `$3$d vector fields' in $\D_t$ as
\begin{align}\label{3dvf}
\underline{\nabla} = \nabla_{x,y} = (\partial_{x_1},\partial_{x_2},\partial_y),
  \qquad \underline{S} = S+y\partial_y, 
  \qquad \underline{\Omega} = \Omega.
\end{align}
In the flattened domain $\R^2_x \!\times\! \{z<0\}$ we slightly abuse notation and still
denote the `$3$d vector fields' by
\begin{align}\label{3dvfflat}
\underline{\nabla} = \nabla_{x,z} = (\partial_{x_1},\partial_{x_2},\partial_z),
  \qquad \underline{S} = S+z\partial_z 
  \qquad \underline{\Omega} = \Omega.
\end{align}
The distinction between these sets of vector fields will always be clear from context.




Let $\Gamma$, respectively $\underline{\Gamma}$, be the collection of $2d$, respectively $3d$ vector fields:
\begin{align}\label{Gammas}
\Gamma = (\partial_{x_1},\partial_{x_2}, S, \Omega),
\qquad \underline{\Gamma} = (\partial_{x_1},\partial_{x_2},\partial_y, \underline{S}, \Omega).
\end{align}
These are respectively $4$- and $5$-component vectors, but we will use the same notation for
multiple applications of them when this causes no confusion, that is,
we will write
$\Gamma^j$, with the understanding that $j\in\Z_+^4$, or $\underline{\Gamma}^j$ with the understanding
that $j\in\Z_+^5$.

Let $W^{s,p} = W^{s,p}(D;\C^m)$, with $H^s = W^{s,2}$ be the standard Sobolev spaces
with $D$ a (sufficiently) smooth domain in $\R^3$, or the plane $\R^2$.
We define the following basic spaces:
\begin{align}
\label{defX}
X^{r,p}_k(\Omega) & := \big\{ f \, : \, \sum_{|j|\leq k}{\| \underline{\Gamma}^j f \|}_{W^{r,p}(\Omega)} < \infty \big\},
  \qquad X^r_k := X^{r,2}_k
\\
\label{defZ}
Z^{r,p}_k(\R^2) & := \big\{ f \, : \, \sum_{|j|\leq k}{\| \Gamma^j f \|}_{W^{r,p}(\R^2)} < \infty \big\},
  \qquad Z^r_k := Z^{r,2}_k.
\end{align}
We denote by ${\| \cdot \|}_{X^{r,p}_k(\Omega)}$ and ${\| \cdot \|}_{Z^{r,p}_k(\R^2)}$ the respective norms.
We will often omit the domain when it is clear from context.

The above spaces play the following roles:
%
$X$ is the space where we measure the 
velocity field in the whole fluid domain,
%
%
while $Z$ is the space where we measure the boundary quantities $h$ and $\varphi$.

Besides these basic spaces, in due course we will also introduce 
other weighted spaces based on mixed $L^q_zL^p_x$ Lebesgue spaces in the flat domain; 
see for example those appearing in Proposition \ref{propalphaintro}.

\subsection{Initial data and main theorem}\label{ssecdata}

\subsubsection{Parameters: smallness and regularity}
Let
\begin{align}\label{param}
\e_1 \ll \e_0, 
  \qquad 0 < 3p_0 < \delta < 1/100, \qquad T_{\e_1} := \bar{c}\e_1^{-1+\delta},
\end{align}
for some sufficiently small absolute constant $\bar{c}>0$ (to be determined
in the course of the proof) and consider three (even) integers $N_0,N_1,N$ such that
\begin{align}\label{paramN}
N_0 \gg N_1 \geq \frac{N_0}{2}+10,  \qquad N := N_1 + 12.
\end{align}
These numbers are associated to various regularities and bounds for the
main unknowns in the problem:

\begin{itemize}

\item $N_0$ corresponds to the maximum number of derivatives and vector fields that we control on the velocity field
and on the height in $L^2$.

\item $N_1$ corresponds to the maximum number of derivatives and vector fields 
for which we prove the sharp decay rate of $(1+|t|)^{-1}$ in $L^\infty$
for the irrotational part of the velocity field and the height $h$.

\item $N$ corresponds to the maximum number of derivatives and vector fields
of the rotational components of the solution 
that we control (almost) uniformly by $\e_1$ on a time-scale of order (almost) $\e_1^{-1}$.

\end{itemize}

\subsubsection{Initial assumptions and main theorem}
We assume that the initial velocity and height satisfy
\begin{align}\label{init0}
& \sum_{r+k \leq N_0} {\| v_0 \|}_{X^r_k(\D_0)}
  + \sum_{r+k \leq N_0} {\| h_0 \|}_{Z^r_k(\R^2)} \leq \e_0.
\end{align}
For the vorticity, we assume that it satisfies the $L^p$-type bounds of high order
\begin{align}\label{init1'}
\sum_{|r|+|k| \leq N_0-20} 
  {\| \nabla_{x,y}^r \underline{\Gamma}^k \omega_0 \|}_{\W(\D_0)} & \leq \e_0, \qquad \W := L^2\cap L^{6/5}
\end{align}
and $L^p$-type bounds of smaller size $\e_1$ for lower order norms:
\begin{align}\label{init1}
\sum_{|r|+|k| \leq N} {\| \nabla_{x,y}^r \underline{\Gamma}^k \omega_0 \|}_{\W(\D_0)} & \leq \e_1.
\end{align}

\begin{remark}
If we define $W_0(x,z) 
= \omega_0(x,z+h(0,x))$, the transformed initial vorticity in the flat domain,
then \eqref{init1'}-\eqref{init1} imply the analogous bounds
\begin{align}\label{propWprdata0}
\sum_{|r|+|k| \leq N_0-20} {\| \nabla^r_{x,z} \underline{\Gamma}^k W_0 \|}_{\W(\R^2_x \times \{z<0\})} \leq C \e_0, 
\qquad
 \sum_{|r|+|k| \leq N} {\| \nabla^r_{x,z} \underline{\Gamma}^k W_0 \|}_{\W(\R^2_x \times \{z<0\})}  \leq C \e_1,
\end{align}
for some absolute constant $C>0$.
\end{remark}


We can now state a more precise version of our main result: 

\begin{theorem}\label{maintheoprecise}
Assume \eqref{init0}-\eqref{init1} and fix $\delta \in (0,1/100)$ and $3p_0 < \delta$.
Assume that $\omega_0$ vanishes on the boundary of\footnote{Since the boundary is a material
surface, this condition is preserved in time.} $\D_0$.
Then, there exists $\overline{\e_0}$ and $\bar{c}>0$ 
such that, for any $\e_1 \leq \e_0 \leq \overline{\e_0}$,
there exists a unique solution of \eqref{freebdy} with initial conditions $v(t=0)=v_0$
and $h(t=0) = h_0$ satisfying \eqref{init0}-\eqref{init1}, 
that remains regular for $|t| \leq T_{\e_1} = \bar{c} \e_1^{-1+\delta}$
and satisfies following: 
the $L^2$ bounds
\begin{align}\label{mtpre2}
& \sum_{r+k \leq N_0} {\| v(t) \|}_{X^r_k(\D_t)}
  + \sum_{r+k \leq N_0-1} {\| \omega(t) \|}_{X^r_k(\D_t)} \lesssim \e_0 \jt^{p_0},
\end{align}
and
\begin{align}\label{mtpre2'}
& \sum_{r+k \leq N_0} {\| h(t) \|}_{Z^r_k(\R^2)}
 \lesssim \e_0 \jt^{p_0},
\end{align}
and the decay bounds
\begin{align}\label{mtpreinfty}
& \sum_{r+k \leq N_1-5} {\| \nabla \nabla\psi 
  (t) \|}_{X^{r,\infty}_k(\D_t)} \lesssim \e_0 \jt^{-1},
\\
& \sum_{r+k \leq N_1-5} {\| \nabla v_\omega(t) \|}_{X^{r,\infty}_k(\D_t)} \lesssim \e_1 \jt^\delta,
\end{align}
and
\begin{align}\label{mtpreinfty'}
& \sum_{r+k \leq N_1} {\| h(t) \big\|}_{Z^{r,\infty}_{k}(\R^2)} \lesssim \e_0 \jt^{-1}.
\end{align}

\end{theorem}



 

\subsection{Main a priori assumptions}
In this subsection we list all the main a priori assumptions that we are going to make.
For convenience, some of these assumptions are stated in the domain $\D_t$, 
while others are stated in the flattened domain $\R^2_x \times \{z<0\}$ 
and some are in terms of the boundary variables.
Then, in Subsection \ref{ssecStra} we are going to explain how all these a priori assumptions 
are bootstrapped on an interval $[0,T]$ with $T\leq T_{\e_1}$, and also provide
some of the main elliptic-type bounds that are needed for the arguments.

- {\it A priori assumption in $\D_t$.}
We make the following a priori assumptions on the high-order energy ($L^2$-based) norms of the velocity,
vorticity, and height:
\begin{align}\label{apriorie0}
& \sup_{[0,T]} \,\jt^{-p_0} \Big(
  \sum_{r+k \leq N_0} {\| v(t) \|}_{X^r_k(\D_t)}
  + \sum_{r+k \leq N_0-1} {\| \omega(t) \|}_{X^r_k(\D_t)}
  + \sum_{r+k \leq N_0} {\| h(t) \|}_{Z^r_k(\R^2)} \Big) \leq 2c_E \e_0
\end{align}
where $c_E$ is an absolute constant to be chosen large enough.

\begin{remark}
Note how we let the highest order energy norms grow like $\jt^{p_0}$,
where $p_0$ is the parameter in \eqref{param}; this parameter can 
be chosen of the form $C \e_0$ for an absolute constant $C>0$.
We will however prove uniform bounds (almost) of $O(\e_1)$ on 
a lower number $N$ of derivatives and vector fields of the vorticity components,
essentially propagating the bound \eqref{init1}.
\end{remark}

\noindent
We also assume a priori decay bounds on the velocity in the interior:
\begin{align}\label{aprioridecayv}
\sum_{r+k \leq N_1-5} {\| 
  v(t) \|}_{X^{r,\infty}_k(\D_t)}
  \leq  2c_v \big( \e_0 \jt^{-1} + \e_1 \jt^\delta\big), \qquad t\in[0,T].
\end{align}
Note that we make decay assumptions (and prove decay bounds) on $v$ and not just on $\nabla v$,
which would be sufficient for the sole purpose of closing standard energy estimates in Sobolev spaces
without vector fields (see Proposition 2.5); these stronger bounds are also needed in other parts of the proof.

- {\it A priori assumptions on the boundary variables.}
We assume sharp pointwise decay bounds for the `boundary variables' $(h,\varphi)$:
\begin{align}\label{aprioridecay}
& \sup_{[0,T]} \, \jt \,
  \sum_{r+k \leq N_1} {\big\| u(t) \big\|}_{Z^{r,\infty}_{k}(\R^2)} \leq 2 c_B \e_0,
  \qquad u := h+i|\nabla|^{1/2}\varphi,
\end{align}
where $c_B$ is an absolute constant to be chosen large enough.

- {\it A priori assumptions on the vorticity in the flat domain.}
Some of the main parts of our argument are performed in the flattened domain $\R^2_x \times \{z<0\}$.
We denote the vorticity in the flattened coordinates as
\begin{align}\label{defW0}
W(t,x,z) := \omega(t,x,z+h(t,x)), \qquad W_0(x,z) := \omega_0(x,z+h_0(x)),
\end{align}
and we will bootstrap three main a priori bounds on it.
For this purpose we introduce the weighted Lebesgue spaces $\mX^n$ defined by the norm
(see \eqref{init1'})
\begin{align}\label{omegaflatspace0}
\begin{split}
{\| f \|}_{\mX^n} := \sum_{|r|+|k| \leq n} {\big\| \underline{\Gamma}^k \nabla^r_{x,z} \,f 
  \big\|}_{\W(\R^2_x \times \{z<0\})}, \qquad \W = L^2\cap L^{6/5}.
\end{split}
\end{align} 

The first two main a priori bounds on $W$ are 
\begin{alignat}{2}
\label{aprioriWL}
& {\| \partial_t^j W(t) \|}_{\mX^{N_1-10-j}} \leq 2c_{L} \e_1, 
&&\qquad t\in[0,T], 
\quad j=0,1,
\\
\label{aprioriWH}
& {\| \partial_t^j W(t) \|}_{\mX^{N_1+12-j}} \leq 2c_{H} \e_0^j \e_1 \jt^{\delta}, 
&&\qquad t\in[0,T], \quad j=0,1,
\end{alignat}
where $c_L < c_H$ are some absolute constants to be chosen large enough 
(we use the same one for $j=0$ or $1$).
In \eqref{aprioriWH}, the growth rate $\delta$ is the parameter in \eqref{param}.
We also assume a high-order (weak) bound
\begin{align}
\label{aprioriW0}
& {\| W(t) \|}_{\mX^{N_0-20}} \leq 2c_{W} \e_0 \jt^{2p_0}, \quad t\in[0,T].
\end{align}

\begin{remark}
Note how we are propagating bounds for $W$ (hence for the vorticity $\omega$) 
of the order $\e_1$ with a small growth factor 
at a level of vector fields larger than $N_1$; we choose $N=N_1+12$ for concreteness.
Along with this, 
we also bootstrap a lower norm with the sharp bound of $\e_1$.  
The need to proceed with this two tier bootstrap
is again attributable to the growth of the highest order weighted energies.




\end{remark}

We now explain our overall strategy for recovering these assumptions and obtaining Theorem \ref{maintheo}.

\subsection{Strategy of the proof and main propositions}\label{ssecStra}
The proof of our Theorem \ref{maintheo} proceeds in a several steps based on some key propositions.
Note that the order in which the various intermediate results are presented here
is not the same as that of the sections in which the proofs are given, 
but follows what we believe to be a more reader-friendly description.

\subsection*{Step 1: Energy estimates and other high-order norms} 

We begin with an energy estimate that controls the increment of the top-order weighted norms.

\begin{prop}[Top order energy inequality]\label{propEv}
Assume that \eqref{apriorie0} 
holds and recall the definition of the spaces \eqref{defX} and \eqref{defZ}.
Then there exist energy functionals $\E_{r,k}$ such that:

\begin{itemize}

\item We have 
\begin{align}\label{propEveq}
\begin{split}
& \E_{r,k}(t) \approx {\| v(t) \|}_{X^r_k(\D_t)}^2 
  + {\| \omega(t) \|}_{X^{r-1}_k(\D_t)}^2
  + {\| h(t) \|}_{Z^r_k(\R^2)}^2,
\end{split}
\end{align}

\item If we define
\begin{align}\label{propEveq0}
& \E_0(t) := \sum_{r+k \leq N_0} \E_{r,k}(t)
  \approx \sum_{r+k \leq N_0} {\| v(t) \|}_{X^r_k(\D_t)}^2 
  + {\| \omega(t) \|}_{X^{r-1}_k(\D_t)}^2 + {\| h(t) \|}_{Z^r_k(\R^2)}^2,
\end{align}
then, for all $t\in[0,T]$,
\begin{align}\label{propEvEE}
\frac{d}{dt} \E_0(t) 
  \lesssim Z_0(t) \cdot \E_0(t) 
\end{align}
where
\begin{align}\label{propEvdec}
Z_0(t) := \sum_{r+k \leq N_0/2 + 4} {\| v(t) \|}_{X^{r,\infty}_k(\D_t)} 
  + {\| h(t) \|}_{Z^{r+2,\infty}_k(\R^2)}.
\end{align}
\end{itemize}

\end{prop}

Note that the initial assumptions \eqref{init0}-\eqref{init1} imply
\begin{align}\label{init10}
\E_0(0) \lesssim \e_0.
\end{align}
$L^2$-based energy estimates are a fairly standard result for this problem,
see for example \cite{CL,GMS2,ShZ2,ShZ3,Wu2} for energy estimates in standard Sobolev spaces 
without vector fields. 
Estimates with vector fields are also essentially standard although, to the best of our knowledge, 
the estimates in Proposition \ref{propEv} do not appear in the literature exactly as stated. 
In the irrotational setting 
\cite{WuAG,Wu3DWW} prove estimates with vector fields for gravity waves,
\cite{GMSC} proves estimates for the problem with surface tension and no gravity,
and \cite{DIPP} proves estimates for the gravity-capillary problem using only the rotation vectorfield
(since the problem is not scaling invariant); 
energy estimates including the scaling vector field 
are also proved in some lower dimensional cases \cite{IoPu2,IoPu4}. 
In section \ref{secEv} we give a brief sketch
of the main ingredients needed in order to carry out the proof of the energy estimate
with vector fields in our setting.

As a consequence of the main energy inequality 
we obtain the following standard result: 

\begin{prop}[Decay implies Energy bootstrap]\label{propEboot}
Assume 
\eqref{init0}, and that, 
for $T \leq T_{\e_1}$, the a priori decay assumptions \eqref{aprioridecay}-\eqref{aprioridecayv} hold.
Then, there exists $c_E$ large enough such that
\begin{align}\label{apriorie0'}
& \sup_{[0,T]} \,\jt^{-p_0} \Big(
  \sum_{r+k \leq N_0} {\| v(t) \|}_{X^r_k(\D_t)}
  + \sum_{r+k \leq N_0-1} {\| \omega(t) \|}_{X^r_k(\D_t)}
  + \sum_{r+k \leq N_0} {\| h(t) \|}_{Z^r_k(\R^2)} \Big) \leq c_E \e_0.
\end{align}
\end{prop}

\begin{proof}[Proof of Proposition \ref{propEboot}]
The a priori assumptions \eqref{aprioridecayv}-\eqref{aprioridecay} directly imply that
\begin{align*}
Z_0(t) \lesssim \e_0 \jt^{-1} + \e_1 \jt^\delta.
\end{align*}
This and \eqref{propEvEE},  together with \eqref{propEveq0} and \eqref{init10}, 
give
\begin{align*}
\begin{split}
\E_0(t) & \leq C \E_0(0) \exp \Big( C \int_0^t \big( \e_0\js^{-1} + \e_1 \js^\delta \big) \, ds \Big)
	\leq C \e_0^2 \jt^{C\e_0} 
\end{split}
\end{align*}
having used the definition of $T_{\e_1}$ from \eqref{param} to bound uniformly the time integral of $\e_1 \js^\delta$.
This implies \eqref{apriorie0'} provided the absolute constant $c_E$ is chosen large enough.
\end{proof}

Proposition \ref{propEboot} closes the bootstrap for the norm in \eqref{apriorie0}.
The main efforts in our proof are then dedicated to bootstraping the a priori decay bounds 
\eqref{aprioridecay} and \eqref{aprioridecayv}.
%
Before moving on to explain how to obtain these, 
we give the bootstrap for the control of the high-order norm 
of $W$, see \eqref{aprioriW0}, and how this is used to bound $|\nabla|^{1/2}\varphi$ in the next two propositions.

\begin{prop}\label{mainpropW0}
Let $W$ be defined as in \eqref{defW0} and let $\mX^n$ be the space defined in \eqref{omegaflatspace0}.
Under the assumptions \eqref{aprioriWL} and \eqref{aprioriW0} on $W$,
the decay assumptions \eqref{aprioridecayv} and \eqref{aprioridecay} on $v$ and $h$, 
and the a priori energy bound \eqref{apriorie0},
we have, for all $t\in[0,T]$, 
\begin{align}
\label{propWconchigh'0}
& {\| W(t) \|}_{\mX^{N_0-20}} \leq c_{W} \e_0 \jt^{2p_0}.
\end{align}
\end{prop}

The proof of Proposition \ref{mainpropW0} is given in Section \ref{secvelpot}
(see Proposition \ref{mainpropW'}).
Using Proposition \ref{mainpropW0} we can obtain bounds on
the vector potential in the flat domain
\begin{align}\label{defVo0}
V_\omega(t,x,z) = v_\omega(t,x,z+h(t,x));
\end{align}
this can be done in appropriate spaces via elliptic estimates for $\alpha$ such that
$\curl \alpha \approx V_\omega$; see \eqref{betaeq} and \eqref{betaflat} for the exact definition.
Using $\nabla \psi = v - v_\omega$ and basic trace estimates we can obtain the following:

\begin{prop}[Bounds on the velocity potential]\label{propvelpot}
Under the a priori assumptions \eqref{aprioriW0} and \eqref{aprioriWL} on $W$, 
and the a priori assumptions \eqref{apriorie0} and \eqref{aprioridecay} on $h$ and $v$, 
it holds
\begin{align}\label{velpotestPsi}
& \sup_{[0,T]} \,\jt^{-3p_0} \sum_{r+k \leq N_0-20} 
  {\big\| \underline{\Gamma}^k \nabla_{x,z}^r \nabla_{x,z} \Psi(t) \big\|}_{L^2(\R^2_x \times \{z<0\})} 
  \leq c_P' \e_0, 
\end{align} 
for some suitably large absolute constant $c_P' > c_E + c_{W}$.
In particular,
\begin{align}\label{velpotest}
& \sup_{[0,T]} \,\jt^{-3p_0} \sum_{r+k \leq N_0-20} 
  {\big\| |\nabla|^{1/2} \varphi (t) \big\|}_{Z^r_k(\R^2)} \leq c_P \e_0. 
\end{align} 
for some suitably large absolute constant $c_P > c_E + c_{W}$.
\end{prop}

\begin{remark}\label{remvelpot}
The choice of the growth rate $\jt^{3p_0}$ in \eqref{velpotestPsi} is 
dictated by the nature of the argument that we use;
this necessitates changing variables from the moving domain to the flat one
and, therefore, taking into account  the growth of the high-order weighted norm of $h$.
This growth could be avoided by bootstrapping additional low norms of the irrotational components. 

Also note that the bound in Proposition \ref{propvelpot} 
for the simple $L^2$ norm of $\nabla\Psi$ (or $|\nabla|^{1/2}\varphi$)
is a direct consequence of the Hodge decomposition.
However, the bound with vector fields requires some non-trivial arguments,
including Proposition \ref{mainpropW0} and elliptic type estimate 
similar to those in Section \ref{secVP} (see also Proposition \ref{propalphaintro} below).  
\end{remark}

We give details for the proof of Proposition \ref{propvelpot} in Section \ref{secvelpot}.

\subsection*{Step 2: Decay estimates on the boundary and in the interior}
Our next main step is the proof of decay for the boundary dispersive variable $u = h + i |\nabla|^{1/2}\varphi$.
First, we derive an equation for $u$ by adapting the classical Zakharov-Craig-Schanz-Sulem formulation;
see also \cite{Gin1} and \cite{CaLa}.
More precisely, in Lemma \ref{lemequ} we obtain that
\begin{align}\label{der30intro}
\partial_t u + i |\nabla|^{1/2} u & = B_0 - i |\nabla|^{-3/2} \nabla \cdot \partial_t P_\omega + \cdots
\end{align}
where $B_0$ denotes quadratic terms in $h$ and $\varphi$ and
$P_\omega$ is the restriction to the interface of the horizontal components of $v_\omega$:
\begin{align}\label{Pomegaintro}
P_\omega(t,x) := \big((v_\omega)_1, (v_\omega)_2\big)(t,x,h(t,x)) = \big((V_\omega)_1, (V_\omega)_2\big)(t,x,0).
\end{align}
The ``$\cdots$'' in \eqref{der30intro} denote other quadratic terms that involve at least one $P_\omega$,
plus other cubic terms, and we disregard them here for the sake of the discussion.
Note that \eqref{der30intro} is forced in a singular way by $\partial_t P_\omega$;
this creates some technical difficulties.
See Section \ref{appder} for the derivation of \eqref{der30intro}.

One can see that in order to prove decay for $u$
through the Duhamel's formula associated to \eqref{der30}
we need, among other things, strong enough control on 
$P_\omega$, at a level of (weighted) regularity which is higher than that 
of the space in which $u$ decays (see \eqref{aprioridecay}). 
We will obtain these estimates on $P_\omega$ in the next step.

Based on \eqref{der30intro} and suitable assumptions on $P_\omega$,
we recover decay for $u$: 

\begin{prop}[Sharp decay of the irrotational component]\label{propdecayirr}
Assume a priori that 
\eqref{aprioridecay} holds, together with 
\begin{align}\label{decayirrasu}
& \sup_{[0,T]} \,\jt^{-3p_0} \sum_{r+k \leq N_0-20} {\| u(t) \|}_{Z^r_k(\R^2)} \lesssim \e_0. 
\end{align} 
Moreover, assume that for some $N \geq N_1+11$ we have, for all $t \leq T_{\e_1}$,
\begin{align}\label{decayirrasVP}
\sum_{r + k \leq N } {\big\| \partial_t^j P_\omega(t) 
  \big\|}_{Z^r_k(\R^2)} \leq c'_{H} \e_1 \e_0^j \jt^\delta, 
  \quad j=0,1.
\end{align} 
Then, for all $t\leq T$, 
\begin{align}\label{decayirrconc}
\sum_{r+k \leq N_1} {\| u (t) \|}_{Z^{r,\infty}_k(\R^2)} \leq c_B \e_0 \jt^{-1}. 
\end{align}
for some large enough absolute constant $c_B > c''_{H}$.
Here $c''_{H} > c_H'$ is the constant in \eqref{lembounda-Voconc0}.
\end{prop}

\begin{remark}\label{Rempropdecayirr}
In our proof of Proposition \ref{propdecayirr} in Section \ref{secdecay},
we will actually show a slightly stronger bound than \eqref{decayirrconc},
with an $\ell^1$ sum over frequencies:
\begin{align}\label{decayirrconcl}
\sum_{\ell \in \Z} \sum_{r+k \leq N_1} {\| P_\ell u (t) \|}_{Z^{r,\infty}_k(\R^2)} \leq c_B \e_0 \jt^{-1};
\end{align}
see Remark \ref{Remdecay}.
This technical improvement helps to show decay for $\nabla\psi$; see Lemma \ref{lemdecayirrot}.
\end{remark}

Note that the assumption \eqref{decayirrasu} is directly guaranteed by \eqref{propvelpot}.
Proposition \ref{propdecayirr} then recovers the a priori decay assumption \eqref{aprioridecay}
closing the bootstrap provided $c_B$ is large enough.
The proof of Proposition \ref{propdecayirr} is given in Section \ref{secdecay} and uses:

\begin{itemize}

\item[-] The boundary evolution equations in the presence of vorticity (see Lemma \ref{lemeqo});

\item[-] A Klainerman-Sobolev type estimate for the flow of $e^{it|\nabla|^{1/2}}$ (see Lemma \ref{lemlinear});

\item[-] Normal form arguments to deal with the quadratic irrotational terms that have slow decay;

\item[-] The estimate \eqref{decayirrasVP}
to bound the nonlinear and forcing terms involving the rotational component $P_\omega$;
particular care needs to be put into handling small frequencies here, due to the singular nature 
of the operator acting on $P_\omega$ in \eqref{der30intro}.

\end{itemize}

The assumption \eqref{decayirrasVP} will follow from a fixed point argument
which essentially constructs and bounds $V_\omega$. 
See in particular the conclusion of Lemmas \ref{lembounda-Vo} and \ref{lemboundaVbulk}.



Before moving on to the next main step in the proof,
we add here the estimates that recover the bootstrap assumption \eqref{aprioridecayv}
on the decay of the velocity in the interior. 
These are obtained in an elliptic way at fixed time $t$ from other bounds that are bootstrapped.
For convenience we split these estimate into two lemmas:

\begin{lemma}[Decay of the irrotational component]\label{lemdecayirrot}
Assume that \eqref{decayirrconcl} 
holds, together with the a priori bounds \eqref{apriorie0} and \eqref{aprioridecay} on $h$.
Then, for all $t \leq T$ we have
\begin{align}\label{decayirrotconc}
\sum_{r+k \leq N_1-5} {\| \nabla \psi(t)  \|}_{X^{r,\infty}_k(\D_t)} \leq c_i \e_0 \jt^{-1}.
\end{align}
for some $c_i > c_B$.
\end{lemma}


\begin{lemma}[Decay of the rotational component]\label{lemdecayrot}
Assume that 
the bounds \eqref{boundaVbulk0} on $V_\omega$ hold for $t \in [0,T]$, with $T\leq T_{\e_1}$. 
Then, for all $t\leq T$, we have
\begin{align}\label{decayrotconc}
\sum_{r+k \leq N_1-5} {\| v_\omega(t) \|}_{X^{r,\infty}_k(\D_t)} \leq c_r \e_1 \jt^\delta,
\end{align}
for some $c_r > c_H'$.
\end{lemma}

Since $v = \nabla \psi + v_\omega$,
Lemmas \ref{lemdecayirrot} and \ref{lemdecayrot} recover the bootstrap assumption \eqref{aprioridecayv}
provided $c_v > c_i + c_r$ is chosen large enough.
The proofs of these lemmas are given in Subsection \ref{ssecdecay}.

\subsection*{Step 3: Estimates for the rotational part of the velocity}
In Section \ref{secVP} we prove estimates for $v_\omega$ which in particular imply 
the assumption \eqref{decayirrasVP} used in Proposition \ref{propdecayirr}.
We first consider the vector potential $\beta$ such that $-\curl\beta = v_\omega$.
$\beta$ satisfies an elliptic system, $\Delta \beta = \omega$, with mixed Dirichlet and Neumann boundary
conditions; see Lemma \ref{lembetaeq}.
We then work in the flat half-space by considering
$\alpha(t,x,z) := \beta(t,x,z+h(t,x))$, and $V_\omega$ as in \eqref{defVo0}.
Then, $\alpha$ satisfies an elliptic system with $h$-dependent coefficients,
which is forced by the vorticity $W$. 
Assuming suitable bounds on the forcing $W$, a fixed point argument,
which relies on (weighted) estimates for the Poisson kernel, 
gives estimates for $\alpha$ in weighted $L^p_zL^q_x$ spaces.
From the bounds obtained on $\alpha$ we can then directly deduce estimates for $V_\omega$ in similar spaces.
The main steps are contained in Proposition \ref{alphaprop}, which constructs and bounds $\alpha$,
and its direct consequence Lemma \ref{lemboundaVbulk}, which gives bounds for $V_\omega$. 
We summarize these results in the following statement:

\begin{prop}[Bounds for $V_\omega$
]\label{propalphaintro}
Assume that \eqref{apriorie0}-\eqref{aprioridecay} hold, 
and that $W$ is given so that \eqref{aprioriWL} and \eqref{aprioriWH} hold, that is, 
for all $t\in[0,T]$, and for $j=0,1$
\begin{align}\label{alphaintroas}
\begin{split}
& {\| \partial_t^j W(t) \|}_{\mX^{N_1-10-j}} \leq c_L \e_1,
\qquad 
{\| \partial_t^j W(t) \|}_{\mX^{N_1+12-j}} \leq c_H \e_0^j \e_1 \jt^{\delta},
\end{split}
\end{align}
where $\mX^n=\mX^n(\R^3_-)$ is the space defined in \eqref{omegaflatspace0}.

Then, for all $t\in[0,T]$, and for $j=0,1$, we have the bounds
\begin{align}\label{boundaVbulk0}
\begin{split}
{\| \pa_t^j V_\omega(t) \|}_{\Yn^{N_1 - 10 - j}} & \leq c_H' \e_1,
\\
{\| \pa_t^j V_\omega(t) \|}_{\Yn^{N_1 + 12 - j}} & \leq c_H' \e_1 \e_0^j \jt^\delta,
\end{split}
\end{align}
for some large enough absolute constant $c_H' > c_H$,
where $\Yn^n=Y^n(\R^3_-)$ is the space defined by the norm 
\begin{align*}
& {\| g \|}_{\Yn^{n}} = \sum_{|r|+|k| \leq n} {\big\| \underline{\Gamma}^k\nabla^r_{x, z} g \big\|}_{\Yn^0},
\\
& {\| g \|}_{\Yn^0} :=
   {\big\||\nabla|^{1/2} g \big\|}_{L^\infty_z L^2_x} + {\|  g \|}_{L^\infty_z L^2_x}
   + {\big\| \nabla_{x,z} g \big\|}_{L^2_z L^2_x}.
\end{align*}

In particular, we also have, for all $t \in [0,T_{\e_1}]$,
and some $c_H'' > c_H'$,
\begin{align}\label{lembounda-Voconc0}
\sum_{r + k \leq N_1+12-j} {\big\| \partial_t^j V_\omega(t,\cdot,0) \big\|}_{Z^r_k(\R^2)} 
  \leq c_H'' \e_1 \e_0^j \jt^\delta, \qquad j=0,1.
\end{align}
\end{prop}

Note that the bound \eqref{lembounda-Voconc0} for the restriction of $V_\omega$ to $z=0$
follows from \eqref{boundaVbulk0} and the definitions of the spaces $Y^n$ and $Z^r_k$,
and implies the validity of the assumption \eqref{decayirrasVP}.

To conclude the proof of our result, we then only need to prove
that the bounds in \eqref{alphaintroas} hold true,
that is, we need to close the bootstrap for the a priori assumptions \eqref{aprioriWL}
and \eqref{aprioriWH}. This is done in the last main step.

\subsection*{Step 4: Estimates for the vorticity}
Our last main step is to bootstrap the estimates \eqref{aprioriWL}-\eqref{aprioriWH}. 
The main point is to obtain estimates of size essentially $\e_1$, comparable to the size of the initial
vorticity. While this is a natural bound to expect it is not straightforward to obtain, as we will explain below.
Notice that the basic energy estimate \eqref{apriorie0'} only gives a bound
on $\omega$ of the order $O(\e_0)$.
This is the main result:

\begin{prop}\label{mainpropWintro}
Assume that the initial conditions \eqref{init1} holds,
and the a priori assumptions \eqref{apriorie0}-\eqref{aprioridecay} hold.
Let $W$ be as defined as above and $\mX^n$ as in \eqref{omegaflatspace0}.
If \eqref{aprioriWL}-\eqref{aprioriWH} hold, that is,
for all $t\in[0,T]$, and for $j=0,1$ 
\begin{align}\label{propWasi}
\begin{split}
& {\| \partial_t^j W(t) \|}_{\mX^{N_1-10-j}} \leq 2c_L \e_1,
\\
& {\| \partial_t^j W(t) \|}_{\mX^{N_1+12-j}} \leq 2c_H \e_0^j \e_1 \jt^{\delta},
\end{split}
\end{align}
then, for all $t\in[0,T]$,
\begin{align}\label{propWconci}
\begin{split}
& {\| \partial_t^j W(t) \|}_{\mX^{N_1-10-j}} \leq c_L \e_1,
\\
& {\| \partial_t^j W(t) \|}_{\mX^{N_1+12-j}} \leq c_H \e_0^j \e_1 \jt^{\delta}.
\end{split}
\end{align}
\end{prop}

The above statement is essentially Proposition \ref{mainpropW}, and its proof occupies all of Section \ref{secVorticity}.
The estimates use crucially the renormalization in Proposition \ref{lemWreno}.

Let us give a few details of the proof of Proposition \ref{mainpropWintro}.
For convenience and ease of cross-reference, we work with the quantities defined in the flat 
coordinates, and let $W$ and $V_\omega$ be as above, with $V$ and $\Psi$ denoting the velocity and velocity potential.
$W$ and $V_\omega$ are the `rotational' components, while $V-V_\omega$ and $\nabla \Psi$ 
(these two only differ by a quadratic term) are the `irrotational' components.
The rotational components together with the height $h$ are also denoted as `dispersive' variables.

The vorticity equation reads (see \eqref{Weq0}-\eqref{Vi})
\begin{align}\label{Weq0intro}
\begin{split}
& {\bf D}_t W  = W \cdot \nabla V + \cdots
\qquad {\bf D}_t := \partial_t + U \cdot \nabla, 
  \qquad U := \nabla\Psi + V_\omega + \cdots 
\end{split}
\end{align}
where we are denoting with ``$\cdots$'' lower order perturbative terms.
Let us simplify \eqref{Weq0intro} further by replacing the material derivative just by $\partial_t$,
but notice that this cannot be done by trivially integrating the Lagrangian 
flow since $U$ is not in $L^1_t([0,T])$
due to the non-integrable time-decay of $\nabla\Psi$.
We then arrive at the model equation
\begin{align}\label{Weq1intro}
& \partial_t W  = W \cdot \nabla V_\omega + W \cdot \nabla \nabla\Psi.
\end{align}
One can see that the part of the quadratic stretching term that only involves the rotational variables 
$W$ and $V_\omega$ is naturally of size $\e_1^2$ (up to small time growing factors), and therefore
is consistent with a bound of order $\e_1$ for $W$ on a time scale of order $O(\e_1^{-1})$.
However, since $|\nabla \Psi| \approx \e_0 \jt^{-1}$ (at best) the last term in \eqref{Weq1intro} 
acts as a long-range perturbation
and does not allow us to propagate bounds on $W$.
Although the loss appears to be only logarithmic here, this type of difficulty 
is a well-known issue when dealing with long-term regularity for nonlinear PDE.

Before explaining our ideas to resolve the above issue,
let us also mention that the situation becomes even more delicate when looking at weighted norms of $W$
as in \eqref{alphaintroas}.
Applying vector fields to \eqref{Weq0intro}, and just concentrating on the rotational-irrotational coupling,
we obtain, schematically,
\begin{align}\label{Weq1introvfs}
\partial_t (\underline{\Gamma}^k W) = (\underline{\Gamma}^k W) \cdot \nabla \nabla\Psi 
  + W \cdot (\underline{\Gamma}^k \nabla \nabla\Psi) + \cdots.
\end{align}
Recall that the evolution of $\Psi$ is forced by the (time derivative) of restriction of $V_\omega$, 
see \eqref{der30intro}-\eqref{Pomegaintro},
and that (in terms of norms) we can think that $\nabla V_\omega \approx W$.
Therefore, one should expect that $k$ vector fields applied to $\nabla \Psi$ 
(or, equivalently, of $|\nabla|^{1/2}\varphi = \Im u$, see \eqref{aprioridecay})
will decay at the sharp rate only provided that strictly more than $k$ vector fields 
of $V_\omega$ are suitably under control.
But this is then inconsistent with the equation \eqref{Weq1introvfs}
where (small) polynomial losses will occur when relying on higher order weighted $L^2$ 
norms of $\nabla \Psi$.

To resolve the issues discussed above, we introduce a ``modified vorticity'' 
which satisfies a better equation than \eqref{Weq1intro} 
where the irrotational components only appear with quadratic or higher homogeneity.
This renormalization procedure can be thought of as a normal form for the vorticity equation
in the three dimensional domain.
The main observation is that, up to perturbative quadratic terms,
\begin{align}
\Psi \approx e^{z|\nabla_x|} \partial_t |\nabla_x|^{-1} h 
\end{align}
and, therefore, $\nabla \Psi$ is approximately the time derivative of a time-decaying component.
Then, the modified vorticity defined by $W - W \cdot \nabla \nabla e^{z|\nabla_x|} |\nabla_x|^{-1} h$
satisfies an equation with truly perturbative nonlinear terms,
and can be used to obtain \eqref{propWconci}. 

%
%

%
%


%

\subsection{Notation}\label{Notation}
Here we give some notation used in the paper. More notation will be introduced in the course of the proofs.

\noindent
- We use standard notations for $L^p$ spaces and Sobolev spaces $W^{s,p}$ and $\dot{W}^{s,p}$, with $H^s = W^{s,2}$.

\noindent
- With $p-$ we denote a number smaller than, but arbitrarily close to, $p$.
$\infty-$ denotes an arbitrarily large number.
Similarly, $p+$ denotes a number larger than, but arbitrarily close to, $p$.

\noindent
- We use $C$ to denote absolute constants; these may vary from line to line of a chain of inequalities,
and may depend on the numbers in \eqref{paramN},
but are independent of the relevant quantities involved, and of $\e_0$ and $\e_1$.

\noindent
- $A \lesssim B$ means that there exists an absolute constant $C>0$ such that $A \leq C B$.
Similarly $A \gtrsim B$ means $B \lesssim A$. $A \approx B$ means $A \lesssim B$ and $B \lesssim A$.

\noindent
- We denote the Fourier transform over $\R^2$ by
\begin{align}
\widehat{f}(\xi) = \mathcal{F}(f) (\xi) := \frac{1}{2\pi}\int_{\R^2} e^{-i x\cdot \xi} f(x) \, dx,
\qquad
\mathcal{F}^{-1}(f) (x) = \frac{1}{2\pi}\int_{\R^2} e^{i x\cdot \xi} f(\xi) \, d\xi.
\end{align}

\noindent
- To define frequency decomposition we fix a smooth even cutoff function  $\varphi: \R \to [0,1]$
supported in $[-8/5,8/5]$ and equal to $1$ on $[-5/4,5/4]$.
By slightly abusing notation we identify $\varphi(x)$ with its radial extension $\varphi(|x|)$, $x\in\R^d$.
For $k \in \Z$ we define $\varphi_k(x) := \varphi(2^{-k}x) - \varphi(2^{-k+1}x)$,
so that the family $(\varphi_k)_{k \in\Z}$ forms a partition of unity,
\begin{equation*}
 \sum_{k\in\Z}\varphi_k(\xi)=1, \quad \xi \in\R^d \smallsetminus \{ 0 \}.
\end{equation*}
We let
\begin{align}\label{cut0}
\varphi_{I}(x) := \sum_{k \in I \cap \Z}\varphi_k, \quad \text{for any} \quad I \subset \R, \quad
\varphi_{\leq a}(x) := \varphi_{(-\infty,a]}(x), \quad \varphi_{> a}(x) = \varphi_{(a,\infty)}(x),
\end{align}
with similar definitions for $\varphi_{< a},\varphi_{\geq a}$.
We will also denote $\varphi_{\sim k}$ a generic smooth cutoff function
that is supported around $|\xi| \approx 2^k$, e.g. $\varphi_{[k-2,k+2]}$ or $\varphi'_k$.
We denote by $P_k$, $k\in \Z$, the Littlewood-Paley projections defined by
\begin{equation}\label{LP0}
\what{P_k f}(\xi) = \varphi_k(\xi) \what{f}(\xi),
  \quad \what{P_{\leq k} f}(\xi) = \varphi_{\leq k}(\xi) \what{f}(\xi),
  \quad \what{P_{\sim k} f}(\xi) = \varphi_{\sim k}(\xi) \what{f}(\xi),
  \quad \textrm{ etc.}
\end{equation}
Note that these projections essentially commute with the vector fields:
\begin{align}\label{Lpcomm}
[\Omega,P_k] = 0, \qquad [S,P_k] = P_{\sim k}.
\end{align}

\noindent
- We denote by $\mathbf{1}_{\pm}$ the characteristic function of $\{ \pm x > 0\}$.


\bigskip
\section{Estimates for the vector potential}\label{secVP}
In this section we establish bounds on the
rotational part of the velocity in the flat domain that, recall, is denoted by
\begin{align}\label{VP00}
V_\omega(t,x,z) = v_\omega(t,x, z+ h(x)),
\end{align}
under some assumption on the vorticity in the flat domain, that is, $W(t,x,z) = \omega(t,x, z+ h(x))$.
In particular, we will prove Proposition \ref{propalphaintro} as a combination of
Proposition \ref{alphaprop} and Lemma \ref{lemboundaVbulk}.



%

%

%

%


\subsection{Preliminaries and flattening of the domain}
We start by relating $v_\omega$ to $\beta$ such that $v_\omega =\curl \beta$,
and reduce matters to estimates for a suitable elliptic system.
This Hodge-type decomposition is fairly standard 
but we provide some details for completeness
and since we are going to need explicit formulas for our estimates; see Appendix \ref{Appalpha}.

\begin{lemma}[Elliptic problem]\label{lembetaeq}
Let $\omega = \curl v$ and $v_\omega$ be defined as above, then we can write
\begin{align}\label{betaeq0}
v_\omega = - \curl \beta
\end{align}
where $\beta$ is the unique solution which decays at infinity
of the system
\begin{subequations}\label{betaeq}
\begin{align}
\label{betaell}
\Delta \beta & = \omega, \qquad \text{ in } \D_t,
\\
\label{betapi}
\Pi \beta &= 0, \qquad \text{ on } \pa \D_t,
\\
\n^i \partial_i (\n^j\beta_j) + H \beta_i \n^i 
  &=0, \qquad \text{ on } \pa \D_t.
\label{betan}
\end{align}
\end{subequations}
Here $H:= \partial_j \n^j$ and $\Pi$ denotes the projection to the tangential components
$\Pi_i^j := \delta^j_i - \n^j\n_i$, with $n$ the outward-pointing unit normal.
\end{lemma}

\begin{proof}
Recall that we define the (irrotational) velocity potential $\psi$ by solving the Neumann problem
\begin{align*}
\begin{array}{rcll}
\Delta \psi & \!\! = \!\!& 0, & \qquad \text{ in } \D_t,
\\
\pa_y \psi - \nabla \psi \cdot \nabla h & \!\! = \!\!& \pa_t h, & \qquad \text{ on } \pa\D_t.
\end{array}
\end{align*}
Since $v_\omega = v -\nabla \psi$
we have the system
\begin{align}\label{vomegaeqs}
\begin{split}
\div v_\omega & = 0, \qquad \text{ in } \D_t,
\\
\curl v_\omega & = \omega, \qquad \text{ in } \D_t,
\\
v_\omega \cdot n & = 0, \qquad \text{ on } \partial \D_t,
\end{split}
\end{align}
where $n = (1+|\nabla h|^2)^{-1/2} (-\nabla h,1)$ denotes the outward unit normal.
In what follows we denote with the same symbol $n$ a regular extension of
the unit normal vector field $n$ defined in a neighborhood of $\pa \D_t$
and such that $|n|=1$ close to $\pa \D_t$.

Note that any $v_\omega$ decaying at infinity that solves this system is unique
since the difference of any two solutions of \eqref{vomegaeqs} is the gradient of an
harmonic function with homogeneous Neumann data, and thus is constant.
Then, letting $\beta$ solve \eqref{betaeq}, we set $w := - \curl \beta$ and want to show that
$w$ satisfies \eqref{vomegaeqs}.

We have $\div w=0$ and $\curl w = \Delta \beta - \nabla (\div \beta) = \omega - \nabla (\div \beta)$.
Observe that $\Delta (\div \beta) = 0$ by \eqref{betaell}.
Decomposing into tangential and normal components,
\fp{
using $v_i = \Pi_i^j v_j + n_i v_n$, and $\pa_i = \Pi_i^j \pa_j + n_i \pa_n$
where we are denoting $\beta_n = \beta \cdot n$ and $\partial_n = \n \cdot \nabla$, we have, on the surface,
\begin{align*}
\partial_i \beta^i & =  \Pi_i^j \partial_j ( \Pi^i_k \beta^k + n^i \beta_n) + n_i \partial_n \beta^i 
 \\
 & = \Pi^j_i \partial_j \Pi^i_k \beta^k + (\Pi^j_i \partial_j n^i) \beta_n + n^i \Pi^j_i \partial_j \beta_n
 + \partial_n \beta_n - (\partial_n n_i) \beta^i.
\end{align*}
The first term in the expression above vanishes since it is the tangential divergence 
of $\Pi\beta$, which is zero on the boundary by assumption;
the second term satisfies 
$$(\Pi^j_i \partial_j n^i) \beta_n = (\partial_i n^i - n_in^j\partial_j n^i) \beta_n = H \beta_n,$$
since $|n|=1$ in a neighborhood of the surface;
the third term vanishes since $n^i \Pi^j_i \equiv 0$;
the last term also vanishes since, using again the boundary conditions and $|n|=1$,
$$(\partial_n n_i) \beta_i = (\partial_n n_i) n_i n^j \beta_j = 0.$$}
We eventually obtain
\begin{align}\label{divS}
\left. \partial_i \beta^i \right|_{\partial \D_t} & = \partial_n \beta_n + H \beta_n.
\end{align}

Therefore, in view of \eqref{betan}, we have $\div \beta = 0$ on $\partial \D_t$
and we can deduce that $\div \beta = 0$ in $\D_t$ so that $\curl w = \Delta \beta = \omega$.
For the last boundary condition in \eqref{vomegaeqs} 
we see that since
\begin{align}\label{curl.n}
\curl z \cdot n = \div ( \Pi z \times n),
\end{align}
using again \eqref{betapi} we get $w\cdot n = -\curl \beta \cdot n = 0$.
Therefore $w$ solves \eqref{vomegaeqs}
and \eqref{betaeq0} follows by uniqueness with $\beta$ solving \eqref{betaeq} as desired.

Finally, observe that solutions to \eqref{betaeq} (with a given divergence-free $\omega$) which decay at infinity
are unique since any solution is divergence-free in view of \eqref{divS},
and the curl of the difference of two solutions solves the homogeneous system associated to \eqref{vomegaeqs}.
\end{proof}

\subsubsection{Change of coordinates}
In order to obtain estimates for $\beta$  
we change coordinates to a flat domain,
going from $(x,y)=(x_1,x_2,y)$ in $\D_t$ to $(x,z)$ with $x\in\R^2,z<0 $ with $z:=y-h(t,x)$,
by defining
\begin{align}\label{betaflat}
\alpha(t,x,z) := \beta(t,x,z+h(t,x)), \qquad
W(t,x,z) := \omega(t,x,z+h(t,x)).
\end{align}
In what follows $\nabla$ and $\Delta$ will only refer to differentiation in $x$ unless otherwise specified.

\begin{remark}
Notice that since we will be working at
a lower level of regularity than the maximal regularity available,
we will not need to worry about the regularity of the coordinate change
and, in particular we can avoid paradifferential calculus.

On the other hand, since we need to work at a level of regularity above $N_1$,
the top order weighted norms of $h$, which enter in the change of coordinates,
cannot be expected to be uniformly bounded in time (see the next remark),
and this creates several technical complications.
\end{remark}

\begin{remark}[A priori bounds on $h$]\label{remaph}
Recall the a priori $L^\infty$ bound \eqref{aprioridecay}
and the $L^2$ bound \eqref{apriorie0} on the height: for all $t\in[0,T]$
\begin{align}\label{apriorih}
\begin{split}
& \sum_{r+k \leq N_1} {\big\| h(t) \big\|}_{Z^{r,\infty}_{k}(\R^2)} \leq c_0 \e_0 \jt^{-1},
\\
& \sum_{r+k \leq N_0} {\| h(t) \|}_{Z^r_k(\R^2)} \leq c_0 \e_0 \jt^{p_0}.
\end{split}
\end{align}

Using standard interpolation of $L^p$ spaces, we also deduce
\begin{align}\label{apriorihp}
& \sum_{r+k \leq N_1} {\big\| h(t) \big\|}_{Z^{r,p}_{k}(\R^2)} \leq
  c_0 \e_0 \jt^{-1+(2/p)(1+p_0)}, \qquad p\geq 2.
\end{align}
for all $t \in [0,T]$.
Note that the last bound above is $\leq c_0\e_0$ for $p\geq 11/5$, $p_0 \leq 1/10$.
\end{remark}

\begin{remark}[A priori bounds on $\partial_t h$]\label{remapdth}
    Using that $\partial_t h = G(h)\varphi$ with the second estimate in \eqref{G1est}, 
\begin{align}\label{aprioridthinfty}
& \sum_{r+k \leq N_1-5} {\big\| \partial_t h(t) \big\|}_{Z^{r,\infty}_{k}(\R^2)} \leq c_0 \e_0 \jt^{-1+}.
\end{align}
Using $\partial_t h = v\cdot (-\nabla h,1)$ on the boundary, together with the a priori bounds on $v$
we can also obtain, for all $t\in[0,T]$,
\begin{align}\label{aprioridth2}
& \sum_{r+k \leq N_0-5} {\| \partial_t h(t) \|}_{Z^r_k(\R^2)} \leq c_0 \e_0 \jt^{p_0}.
\end{align}
The proof of \eqref{aprioridth2} follows from the a priori bounds on $v$ and $h$
in \eqref{apriorie0}, \eqref{aprioridecayv} and \eqref{aprioridecay}, and elementary 
composition and product identities;
we postpone the proof until after the proof of Lemma \ref{Lemrestr} 
since it can be more conveniently written out using some 
notation that will be introduced later.

Interpolating \eqref{aprioridthinfty} and \eqref{aprioridth2} we have
\begin{align}\label{aprioridthp}
& \sum_{r+k \leq N_1-5} {\big\| \partial_t h(t) \big\|}_{Z^{r,p}_{k}(\R^2)} \leq
  c_0 \e_0 \jt^{(2/p)(1+p_0)-1+}, \qquad p\geq 2,
\end{align}
Note that the right-hand side of \eqref{aprioridthp} is bounded by
$\leq c_0\e_0$ for, say, $p\geq 3$
\end{remark}

Our goal is to establish bounds for $V_\omega$ and its time derivative. 
Since we will do this by establishing bounds for $\alpha$, we first relate their norms.

\subsubsection{Basic formulas and norms}
For given $f:[0,T] \times \D_t \rightarrow \R$, let us define for $t\in[0,T]$,
$x\in\R^2$ and $z\leq 0$ the function
\begin{align*}
F(t,x,z) := f(t,x,z+h(t,x)), \qquad  f(t,x,y) = F(t,x,y-h(t,x)),
\end{align*}
and record the basic identities
\begin{align}\label{equivpr1}
\begin{split}
\partial_{t} F(t,x,z) & = (\partial_{t}f)(t,x,z+h(t,x)) + (\partial_{y}f)(t,x,z+h(t,x)) \partial_{t}h,
\\
\partial_{x_i}F(t,x,z) & = (\partial_{x_i}f)(t,x,z+h(t,x)) + (\partial_{y}f)(t,x,z+h(t,x)) \partial_{x_i}h,
\\
\partial_{z} F(t,x,z) & = (\partial_y f)(t,x,z+h(t,x)),
\\
\Omega F(t,x,z) & = (\Omega f)(t,x,z+h(t,x)) + (\partial_{y}f)(t,x,z+h(t,x)) \Omega h,
\\
(S + z\partial_z) F(t,x,z) 
  & = (\underline{S} f)(t,x,z+h(t,x)) + (\partial_{y}f)(t,x,z+h(t,x))  (Sh-h),
\end{split}
\end{align}
or, equivalently,
\begin{align}\label{equivpr1'}
\begin{split}
    (\partial_{t}f)(t,x,y) & = (\partial_{t} F)(t,x,y-h(x)) - (\partial_{z}F)\dg{(t,x,y-h(x))} \partial_{t}h(t,x),
\\
(\partial_{x_i}f)(t,x,y) & = (\partial_{x_i}F)(t,x,y-h(x)) - (\partial_{z} F)(t,x,y-h(t,x)) \partial_{x_i}h(t,x),
\\
(\partial_y f)(t,x,y) & = (\partial_{z} F) (t,x,y-h(x)),
\\
(\Omega f)(t,x,y) & = (\Omega F) (t,x,y-h(x)) - (\partial_z F)(t,x,y-h(x)) \Omega h(t,x),
\\
(\underline{S} f)(t,x,y) & = ((S + z\partial_z) F)(t,x,y-h(x)) - (\partial_z F)(t,x,y-h(t,x)) (Sh(t,x)-h(t,x)).
\end{split}
\end{align}
In particular, evaluating at the boundary
\begin{align}\label{restrpr1}
\begin{split}
(\partial_t f |_{\partial\D_t}\big)(t,x)
  & = (\partial_{t}F)(t,x,0) - (\partial_z F)(t,x,0) \partial_{t}h,
\\
\big(\partial_{x_i}f |_{\partial\D_t}\big)(t,x)
  & = (\partial_{x_i}F)(t,x,0) - (\partial_zf)(t,x,0) \partial_{x_i}h,
\\
\big( \partial_y f  |_{\partial\D_t} \big)(t,x) & = (\partial_{z} F) (t,x,0),
\\
\big(\Omega f |_{\partial\D_t}\big)(t,x) & = (\Omega F)(t,x,0)
  - (\partial_zF)(t,x,0) \Omega h,
\\
\big(\underline{S} f |_{\partial\D_t}\big)(t,x)
  & = (S F)(t,x,0) - (\partial_z F)(t,x,0)  (Sh-h),
\end{split}
\end{align}
and
\begin{align}\label{restrpr2}
\begin{split}
\partial_{t} \big(f |_{\partial\D_t}\big)(t,x)
  & = \big(\partial_{t} f) |_{\partial\D_t}(t,x) + \big(\partial_y f) |_{\partial\D_t}(t,x) \partial_t h
\\
\partial_{x_i} \big(f |_{\partial\D_t}\big)(t,x)
  & = \big( \partial_{x_i} f) |_{\partial\D_t}(t,x) + (\partial_y f) |_{\partial\D_t} \partial_{x_i}h 
\\
\Omega \big(f |_{\partial\D_t}\big)(t,x)
  & = \big( \Omega f) |_{\partial\D_t}(t,x) + (\partial_y f) |_{\partial\D_t} \Omega h 
\\
S \big(f |_{\partial\D_t}\big)(t,x)
  & = \big( \underline{S} f) |_{\partial\D_t}(t,x) + (\partial_y f) |_{\partial\D_t} (Sh - h).
\end{split}
\end{align}

The identities \eqref{equivpr1}, with the definitions \eqref{VP00} and \eqref{betaflat}, imply
\dg{
\begin{align}\label{Voalpha}
\begin{split}
V_\omega(t,x,z)
 & = (\curl_{x,y} \beta)(t,x,z+h(x))
 \\
 & = 
 (\curl_{x,z} \alpha)(t,x,z) - \partial_z \alpha_3 (t,x,z) \,
 (\pa_2 h(t,x), -\pa_1h(t,x), 0)
 \\
 &
 -(0,0,\pa_1 h(t,x) \pa_z \alpha_2(t,x,z) - \pa_2 h(t,x)\pa_z \alpha_1(t,x,z))
 \\
 & = \big[ \curl \alpha + \pa_z \alpha_3 \nabla^\perp h - (\pa_z \alpha \cdot \nabla^\perp h) e_3\big](t,x,z),
\end{split}
\end{align}
with the convention that $\nabla^\perp = (-\pa_2, \pa_1, 0)$. }
\dg{We also have
\begin{align}
    \label{Voalphadt}
\begin{split}
\partial_t V_\omega(t,x,z) &\! =\!
\big[ \curl \pa_t \alpha + \pa_z \pa_t \alpha_3 \nabla^\perp h + \pa_z \alpha_3 \nabla^\perp \pa_t h
\!-\! (\pa_z \pa_t \alpha \cdot \nabla^\perp h + \pa_z \alpha \cdot \nabla^\perp \pa_t h) e_3\big]
(t,x,z).
\end{split}
\end{align}

} 
Let us record that the identities \eqref{Voalpha} and \eqref{Voalphadt} schematically read as follows:
\begin{align}\label{Voaschem}
\begin{split}
V_\omega & = \curl \alpha + \partial_z \alpha \cdot \nabla h,
\\
\partial_t V_\omega & = \curl \partial_t \alpha + \partial_z \partial_t \alpha \cdot \nabla h
 + \partial_z \alpha \cdot \nabla \partial_t h.
\end{split}
\end{align}
Then, using simple product estimates for weighted norms
we will be able to obtain bounds for $V_\omega$ 
in terms of certain norms of $\alpha$. Here are the norms that we are going to use:

\begin{defi}[Norms]
\label{defnormsalpha}
For a non-negative integer $n$, let
\begin{align}\label{alphaspace0}
{\|g\|}_{\Yn^n} = \sum_{|r|+|k| \leq n} {\big\|g^{r, k} \big\|}_{\Yn^0},
\qquad g^{r, k} := \underline{\Gamma}^k\nabla^r_{x, z} g
\end{align}
where 
\begin{align}\label{alphaspace0'}
{\|g\|}_{\Yn^0} =
   {\big\||\nabla|^{1/2} g \big\|}_{L^\infty_z L^2_x}
  + {\big\| \nabla_{x,z} g \big\|}_{L^2_z L^2_x}
    + {\|  g \|}_{L^\infty_z L^2_x}.
\end{align}
Also, let us define the ``homogeneous'' versions of the above spaces by
\begin{align}\label{dotalphaspace0}
{\|g\|}_{\dYn^{n}} = \sum_{|r|+|k| \leq n} {\big\|g^{r, k} \big\|}_{\dYn^0},\qquad g^{r, k} 
  = \underline{\Gamma}^k\nabla^r_{x, z} g
\end{align}
where
\begin{align}\label{dotalphaspace0'}
{\|g\|}_{\dYn^0} =
  \sum_{0 \leq |a| \leq 1}\left( {\big\| \nabla_{x,z}^a |\nabla|^{1/2} g \big\|}_{L^\infty_z L^2_x}
  + {\big\| \nabla_{x,z}^a 
  \nabla_{x,z} g \big\|}_{L^2_z L^2_x}\right)
    + {\| \partial_z g \|}_{L^\infty_z L^2_x}.
\end{align}
\end{defi}

Note that the above norms are defined so that
\begin{equation} \label{dotequiv}
\|\nabla_{x,z} g\|_{\Yn^0} + \| |\nabla|^{1/2} g\|_{L^\infty_zL^2_x} +
 \| \nabla_{x,z} g \|_{L^2_zL^2_x} \lesssim \|g\|_{\dYn^0}.
\end{equation}
In the upcoming section, we will prove estimates for $V_\omega$ in the $\Yn^n$
spaces by first bounding the vector potential $\alpha$ in the $\dYn^n$ spaces.
We will also use the bounds on $\alpha$
to get estimates for $\widetilde{v}_\omega = V_{\omega}|_{z = 0}$, hence on $P_\omega$,
in $Z^{r}_k$ spaces; see \eqref{lembounda-Voconc0} and \eqref{decayirrasVP}.
The bounds for $\alpha$ will follow from a fixed-point argument
for a Poisson-type problem for which the $\dYn^n$ norms are well suited.

\subsubsection{Consequences of bounds for $\alpha$}
The next two lemmas show how bounds for $\widetilde{v}_\omega$ in the spaces $Z^r_k$,
and bounds for $V_\omega$ in the spaces $\Yn^n$,
follow from bounds for the vector potential $\alpha$ the $\dYn^n$ spaces.
Then, in the remainder of the section, we will prove
that the needed bounds on $\alpha$ stated in \eqref{vpboundlo}-\eqref{vpboundhi}
can be obtained from our bootstrap
assumptions \eqref{apriorih}-\eqref{apriorihp} on the dispersive variables,
assumptions \eqref{aprioriWL}-\eqref{aprioriWH} on the vorticity 
and a fixed point argument.

\begin{lemma}[Bounds for $\alpha$ imply bounds for $\wt{v}_\omega$]\label{lembounda-Vo}
Assume \eqref{apriorih}-\eqref{apriorihp}.
Assume that for all $t \in [0,T]$ and for
$j =0$ and $1$, with notation as in \eqref{dotalphaspace0}, we have
\begin{align}
\label{vpboundlo}
\|\pa_t^j \alpha(t)\|_{\dYn^{N_1-10-j}} & \lesssim \e_1,
\\
\label{vpboundhi}
\|\pa_t^j \alpha(t)\|_{\dYn^{N_1 + 12-j}} &\lesssim \e_1 \e_0^j \jt^\delta.
\end{align}
%
Then, for $j=0,1$,
\begin{align}\label{lembounda-Voconc}
    \sum_{\dg{r + n} \leq N_1+12-j} {\big\| \partial_t^j \wt{v_\omega}(t) \big\|}_{Z^r_n(\R^2)}
\lesssim \e_1 \e_0^j \jt^\delta.
\end{align}
\end{lemma}


\begin{lemma}[Bounds for $\alpha$ imply bounds for $V_\omega$]\label{lemboundaVbulk}
Assume \eqref{apriorih}-\eqref{apriorihp}. With notation as in the previous lemma,
if the bounds \eqref{vpboundlo}-\eqref{vpboundhi}
for the quantity $\alpha$ hold for $t\in[0,T]$, then,
with the norms $\Yn^n$ defined as in \eqref{alphaspace0}, we have
\begin{align}\label{boundaVbulk}
\begin{split}
{\| \pa_t^j V_\omega(t)\|}_{\Yn^{N_1 - 10 - j}} &\lesssim \e_1,
\\
{\| \pa_t^j V_\omega(t)\|}_{\Yn^{N_1 + 12 - j}} &\lesssim \e_1 \e_0^j \jt^\delta.
\end{split}
\end{align}
\end{lemma}

\begin{proof}[Proof of Lemma \ref{lembounda-Vo}]
Starting from \eqref{Voalpha} and applying the product estimate \eqref{prodest0'}, 
we estimate for any $r+n \leq N=N_1+12$ 
\begin{align}\label{aVopr1}
\begin{split}
{\| \wt{v_\omega}(t,x) \|}_{Z^r_n}
 & \lesssim
 {\| (\curl_{x,z} \alpha)(t,x,0) \|}_{Z^r_n}
 \\
 & + \sum_{r+n \leq N/2} {\| \partial_z \alpha (t,x,0) \|}_{Z^{r,\infty}_n}
 \, \sum_{r+n \leq N} {\| \nabla' h(t,x) \|}_{Z^r_n}
 \\
 & + \sum_{r+n \leq N} {\| \partial_z \alpha (t,x,0) \|}_{Z^r_n}
 \, \sum_{r+n \leq N/2} {\| \nabla' h(t,x) \|}_{Z^{r,\infty}_n}
\end{split}
\end{align}
where we denoted $\nabla' = (\nabla^\perp, \partial_{x_1} - \partial_{x_2})$.

The first term on the right-hand side of \eqref{aVopr1}
is directly bounded using the assumption \eqref{vpboundhi}:
for all $r+n \leq N$, with the notation $\alpha^{r,n}=\underline{\Gamma}^n\nabla^r_{x,z}\alpha$
as in \eqref{alphaspace0},
we have 
\begin{align}\label{aVopr2}
\begin{split}
{\| (\curl_{x,z} \alpha)(t,x,0) \|}_{Z^r_n}
 & \lesssim \sum_{r+n \leq N} {\| (\nabla_x,\partial_z) \nabla^r_{x} \Gamma^n \alpha (t,\cdot,0) \|}_{L^2_x}
 \\
 & \lesssim \sum_{r+n \leq N} {\| (\nabla_x,\partial_z) \alpha^{r,n}(t) \|}_{L^\infty_z L^2_x}
 \lesssim {\| \alpha(t) \|}_{\dYn^N} \lesssim \e_1 \jt^\delta.
\end{split}
\end{align}
having also used \eqref{dotequiv} (recall also \eqref{alphaspace0'}) and that $\underline{S}|_{z=0} = S$.

The second term in \eqref{aVopr1} is bounded by
\begin{align}\label{aVopr3}
C \sum_{r+n \leq N/2+2} {\| \partial_z \alpha (t,\cdot,0) \|}_{Z^r_n}
 \, \sum_{r+n \leq N} {\| \nabla h(t,\cdot) \|}_{Z^r_n}
 \lesssim \e_1 \cdot \e_0 \jt^\delta
 \end{align}
having used Sobolev's embedding,
the control the first norm in \eqref{vpboundlo} (since $N_1-10 \geq N/2+2$)
and the a priori assumption \eqref{apriorih} (since $N+1 \leq N_0$).

We can instead bound the last term in \eqref{aVopr1} using \eqref{vpboundhi}
and \eqref{apriorih} (since $N_1 \geq N/2+1)$, as follows:
\begin{align}\label{aVopr4}
C \sum_{r+n \leq N} {\| \partial_z \alpha (t,\cdot,0) \|}_{Z^r_n}
 \, \sum_{r+n \leq N/2+1} {\| h(t,\cdot) \|}_{Z^{r,\infty}_n}
 \lesssim \e_1 \jt^\delta \cdot \e_0 \jt^{-1}
 \end{align}
which is more than sufficient.
This proves \eqref{lembounda-Voconc} when $j=0$.


To obtain the estimate for the time derivative we can proceed similarly, starting from the formula \eqref{Voalphadt}.
To control the term $\curl_{x,z} \partial_t \alpha(t,x,0)$ we can estimate
as in \eqref{aVopr2} replacing $\alpha$ with $\partial_t \alpha$
and using \eqref{vpboundhi} with $j=1$.
All the other terms in \eqref{Voalphadt} are of the form
$\partial_z \partial_t \alpha_i(t,x,0) \cdot \partial_{x_k} h$
or $\partial_z \alpha_i(t,x,0) \cdot \partial_{x_k} \partial_t h$ for some $i=1,2,3$ and $k=1,2$.
We can then estimate all of these
using \eqref{vpboundhi} and \eqref{vpboundlo} also with $j=1$
and the bounds \eqref{aprioridth2} and \eqref{aprioridthp} for the terms involving $\partial_t h$;
these estimates are analogous to \eqref{aVopr3} and \eqref{aVopr4} so we omit the details.
\end{proof}

\begin{proof}[Proof of Lemma \ref{lemboundaVbulk}]
We argue in a similar way as in the previous lemma.
For the first term on the right-hand side of \eqref{Voalpha} we observe, using \eqref{dotequiv},
that ${\| \curl \alpha \|}_{Y^n} \lesssim {\| \alpha \|}_{\dot{Y}^n}$,
for $n= N_1-10$ or $N_1+12$, which is consistent with the desired \eqref{boundaVbulk}.
We then only need to look at the nonlinear terms on the right-hand side of \eqref{Voalpha}.
According to the schematic version \eqref{Voaschem}, applying vector fields
and using the notation \eqref{alphaspace0}, we have, for all $|r|+|k| \leq N$,
\begin{align}
\label{Vbulkpr0}
\underline{\Gamma}^k \nabla_{x,z}^r (V_\omega - \curl \alpha)
 & = 
 \sum_{\substack{r_1 + r_2 = r,\\ k_1 + k_2 = k}}
 (\partial_z \alpha)^{r_1, k_1} \cdot (\nabla h)^{r_2, k_2}.
\end{align}
Let us concentrate on proving the bounds in the high-norm, that is $|r|+ |k| = N := N_1+12$,
since the bounds in the low norm can be obtained similarly.
From \eqref{Vbulkpr0} we see that it suffices to estimate the $Y^0$-norm of the terms
\begin{align}
\label{Vbulkpr1}
I & := \sum_{|r_1| + |k_1| \leq N/2} (\partial_z \alpha)^{r_1, k_1} \cdot (\nabla h)^{r_2, k_2},
\\
\label{Vbulkpr2}
II & := \sum_{|r_2| + |k_2| \leq N/2} (\partial_z \alpha)^{r_1, k_1} \cdot (\nabla h)^{r_2, k_2}.
\end{align}
The constraints $r_1 + r_2 = r$ and $k_1 + k_2 = k$ are implicit in the above sums.

The first component of the $Y_0$-norm of $I$ can be estimated using \eqref{prodestuse00}:
\begin{align*}
\begin{split}
{\| |\nabla|^{1/2} I \|}_{L^\infty_zL^2_x}
  \lesssim \sum_{|r_1| + |k_1| \leq N/2}
  {\| |\nabla|^{1/2} (\partial_z \alpha)^{r_1, k_1} \|}_{L^\infty_zL^2_x}
  \cdot \sum_{|r_2| + |k_2| \leq N} {\| (\nabla h)^{r_2, k_2} \|}_{W^{1,3}}
  \\
  \lesssim \e_1 \cdot \e_0\jt^{p_0}
\end{split}
\end{align*}
having also used \eqref{vpboundlo} (recall \eqref{dotalphaspace0'} and that $N/2 \leq N_1-10$) and \eqref{apriorih}.
Note that we have also used the commutation relation \eqref{comm0}. 

For the second component of the $Y_0$-norm of $I$ we use H\"older
and the same assumptions above:
\begin{align}
	\label{nablaIL2L2}
\begin{split}
{\| \nabla_{x,z} I \|}_{L^2_zL^2_x}
  \lesssim \sum_{\substack{|r_1| + |k_1| \leq N/2 \\ |a|\leq 1}}
  {\| \nabla_{x,z}^a (\partial_z \alpha)^{r_1, k_1} \|}_{L^2_zL^2_x}
  \cdot \sum_{|r_2| + |k_2| \leq N+1} {\| (\nabla h)^{r_2, k_2} \|}_{L^\infty}
  \\
  \lesssim \e_1 \cdot \e_0\jt^{p_0}.
\end{split}
\end{align}
The last $L^\infty_zL^2_x$ piece of the norm is immediate to estimate, so we skip it.

For the term $II$, we estimate the first component of the $Y_0$-norm
using again the product estimate \eqref{prodestuse00},
and then the assumption \eqref{vpboundhi} and the a priori bound \eqref{apriorihp}:
\begin{align*}
\begin{split}
{\| |\nabla|^{1/2} II \|}_{L^\infty_zL^2_x}
\lesssim \sum_{\dg{|r_1| + |k_1|} \leq N}
  {\| |\nabla|^{1/2} (\partial_z \alpha)^{r_1, k_1} \|}_{L^\infty_zL^2_x}
  \cdot \sum_{\dg{|r_2| + |k_2|} \leq N/2} {\| (\nabla h)^{r_2, k_2} \|}_{W^{1,3}}
  \\
  \lesssim \e_1 \jt^{\delta} \cdot \e_0.
\end{split}
\end{align*}
The second component of the $Y_0$-norm is estimated just using H\"older and the same assumptions above:
\begin{align}
		\label{nablaIIL2L2}
\begin{split}
{\| \nabla_{x,z} II \|}_{L^2_zL^2_x}
\lesssim \sum_{\substack{\dg{|r_1| + |k_1|} \leq N \\ |a|\leq 1}}
  {\| \nabla_{x,z}^a (\partial_z \alpha)^{r_1, k_1} \|}_{L^2_zL^2_x}
  \cdot \sum_{\dg{|r_2| + |k_2|} \leq N/2+1} {\| (\nabla h)^{r_2, k_2} \|}_{L^\infty}
  \\
  \lesssim \e_1 \jt^{\delta} \cdot \e_0.
\end{split}
\end{align}
The last piece of the norm, that is, ${\|II\|}_{L^\infty_zL^2_x}$ can be bounded in the same way.

To obtain the estimates for the time derivative we can proceed in the same way, starting
from the second formula in \eqref{Voaschem}, using \eqref{prodestuse00} as above,
H\"older, and the assumption on $\partial_t\alpha$ in \eqref{vpboundlo}-\eqref{vpboundhi}
and on $\partial_t h$ in \eqref{aprioridthinfty}-\eqref{aprioridth2}.
\end{proof}

%

For some of our applications (specifically for the estimates
in Section \ref{secvelpot}), we will need a slight variation of
the above bounds for $V_\omega$ where we both control the $L^2_{x, z}$
norms of $V_\omega$ directly (technically, this is not included in the $Y^n$ spaces)
and additionally control a higher-order
norm of $V_\omega$ (but with a worse bound) provided we have additional high-order
control of the vorticity. This is the lemma that we will need:

\begin{lemma}[High-order bounds for $\alpha$ imply high-order $L^2_{z}L^2_x$ bounds for $V_\omega$]\label{higherorderalpha}
Under the hypotheses of Lemma \ref{lemboundaVbulk}, 
and using the notation $g^{r,k} = \underline{\Gamma}^k \nabla_{x,z}^r g$ from \eqref{alphaspace0},
we have
\begin{align}\label{boundaVbulk2}
\sum_{|r| + |k| \leq N_1-10-j} \| \pa_t^j V_\omega^{r,k}(t)\|_{L^2_zL^2_x} &\lesssim \e_1,
\\
\label{boundaVbulk2'}
\sum_{|r| + |k| \leq N_1+12-j} {\| \pa_t^j V_\omega^{r,k}(t)\|}_{L^2_zL^2_x} &\lesssim \e_1 \e_0^j \jt^\delta.
\end{align}
Moreover, if
\begin{equation}\label{extraonalpha}
{\| \alpha(t) \|}_{\dYn^{N_0-20}} \lesssim \e_0 \jt^{2p_0},  
\end{equation}
then
\begin{equation}\label{boundaVbulkweaker}
\sum_{|r| + |k| \leq N_0-20} {\| V_\omega^{r,k}(t)\|}_{L^2_zL^2_x}\lesssim \e_0 \jt^{2p_0}.
\end{equation}
\end{lemma}

%
%
In Section \ref{secvelpot} we will see how the high order norm assumption \eqref{extraonalpha}
on $\alpha$ follows from the assumption on high-order norms of the vorticity \eqref{aprioriWH};
see Proposition \ref{mainpropW'} and Lemma \ref{alphaprop'}. 

\begin{proof}[Proof of Lemma \ref{higherorderalpha}]
The argument is nearly identical to the proof of Lemma \ref{lemboundaVbulk}
using \eqref{Voaschem} and simple product estimates. 
The only additional observation needed is that
$\|\underline{\Gamma}^k\nabla^r\curl \alpha\|_{L^2_zL^2_x}
 + \|\underline{\Gamma}^k\nabla^r\nabla_{x,z} \alpha\|_{L^2_zL^2_x} \lesssim \|\alpha\|_{\dYn^n}$
for $|k| + |r| \leq n$ by definition (note the $a = 0$ term in the second
term of the definition \eqref{dotalphaspace0'} of the $\dYn^n$ norms). 
Then, we follow the same steps as in the above proof with $I$ and $II$ defined as in \eqref{Vbulkpr1}-\eqref{Vbulkpr2}.
In place of \eqref{nablaIL2L2} we bound
\begin{align*}
\begin{split}
 {\|  I \|}_{L^2_zL^2_x}
   \lesssim \sum_{|r_1| + |k_1| \leq N/2}
   {\| (\partial_z \alpha)^{r_1, k_1} \|}_{L^2_zL^2_x}
   \cdot \sum_{|r_2| + |k_2| \leq N+1} {\| (\nabla h)^{r_2, k_2} \|}_{L^\infty}
   \lesssim \e_1 \cdot \e_0\jt^{p_0};
\end{split}
\end{align*}
for $N = N_1+12$, which is consistent with \eqref{boundaVbulk2'};
with an obvious modification when $N$ is replaced with $N_0-20$ this is consistent with \eqref{boundaVbulkweaker}.
Similarly, in place of \eqref{nablaIIL2L2} we have
\begin{align*}
\begin{split}
{\|  II \|}_{L^2_zL^2_x}
  \lesssim \sum_{|r_1| + |k_1| \leq N }
  {\| (\partial_z \alpha)^{r_1, k_1} \|}_{L^2_zL^2_x}
  \cdot \sum_{|r_2| + |k_2| \leq N/2+1} {\| (\nabla h)^{r_2, k_2} \|}_{L^\infty}
  \lesssim \e_1 \jt^{\delta} \cdot \e_0;
\end{split}
\end{align*}
once again this is consistent with \eqref{boundaVbulk2'} if $N=N_1+12$;
the obvious modification when $N$ is replaced with $N_0-20$ gives \eqref{boundaVbulkweaker}.

The lower order norm in \eqref{boundaVbulk2} can be estimated similarly, using the 
uniform bound \eqref{apriorihp}.
The estimates for the time derivatives can also be obtain in a completely analogous fashion,
using the estimates on $\partial_t h$ from Remark \ref{remapdth}
\end{proof}

\subsection{Fixed point formulation for $\alpha$}
From the system \eqref{betaeq} satisfied by $\beta$ 
we derive a fixed point formulation for $\alpha$.
We first write out the elliptic system satisfied by $\alpha$:

\begin{lemma}[The Elliptic system in the flat domain]\label{flatlem}
Let $\alpha$ and $W$ be defined as in \eqref{betaflat},
with $\beta$ the solution of \eqref{betaeq}. Then we have 
\begin{subequations}\label{eqflat}
\begin{alignat}{2}
\label{eqflat0}
(\partial_z^2 + \Delta_x) \alpha & = \partial_z \dg{E^a} + |\nabla| \dg{E^b} + F, && \qquad \text{ in } z<0,
\\
\label{eqflat1}
\alpha_1 & = B_1, && \qquad \text{ on } z=0,
\\
\label{eqflat2}
\alpha_2 & = B_2, && \qquad \text{ on } z=0,
\\
\label{eqflat3}
\partial_z \alpha_3 & = B_3 , && \qquad \text{ on } z=0,
\end{alignat}
\end{subequations}
where
\begin{equation}\label{EaEb}
    {E^a}(\alpha) := \frac{\nabla}{|\nabla |} \cdot (\nabla h \pa_z \alpha)
\qquad
{E^b}(\alpha) := -|\nabla h|^2 \pa_z\alpha + \nabla h\cdot \nabla \alpha,
\qquad
F = W,
\end{equation}
\begin{align}\label{Aidef}
B_i(\alpha,\nabla \alpha) & =
\left((1+ |\nabla h|^2) \pa_i h (\alpha_3 - \nabla h\cdot \alpha) \right)|_{z = 0}, \qquad i = 1,2,
\end{align}
and
\begin{align}
\begin{split}\label{B3formula}
	B_3(\alpha, \nabla\alpha) &
	=
	\nabla h\cdot \partial_z (\alpha_1,\alpha_2)
  + \nabla \cdot \Big[ (1+|\nabla h|^2)^{-1} \nabla h \, (\alpha_3 - \nabla h \cdot \alpha) \Big]\bigg|_{z=0}.
\end{split}
\end{align}

\end{lemma}

{\it Notation}.
Note that in \eqref{EaEb} we are omitting the dependence on $h$ and implicitly on the position $(x,z)$.
Later on, e.g. in \eqref{alphaifp}, we will denote these terms with $E_{a}(z)$ to make the dependence on
the vertical variable explicit.

The proof of the above lemma is an explicit computation, see Appendix \ref{Appalpha}.
Regardless of the exact formulas, we point out that we are dealing with an elliptic system
for the vector field $\alpha$ with mixed Dirichlet (for the first two components)
and Neumann (for the third component) boundary conditions.
Note that the quantity $\alpha_3$, which is more singular than $\nabla_x\alpha_3$ or $\partial_z\alpha_3$,
appears in the boundary data multiplied by a linear factor of $h$;
this will create some technical difficulties in proving bounds for $\alpha$.



Using Lemma \ref{flatlem} we write a fixed point formulation for $\alpha$,
which we record in the following:

\begin{lemma}[Fixed point formulation]\label{lemmafp}
Let $\alpha$ be the solution of \eqref{eqflat}-\eqref{B3formula}.
Then, it is formally a fixed point of the map
\begin{align}\label{alphafp}
\alpha \rightarrow L(\alpha) = (L_1(\alpha),L_2(\alpha),L_3(\alpha))
\end{align}
where 
\begin{multline}\label{alphaifp}
L_i(\alpha)(z) := e^{z|\nabla|} B_i(\alpha) -\frac{1}{2} \int_{-\infty}^0
  e^{(z+s)|\nabla|}(E^a_i(s) - E^b_i(s) - |\nabla|^{-1} F_i(s))\, ds
  \\
  + \frac{1}{2} \int_{-\infty}^0
  e^{-|z-s||\nabla|} (\mathrm{sign}(z-s) E^a_i(s) - E^b_i(s) - |\nabla|^{-1}F_i(s))\, ds,
  \quad i = 1,2,
\end{multline}
with \eqref{Aidef}, and
\begin{align}\label{alpha3fp}
\begin{split}
L_3(\alpha)(z) & :=  e^{z|\nabla|} B_{3, a}(\alpha) + |\nabla|^{-1} e^{z|\nabla|} B_{3,b}(\alpha)
 \\
 & + \frac{1}{2} \int_{-\infty}^0 e^{(z+s)|\nabla|}(E^a_3(s) - E^b_3(s) - |\nabla|^{-1} F_3(s))\, ds
 \\
 & + \frac{1}{2} \int_{-\infty}^0 e^{-|z-s||\nabla|} (\mathrm{sign}(z-s) E^a_3(s) - E^b_3(s) - |\nabla|^{-1} F_3(s) )\, ds,
\end{split}
\end{align}
with
\begin{align}\label{alpha3a}
\begin{split}
B_{3,a}(\alpha)
 & = \frac{\nabla}{|\nabla|} \cdot
 \big[ (1+|\nabla h|^2)^{-1} \nabla h \, (\alpha_3 - \nabla h \cdot \alpha) \big],
\end{split}
\end{align}
\begin{align}\label{alpha3b}
\begin{split}
B_{3,b}(\alpha, \nabla \alpha) & = \nabla h \cdot \pa_z \alpha - \frac{\nabla}{|\nabla|}\cdot(\nabla h \pa_z\alpha_3).
\end{split}
\end{align}

\end{lemma}





\begin{proof}[Proof of Lemma \ref{lemmafp}]
The fixed point formulation \eqref{alphafp}-\eqref{alpha3b}
is obtained using the solution of Laplace's equation given in Lemma \ref{lemmaAfp}.
Using \eqref{dirichletformula} we directly obtain \eqref{alphaifp}.

For the third component $\alpha_3$, an application of \eqref{neumformula} gives us
the bulk integrals in \eqref{alpha3fp}, so we only need to verify the formulas
for the boundary contributions, which are given by
\begin{align}\label{alpha3fpBC}
\frac{1}{|\nabla|} e^{z|\nabla|} B_3 - \frac{1}{|\nabla|} e^{z|\nabla|} 
  \frac{\nabla}{|\nabla |} \cdot (\nabla h \, \pa_z \alpha_3),
\end{align}
with $B_3$ as in \eqref{B3formula}, which gives the result.
\end{proof}

\subsection{Norms and main proposition}
Based on the above fixed point formulation and using the a priori bounds on $h$ from \eqref{apriorihp},
we want to show existence and uniqueness of $\alpha$ and
bound it as in \eqref{vpboundlo}-\eqref{vpboundhi}.
As mentioned above, we will work in terms of the norms from Definition \ref{defnormsalpha}; 
%
%
%
in particular, we will prove a contraction for the map \eqref{alphafp} in the `low norm' $\dYn^{N_1-10}$
and bounds in the `high norm' $\dYn^{N}$, $N:=N_1+12$.

\begin{remark}\label{remalphaspace}
    Directly from the definition, for all $\dg{r+|k|}\leq n$, we see that
\begin{align}\label{alphaspacebd'full}
\sum_{|k'|\leq |k|} {\big\| \jnab^{1/2} \Gamma^{k'} \nabla_{x,z} \alpha(0) \big\|}_{H^r (\R^2)}
\lesssim \sum_{|k'|\leq |k|} {\big\| \jnab^{1/2} \underline{\Gamma}^{k'} \nabla_{x,z} \alpha \big\|}_{L^\infty_z H^r (\R^2)}
  \lesssim {\|\alpha\|}_{\dYn^n},
\end{align}
This will be used to control some of the homogeneous boundary terms that we will encounter.
\end{remark}

We also define (see \eqref{omegaflatspace0})
\begin{align}\label{omegaflatspace2}
\begin{split}
{\| f \|}_{\mX^n} := \sum_{|r|+|k| \leq n}
  {\big\| \underline{\Gamma}^k \nabla^r_{x,z} \,f \big\|}_{L^2_z L^2_x \cap L^{6/5}_{x,z}}
.
\end{split}
\end{align}
This is the norm that we use to measure the vorticity $W$ (see \eqref{betaflat})
which appears as a forcing term in \eqref{eqflat0}. 
Bounds on $W$ and its time derivative in the above spaces
will be bootstrapped in Section \ref{secVorticity}.

To obtain \eqref{vpboundlo}-\eqref{vpboundhi} it will suffice to show the following
proposition:

\begin{prop}[Bounds for $\alpha$]\label{alphaprop}
Let $\alpha: [0,T] \times \R^2 \times \R_- \mapsto \R^3$
be defined by $\alpha(t,x,z) := \beta(t,x,z+h(t,x))$
where $\beta$ solves the system \eqref{betaeq} in $\D_t$.
Assume that $h$ satisfies \eqref{apriorih}-\eqref{apriorihp} and \eqref{aprioridth2}-\eqref{aprioridthp},
and let $W$ be given so that, for $t\in[0,T]$, and for $j=0,1$ 
\begin{align}
\label{alphaaso0}
& {\| \partial_t^j W(t) \|}_{\mX^{N_1-10-j}} \lesssim \e_1,
\\
\label{alphaaso1}
& {\| \partial_t^j W(t) \|}_{\mX^{N_1+12-j}} \lesssim \e_0^j \e_1 \jt^{\delta}.
\end{align}
Then, there exists a unique fixed point $\alpha$ of the map in \eqref{alphafp}
in the space $\dYn^{N_1-10}$, which satisfies
\begin{align}
\label{alphaconcL}
& {\| \partial_t^j \alpha(t) \|}_{\dYn^{N_1-10-j}} \lesssim \e_1,
\\
\label{alphadtconcH}
& {\| \partial_t^j \alpha(t) \|}_{\dYn^{N_1+12-j}} \lesssim \e_1 \e_0^j \jt^\delta.
\end{align}
\end{prop}

The proof of Proposition \eqref{alphaprop} is carried out in the next subsection.
The desired conclusions will be a consequence of the following main estimates:
\begin{align}
\label{alphapropL}
{\| L(\alpha) \|}_{\dYn^{N_1-10}} & \lesssim \e_0 {\| \alpha \|}_{\dYn^{N_1-10}} + {\| W \|}_{\mX^{N_1-10}},
\\
\label{alphapropH}
{\| L(\alpha) \|}_{\dYn^{N_1+12}} & \lesssim \e_0 {\| \alpha \|}_{\dYn^{N_1+12}} + \e_0 \jt^{\delta}
 {\| \alpha \|}_{\dYn^{N_1-10}}
  + {\| W \|}_{\mX^{N_1+12}},
\end{align}
and
\begin{align}
\label{alphapropL'}
{\| \partial_t L(\alpha) \|}_{\dYn^{N_1-11}} & \lesssim \e_0 \big( {\| \alpha \|}_{\dYn^{N_1-10}}
 + {\| \partial_t \alpha \|}_{\dYn^{N_1-11}} \big) + {\| \partial_t W \|}_{\mX^{N_1-11}},
\\
\label{alphapropH'}
{\| \partial_t L(\alpha) \|}_{\dYn^{N_1+11}} & \lesssim \e_0
\big({\| \alpha \|}_{\dYn^{N_1+12}} + {\| \partial_t \alpha \|}_{\dYn^{N_1+11}}\big)
  + \e_0 \jt^{\delta} \big( {\| \alpha \|}_{\dYn^{N_1-10}} + {\| \partial_t \alpha \|}_{\dYn^{N_1-11}} \big)
	\\
	\nonumber
	&  + {\| \partial_t W \|}_{\mX^{N_1+11}}
\end{align}

%
%
%
%

\subsection{Proof of Proposition \ref{alphaprop}}\label{ssecpralphaprop}

\subsubsection{Bounds for the Poisson kernel}
We first need some bounds on the Poisson kernel.

\begin{lemma}\label{explem}
For $f:\R^2 \rightarrow \R$, any $p \in (1,\infty)$, and $k=0,1,\dots$, we have
\begin{align}
\label{expbounds1}
{\big\| \underline{\Gamma}^k e^{z|\nabla|} f \big\|}_{L^\infty_z W^{r,p}_x}
  & \lesssim 
  {\| f \|}_{Z^{r,p}_k}, \qquad 1 < p <\infty,
\end{align}
and
\begin{align}
\label{expbounds2}
{\big\| \underline{\Gamma}^k |\nabla|^{1/2} e^{z|\nabla|} f \big\|}_{L^2_z H^r_x} & \lesssim {\| f \|}_{H^r}.
\end{align}
Moreover, for $f:\R^2\times \{z<0\} \rightarrow \R$ we have
\begin{align}\label{expbounds3}
\begin{split}
& {\Big\| \underline{\Gamma}^k
  |\nabla|^{1/2} \int_{-\infty}^0 e^{-|z-s||\nabla|}  \mathbf{1}_\pm(s-z) f(x,s) \, ds \Big\|}_{L^\infty_z H^r}
  \\
  & + {\Big\| \underline{\Gamma}^k
  |\nabla| \int_{-\infty}^0 e^{-|z-s||\nabla|} \mathbf{1}_\pm(s-z) f(x,s) \, ds \Big\|}_{L^2_z H^r}
  \\
  & \lesssim
  \min \Big( \sum_{k'\leq k} {\| \underline{\Gamma}^{k'} f \|}_{L^2_z H^r}, \,
  \sum_{k'\leq k} {\big\| |\nabla|^{1/3} \underline{\Gamma}^{k'} f \big\|}_{L^{6/5}_z H^r} \Big).
\end{split}
\end{align}
Recall that $\mathbf{1}_\pm(x)$ is the indicator function of $\pm x >0$.
\end{lemma}


Lemma \ref{explem} follows from standard bounds for the Poisson kernel and commutation identities for vector fields.
The proof is given in \ref{ssecPoisson}.
Let us make a few remarks.

\begin{remark}\label{remexp}

1. Note that \eqref{expbounds3} implies the same bounds for the operators
\begin{align}\label{Tidefs}
T_1 f := & \int_{z}^0 e^{(z-s)|\nabla|} f(x,s)\, ds,
\quad T_2 := \int_{-\infty}^z e^{(s-z)|\nabla|} f(x,s)\, ds,
\quad T_3 := \int_{-\infty}^0 e^{(z+s)|\nabla|} f(x,s) \, ds,
\end{align}
which are those that appear in \eqref{alphafp};
the first two are immediate, while for the last one we just observe that
$T_3 = T_1e^{2z|\nabla|} + e^{2z|\nabla|} T_2$.

2. Also note that the estimate for the second term in \eqref{expbounds3} implies
a similar estimate with $\partial_z$ replacing $|\nabla|$:
\begin{align}\label{expboundsdz2}
\begin{split}
{\Big\| \underline{\Gamma}^k \partial_z
  \int_{-\infty}^0 e^{-|z-s||\nabla|} \mathbf{1}_\pm(s-z) f(x,s) \, ds \Big\|}_{L^2_z H^r}
  \lesssim \sum_{k'\leq k} {\| \underline{\Gamma}^{k'} f \|}_{L^2_z H^r};
\end{split}
\end{align}
this follows since we have the identities
\begin{align}\label{dzids}
\partial_z T_1 = -\mathrm{id} + T_1 |\nabla|,
\quad \partial_z T_2 = \mathrm{id} - T_2 |\nabla|,
\quad \partial_z T_3 = T_3 |\nabla|, \qquad [T_i,|\nabla|]=0.
\end{align}
Using these identities we can also estimate
\begin{align}\label{expboundsdzinfty}
\begin{split}
{\Big\| \underline{\Gamma}^k \partial_z
  \int_{-\infty}^0 e^{-|z-s||\nabla|} \mathbf{1}_\pm(s-z) f(x,s) \, ds \Big\|}_{L^\infty_z H^r}
  \lesssim \sum_{k'\leq k} {\| |\nabla|^{1/2} \underline{\Gamma}^{k'} f \|}_{L^2_z H^r}
  + {\|\underline{\Gamma}^{k'} f \|}_{L^\infty_z H^r}.
\end{split}
\end{align}

3. Bounds for higher-order $z$-derivatives also hold true: for $\ell \geq 1$,
\begin{align}\label{expboundsdz3}
\begin{split}
\Big\| \underline{\Gamma}^k \partial_z^\ell
  \int_{-\infty}^0 e^{-|z-s||\nabla|} \mathbf{1}_\pm(s-z) f(x,s) \, ds \Big\|_{L^2_z H^r}
  \lesssim \sum_{k'\leq k} \sum_{\ell_1 + \ell_2 \leq \ell-1}
	\| \underline{\Gamma}^{k'}\pa_z^{\ell_1} f \|_{L^2_z H^{r + \ell_2}};
\end{split}
\end{align}
and
\begin{align}\label{expboundsdz3'}
\begin{split}
& \Big\| \underline{\Gamma}^k \partial_z^\ell
  \int_{-\infty}^0 e^{-|z-s||\nabla|} \mathbf{1}_\pm(s-z) f(x,s) \, ds \Big\|_{L^\infty_z H^r}
  \\
  & \lesssim \sum_{k'\leq k}\sum_{\ell_1 + \ell_2 \leq \ell-1}
  {\big\| |\nabla|^{1/2} \underline{\Gamma}^{k'}\pa_z^{\ell_1}  f \big\|}_{L^2_z H^{r+\ell_2}}
  + {\big\| \underline{\Gamma}^{k'} \pa_z^{\ell_1}  f \big\|}_{L^\infty_z H^{r+\ell_2}};
\end{split}
\end{align}
these follow by repeatedly applying \eqref{dzids} to see that for $T \in\{T_1, T_2, T_3\}$
we can write $\pa_z^\ell T$ as a sum of terms of the form
\begin{equation}
 \pa_z^{\ell_1} |\nabla|^{\ell_2}, \qquad T |\nabla|^{\ell},
 \qquad \ell_1 + \ell_2 \leq \ell-1,
 \label{}
\end{equation}
and then applying \eqref{expboundsdzinfty}.

4. Finally, we remark that the first norm on the right-hand side of \eqref{expbounds3}
will be enough to control all the terms on the right-hand sides of \eqref{alphaifp}-\eqref{alpha3fp}
except the forcing term involving the (inverse gradient of the) vorticity
for which we need to use the second norm.
\end{remark}

We now proceed to estimate the map $L(\alpha)$ in \eqref{alphafp}-\eqref{alpha3b}
in the spaces $\dYn^n$ defined in \eqref{dotalphaspace0}-\eqref{dotalphaspace0'}.
We first prove \eqref{alphapropL}-\eqref{alphapropH}
by estimating the quantities arising from the boundary conditions \eqref{alphaifp}-\eqref{alpha3fp}
in \ref{alphapropssec1}, and the nonlinear bulk terms in \ref{alphapropssec2}.
In \ref{linforcesec} we control the forcing term.
Finally, in \ref{alphapropsseclast} we prove \eqref{alphapropL'}-\eqref{alphapropH'}.

\subsubsection{Estimate for the homogeneous terms}\label{alphapropssec1}
In view of the bounds \eqref{expbounds1}-\eqref{expbounds2} for $e^{z|\nabla|}$,
the fact that $\pa_z e^{z|\nabla|} = |\nabla| e^{z|\nabla|}$, and the definition of the space $\dYn^n$,
we have the estimate
\begin{align}
\label{expbounds1'}
{\big\| e^{z|\nabla|} f \big\|}_{\dYn^n}
  & \lesssim \sum_{\substack{r+k \leq n \\ a=0, 1}} {\big\| |\nabla|^{1/2 + a} f \big\|}_{Z^{r}_k}.
\end{align}

Using this we can bound 
\begin{align}\label{prbd1}
{\big\| e^{z|\nabla|} B_i \big\|}_{\dYn^n}
 & \lesssim \sum_{\substack{r+k \leq n \\ a=0, 1}}
 {\| |\nabla|^{1/2+a} B_i \|}_{Z^{r}_k}, \qquad B \in \{B_1,B_2,B_{3,a}\},
\\
\label{prbd3}
{\big\| |\nabla|^{-1} e^{z|\nabla|} B_{3,b} \big\|}_{\dYn^n}
  & \lesssim \sum_{\substack{r+k \leq n \\ a=0, 1}} {\| |\nabla|^{-1/2+a} B_{3,b} \|}_{Z^{r}_k}.
\end{align}

To get the needed estimates for $\alpha$,
we therefore want to estimate the right-hand sides of \eqref{prbd1}-\eqref{prbd3} and show the following:
\begin{align}\label{prbdmainL}
\begin{split}
\sum_{\substack{r+k \leq  N_1-10 \\ a=0, 1}} {\| |\nabla|^{1/2+a} (B_1,B_2,B_{3,a}) \|}_{Z^r_k}
 + \sum_{\substack{r+k \leq  N_1-10 \\ a=0, 1}} {\| |\nabla|^{-1/2+a} B_{3,b} \|}_{Z^r_k}
 \lesssim \e_0 {\| \alpha \|}_{\dYn^{N_1-10}},
\end{split}
\end{align}
and
\begin{align}\label{prbdmainH}
\begin{split}
\sum_{\substack{r+k \leq  N_1+12 \\ a=0, 1}} {\| |\nabla|^{1/2+a} (B_1,B_2,B_{3,a}) \|}_{Z^r_k}
  + \sum_{\substack{r+k \leq  N_1+12 \\ a=0, 1}} {\| |\nabla|^{-1/2+a} B_{3,b} \|}_{Z^r_k}
  \\
  \lesssim \e_0  {\| \alpha \|}_{\dYn^{N_1+12}} + \e_0 \jt^\delta  {\| \alpha \|}_{\dYn^{N_1-10}}.
\end{split}
\end{align}

\noindent
{\it Some reductions and useful estimates}.
From the definitions \eqref{Aidef}, \eqref{alpha3a} and \eqref{alpha3b}
we see that there are many terms that need to be estimated to prove \eqref{prbdmainL} and \eqref{prbdmainH}.
However, many of them are similar and they can all be written as linear combinations of simpler terms, as we now argue.
First, \eqref{Aidef} and \eqref{alpha3a} are all linear combinations of terms of the form
\begin{align}\label{prbd5}
b(\nabla h, \alpha) & := c(\nabla h) \alpha_j(0), \quad j=1,2,3,
\end{align}
with $c$ denoting a generic coefficient satisfying
\begin{align}
\label{prbd5cL}
& \sum_{\dg{r+k} \leq N_1-1}{\| c(\nabla h) \|}_{Z^{r,p}_k} \lesssim \e_0, \qquad p \geq 3, 
\\
\label{prbd5cH}
& \sum_{r+k \leq N_0-3} {\| c(\nabla h) \|}_{Z^{r,p}_k}\lesssim \e_0 \jt^\delta, \qquad p \geq 2.
\end{align}
Note that we have disregarded the Riesz transform $\nabla|\nabla|^{-1}$ in front of \eqref{alpha3a}
since this plays no role in the desired $L^2$-based estimates.
Also note that, for all practical purposes, one may think that $c = \nabla h$.

To verify \eqref{prbd5} with \eqref{prbd5cL}-\eqref{prbd5cH} we inspect \eqref{Aidef} and see that
$B_i$ is a linear combination of terms as in \eqref{prbd5} where the coefficients are
of the form $c(\nabla h) = (1+ |\nabla h|^2) \pa_i h$ and $\nabla h \, c(\nabla h)$;
using the product estimate \eqref{prodest0'} 
and the a priori assumptions \eqref{apriorihp},
we can verify directly that \eqref{prbd5cL} holds: for all $p\geq 11/5$
\begin{align*}
\sum_{r+k \leq N_1-1}{\| (1+ |\nabla h|^2) \pa_i h  \|}_{Z^{r,p}_k}
& \lesssim (1+\e_0^2) \sum_{r+k \leq N_1-1}{\| \pa_i h  \|}_{Z^{r,p}_k}
\lesssim \e_0
\end{align*}
Similarly, we can use also \eqref{apriorih} to verify \eqref{prbd5cH}:
\begin{align*}
\sum_{r+k \leq N_0-3} {\|(1+ |\nabla h|^2) \pa_i h  \|}_{Z^{r,p}_k}
& \lesssim \Big( 1 + \sum_{r+k \leq (N_0-3)/2} {\| \nabla h \|}^2_{Z^{r,\infty}_k}  \Big)
  \sum_{r+k \leq N_0-3}{\| \nabla h  \|}_{Z^{r,p}_k}
\\
& \lesssim (1 + \e_0^2) \e_0 \jt^\delta,
\end{align*}
having used \eqref{apriorihp} and $N_1 \geq N_0/2$.

Again omitting the Riesz transform, the term $B_{3,b}$ in \eqref{alpha3b}
 is a linear combination of terms of the type:
\begin{align}
\label{prbd6}
b_3(\nabla h,  \alpha) & := c(\nabla h) \,\nabla_{j} \alpha_k(0),
\end{align}
where $c_3$ denotes a generic coefficient satisfying
\begin{align}\label{prbd6cL}
& \sum_{r+k \leq N_1-3} {\| c_3(\nabla h) \|}_{Z^{r,p}_k}
  \lesssim \e_0 \jt^{-1 + (2/p)(1+\delta)},\qquad p \geq 2,
  \\
\label{prbd6cH}
& \sum_{r+k \leq N_0-5} {\| c_3(\nabla h) \|}_{Z^{r,p}_k} \lesssim \e_0 \jt^\delta, \qquad p \geq 2.
\end{align}
In fact each $c_3$ we consider is just a component of $\nabla h$
and so these bounds follow directly from \eqref{apriorih}-\eqref{apriorihp}.

In view of the above reductions, we see that in order to prove
the desired bounds \eqref{prbdmainL} and \eqref{prbdmainH}, 
it suffices to show that for coefficients $c, c_3$ satisfying the above bounds,
we have
\begin{align}
\label{prbdredL1}
& \sum_{\substack{r+k \leq  N_1-10 \\ a=0, 1}} {\| |\nabla|^{1/2+a}  c(\nabla h) \alpha(0) \|}_{Z^r_k}
  \lesssim \e_0 {\| \alpha \|}_{W^{N_1-10}},
\\
\label{prbdredL2}
& \sum_{\substack{r+k \leq  N_1-10 \\ a=0, 1}}
  {\| |\nabla|^{-1/2+a} c_3(\nabla h) \nabla_{x,z} \alpha(0) \|}_{Z^r_k}
  \lesssim \e_0 {\| \alpha \|}_{W^{N_1-10}},
\end{align}
and
\begin{align}
\label{prbdredH1}
& \sum_{\substack{r+k \leq  N_1+12 \\ a=0, 1}} {\| |\nabla|^{1/2+a}  c(\nabla h) \alpha(0) \|}_{Z^r_k}
  \lesssim \e_0 {\| \alpha \|}_{W^{N_1+12}} + \e_0 \jt^\delta  {\| \alpha \|}_{W^{N_1-10}},
\\
\label{prbdredH2}
& \sum_{\substack{r+k \leq  N_1+12 \\ a=0, 1}}
  {\| |\nabla|^{-1/2+a} c_3(\nabla h) \nabla_{x,z} \alpha(0) \|}_{Z^r_k}
  \lesssim \e_0 {\| \alpha \|}_{W^{N_1+12}} + \e_0 \jt^\delta  {\| \alpha \|}_{W^{N_1-10}}.
\end{align}

Before proving the above estimates, we record a simple but useful product estimate
that we are going to use repeatedly below:
\begin{align}\label{prodestuse0}
{\big\| |\nabla|^{1/2} (f g) \big\|}_{L^2} \lesssim
  {\| f \|}_{W^{1,3}} {\big\| |\nabla|^{1/2} g \big\|}_{L^2},
\end{align}
see Lemma \ref{lemprodestuse0}.
In what follows $g$ will essentially play the role of $\alpha(0)$, and $f$ will be nonlinear
expressions in $h$ and its derivatives.


\noindent
{\it Proof of \eqref{prbdredL1}.}
Distributing vector fields using also \eqref{comm0}, and applying the estimate \eqref{prodestuse0},
we can bound
\begin{align*}
& \sum_{\substack{r+k \leq  N_1-10 \\ a=0, 1}} {\| |\nabla|^{1/2+a}  c(\nabla h) \alpha(0) \|}_{Z^r_k}
\\
& \lesssim \sum_{\substack{r+k \leq  N_1-3}} {\| c(\nabla h) \|}_{Z^{r,3}_k}
  \sum_{\substack{r+k \leq  N_1-10 \\ a=0, 1}} {\| |\nabla|^{1/2+a} \alpha(0) \|}_{Z^r_k}
  \lesssim \e_0 {\| \alpha \|}_{\dYn^{N_1-10}},
\end{align*}
having used \eqref{prbd5cL} to control the coefficient.

\noindent

{\it Proof of \eqref{prbdredL2}.}
Due to the possibly singular factor of $|\nabla|^{-1/2}$, here we distinguish the cases $a=0$ and $a=1$.
If $a=0$ we first apply fractional integration followed by \eqref{prodest0'}:
\begin{align*}
\sum_{\substack{r+k \leq  N_1-10}}
  {\| |\nabla|^{-1/2} c_3(\nabla h) \nabla_{x,z} \alpha(0) \|}_{Z^r_k}
  \lesssim \sum_{\substack{r+k \leq  N_1-10}}
  {\| c_3(\nabla h) \nabla_{x,z} \alpha(0) \|}_{Z^{r,4/3}_k}
  \\
  \lesssim \sum_{\substack{r+k \leq  N_1-10}} {\| c_3(\nabla h)  \|}_{Z^{r,4}_k}
  \sum_{\substack{r+k \leq  N_1-10}} {\| \nabla_{x,z} \alpha(0) \|}_{Z^{r}_k}
  \\
  \lesssim \e_0 {\| \alpha \|}_{\dYn^{N_1-10}},
\end{align*}
having used \eqref{prbd6cL} for the coefficient, and \eqref{alphaspacebd'full}.
When $a=1$ we use \eqref{prodestuse0}:
\begin{align*}
& \sum_{\substack{r+k \leq  N_1-10}}
  {\| |\nabla|^{1/2} c_3(\nabla h) \nabla_{x,z} \alpha(0) \|}_{Z^r_k}
  \\
  & \lesssim \sum_{\substack{r+k \leq  N_1-4}} {\| c_3(\nabla h)  \|}_{Z^{r,3}_k}
  \sum_{\substack{r+k \leq  N_1-10}} {\| |\nabla|^{1/2} \nabla_{x,z} \alpha(0) \|}_{Z^r_k}
  \lesssim \e_0 {\| \alpha \|}_{\dYn^{N_1-10}}.
\end{align*}

\noindent
{\it Proof of \eqref{prbdredH1}.}
Distributing vector fields we can estimate
\begin{align*}
& \sum_{r+k \leq  N_1+12} {\| |\nabla|^{1/2+a}  c(\nabla h) \alpha(0) \|}_{Z^r_k}
  \\
  & \lesssim \sum_{\substack{|r_1|+|k_1| \leq n_1 \\ |r_2|+|k_2| \leq n_2}}
  {\big\| |\nabla|^{1/2+a} \big( \nabla^{r_1} \Gamma^{k_1} c(\nabla h) \, \nabla^{r_2} \Gamma^{k_2} \alpha(0) \big)
  \big\|}_{L^2} := M_{n_1,n_2},
\end{align*}
where $n_1+n_2 = N_1+12$, and we do not make explicit the dependence on $a=0,1$ which is unimportant here.
%
%
We distinguish two cases depending which of the indexes $n_1$ and $n_2$ is smaller.
If $n_1 \leq N_1-15$ we use \eqref{prodestuse0},
\begin{align*}
M_{n_1,n_2} \lesssim \sum_{|r_1|+|k_1| \leq N_1-13}
  {\| \nabla^{r_1} \Gamma^{k_1} c(\nabla h) \|}_{L^3}
  \sum_{\substack{|r_2|+|k_2| \leq N_1+12 \\ a=0, 1}}
  {\| |\nabla|^{1/2+a} \nabla^{r_2} \Gamma^{k_2} \alpha(0) \|}_{L^2}
  \\
  \lesssim \e_0 {\| \alpha \|}_{\dYn^{N_1+12}}
\end{align*}
having used \eqref{prbd6cL} to estimate the coefficient.
If instead $n_2 \leq N_1-15$, using again \eqref{prodestuse0}, followed by \eqref{prbd6cH}, we get
\begin{align*}
M_{n_1,n_2} \lesssim \sum_{|r_1|+|k_1| \leq N_1+13}
  {\| \nabla^{r_1} \Gamma^{k_1} c(\nabla h) \|}_{L^3}
  \sum_{\substack{|r_2|+|k_2| \leq N_1-15 \\ a=0, 1}}
  {\| |\nabla|^{1/2+a} \nabla^{r_2} \Gamma^{k_2} \alpha(0) \|}_{L^2}
\\
\lesssim \e_0 \jt^\delta \cdot {\| \alpha \|}_{\dYn^{N_1-10}}.
\end{align*}
These last two bounds above give \eqref{prbdredH1}.

\noindent
{\it Proof of \eqref{prbdredH2}.}
Distributing vector fields we have, for $a=0,1$,
\begin{align*}
& \sum_{r+k \leq  N_1+12}
  {\big\| |\nabla|^{-1/2+a} \big( c_3(\nabla h) \nabla_{x,z} \alpha(0) \big) \big\|}_{Z^r_k}
  \\
  & \lesssim \sum_{\substack{|r_1|+|k_1| \leq n_1 \\ |r_2|+|k_2| \leq n_2}}
  {\big\| |\nabla|^{-1/2+a} \big( \nabla^{r_1} \Gamma^{k_1} c_3(\nabla h) \cdot
  \nabla^{r_2} \Gamma^{k_2} \nabla_{x,z} \alpha(0) \big) \big\|}_{L^2}
  := M_{n_1,n_2}^a,
\end{align*}
where $n_1+n_2 = N_1+12$.

We look at the case $a=0$ first, apply fractional integration as before and then H\"older
to bound first
\begin{align*}
M_{n_1,n_2}^0 \lesssim \sum_{|r_1|+|k_1| \leq n_1}
  {\| \nabla^{r_1} \Gamma^{k_1} c_3(\nabla h) \|}_{L^4}
  \sum_{|r_2|+|k_2| \leq n_2} {\| \nabla^{r_2} \Gamma^{k_2} \nabla_{x,z} \alpha(0) \|}_{L^2};
\end{align*}
then, when $n_1 \leq N_1-15$ we use \eqref{prbd6cL} and \eqref{alphaspacebd'full}
to obtain $M_{n_1,n_2}^0 \lesssim \e_0 {\| \alpha \|}_{W^{N_1+12}}$;
when, instead, $n_2 \leq N_1-15$ we use \eqref{prbd6cH} to obtain
$M_{n_1,n_2}^0 \lesssim \e_0 \jt^\delta \cdot {\| \alpha \|}_{\dYn^{N_1-10}}$.

In the case $a=1$ we can use the product estimate \eqref{prodestuse0} to see that
\begin{align*}
M_{n_1,n_2}^1 \lesssim \sum_{|r_1|+|k_1| \leq n_1+1}
  {\| \nabla^{r_1} \Gamma^{k_1} c_3(\nabla h) \|}_{L^3}
  \sum_{|r_2|+|k_2| \leq n_2} {\| |\nabla|^{1/2} \nabla^{r_2} \Gamma^{k_2} \nabla_{x,z} \alpha(0) \|}_{L^2};
\end{align*}
then, for $n_1 \leq N_1-15$ we use \eqref{prbd6cL} and \eqref{alphaspacebd'full} to bound
$M_{n_1,n_2}^1 \lesssim \e_0 {\| \alpha \|}_{W^{N_1+12}}$,
and for $n_2 \leq N_1-15$ we use \eqref{prbd6cH}  and \eqref{alphaspacebd'full} to get
$M_{n_1,n_2}^1 \lesssim \e_0 \jt^\delta \cdot {\| \alpha \|}_{\dYn^{N_1-10}}$.
This concludes the proof
the bounds \eqref{prbdmainL} and \eqref{prbdmainH}.

\subsubsection{Bounds for the nonlinear bulk terms}\label{alphapropssec2}
To estimate the nonlinear expressions in the bulk
integrals on the right-hand side of \eqref{alphaifp}-\eqref{alpha3fp} we proceed similarly as above,
this time using the bounds in Lemma \ref{explem} and Remark \ref{remexp} first,
and then product estimates in weighted spaces.
Define
\begin{align}\label{alphapropbulk}
\begin{split}
N_i^a(\alpha)(z) & := \int_{-\infty}^0 \big( e^{(z+s)|\nabla|}
  - e^{-|z-s||\nabla|} \mathrm{sign}(z-s) \big) E^a_i(s) \, ds,
\\
N_i^b(\alpha)(z) & := \int_{-\infty}^0 \big( e^{(z+s)|\nabla|} - e^{-|z-s||\nabla|} \big) E^b_i(s)\, ds,
  \qquad i=1,2,3,
\end{split}
\end{align}
with $E^a$ and $E^b$ defined in \eqref{EaEb}.
We then want to show, for $i=1,2,3$,
\begin{align}
\label{alphapropLbulk}
{\| N_i^\ast (\alpha) \|}_{\dYn^{N_1-10}} & \lesssim \e_0 {\| \alpha \|}_{\dYn^{N_1-10}},
\\
\label{alphapropHbulk}
{\| N_i^\ast (\alpha) \|}_{\dYn^{N_1+12}} & \lesssim \e_0 {\| \alpha \|}_{\dYn^{N_1+12}}
  + \e_0 \jt^{\delta} {\| \alpha \|}_{\dYn^{N_1-10}}, \qquad \ast \in \{a,b\},
\end{align}
consistently with \eqref{alphapropL} and \eqref{alphapropH}.

We start by noting that the $N_i^\ast$
can be written in terms of the operators $T_1, T_2, T_3$ in \eqref{Tidefs}: 
\begin{equation*}\label{}
 N_i^a(\alpha) = T_1(E^a) - T_2(E^a) + T_3(E^a),
 \qquad
 N_i^b(\alpha) =\dg{-} T_1(E^b) \dg{-} T_2(E^b) + T_3(E^b).
\end{equation*}
Then, from the definition of the $\dYn^n$ norm in \eqref{dotalphaspace0}-\eqref{dotalphaspace0'}
the estimates \eqref{expbounds3}, and \eqref{expboundsdz3}-\eqref{expboundsdz3'},
for $\ast \in\{a,b\}$, we have
\begin{align}
\nonumber
& {\big\| N_i^\ast(\alpha) \|}_{\dYn^n}
  \\
  \nonumber
  & = \sum_{ \substack{|r|+|k| \leq n \\ 0\leq a \leq 1}}
  {\big\| \nabla_{x,z}^a |\nabla|^{1/2} \underline{\Gamma}^k \nabla^r_{x,z} N_i^\ast(\alpha) \big\|}_{L^\infty_zL^2_x}
  + {\big\|\nabla_{x,z}^a
  \nabla_{x,z} \underline{\Gamma}^k \nabla^r_{x,z} N_i^\ast(\alpha) \big\|}_{L^2_zL^2_x}
  \\ \nonumber
  & + {\big\| \partial_{z} \underline{\Gamma}^k \nabla_{x,z}^r N_i^\ast(\alpha) \big\|}_{L^\infty_zL^2_x}
\\
\label{alphapropbulkest1}
& \lesssim \sum_{ \substack{|r|+|k| \leq n \\ a = 0,1}}
  {\big\| |\nabla|^a \, \underline{\Gamma}^k \nabla^r_{x,z} E^\ast_i \big\|}_{L^2_zL^2_x}
+ \sum_{|r|+|k| \leq n}
  {\big\| \jnab^{1/2} \underline{\Gamma}^k \nabla^r_{x,z} E^\ast_i \big\|}_{L^\infty_zL^2_x}.
\end{align}
Therefore, in view of the definitions \eqref{EaEb},
and the commutation identity \eqref{comm0} to handle the Riesz transform in front of $E^a$,
for \eqref{alphapropLbulk}-\eqref{alphapropHbulk} it suffices to prove the following  bounds
\begin{align}
\label{alphapropLbulk'}
  \sum_{ \substack{|r|+|k| \leq N_1-10 \\ |\ell| \leq 1}}
  {\big\| |\nabla|^\ell 
  \underline{\Gamma}^k \nabla_{x,z}^r c(\nabla h) \nabla_{x,z} \alpha \big\|}_{L^2_zL^2_x}
  & \lesssim \e_0 {\| \alpha \|}_{\dYn^{N_1-10}},
\\
\label{alphapropHbulk'}
\sum_{ \substack{|r|+|k| \leq N_1+12 \\ |\ell| \leq 1}}
  {\big\| |\nabla|^\ell \underline{\Gamma}^k \nabla_{x,z}^r c(\nabla h) \nabla_{x,z} \alpha \big\|}_{L^2_zL^2_x}
  & \lesssim \e_0 {\| \alpha \|}_{\dYn^{N_1+12}} + \e_0 \jt^{\delta} {\| \alpha \|}_{\dYn^{N_1-10}},
\end{align}
and
\begin{align}
\label{alphapropLbulk''}
  \sum_{|r|+|k| \leq N_1-10}
  {\big\| \jnab^{1/2}
  \underline{\Gamma}^k \nabla_{x,z}^r c(\nabla h) \nabla_{x,z} \alpha \big\|}_{L^\infty_zL^2_x}
  & \lesssim \e_0 {\| \alpha \|}_{\dYn^{N_1-10}},
\\
\label{alphapropHbulk''}
\sum_{|r|+|k| \leq N_1+12}
  {\big\| \jnab^{1/2} \underline{\Gamma}^k \nabla_{x,z}^r c(\nabla h) \nabla_{x,z} \alpha \big\|}_{L^\infty_zL^2_x}
  & \lesssim \e_0 {\| \alpha \|}_{\dYn^{N_1+12}} + \e_0 \jt^{\delta} {\| \alpha \|}_{\dYn^{N_1-10}},
\end{align}
where $c(\nabla h)$ is a component of $\nabla h$ or is $|\nabla h|^2$ so that,
in particular, it satisfies
\begin{align}
\label{cbulkL}
& \sum_{r+k \leq N_1-1}{\| c(\nabla h) \|}_{Z^{r,p}_k} \lesssim \e_0 \jt^{-1 + (2/p)(1+\delta)},
\\
\label{cbulkH}
& \sum_{r+k \leq N_0-3} {\| c(\nabla h) \|}_{Z^{r,p}_k}\lesssim \e_0 \jt^\delta, \qquad p \geq 2,
\end{align}
in view of \eqref{apriorih}-\eqref{apriorihp}.

\noindent
{\it Proof of \eqref{alphapropLbulk'}.}
Distributing vector fields and using H\"older we can simply bound
the left-hand side of \eqref{alphapropLbulk'} by
\begin{align*}
& \sum_{|r|+|k| \leq N_1-5}
  {\big\| \jnab \Gamma^k \nabla^r c(\nabla h) \big\|}_{L^\infty_x}
  \sum_{|r|+|k| \leq N_1-5}
  {\big\| \jnab \underline{\Gamma}^k \nabla^r \nabla_{x,z} \alpha \big\|}_{L^2_zL^2_x}
  \lesssim \e_0 {\| \alpha \|}_{\dYn^{N_1-5}}
\end{align*}
in view of \eqref{cbulkL} and the definition of the $\dot{Y}^n$ norm.

\noindent
{\it Proof of \eqref{alphapropHbulk'}.}
Distributing vector fields we see that the left-hand side of \eqref{alphapropHbulk'} is bounded by
the terms
\begin{align}\label{Hbulk1}
\sum_{\substack{|r_1|+|k_1| \leq n_1 \\ |r_2|+|k_2| \leq n_2}}
  {\big\| \nabla_x^\ell \big( \Gamma^{k_1} \nabla^{r_1} c(\nabla h)
  \cdot \underline{\Gamma}^{k_2} \nabla^{r_2}_{x,z} \nabla_{x,z} \alpha \big) \big\|}_{L^2_{x,z}}
  := B_{n_1,n_2}^\ell,
\end{align}
with $n_1+n_2 = N_1+12$ and $\ell=0,1$.
In the case $n_1 \leq N_1-15$ we can bound \eqref{Hbulk1} by
\begin{align*}
B_{n_1,n_2}^\ell \lesssim \sum_{|r_1|+|k_1| \leq N_1-14}
  {\big\| \Gamma^{k_1} \nabla^{r_1} c(\nabla h) \big\|}_{L^\infty_x}
  \sum_{|r_2|+|k_2| \leq N_1+12}
  {\big\| \jnab \underline{\Gamma}^{k_2} \nabla^{r_2}_{x,z} \nabla_{x,z} \alpha \big\|}_{L^2_{x,z}}
  \\
  \lesssim \e_0 {\| \alpha \|}_{\dYn^{N_1+12}},
\end{align*}
having used \eqref{cbulkL}.
When instead $n_2 \leq N_1-15$ we can bound similarly
\begin{align*}
B_{n_1,n_2}^\ell \lesssim \sum_{|r_1|+|k_1| \leq N_1+12}
  {\big\| \Gamma^{k_1} \nabla^{r_1} c(\nabla h) \big\|}_{L^\infty_x}
  \sum_{|r_2|+|k_2| \leq N_1-10}
  {\big\| \underline{\Gamma}^{k_2} \nabla^{r_2}_{x,z} \nabla_{x,z} \alpha \big\|}_{L^2_{x,z}}
  \\
  \lesssim \e_0 \jt^{\delta} {\| \alpha \|}_{\dYn^{N_1-10}},
\end{align*}
having used \eqref{cbulkH}.

\noindent
{\it Proof of \eqref{alphapropLbulk''}.}
Distributing vector fields and using \eqref{prodestuse0} we can bound
the left-hand side of \eqref{alphapropLbulk''} by
\begin{align*}
& \sum_{|r|+|k| \leq N_1-9}
  {\big\| \Gamma^k \nabla^r c(\nabla h) \big\|}_{L^3_x}
  \sum_{|r|+|k| \leq N_1-10}
  {\big\| \jnab^{1/2} \underline{\Gamma}^k \nabla^r_{x,z} \nabla_{x,z} \alpha \big\|}_{L^\infty_zL^2_x}
  \lesssim \e_0 {\| \alpha \|}_{\dYn^{N_1-10}}
\end{align*}
where we have used \eqref{alphaspacebd'full} for the last inequality.


\noindent
{\it Proof of \eqref{alphapropHbulk''}.}
Finally, we examine the left-hand side of \eqref{alphapropHbulk''},
distribute the vector fields and estimate it, using again \eqref{prodestuse0},
by a linear combination of the terms
\begin{align}\label{Hbulk1'}
\begin{split}
& \sum_{\substack{r_1+|k_1| \leq n_1 \\ r_2+|k_2| \leq n_2}}
  {\big\| \jnab^{1/2} \big( \Gamma^{k_1} \nabla^{r_1} c(\nabla h)
  \cdot \underline{\Gamma}^{k_2} \nabla^{r_2}_{x,z} \nabla_{x,z} \alpha \big) \big\|}_{L^\infty_z L^2_x}
  \\
  & \lesssim \sum_{r_1+|k_1| \leq n_1+1}
  {\big\| \Gamma^{k_1} \nabla^{r_1} c(\nabla h) \big\|}_{L^3_x}
  \sum_{r_2+|k_2| \leq n_2}
  {\big\| \jnab^{1/2} \underline{\Gamma}^{k_2} \nabla^{r_2}_{x,z} \nabla_{x,z} \alpha \big\|}_{L^\infty_zL^2_x}
  := C_{n_1,n_2}.
\end{split}
\end{align}
Then, in the case $n_1 \leq N_1-15$ we can bound \eqref{Hbulk1'} as follows:
\begin{align*}
C_{n_1,n_2} \lesssim \sum_{|r_1|+|k_1| \leq N_1-14}
  {\big\| \Gamma^{k_1} \nabla^{r_1} c(\nabla h) \big\|}_{L^3_x}
  \sum_{|r_2|+|k_2| \leq N_1+12}
  {\big\| \jnab^{1/2} \underline{\Gamma}^{k_2} \nabla^{r_2}_{x,z} \nabla_{x,z} \alpha \big\|}_{L^\infty_z L^2_x}
  \\
  \lesssim \e_0 {\| \alpha \|}_{\dYn^{N_1+12}},
\end{align*}
having used \eqref{cbulkL}.
When instead $n_2 \leq N_1-15$ we estimate
\begin{align*}
C_{n_1,n_2} \lesssim \sum_{|r_1|+|k_1| \leq N_1+12}
  {\big\| \Gamma^{k_1} \nabla^{r_1} c(\nabla h) \big\|}_{L^3_x}
  \sum_{|r_2|+|k_2| \leq N_1-10}
  {\big\| \jnab^{1/2} \underline{\Gamma}^{k_2} \nabla^{r_2}_{x,z} \nabla_{x,z} \alpha \big\|}_{L^\infty_z L^2_x}
  \\
  \lesssim \e_0 \jt^{\delta} {\| \alpha \|}_{\dYn^{N_1-10}},
\end{align*}
having used \eqref{cbulkH} and \eqref{alphaspacebd'full}.

\subsubsection{Bounds for the linear forcing terms}\label{linforcesec}
We now estimate the forcing term involving the vorticity on the right-hand side of \eqref{alphaifp}-\eqref{alpha3fp}.
These are (up to constants) given by
\begin{align}\label{fpF1}
T_3F = \int_{-\infty}^0 e^{(z+s)|\nabla|} |\nabla|^{-1} F(s)\, ds,
  \quad \mbox{or} \quad (T_1+T_2)F = \int_{-\infty}^0 e^{-|z-s||\nabla|} |\nabla|^{-1} F(s)\, ds,
\end{align}
recall $F = W$ and the notation in \eqref{Tidefs}.
From the identities \eqref{dzids} we have
\begin{equation}\label{pazG}
\pa_z T_3 = |\nabla| T_3, \qquad \pa_z (T_1+T_2) = |\nabla| (T_1 - T_2).
\end{equation}

Let $P_{> 0}$ and $P_{\leq 0}$ denote the standard Littlewood-Paley projections (in the $x$ variable)
defined according to \eqref{cut0}-\eqref{LP0}, and note that they commute with the operators in \eqref{fpF1} above.
Let us denote by $G$ any of the two expressions in \eqref{fpF1}.
Using the bound in \eqref{expbounds3} by the first argument on the right-hand side,
recalling the definition \eqref{omegaflatspace2} and \eqref{dotalphaspace0}-\eqref{dotalphaspace0'}, and
using Remark \ref{remexp} and 
\eqref{pazG} to handle the $z$-derivatives, we have
\begin{align}\label{fpFhigh}
\begin{split}
{\big\| P_{>0} G \|}_{\dYn^n} \lesssim
\sum_{\dg{|r|+|k|} \leq n} {\big\| \jnab \underline{\Gamma}^{k} \nabla_{x,z}^r 
  P_{> 0} (|\nabla|^{-1} F(s)) \big\|}_{L^2_z L^2_x}
  \lesssim {\| W \|}_{\mX^n}.
\end{split}
\end{align}

For small frequencies, we instead use the bound in \eqref{expbounds3} by the second argument on the right-hand side,
followed by fractional integration, and obtain 
\begin{align}\label{fpFlow}
\begin{split}
{\big\| P_{\leq 0} G \|}_{\dYn^n} & \lesssim
\sum_{\dg{|r|+|k|} \leq n}
  {\big\| \jnab \underline{\Gamma}^k \nabla^r_{x,z} \, |\nabla|^{1/3} P_{\leq 0} (|\nabla|^{-1} F(s)) \big\|}_{L^{6/5}_z L^2_x}
  \\
                      & \lesssim \sum_{\dg{|r|+|k|} \leq n}
  {\big\| \underline{\Gamma}^k \nabla^r_{x,z} \,|\nabla|^{-2/3} W \big\|}_{L^{6/5}_z L^2_x}
  \\
                      & \lesssim \sum_{\dg{|r|+|k|} \leq n} {\big\| \underline{\Gamma}^k\nabla^r_{x,z} W \big\|}_{L^{6/5}_{x,z}}
  \lesssim {\| W \|}_{\mX^n}.
\end{split}
\end{align}
Using \eqref{fpFhigh} and \eqref{fpFlow} with $n=N_1-10$ 
and $n=N_1+12$,
we get bounds 
consistent with the desired inequalities \eqref{alphapropL} and \eqref{alphapropH}.
The proof of \eqref{alphapropL}-\eqref{alphapropH} is thus concluded.

\subsubsection{Proof of the bounds \eqref{alphapropL'}-\eqref{alphapropH'}}\label{alphapropsseclast}
We now prove the bounds for the time derivatives of the map $L(\alpha)$ in Lemma \ref{flatlem}.
These can be obtained in the same way as the bounds \eqref{alphapropL}-\eqref{alphapropH} proved above,
using in addition the bounds on $\partial_t h$ from \eqref{aprioridth2} and \eqref{aprioridthp}.
We give some details for completeness.
Let us define the map $\dot{L}$ through the identity
\begin{align}\label{dtL}
\partial_t L(\alpha) = L(\partial_t \alpha) + \dot{L}(\alpha);
\end{align}
Under the assumptions \eqref{vpboundhi}-\eqref{vpboundlo} the same exact arguments above give 
\begin{align}\label{Ldtbounds}
\begin{split}
{\| L(\partial_t \alpha) \|}_{\dYn^{N_1-11}} & \lesssim \e_0 {\| \partial_t \alpha \|}_{\dYn^{N_1-11}}
  + {\| \partial_t W \|}_{\mX^{N_1-11}},
\\
{\| L(\partial_t \alpha) \|}_{\dYn^{N_1+11}} & \lesssim \e_0 {\| \partial_t \alpha \|}_{\dYn^{N_1+11}}
  + \e_0 \jt^{\delta} {\| \partial_t \alpha \|}_{\dYn^{N_1-11}}
  + {\| \partial_t W \|}_{\mX^{N_1+11}}.
\end{split}
\end{align}
Therefore, it suffices to prove that
\begin{align}\label{dotLbounds}
\begin{split}
{\| \dot{L}(\alpha) \|}_{\dYn^{N_1-11}} & \lesssim \e_0 {\| \alpha \|}_{\dYn^{N_1-11}},
\\
{\| \dot{L}(\alpha) \|}_{\dYn^{N_1+11}} & \lesssim \e_0 {\| \alpha \|}_{\dYn^{N_1+11}}
  + \e_0 \jt^{\delta} {\| \alpha \|}_{\dYn^{N_1-10}}.
\end{split}
\end{align}
By definition the map $\dot{L}$ is given by the right-hand side of \eqref{alphaifp} and \eqref{alpha3fp}
with $F=0$, and where we replace $(E_i^a,E_i^b)$ and $(B_1,B_2,B_{3,a},B_{3,b})$ by
new quantities $(\dot{E}_i^a,\dot{E}_i^b)$ and $(\dot{B}_1,\dot{B}_2,\dot{B}_{3,a},\dot{B}_{3,b})$
defined by differentiating the coefficients that multiply $\alpha$; in other words,
for $X(\alpha) = E_i^a(\alpha), E_i^b(\alpha), B_1(\alpha)$ and so on, we define
\begin{align*}
\dot{X}(\alpha) = \partial_t X(\alpha) - X(\partial_t \alpha).
\end{align*}
More explicitly, we have
\begin{align}\label{Edot}
\begin{split}
\dot{E}^a(\alpha) & := \frac{\nabla}{|\nabla |} \cdot (\partial_t \nabla h ) \, \pa_z \alpha
\qquad \dot{E}^b(\alpha) := \big( - \partial_t |\nabla h|^2 \big) \, \pa_z\alpha + (\partial_t \nabla h) \cdot \nabla \alpha,
\end{split}
\end{align}
and
\begin{align}\label{Bdot}
\begin{split}
\dot{B}_i(\alpha) & := \partial_t \big[ (1+ |\nabla h|^2) \pa_i h \big]  \alpha_3(0)
  - \partial_t \big[ (1+ |\nabla h|^2) \pa_i h \nabla h \big] \cdot \alpha(0), \qquad i = 1,2,
\\
\dot{B}_{3,a}(\alpha) & := \frac{\nabla}{|\nabla|} \cdot
 \partial_t \big( \nabla h (1+ |\nabla h|^2)^{-1/2} \big) \, \alpha_3(0).
\end{split}
\end{align}
For the last boundary term, that is, $\dot{B}_{3,b}(\alpha) = \partial_t B_{3,b} (\alpha) - B_{3,b} (\partial_t \alpha)$,
we use its schematic representation from \eqref{prbd6}
 to write it as a linear combination
of terms of the form
\begin{align}
\label{prbd6dt}
b_3(\nabla h, \nabla^2h, \alpha) & := \partial_t c_3(\nabla h) \,\nabla_{x,z} \alpha(0),
\end{align}
with the natural definitions of the coefficients $c_3$ according to the formula \eqref{alpha3b}.
In particular, we can verify that the following analogues of \eqref{prbd5}-\eqref{prbd6cH} hold:

\noindent
$-$ The terms $\dot{B}_1, \dot{B}_2$ and $\dot{B}_{3,a}$ are linear combinations of terms of the form
\begin{align}\label{prbd5dt}
\partial_t c(\nabla h) \, \alpha_j(0), \quad j=1,2,3,
\end{align}
with
\begin{align}
\label{prbd5cLdt}
& \sum_{r+k \leq N_1-6}{\| \partial_t c(\nabla h) \|}_{Z^{r,p}_k} \lesssim \e_0, \qquad p \geq 3,
\\
\label{prbd5cHdt}
& \sum_{r+k \leq N_0-8} {\| \partial_t c(\nabla h) \|}_{Z^{r,p}_k}\lesssim \e_0 \jt^\delta, \qquad p \geq 2;
\end{align}
to check the above bounds one can just
proceed as in the proofs of \eqref{prbd5cL}-\eqref{prbd5cH}, using Lemma \ref{lemprod},
and also the bounds for $\partial_t h$ in \eqref{aprioridth2}-\eqref{aprioridthp}
besides the usual \eqref{apriorih}-\eqref{apriorihp}.

\noindent
$-$ The term $\dot{B_{3,b}}(\alpha)$ is a linear combination of terms of the form \eqref{prbd6dt}
with
\begin{align}
\label{prbd6cLdt}
& \sum_{r+k \leq N_1-8} {\| \partial_t c_3(\nabla h) \|}_{Z^{r,p}_k}
  \lesssim \e_0 \jt^{-3/4 + (2/p)(1+\delta)}, \qquad p\geq 3,
  \\
\label{prbd6cHdt}
& \sum_{r+k \leq N_0-10} {\| \partial_t c_3(\nabla h) \|}_{Z^{r,p}_k} \lesssim \e_0 \jt^\delta, \qquad p \geq 2;
\end{align}
the above bounds are analogous to \eqref{prbd6cL} and \eqref{prbd6cH} and can be obtained in the same
way using in addition \eqref{aprioridth2}-\eqref{aprioridthp}.
Similarly to before, for all practical purposes one may think that $c$ and $c_3 $
are both just $\nabla h$.

With the formulas \eqref{Edot}, \eqref{prbd5dt} and \eqref{prbd6dt},
and the estimate \eqref{prbd5cLdt}-\eqref{prbd5cHdt} and \eqref{prbd6cLdt}-\eqref{prbd6cHdt}
we can then proceed in a way completely analogous to Subsections \ref{alphapropssec1} and \ref{alphapropssec2}.
More precisely, as in Subsection \ref{alphapropssec1} we can use \eqref{expbounds1},
and reduce matters to showing the analogues of \eqref{prbdmainL}-\eqref{prbdmainH} for the dotted quantities, that is,
we want to show
\begin{align}\label{prbdmainLdot}
\begin{split}
\sum_{\substack{r+k \leq  N_1-11 \\ a=0, 1}} {\| |\nabla|^{1/2+a} (\dot{B}_1,\dot{B}_2,\dot{B}_{3,a}) \|}_{Z^r_k}
 + \sum_{\substack{r+k \leq  N_1-11 \\ a=0, 1}} {\| |\nabla|^{-1/2+a} \dot{B}_{3,b} \|}_{Z^r_k}
 \\
 \lesssim \e_0 {\| \alpha \|}_{\dYn^{N_1-11}},
\end{split}
\end{align}
and
\begin{align}\label{prbdmainHdot}
\begin{split}
\sum_{\substack{r+k \leq  N_1+11 \\ a=0, 1}} {\| |\nabla|^{1/2+a} (\dot{B}_1,\dot{B}_2,\dot{B}_{3,a}) \|}_{Z^r_k}
  + \sum_{\substack{r+k \leq  N_1+11 \\ a=0, 1}} {\| |\nabla|^{-1/2+a} \dot{B}_{3,b} \|}_{Z^r_k}
  \\
  \lesssim \e_0  {\| \alpha \|}_{\dYn^{N_1+11}} + \e_0 \jt^\delta  {\| \alpha \|}_{\dYn^{N_1-11}}.
\end{split}
\end{align}

The proofs of \eqref{prbdmainLdot} and \eqref{prbdmainHdot}
can then be obtained in the same exact way as the proof of the bounds \eqref{prbdmainL} and \eqref{prbdmainH},
through analogues of \eqref{prbdredL1}-\eqref{prbdredH2} where the coefficients $c$ and $c_3$
are replaced by $\partial_tc$ and $\partial_tc_3$,
and using \eqref{prbd5cLdt}-\eqref{prbd5cHdt} and \eqref{prbd6cLdt}-\eqref{prbd6cHdt}; we omit the details.

Proceeding as in Subsection \ref{alphapropssec2} we can let
\begin{align}\label{alphapropbulkdt}
\begin{split}
\dot{N}_i(\alpha)(z) & := \int_{-\infty}^0 e^{(z+s)|\nabla|}(\dot{E}^a_i(s) - \dot{E}^b_i(s)) \, ds
  \\
  & + \int_{-\infty}^0
  e^{-|z-s||\nabla|} (\mathrm{sign}(s-z) \dot{E}^a_i(s) - \dot{E}^b_i(s))\, ds, \qquad i=1,2,3,
\end{split}
\end{align}
and obtain the analogues of \eqref{alphapropLbulk} and \eqref{alphapropHbulk}
using the estimates in Lemma \ref{explem} and Remark \ref{remexp},
the product estimate from Lemma \ref{lemprod}, \eqref{prodestuse0},
and the bounds for $h$ and $\partial_th$ in \eqref{apriorihp} and \eqref{aprioridthp}:
\begin{align*}
{\| \dot{N}_i(\alpha) \|}_{\dYn^{N_1-11}} & \lesssim \e_0 {\| \alpha \|}_{\dYn^{N_1-11}},
\\
{\| \dot{N}_i(\alpha) \|}_{\dYn^{N_1+11}} & \lesssim \e_0 {\| \alpha \|}_{\dYn^{N_1+11}}
  + \e_0 \jt^{\delta} {\| \alpha \|}_{\dYn^{N_1-11}}.
\end{align*}
These are consistent with the desired \eqref{dotLbounds},
and conclude the proof of \eqref{alphapropL'}-\eqref{alphapropH'}. 

\subsubsection{Conclusion}
Finally, we show how the main bounds \eqref{alphapropL}-\eqref{alphapropH}
and \eqref{alphapropL'}-\eqref{alphapropH'} imply Proposition \ref{alphaprop}.
Consider the sequence
\begin{equation*}
\alpha^{(0)} = 0, \qquad \alpha^{(m+1)} = L(\alpha^{(m)}), \qquad m \geq 0.
\end{equation*}
so that, under the assumptions \eqref{alphaaso0}-\eqref{alphaaso1},
the bounds \eqref{alphapropL}--\eqref{alphapropH} read
\begin{align}\label{seq1}
\begin{split}
{\| \alpha^{(m+1)} \|}_{\dYn^{N_1-10}} & \lesssim \e_0 {\| \alpha^{(m)} \|}_{\dYn^{N_1-10}} + \e_1
\\
{\| \alpha^{(m+1)} \|}_{\dYn^{N_1+12}} & \lesssim \e_0 {\| \alpha^{(m)} \|}_{\dYn^{N_1+12}}
  + \e_0 \jt^{\delta} {\| \alpha^{(m)} \|}_{\dYn^{N_1-10}} + \e_1 \jt^\delta,
\end{split}
\end{align}
and \eqref{alphapropL'}--\eqref{alphapropH'} read
\begin{align}\label{seq1'}
\begin{split}
{\| \partial_t \alpha^{(m+1)} \|}_{\dYn^{N_1-11}} & \lesssim \e_0
  \big( {\| \alpha^{(m)} \|}_{\dYn^{N_1-10}} + {\| \partial_t \alpha^{(m)} \|}_{\dYn^{N_1-11}} \big) + \e_1
\\
{\| \partial_t \alpha^{(m+1)} \|}_{\dYn^{N_1+11}} & \lesssim \e_0\big(
{\| \partial_t \alpha^{(m)} \|}_{\dYn^{N_1+11}} \dg{+ \|\alpha^{(m)}\|_{\dYn^{N_1+12}}}\big) 
  + \e_0 \jt^{\delta} \big( {\|\alpha^{(m)} \|}_{\dYn^{N_1-10}} + {\| \partial_t^j \alpha^{(m)} \|}_{\dYn^{N_1-11}} \big)
  + \e_0 \e_1 \jt^\delta.
\end{split}
\end{align}
From these we see that, for $\e_0$ small enough, and all $t\leq T$,
\begin{align}\label{seq2}
\begin{split}
{\| \alpha^{(m)}(t) \|}_{\dYn^{N_1-10}} & \lesssim \e_1,
  \qquad {\| \alpha^{(m)}(t) \|}_{\dYn^{N_1+12}} \lesssim \e_1 \jt^\delta,
\\
{\| \partial_t \alpha^{(m)}(t) \|}_{\dYn^{N_1-11}} & \lesssim \e_1,
  \qquad {\| \partial_t \alpha^{(m)}(t) \|}_{\dYn^{N_1+11}} \lesssim \e_0 \e_1 \jt^\delta,
\end{split}
\end{align}
Moreover, since $L$ is linear in $\alpha$, we also have that, for $j=0,1$
\begin{align}\label{seqcon}
{\big\| \partial_t^j \big( L(\alpha_1) - L(\alpha_2) \big) \big\|}_{\dYn^{N_1-10-j}}
  \lesssim \e_0 \big( {\|\alpha_1 - \alpha_2 \|}_{\dYn^{N_1-10}}
  + {\|\partial_t(\alpha_1 - \alpha_2) \|}_{\dYn^{N_1-11}} \big),
\end{align}
so that $L$ is a contraction in a ball of radius $C\e_1$, with some absolute constant $C$,
in the space $C^0([0,T], \dYn^{N_1-10}) \cap C^1([0,T], \dYn^{N_1-11})$.
Let us denote by $\alpha$ the unique fixed point of $L$ in this space; we have for $j=0,1$,
\begin{align}\label{seqlim}
{\| \partial_t^j \alpha(t) \|}_{\dYn^{N_1-10-j}} \lesssim \e_1
\end{align}
for all $t \leq T$.
In addition, from \eqref{seq2},  
we have 
${\| \partial_t^j \alpha^{(m)}(t) \|}_{\dYn^{N_1+12-j}} \lesssim \e_1 \jt^\delta \e_0^j$
and therefore, up to passing to a sub-sequence, we get that $\partial_t^j \alpha^{(m)}(t)$ converges
weak-$\ast$ in $\dYn^{N_1+12-j}$ to a limit $\alpha_j'(t)$, $j=0,1$.
Then we have $\partial_t \alpha(t)=\alpha_j'(t)$ for all $t\leq T$,
and by lower semi-continuity
\begin{align}\label{seqlim2}
{\| \alpha(t) \|}_{\dYn^{N_1+12}} \lesssim \e_1 \jt^\delta,
  \qquad {\| \partial_t \alpha(t) \|}_{\dYn^{N_1+11}} \lesssim \e_0 \e_1 \jt^\delta.
\end{align}
With \eqref{seqlim} and \eqref{seqlim2} we conclude the proof of Proposition \ref{alphaprop}
$\hfill \Box$


\bigskip
\section{Estimates for the vorticity}\label{secVorticity}
In this section we bootstrap weighted bounds for the vorticity in the three dimensional (flat) domain
proving the main Proposition \ref{mainpropWintro}.
For the convenience of the reader, and ease of reference, we restate this result below as Proposition \ref{mainpropW}.
In particular this will prove the validity of the assumptions \eqref{alphaaso0}-\eqref{alphaaso1}
used in Proposition \ref{alphaprop} to obtain \eqref{alphaconcL}-\eqref{alphadtconcH},
which in turn give bounds on the vector potential $V_\omega$
and its restriction to the boundary $\widetilde{v_\omega}$ 
(see Lemma \ref{lemboundaVbulk} and the conclusions of Proposition \ref{propalphaintro}).


Let us recall here some of our notation:  
with $v$ the velocity field, $\omega = \curl v$, we let, for $x \in \R^2, z\leq 0$
\begin{align}\label{defVW}
V(t,x,z) := v(t,x,z+h(t,x)), \qquad  W(t,x,z) := \omega(t,x,z+h(t,x)), 
\end{align}
and will generally use capital letters for quantities defined in the transformed domain $\mathbb{R}^2 \times \{ z\leq 0\}$;
accordingly, given the Hodge decomposition \eqref{setup1},
set
\begin{align}\label{defPsiV_o}
\Psi(t,x,z) := \psi(t,x,z+h(t,x)), \qquad  V_\omega(t,x,z) := v_\omega(t,x,z+h(t,x)).
\end{align}

\begin{prop}\label{mainpropW}
Assume that $h$ satisfies \eqref{apriorih}-\eqref{apriorihp} and \eqref{aprioridth2}-\eqref{aprioridthp},
and let $W$ be as defined in \eqref{defVW}.
Let $\mX^n$ be the space defined in \eqref{omegaflatspace0}, which for convenience we recall here:
\begin{align}\label{omegaflatspacebis}
\begin{split}
    {\| f \|}_{\mX^n} := \sum_{\dg{|r|+|k|} \leq n} 
  {\big\| \underline{\Gamma}^k \nabla^r_{x,z} \,f \big\|}_{L^2_z L^2_x \cap L^{6/5}_{x,z}}.
\end{split}
\end{align}

Assume that, for all $t\in[0,T]$ for some $T\leq T_{\e_1}$, and for $j=0,1$ 
\begin{align}
\label{propWaslow}
& {\| \partial_t^j W(t) \|}_{\mX^{N_1-10-j}} \leq 2C \e_1,
\\
\label{propWashigh}
& {\| \partial_t^j W(t) \|}_{\mX^{N_1+12-j}} \leq 2C \e_0^j \e_1 \jt^{\delta},
\end{align}
for some absolute constant $C>0$ large enough, and where $\delta$ satisfies \eqref{param}.
Then, for all $t\in[0,T_{\e_1}]$, we have the improved bounds
\begin{align}
\label{propWconclow}
& {\| \partial_t^j W(t) \|}_{\mX^{N_1-10-j}} \leq C \e_1,
\\
\label{propWconchigh}
& {\| \partial_t^j W(t) \|}_{\mX^{N_1+12-j}} \leq C \e_0^j \e_1 \jt^{\delta}.
\end{align}


\end{prop}

Notice that the a priori assumptions \eqref{propWaslow} and Sobolev's embedding in $x$ and $z$, imply
\begin{align}\label{propWasinfty}
\begin{split}
\sum_{|r|+|k| \leq N_1-15} {\big\| \underline{\Gamma}^k \nabla^r_{x,z} \partial_t^j W(t) \big\|}_{L^\infty_{x,z}} 
  \leq 2C\e_1.
\end{split}
\end{align}

The proof of Proposition \ref{mainpropW} will follow from estimates on the transport equation
satisfied by the vorticity in the transformed flat domain. Several difficulties arise in
trying to control this flow, in particular the fact that the transport velocity is not integrable in time.
We will explain in more detail below how to resolve this by a normal form type argument in the full three dimensional
(transformed) fluid domain.


\subsection{Vorticity equation and transport velocity}\label{ssecVorticity1}
As a first step we write the transport of vorticity equation in the transformed flat domain.
We think of $\omega$ as a vector field letting $\omega^i = \epsilon^{ijk}\pa_j v_k$ and from the vorticity equation
$(\pa_t + v^\ell \pa_\ell) \omega^i = \omega^\ell \pa_\ell v^i$
and \eqref{defVW} we get 
\begin{align}\label{Weq00}
\pa_t W^i - \partial_t h \partial_z W^i + V^\ell \partial_\ell W^i - (V^\ell \partial_\ell h) \partial_z W^i
 = W^\ell (\partial_\ell - \partial_\ell h \partial_z ) V^i.
\end{align}
Above, and in what follows, we adopt the natural convention that $\partial_3 h = 0$.
We then write \eqref{Weq00} as
\begin{align}\label{Weq0}
\begin{split}
& {\bf D}_t W  = W \cdot \nabla V - W^\ell \partial_\ell h \partial_z V^i
\\
& {\bf D}_t := \partial_t + U \cdot \nabla, \qquad U := V - (\partial_t h + V^\ell \partial_\ell h) e_z
\end{split}
\end{align}
where $\nabla = \nabla_{x,z}$ and $e_z = (0,0,1)$.

Next, from \eqref{defPsiV_o} we write, for $x \in \R^2, z\leq 0$,
\begin{align}\label{Vi}
V^i = \partial^i \Psi - \partial^i h \partial_z \Psi + V^i_\omega
\end{align}
and then rewrite \eqref{Weq0} as follows:
\begin{align}\label{Weq1}
\begin{split}
{\bf D}_t W  = W \cdot \nabla X + F, \qquad X & := V_\omega + \nabla \Psi,
\\
F^i & := - W^\ell \partial_\ell h \partial_z V^i - W \cdot \nabla ( \partial^i h \partial_z \Psi).
\end{split}
\end{align}

The system \eqref{Weq1} is a transport equation for $W$ with a quadratic stretching term $W \cdot \nabla X$
and an additional nonlinearity $F$ that contains only cubic terms. 
Note that $W$ is transported by the vector field $U$ which has some components
that are not integrable in time, e.g. because of the presence of the `dispersive' components
$\nabla \Psi$ and $\partial_t h $;
at the same time, some of the quadratic terms are also weakly decaying because of the presence of
$\nabla \Psi$ in $X$. In particular we cannot close a bootstrap argument for norms of $W$ using \eqref{Weq1} directly.
Instead, 
we will need to use the fact that $U$ and $X$ have additional structure.
This is the content of the next lemma.

\begin{lemma}\label{lemmaU}
The vector fields $U$ and $X$ defined respectively in \eqref{Weq0} and \eqref{Weq1}
can be written as
\begin{align}\label{lemUdt}
& U = \partial_t (A - h \, e_z) + V_\omega - V_\omega \cdot \nabla h \, e_z + R,
\\
\label{lemUXdt}
& X = \partial_t A + V_\omega + R_2, \qquad \mbox{with} \quad A := \nabla_{x,z} |\nabla|^{-1} e^{z|\nabla|} h, 
\end{align}
where $R$ is given by 
\begin{align}\label{lemUR}
\begin{split}
& R^i =  R_1^i + R_2^i,
\\
& R_1^i := - \partial^i h \partial_z \Psi - (V^\ell - V_\omega^\ell) \partial_\ell h \, e_z^i,
\\
& R_2^i := \partial_i e^{z|\nabla|} |\nabla|^{-1} \big( |\nabla| \varphi - G(h)\varphi \big)
  + \partial_i (\Psi - e^{z|\nabla|} \varphi), \qquad i=1,2,3.
\end{split}
\end{align}
Recall that $\varphi(t,x) = \Psi(t,x,h(x))$.
\end{lemma}

More precise bounds on $U$ and $X$ are postponed for the moment (see, for example, Lemma \ref{lemGU}).
The main point of the above lemma is that $U$ (and $X$) can be written as a perfect time derivative
of one of the dispersive variables, plus terms that involve $V_\omega$ and other
terms that will be proven to have good time-integrability properties.
In this whole section we will distinguish and treat differently,
the `dispersive variables', e.g. $h$ and $\nabla\Psi$, which mainly play the role of coefficients,
and the rotational variables, e.g. $V_\omega$ and $W$.

\begin{proof}[Proof of Lemma \ref{lemmaU}]
From \eqref{Weq0} and \eqref{Vi} we write
\begin{align}
\begin{split}
& U^i = \partial^i \Psi - \partial_t h \, e_z^i + V^i_\omega - V_\omega^\ell \partial_\ell h \, e_z^i + R_1^i,
\\
& R_1^i := - \partial^i h \partial_z \Psi - (V^\ell - V_\omega^\ell) \partial_\ell h \, e_z^i.
\end{split}
\end{align}
We recall that $\psi = \psi(t,x,y)$ is the harmonic extension of $\varphi(t,x) := \psi(t,x,h(x))$
in the original domain, and that $\partial_t h = G(h)\varphi = |\nabla| \varphi + \mbox{quadratic terms}$,
and  write, for $z\leq 0$,
\begin{align}\label{Psi=dt}
\begin{split}
\Psi & = e^{z|\nabla|} \varphi + (\Psi - e^{z|\nabla|} \varphi)
\\
& = e^{z|\nabla|} \partial_t |\nabla|^{-1} h
  + e^{z|\nabla|} |\nabla|^{-1} \big( |\nabla| \varphi - G(h)\varphi \big)
  + (\Psi - e^{z|\nabla|} \varphi).
\end{split}
\end{align}
Therefore,
\begin{align*}
\begin{split}
U^i =  \partial_t \big( \partial^i e^{z|\nabla|} |\nabla|^{-1} h - h \, e_z^i \big)
  + \partial^i \big[ e^{z|\nabla|} |\nabla|^{-1} \big( |\nabla| \varphi - G(h)\varphi \big)
  + (\Psi - e^{z|\nabla|} \varphi) \big]
  \\
  + V^i_\omega - V_\omega^\ell \partial_\ell h \, e_z^i + R_1^i,
\end{split}
\end{align*}
which is the desired conclusion \eqref{lemUdt}-\eqref{lemUR}.

The identity \eqref{lemUXdt} is directly obtained from $X = \nabla \Psi + V_\omega$
and the above formula \eqref{Psi=dt} for $\Psi$.
\end{proof}


\subsection{Vector fields and function classes}\label{secvfO}
Our next task is to apply vector fields to the equation \eqref{Weq1}.
To deal with this and handle various identities
and manipulations involving vector fields, products, commutators etc. 
we introduce convenient shorthand notation below.
We then also define useful classes of functions
satisfying linear and quadratic bounds consistent with our energy and dispersive estimates. 


\subsubsection{Notation for vector fields}\label{vfnotation}
For $\underline{\Gamma} := (\partial_{x_1},\partial_{x_2},\partial_z,\Omega,\underline{S})$
and for $\alpha \in \Z_+^5$, let 
\begin{align*}
\underline{\Gamma}^\alpha := \partial_{x_1}^{\alpha_1} \cdot \partial_{x_2}^{\alpha_2} \cdot \partial_z^{\alpha_3}
  \cdot \Omega^{\alpha_4} \cdot \underline{S}^{\alpha_5}.
\end{align*}
For $n \in \Z_+$, we define the sets of vector fields of order $n$, respectively, less or equal to $n$, by
\begin{align*}
\mathcal{V}^n = \big\{ \underline{\Gamma}^\alpha, \quad \alpha \in \Z_+^5
  \quad \mbox{with} \quad \alpha_1+\alpha_2+\alpha_3+\alpha_4+\alpha_5= n \big\},
  \\
\mathcal{V}^{\leq n} = \big\{ \underline{\Gamma}^\alpha, \quad \alpha \in \Z_+^5
  \quad \mbox{with} \quad \alpha_1+\alpha_2+\alpha_3+\alpha_4+\alpha_5 \leq n \big\}.
\end{align*}
$\mathcal{V}^0 = \{ 1 \}$, where $1$ is intended as a multiplication operator.
%
%

\begin{itemize}

\item {\normalfont ({\bf Action of vector fields})}
Given a function $f$ we will denote by $\Gamma^n f$ a generic element of $\mathcal{V}^n f$, that is,
$g = \Gamma^n f$ if there exists $V \in \mathcal{V}^n$ such that $g = V f$.
We further denote with $\Gamma^{\leq n} f$ a generic linear combination of elements in $\mathcal{V}^{\leq n} f$, that is,
$g = \Gamma^{\leq n} f$ if
\begin{align}\label{G<nf}
g = \sum_{V \in \mathcal{V}^{\leq n}} c_V \, V f .
\end{align}
for some real constants $c_V$.
In particular, $\Gamma^{\leq n} f$ denotes a linear combination of elements of $\mathcal{V}^n f$.
We have the following schematic identities:
\begin{align}
\Gamma^{n_1} (\Gamma ^{n_2} f ) = \Gamma^{n_1+n_2} f,
	\qquad  \Gamma^{n_1} (\Gamma^{\leq n_2} f ), \, \Gamma^{\leq n_1} (\Gamma^{n_2} f ) = \Gamma^{\leq n_1+n_2} f.
\end{align}
Note that while the $3$d vector fields are denoted $\underline{\Gamma}$ using an underline,
we are not underlining the $\Gamma^n$ expressions for ease of notation.
This should generate no confusion since in this section we work solely in the full $3$d 
(flattened) domain. 
When the $z$-independent function $h$ is involved, the action of $\underline{\Gamma}$ is the obvious one,
and we adopt the same notation $\Gamma^n h$ and $\Gamma^{\leq n} h$.



\item {\normalfont ({\bf Products})} For any $V \in \mathcal{V}^n$ one sees that 
\begin{align*}
V(fg) = \sum_{\substack{V_1 \in \mathcal{V}^{n_1}, \, V_2\in \mathcal{V}^{n_2} \\ n_1+n_2 = n} }
  c_{V_1V_2} \, V_1 f \cdot V_2 g
\end{align*}
for some real constants $c_{V_1V_2} = c_{V_1V_2}(V)$ determined by $V$.
We then adopt a short-hand notation for linear combinations of the form above
by omitting the constants and denoting a generic term in $\mathcal{V}^n(fg)$ as
\begin{align}\label{Gnfg}
\Gamma^n (fg) = \sum_{n_1+n_2=n} \Gamma^{n_1} f \cdot \Gamma^{n_2} g.
\end{align}
A linear combination of terms in $\mathcal{V}^n(fg)$ will also be denoted in the same way.
Similarly, we will write a generic linear combination of elements of $\mathcal{V}^{\leq n}(fg)$,
as
\begin{align}\label{G<nfg}
\Gamma^{\leq n} (fg) = \sum_{n_1+n_2 \leq n} \Gamma^{n_1} f \cdot \Gamma^{n_2} g.
\end{align}
We will also use a sum as in the right-hand side of \eqref{G<nfg} 
to denote a generic linear combination (of linear combinations) of products of terms
in $\mathcal{V}^{n_1}$ and $\mathcal{V}^{n_2}$ for $n_1 + n_2 = 0, \dots, n$.
Note that, by this convention, we can identify
\begin{align}\label{...}
\sum_{n_1+n_2 \leq n} \Gamma^{\leq n_1} f \cdot \Gamma^{n_2} g = \sum_{n_1+n_2 \leq n} \Gamma^{n_1} f \cdot \Gamma^{n_2} g.
\end{align}

\item {\normalfont({\bf Commutators})}
From standard commutation relations one sees that for any $V_1 \in \mathcal{V}^{n_1}$ and $V_2 \in \mathcal{V}^{n_2}$
the commutator satisfies $[V_1, V_2] = V'$ where
$V' \in 
\mathcal{V}^{n_1+n_2-1}$. Consistently with this and our short-hand notation from above, we write
\begin{align}
[\Gamma^{\leq n_1}, \Gamma^{\leq n_2} ] = \Gamma^{\leq n_1+n_2-1}
\end{align}
to express the fact that the commutation of any linear combinations of vector fields of order
at most $n_1$ and $n_2$, gives a linear combination of vector fields of order less or equal to $n_1+n_2-1$.

\item {\normalfont({\bf Norms and estimates})}
Consistently with the above convention,
for generic terms $\Gamma^n f$ and $\Gamma^{\leq n} f$ we have
\begin{align}
|\Gamma^n f | \leq \sup_{V \in \mathcal{V}^n} |V f|, \qquad
  |\Gamma^{\leq n} f | \lesssim \sum_{V \in \mathcal{V}^{\leq n}} |V f|.
\end{align}
In particular,
\begin{align}\label{notnorms1}
{\|\Gamma^{\leq n} f \|}_{L^q_zL^p_x} \lesssim \sum_{V \in \mathcal{V}^{\leq n}} {\| V f \|}_{L^q_zL^p_x}
  \lesssim  \sum_{\substack{r,k \in \Z^2_+ \\ |r|+|k| \leq n}}
  {\big\| (\partial_{x_1},\partial_{x_2})^r (\Omega,S)^k f \big\|}_{L^q_zL^p_x}.
\end{align}
At the same time, we also have
\begin{align}\label{notnorms2}
\sum_{\substack{r,k \in \Z^2_+ \\ |r|+|k| \leq n}}
  {\big\| (\partial_{x_1},\partial_{x_2})^r (\Omega, \underline{S})^k f \big\|}_{L^q_zL^p_x}
  \lesssim \sup_{0\leq \ell \leq n} \sup_{V \in \mathcal{V}^\ell} {\| V f \|}_{L^q_zL^p_x}.
\end{align}
Therefore, in order to estimate the desired $\mX^n$ type norms of $W$, see Proposition \ref{mainpropW},
it suffices to estimate generic terms of the form $\Gamma^\ell W$, $0\leq \ell \leq n$ 
in the appropriate $L^p_{x,z}$ spaces,
and under the a priori assumptions inferred from \eqref{notnorms1},
\eqref{propWaslow}-\eqref{propWashigh} and \eqref{propWasinfty}.

\item {\normalfont({\bf The two dimensional case})}
We will adopt an analogous notation for the $2$d vector fields,
with corresponding product and commutator identities, 
together with corresponding norms estimates as in \eqref{notnorms1}.
The distinction will always be clear from context.

\end{itemize}

\subsubsection{Classes of functions}\label{ssecO}
To deal with multilinear expressions involving the vorticity and the dispersive variables,
we introduce classes of linear and quadratic functions satisfying dispersive type bounds.
We will adopt the shorthand notation from Subsection \ref{vfnotation}.
Also, recall the notation $x+$ for a real number $x$ (see Subsection \ref{Notation}) 
and that $3p_0 < \delta$.

\begin{defi}[$\mathcal{O}_i$ classes]\label{defOi}
We say that a function $F = F(x,z)$ defined on $\R^2 \times \R_-$
is of class $\mathcal{O}_1$, and write $F \in \mathcal{O}_1$, if
\begin{subequations}
    \begin{alignat}{2}
\label{defO1b}
& {\big\| \Gamma^n F \big\|}_{L^\infty_{x,z}} \lesssim \e_0 \jt^{-1+}, &&\qquad n \leq N_1-10,
\\
\label{defO1a}
& {\big\| \Gamma^n F \big\|}_{L^\infty_{x,z}} \lesssim \e_0 \jt^{3p_0}, &&\qquad n \leq N_1+12.
\end{alignat}
\end{subequations}


\noindent
We say that $F$ is of class $\mathcal{O}_2$, and write $F \in \mathcal{O}_2$, if
\begin{subequations}
\begin{align}
\label{defO2b}
& {\big\| \Gamma^n F \big\|}_{L^\infty_{x,z}} \lesssim \e_0^{1+} \jt^{-1-\delta}, \qquad n \leq N_1-10,
\\
\label{defO2a}
& {\big\| \Gamma^n F \big\|}_{L^\infty_{x,z}} \lesssim \e_0^{1+} \jt^{-1+\delta}, \qquad n \leq N_1+12.
\end{align}
\end{subequations}

\end{defi}

\begin{remark}[About the $\mathcal{O}$ classes]\label{remOi}
Here are a few remarks and simple consequences of the above definitions:

\begin{enumerate}

 \label{Oi0}
\item The classes above are consistent with the bounds we have on high order energies and low order dispersive norms
for our irrotational variables.
Indeed, notice that a typical $\mathcal{O}_1$ function is $h$
together with a few (up to ten) of its derivatives, in view of \eqref{apriorih}.
The quantity $A$ in \eqref{lemUdt} can also be easily seen to be in $\mathcal{O}_1$; see Lemma \ref{lemO1} below.
Also, clearly $\mathcal{O}_2 \subset \mathcal{O}_1$.

\item \label{Oi11=2}
If $F,G \in \mathcal{O}_1$, then the product $F \cdot G \in \mathcal{O}_2$.
This follows immediately from the definitions,  distributing vector fields as in \eqref{Gnfg},
and using that $N_1$ is large.
More precisely, we have
\begin{alignat}{2}
\label{O1dotO1b}
& {\big\| \Gamma^{n_1} F \cdot \Gamma^{n_2} G \big\|}_{L^\infty_{x,z}} \lesssim \e_0^2 \jt^{-2+}, 
&&\qquad n_1+n_2 \leq N_1-10,
\\
\label{O1dotO1a}
& {\big\| \Gamma^{n_1} F \cdot \Gamma^{n_2} G \big\|}_{L^\infty_{x,z}} \lesssim \e_0^2
\jt^{-1+3p_0+},
&&\qquad n_1+n_2 \leq N_1+12,
\end{alignat}
which are sufficient since $3p_0 < \delta$.

\item \label{OiO1W}
If $F \in \mathcal{O}_1$, and $W$ satisfies the a priori estimates \eqref{propWaslow}-\eqref{propWashigh},
then the product $W \cdot F$ satisfies estimates that are at least $\e_0$ better.
More precisely
\begin{align}
\label{WdotO1b}
& {\big\| \Gamma^{n_1} W \cdot \Gamma^{n_2} F \big\|}_{L^2_{x,z} \cap L^{6/5}_{x,z}}
  \lesssim \e_0 \cdot \e_1 \cdot \jt^{-1+}, \qquad n_1+n_2 \leq N_1-10,
\\
\label{WdotO1a}
& {\big\| \Gamma^{n_1} W \cdot \Gamma^{n_2} F \big\|}_{L^2_{x,z} \cap L^{6/5}_{x,z}}
\lesssim \e_0 \cdot \e_1 \cdot \jt^{3p_0}, \qquad n_1+n_2 \leq N_1+12,
\end{align}
which can be verified directly using H\"older's inequality, estimating the terms $\Gamma^{n_2} F$
always in $L^\infty_{x,z}$.

%



\item
The main property that we will eventually use to bound the large number of cubic remainder terms in the
renormalized vorticity equation
is that the $L^1_t([0,T_{\e_1}])$-norm of the product of $W$ with an $\mathcal{O}_2$ function satisfies bounds
consistent with the conclusions \eqref{propWconclow}-\eqref{propWconchigh}.
This is the content of Lemma \ref{lemWO2} below.

\end{enumerate}

\end{remark}

Before proceeding with some general lemmas about the behavior of $\mathcal{O}_i$ functions,
let us recall here, for ease of reference, the bounds we established on the vector potential,
(see Lemma \ref{lemboundaVbulk} and Definition \ref{defnormsalpha}):
with the notation
\begin{align}\label{defVojrn}
V_{\omega,j}^{r,k} = \partial_t^j \underline{\Gamma}^k \nabla_{x,z}^r V_\omega, \quad j=0,1,
\end{align}
we have
\begin{align}\label{Voboundhi}
\begin{split}
    \sum_{\dg{|r|+|k|} \leq N_1+12-j} 
  {\big\| 
  |\nabla|^{1/2} V_{\omega,j}^{r,k}(t) \big\|}_{L^\infty_zL^2_x}
  + {\big\| 
  \nabla_{x,z} V_{\omega,j}^{r,k}(t) \big\|}_{L^2_zL^2_x}
  + {\| 
  V_{\omega,j}^{r,k}(t)\|}_{L^\infty_zL^2_x} \lesssim \e_1 \e_0^j \jt^\delta,
\end{split}
\end{align}
and
\begin{align}\label{Voboundlo}
\begin{split}
    \sum_{\dg{|r|+|k|} \leq N_1-10 -j}
  {\| |\nabla|^{1/2} V_{\omega,j}^{r,k}(t)\|}_{L^\infty_zL^2_x}
  + {\| \nabla_{x,z} V_{\omega,j}^{r,k}(t)\|}_{L^2_zL^2_x}
  + {\| V_{\omega,j}^{r,k}(t)\|}_{L^\infty_zL^2_x} \lesssim \e_1.
\end{split}
\end{align}
In particular using the notation from \ref{vfnotation} for vector fields,
and \dg{the bounds \eqref{Voboundhi}-\eqref{Voboundlo}} with Sobolev's embedding in $x$, we have the bounds
\begin{alignat}{2}
\label{Volo}
& {\| 
\Gamma^n V_\omega(t) \|}_{L^\infty_{x,z}} \lesssim \e_1, &&\qquad n \leq N_1-12,
\\
\label{Vohi}
                                                         & {\| \Gamma^n \nabla_{x,z} V_\omega(t) \|}_{L^2_{x,z}} \lesssim \e_1 \jt^\delta, &&\qquad n \leq N_1+12,
\end{alignat}
for all $t \leq T_{\e_1}$.

Here is a useful Lemma about products of $W$ with elements of $\mathcal{O}_2$.

\begin{lemma}[Bounds on trilinear expression]\label{lemWO2}
Let $W$ be defined as in \eqref{defVW} and  let $H \in \mathcal{O}_2$ as in Definition \ref{defOi}.
Then,
\begin{subequations}\label{WO2}
\begin{align}
\label{WO2b}
& {\big\| \Gamma^{n_1} W \cdot \Gamma^{n_2} H \big\|}_{L^2_{x,z}\cap L^{6/5}_{x,z}}
  \lesssim \e_1 \e_0^{1+} \jt^{-1-\delta}, \qquad n_1+n_2 \leq N_1-10,
\\
\label{WO2a}
& {\big\| \Gamma^{n_1} W \cdot \Gamma^{n_2} H \big\|}_{L^2_{x,z}\cap L^{6/5}_{x,z}}
  \lesssim \e_1 \e_0^{1+} \jt^{-1+\delta}, \qquad n_1+n_2 \leq N_1+12.
\end{align}
\end{subequations}

Similarly, if $H,K \in \mathcal{O}_1$, then
\begin{subequations}\label{WO1O1}
\begin{align}
\label{WO1O1b}
& {\big\| \Gamma^{n_1} W \cdot \Gamma^{n_2} H \cdot \Gamma^{n_3} K  \big\|}_{L^2_{x,z}\cap L^{6/5}_{x,z}}
	\lesssim \e_1 \e_0^{1+} \jt^{-1-\delta}, \qquad n_1+n_2+n_3 \leq N_1-10,
\\
\label{WO1O1a}
& {\big\| \Gamma^{n_1} W \cdot \Gamma^{n_2} H \cdot \Gamma^{n_3} K \big\|}_{L^2_{x,z}\cap L^{6/5}_{x,z}}
	\lesssim \e_1 \e_0^{1+} \jt^{-1+\delta}, \qquad n_1+n_2+n_3 \leq N_1+12.
\end{align}
\end{subequations}

\end{lemma}

\begin{proof}
The proof is an immediate verification. We give some details for the sake of completeness.
Using H\"older's inequality we have
\begin{align}\label{WO2pr1}
 {\big\| \Gamma^{n_1} W \cdot \Gamma^{n_2} H \big\|}_{L^2_{x,z}} & \lesssim
 	\sup_{n_1+n_2=n} {\| \Gamma^{n_1} W \|}_{L^2_{x,z}} {\| \Gamma^{n_2} H \|}_{L^\infty_{x,z}}.
 \end{align}
Let us look at the case $n_1+n_2 = n \leq N_1+12$ first.
When $n_1 \geq n_2$, we use the a priori bounds \eqref{propWashigh} and \eqref{defO2b} to estimate
\begin{align*}
{\| \Gamma^{n_1} W \|}_{L^2_{x,z}} {\| \Gamma^{n_2} H \|}_{L^\infty_{x,z}}
 	& \lesssim {\| \Gamma^{\leq N_1+12} W \|}_{L^2_{x,z}} {\| \Gamma^{\leq N_1-10} H \|}_{L^\infty_{x,z}}
 	\\
	& \lesssim \e_1 \jt^\delta \cdot \e_0^{1+} \jt^{-1-\delta} \lesssim \e_1 \cdot \e_0^{1+} \jt^{-1},
\end{align*}
which suffices.
When instead $n_2 \geq n_1$, we use the a priori bounds \eqref{propWaslow} and \eqref{defO2a} to estimate
\begin{align*}
{\| \Gamma^{n_1} W \|}_{L^2_{x,z}} {\| \Gamma^{n_2} H \|}_{L^\infty_{x,z}}
 	& \lesssim {\| \Gamma^{\leq N_1-10} W \|}_{L^2_{x,z}} {\| \Gamma^{\leq N_1+12} H \|}_{L^\infty_{x,z}}
 	\\
	& \lesssim \e_1 \cdot \e_0^{1+} \jt^{-1+\delta} \lesssim \e_1 \cdot \e_0^{1+} \jt^{-1+\delta}.
\end{align*}
Identical estimates hold with $L^{6/5}_{x,z}$ replacing $L^2_{x,z}$.
This gives us the desired bound \eqref{WO2a}.

To obtain \eqref{WO2b} we use the bounds on low norms \eqref{propWaslow} and \eqref{defO2b} and see that
for all $n_1+n_2 \leq N_1-10$,
\begin{align*}
{\| \Gamma^{n_1} W \|}_{L^2\cap L^{6/5}} {\| \Gamma^{n_2} H \|}_{L^\infty_{x,z}}
& \lesssim {\| \Gamma^{\leq N_1-10} W \|}_{L^2\cap L^{6/5}} {\| \Gamma^{\leq N_1-10} H \|}_{L^\infty_{x,z}}
	\lesssim \e_1 \cdot \e_0^{1+} \jt^{-1-\delta},
 \end{align*}
as claimed.

The proof of \eqref{WO1O1} is similar. It suffices to notice that terms of the form
$\Gamma^{n_2} \mathcal{O}_1 \cdot \Gamma^{n_3} \mathcal{O}_1$ satisfy better bounds than
$\Gamma^{\leq n_2+n_3} \mathcal{O}_2$, see Remark \ref{remOi}\eqref{Oi11=2}.
\end{proof}

%

We conclude this section with a list of functions that are of class $\mathcal{O}_1$ and $\mathcal{O}_2$.

\begin{lemma}[Functions of class $\mathcal{O}_1$]\label{lemO1}
With the classes $\mathcal{O}_1$ defined as in Definition \ref{defOi},
and under our a priori assumptions, we have
\begin{align}\label{O1h}
h, \,\, \partial_t h,  \,\, A, \,\, \partial_t A  \,\, \in \mathcal{O}_1,
\end{align}
recall the definition \eqref{lemUXdt}, and
\begin{align}\label{O1V}
V-V_\omega, \,\, \nabla_{x,z} \Psi \,\, \in \mathcal{O}_1,
\end{align}
see \eqref{defVW} and \eqref{defPsiV_o}.
The same holds true for $\Gamma^{\leq 2}$ applied to all the quantities in \eqref{O1h}-\eqref{O1V}.
\end{lemma}

\begin{proof}
The fact that $h \in \mathcal{O}_1$ follows directly from 
the a priori assumption \eqref{apriorie0} and \eqref{aprioridecay}.
For $\partial_t h$ we use instead 
that $\partial_th = G(h)\varphi$ and the bounds in \eqref{G1est}
to deduce bounds as in \eqref{defO1b}-\eqref{defO1a}.

For $A = \nabla_{x,z} |\nabla|^{-1} e^{z|\nabla|} h$ we use 
Sobolev's embedding in $x$ followed by \eqref{expbounds1}:
for $n \leq N_1+12$, and denoting $\mathcal{R} = \nabla|\nabla|^{-1}$ the (vector) Riesz transform,
\begin{align*}
{\big\| 
  \Gamma^n \nabla_{x,z} |\nabla|^{-1} e^{z|\nabla|} h \big\|}_{L^\infty_{x,z}}
& \lesssim  {\big\| \Gamma^n  (1,\mathcal{R}) e^{z|\nabla|} h \big\|}_{L^\infty_z H^2_x}
  \lesssim \sup_{|r|+|k| \leq n} {\| h \|}_{Z_k^{r+2}} \lesssim \e_0 \jt^{p_0}
\end{align*}
which is sufficient for the bound \eqref{defO1a} for this term.
To verify the bound \eqref{defO1b} we proceed similarly: for $n \leq N_1-10$ we have,
by the maximum principle, Sobolev's embedding, and \eqref{apriorihp},
\begin{align*}
& {\big\| \Gamma^n \nabla_{x,z} |\nabla|^{-1} e^{z|\nabla|} h \big\|}_{L^\infty_{x,z}} \lesssim
    {\| (1,\mathcal{R}) \Gamma^n h \|}_{L^\infty}
    \lesssim
    \sup_{|r|+|k| \leq n} {\| h \|}_{Z_k^{r+1,\infty-}} \lesssim \e_0 \jt^{-1+}.
\end{align*}
We also see that applying $\Gamma^{\leq 2}$ to any of the quantities in \eqref{O1h}
we still obtain $\mathcal{O}_1$ functions.

Since $V - V_\omega = \nabla_{x,z} \Psi - \nabla h \, \partial_z \Psi$,
in view of Remark \ref{remOi}, we see that in order to obtain \eqref{O1V} it suffices
to prove that $\nabla_{x,z} \Psi \in \mathcal{O}_1$.
This is a consequence of Lemma \ref{LemmaPsi}.
Indeed, the property \eqref{defO1a} follows directly from the first bound in \eqref{PsiL2}
and Sobolev's embedding.
The bound \eqref{defO1b} follows from \eqref{Psiinfty}
which gives a stronger bound for $\sum_\ell P_\ell \nabla_{x,z}\Psi$ with $2^\ell \in [\jt^{-5},\jt^{5}]$,
combined with the $L^2$ bound \eqref{PsiL2} for the remaining very small and very large frequencies.
\end{proof}

\begin{lemma}[Functions of class $\mathcal{O}_2$]\label{lemO2}
    With the definitions \eqref{lemUR}, we have
\begin{align}\label{O2terms}
R_1, \quad  \nabla_{x,z} (\Psi - e^{z|\nabla|} \varphi),
  \quad \nabla_{x,z} e^{z|\nabla|} |\nabla|^{-1} \big( |\nabla| \varphi - G(h)\varphi \big)
  \, \in \, \mathcal{O}_2,
\end{align}
The same is also true for $\Gamma^{\leq 2}$ applied to all of the above quantities.
In particular, 
\begin{align}\label{O2terms'}
\Gamma^{\leq 2} R_2, \quad \Gamma^{\leq 2} R \, \in \, \mathcal{O}_2.
\end{align}
\end{lemma}

\begin{proof}
Since $R_1^i = - \partial^i h \partial_z \Psi - (V^\ell - V^\ell_\omega) \partial_\ell h \, e_z^i$
we see that this is in the $\mathcal{O}_2$ class in view of \eqref{O1h}-\eqref{O1V}
and \eqref{Oi11=2} of Remark \ref{remOi}. 

The second and third term in \eqref{O2terms} are almost the same.
For the second term we can use directly \eqref{PsiL2quad}, respectively \eqref{Psiinftyquad},
and Sobolev's embedding to deduce \eqref{defO2a} (recall $\delta>3p_0$), respectively, \eqref{defO2b}.

For the third term in \eqref{O2terms} we use the notation from \eqref{derG2}, that is,
$G_{\geq 2}(h)\varphi := G(h)\varphi - |\nabla| \varphi$, and invoke the bounds \eqref{G2est}.
Using the maximum principle 
and Sobolev's embedding we obtain, for $n \leq N_1+12$,
\begin{align*}
{\big\| 
  \Gamma^n \nabla_{x,z} |\nabla|^{-1} e^{z|\nabla|} G_{\geq 2}(h)\varphi \big\|}_{L^\infty_{x,z}}
  & \lesssim \sup_{|r|+|k|\leq n} {\| (1,\mathcal{R}) G_{\geq 2}(h)\varphi \|}_{Z_k^{r,\infty}}
  \\
  & \lesssim \sup_{|r|+|k|\leq n} {\| (1,\mathcal{R}) G_{\geq 2}(h)\varphi \|}_{Z_k^{r+2}}
  \lesssim \e_0^2 \jt^{-1+3p_0}.
\end{align*}
Similarly, we can use \eqref{expbounds1} 
to obtain, for $n \leq N_1-10$,
\begin{align*}
{\big\| \Gamma^n \nabla_{x,z} |\nabla|^{-1} e^{z|\nabla|} G_{\geq 2}(h)\varphi \big\|}_{L^\infty_{x,z}}
  & \lesssim \sup_{|r|+|k| \leq n} {\| (1,\mathcal{R}) G_{\geq 2}(h)\varphi \|}_{Z_k^{r,\infty}}
  \\
  & \lesssim \sup_{|r|+|k| \leq n} {\| G_{\geq 2}(h)\varphi \|}_{Z_k^{r+1,\infty-}} \lesssim \e_0^2 \jt^{-6/5}.
\end{align*}
having used $L^p$ interpolation between the bounds in \eqref{G2est}.
\end{proof}

\subsection{Commutation with vector fields}\label{ssecVorticitycomm}
We proceed to derive a transport equation for $\Gamma^n W$.
\begin{lemma}\label{lemDtW}
Let ${\bf D}_t := \partial_t + U \cdot \nabla$ and recall the notation in \ref{vfnotation}.
We have:

\setlength{\leftmargini}{2em}
\begin{enumerate}

\item The following basic commutation identities hold
\begin{align}\label{lemDtW0}
\begin{split}
[ {\bf D}_t, \partial_{x_i} ] & = - \partial_{x_i} U \cdot \nabla, \quad i=1,2,3, \quad (x_3=z),
\\
[ {\bf D}_t, \Omega ]  & = U_1 \partial_{x_2} - U_2 \partial_{x_1} - \Omega U \cdot \nabla,
\\
[ {\bf D}_t, S ] & = (1/2) {\bf D}_t + (1/2) U \cdot \nabla  - SU \cdot \nabla.
\end{split}
\end{align}

\item
For any $\alpha \in \Z_+^5$, $|\alpha| = 1$,
there exists constants $c_1,c_2,c_2'
\in \R$,
such that
\begin{align}\label{lemDtWcomm'}
[{\bf D}_t, \Gamma^\alpha]  = c_1 {\bf D}_t + c_2 U\cdot \nabla + c_2' U \cdot \nabla_x^\perp
 - \Gamma^\alpha U \cdot \nabla.
\end{align}
Since the presence of the ${}^\perp$ and of all the constants will not play any role,
we will drop them from \eqref{lemDtWcomm'} and use the notation conventions from \ref{vfnotation} to write,
for some $c\in\R$,
\begin{align}\label{lemDtWcomm}
[{\bf D}_t, \Gamma^1]f  = c {\bf D}_t f + ( \Gamma^{\leq 1} U ) \cdot \nabla f.
\end{align}

\item If ${\bf D}_t W = B$ then, for all $n\geq 1$,
\begin{align}\label{lemDtWcommN}
{\bf D}_t \Gamma^n W = \Gamma^{\leq n} B
  + \sum_{n_1+n_2\leq n-1} \Gamma^{\leq n_1+1} U \cdot \Gamma^{n_2} \nabla W;
\end{align}
here the sum is a linear combination as per our conventions, see \eqref{G<nfg} and the paragraph following that.


\item If $W$ is the solution of \eqref{Weq1}, then, for any integer $n$ we have
\begin{align}\label{lemDtWeq}
\begin{split}
{\bf D}_t \Gamma^n W & = \sum_{n_1+n_2\leq n-1} \Gamma^{\leq n_1+1} U \cdot \Gamma^{n_2} \nabla W
  + \sum_{n_1+n_2 \leq n} \Gamma^{n_1} W \cdot \Gamma^{n_2} \nabla X
  + \Gamma^{\leq n} F.
\end{split}
\end{align}

\end{enumerate}

\end{lemma}

\begin{proof}
The identities \eqref{lemDtW0} follow from standard calculations.
Indeed, for $x \in \R^2 \times \R^2_-$ we have
\begin{align*}
[ {\bf D}_t, S ]  & = (1/2) [\partial_t, t\partial_t ] + [U \cdot \nabla, (1/2) t\partial_t + x \cdot \nabla]
\\
& = (1/2) \partial_t - SU \cdot \nabla - U \cdot [S,\nabla]
\\
& = (1/2) {\bf D}_t - (1/2) U \cdot \nabla - SU \cdot \nabla - U \cdot (-\nabla),
\end{align*}
which is of the desired form.
Similarly,
$[ {\bf D}_t, \Omega ] = [U \cdot \nabla, x_1 \partial_{x_2} - x_2 \partial_{x_1}] =
  U_1 \partial_{x_2} - U_2 \partial_{x_1} - \Omega U \cdot \nabla.$
The identity \eqref{lemDtWcomm'}, and its shorthand version \eqref{lemDtWcomm},
then follow directly from \eqref{lemDtW0}.

To prove \eqref{lemDtWcommN} we proceed by induction.
The case $n=1$ follows directly from 
\eqref{lemDtWcomm}.
Assuming \eqref{lemDtWcommN} holds true for $n=1$ 
we calculate using the inductive hypothesis and \eqref{lemDtWcomm}:
\begin{align*}
{\bf D}_t \Gamma^n W & = \Gamma^1 {\bf D}_t \Gamma^{n-1} W + [{\bf D}_t, \Gamma^1] \Gamma^{n-1} W
\\
& = \Gamma^1 \Big( \Gamma^{\leq n-1} B
  + \sum_{n_1+n_2\leq n-2} \Gamma^{\leq n_1+1} U \cdot \Gamma^{n_2} \nabla W \Big)
  \\
  & + {\bf D}_t \Gamma^{n-1} W + \Gamma^{\leq 1} U \cdot \nabla \Gamma^{n-1} W
  ;
\end{align*}
distributing the vector field $\Gamma$ in the first line (see \eqref{Gnfg}),
applying again the inductive hypothesis to ${\bf D}_t \Gamma^{n-1} W$,
and commuting the $\Gamma$'s and $\nabla$ in the last term, 
we see that the above expression is of the desired form \eqref{lemDtWcommN}.

Finally, the equation \eqref{lemDtWeq} follows from applying the previous identity \eqref{lemDtWcommN}
to the equation \eqref{Weq1}.
\end{proof}

We also have the following lemma for the transporting vector field $U$.

\begin{lemma}\label{lemGU}
With the notation of \ref{vfnotation} 
and under the assumptions of Proposition \ref{mainpropW}, 
for $U$ and $X$ as in \eqref{lemUdt} we have:
\begin{align}\label{lemGU0}
U = \mathcal{O}_1 + V_\omega + V_\omega \cdot \mathcal{O}_1, \qquad \nabla X = \mathcal{O}_1 + \nabla V_\omega,
\end{align}
In particular
\begin{align}\label{GUbdlo}
& {\| \Gamma^n U \|}_{L^\infty_{x,z}} \lesssim \e_0 \jt^{-1+} + \e_1, \qquad n\leq N_1-12.
\end{align}
\end{lemma}

\begin{proof}
We see from the formulas in Lemma \ref{lemmaU},
and using Lemmas \ref{lemO1} and \ref{lemO2}, that
\begin{align}\label{GUpr0}
U = \mathcal{O}_1 + V_\omega - V_\omega \cdot \nabla h \, e_z + \mathcal{O}_2,
\end{align}
and \eqref{lemGU0} follows.
\end{proof}

\subsection{Renormalization of the vorticity equation}
In this subsection we manipulate the equation for $W$, see \eqref{lemDtWeq},
using the identities from Lemma \ref{lemmaU}, in order to write it in a better form
that allows to propagate the desired $\X^n$ norms and prove Proposition \ref{mainpropW}.
These manipulations are akin to a (partial) normal form transformation on the vorticity equation
in the full three dimensional fluid domain that effectively renormalizes the irrotational components.
The next proposition is the main result of this section.

\begin{prop}[Renormalized vorticity equation]\label{lemWreno}
Under the assumptions of Proposition \ref{mainpropW}
we have following:
for all $n \leq N_1+12$, there exist
corrections $G^n=G^n(t,x,z)$ such that
\begin{align}\label{Wreno1}
\begin{split}
{\mathbf D_t} \big(\Gamma^n W - G^n \big) = Q_1^n + Q_2^n
+ C_1^n + C_2^n + F^n
\end{split}
\end{align}
where the following holds:

\begin{itemize}

\item The correction $G^n$ satisfies for all $|t| \leq T$ 
\begin{subequations}\label{WrenoG}
\begin{align}
\label{WrenoG1}
& {\| G^n(t) \|}_{L^2_{x,z} \cap L^{6/5}_{x,z}} \lesssim \e_0 \e_1, \qquad n\leq N_1-10,
\\
\label{WrenoG2}
& {\| G^n(t) \|}_{L^2_{x,z} \cap L^{6/5}_{x,z}} \lesssim \e_0 \e_1 \jt^\delta, \qquad n\leq N_1+12.
\end{align}
\end{subequations}

\item $Q_1^n,Q_2^n$ are quadratic terms (in the rotational variables only) given by
\begin{subequations}\label{WrenoQuad}
\begin{align}\label{WrenoQuad1}
Q_1^n
  & := \sum_{n_1+n_2\leq n} \Gamma^{n_1} W \cdot \Gamma^{n_2} \nabla V_\omega,
\\
\label{WrenoQuad2}
Q_2^n & := \sum_{n_1+n_2\leq n-1} \Gamma^{\leq n_1+1} V_\omega \cdot \Gamma^{n_2} \nabla W.
\end{align}
\end{subequations}

\item $C^n_1, C^n_2$ are cubic terms of the form
\begin{align}\label{WrenoCub}
\begin{split}
C_1^n
  & := \sum_{n_1+n_2+n_3\leq n} \Gamma^{n_1} W \cdot \Gamma^{n_2} \nabla V_\omega
  \cdot \Gamma^{n_3} \mathcal{O}_1,
\\
C_2^n & := \sum_{n_1+n_2+n_3\leq n-1} \Gamma^{\leq n_1+1} V_\omega \cdot \Gamma^{n_2} \nabla W
  \cdot \Gamma^{n_3} \mathcal{O}_1;
\end{split}
\end{align}
recall Definition \ref{defOi} and Lemma \ref{lemO1}.

\item The remaining nonlinear terms satisfy 
\begin{subequations}\label{WrenoF}
\begin{align}
\label{WrenoF1}
& {\| F^n \|}_{L^2_{x,z} \cap L^{6/5}_{x,z}} \lesssim \e_1\e_0^{1+} \jt^{-1-\delta}, \qquad n\leq N_1-10,
\\
\label{WrenoF2}
& {\| F^n \|}_{L^2_{x,z} \cap L^{6/5}_{x,z}} \lesssim \e_1 \e_0^{1+} \jt^{-1+\delta},
  \qquad n\leq N_1+12.
\end{align}
\end{subequations}

\end{itemize}

\end{prop}

\begin{remark}[``Correction'' and ``Acceptable Remainders'']
Here are some remarks on Proposition \ref{lemWreno}: 

\begin{itemize}
\item
$G^n$ is a normal-form type ``correction" of $\Gamma^n W$ since its norms are $\e_0$ smaller than those of $\Gamma^n W$
(compare \eqref{WrenoG} and \eqref{propWaslow}-\eqref{propWashigh}) and the equation satisfied by $\Gamma^n W - G^n$
has nonlinear terms that are more perturbative then the ones in \eqref{lemDtWeq}.

\item
We call a (cubic) term that satisfies \eqref{WrenoF} an ``acceptable remainder"
since such a term gives a small perturbation of the transported vector field $\Gamma^n W - G^n$ (hence of $\Gamma^n W$)
when integrated in over time $t \in [0,T_{\e_1}]$.
Many terms will be shown to be acceptable remainders directly using the lemmas from the previous subsection.

\item The quadratic terms on the right-hand side of \eqref{Wreno1} only depend on $W$ and $V_\omega$.
Technically, they are not acceptable remainders in the sense specified above and so will be estimated separately in Subsection \ref{secTransp}.
\end{itemize}

\end{remark}

\begin{proof}[Proof of Proposition \ref{lemWreno}]
We start from \eqref{lemDtWeq} and use the structure of the vector fields $U$ and $X$,
see Lemmas \ref{lemmaU} and \ref{lemGU}, to eventually obtain \eqref{Wreno1}.
In the course of the proof we are going to collect several remainders denoted by $F^n_1, F^n_2$ and similar,
that will eventually contribute to the nonlinear remainder $F^n$.

{\it Step 1: Renormalized equation}.
First, for convenience of the reader, we recall \eqref{lemDtWeq}
\begin{align}\label{lemDtWeq'}
\begin{split}
{\bf D}_t \Gamma^n W & = \sum_{n_1+n_2\leq n-1} \Gamma^{\leq n_1+1} U \cdot \Gamma^{n_2} \nabla W
  + \sum_{n_1+n_2 \leq n} \Gamma^{n_1} W \cdot \Gamma^{n_2} \nabla X
  + \Gamma^{\leq n} F.
\end{split}
\end{align}
Using the formulas for $A$ and $X$ in \eqref{lemUXdt}, and the definition of $F$ in \eqref{Weq1},
we rewrite \eqref{lemDtWeq'} in the following form:
\begin{subequations}\label{Wrenopreq}
\begin{align}\label{Wrenopr1}
{\bf D}_t \Gamma^n W & = \sum_{n_1+n_2\leq n-1} \Gamma^{\leq n_1+1} \partial_t (A - h e_z) \cdot \Gamma^{n_2} \nabla W
   + \sum_{n_1+n_2\leq n} \Gamma^{n_2} W \cdot \Gamma^{n_1} \nabla \partial_t A
\\
\label{Wrenopr1b}
  & + \sum_{n_1+n_2\leq n-1} \Gamma^{\leq n_1+1 } V_\omega \cdot \Gamma^{n_2} \nabla W
  + \sum_{n_1+n_2\leq n} \Gamma^{n_1} W \cdot \Gamma^{n_2} \nabla V_\omega
  \\
\label{Wrenopr1c}
  & - \sum_{n_1+n_2\leq n-1} \Gamma^{\leq n_1+1 } ( V_\omega \cdot \nabla h e_z) \cdot \Gamma^{n_2} \nabla W
  + \Gamma^{\leq n} \big( -W \cdot \nabla h \, \partial_z V_\omega \big)
\\
\label{Wrenopr1R}
  & + F^n_0 + F^n_1 + F^n_2,
\end{align}
\end{subequations}
where we define
\begin{align}
\label{F0}
F^n_0 & := \Gamma^{\leq n} \big( - W \cdot \nabla h \, \partial_z (V - V_\omega)
  - W \cdot \nabla ( \nabla h \partial_z \Psi) \big),
\\
\label{F1}
F^n_1 & :=  \sum_{n_1+n_2\leq n-1} \Gamma^{\leq n_1+1 } R \cdot \Gamma^{n_2} \nabla W,
\\
\label{F2}
F^n_2 & := \sum_{n_1+n_2\leq n} \Gamma^{n_1} W \cdot \Gamma^{n_2} \nabla R_2.
\end{align}

We analyze term by term the right-hand side of \eqref{Wrenopreq}.
First, we see that the terms in \eqref{Wrenopr1b} contribute to the quadratic
terms on the right-hand side of \eqref{Wreno1} as they match exactly \eqref{WrenoQuad}, and so they are accounted for.
Similarly, the cubic terms in \eqref{Wrenopr1c} are accounted for in the terms \eqref{WrenoCub},
since up to two derivatives of $h$ are in $\mathcal{O}_1$, see Lemma \ref{lemO1}.

Then, since up to two derivatives of $h$, and one derivative of $\partial_z\Psi$,
and the term $\partial_z(V-V_\omega)$ are in $\mathcal{O}_1$, see Lemma \ref{lemO1},
we see that \eqref{F0} is of the form
$$\Gamma^{\leq n} (W \cdot \mathcal{O}_1 \cdot \mathcal{O}_1) =
\Gamma^{\leq n} (W \cdot \mathcal{O}_2),$$
see Remark \ref{remOi}; applying Lemma \ref{lemWO2} and \eqref{Gnfg}
we obtain that $F^n_0$ is an acceptable remainder in that it satisfies the bounds \eqref{WrenoF}.

We can easily see that the term \eqref{F1} is an acceptable remainder using that $R \in \mathcal{O}_2$,
see Lemma \ref{lemO2} and \eqref{lemUR}, and an application of Lemma \ref{lemWO2}.
Similarly, the term \eqref{F2} is an acceptable remainder using that $\nabla R_2 \in \mathcal{O}_2$,
see Lemma \ref{lemO2}. 

We now analyze the two terms in \eqref{Wrenopr1}. For the first one
we use $\Gamma^{\leq n_1+1} \partial_t = \partial_t \Gamma^{\leq n_1+1}
= {\mathbf D_t}  \Gamma^{\leq n_1+1} + U \cdot  \Gamma^{\leq n_1+1}$ to write 
\begin{align}\label{Wrenopr2}
& \sum_{n_1+n_2\leq n-1} 
  \Gamma^{\leq n_1+1} \partial_t (A - h e_z) \cdot \Gamma^{n_2} \nabla W
  = \sum_{n_1+n_2\leq n-1} {\mathbf D_t} \Gamma^{\leq n_1+1} (A - h e_z)  \cdot \Gamma^{n_2} \nabla W
  + F^n_3,
\\
\nonumber
& \qquad F^n_3 := \sum_{n_1+n_2\leq n-1} 
 U \cdot \nabla \Gamma^{\leq n_1+1}  (A - h e_z)  \Gamma^{n_2} \nabla W .
\end{align}

The first term on the right-hand side of \eqref{Wrenopr2} will be analyzed shortly below.
First, we verify that $F^n_3$ is an acceptable remainder satisfying \eqref{WrenoF}; indeed
we can write
\begin{align*}
F^n_3 := \sum_{n_1+n_2\leq n-1} 
 U \cdot \Gamma^{\leq n_1+1} \mathcal{O}_1 \cdot \Gamma^{n_2} \nabla W
\end{align*}
and observe that $U \cdot \Gamma^{\leq n_1+1} \mathcal{O}_1 \in \Gamma^{n_1+1}\mathcal{O}_2$,
since $U$ satisfies \eqref{GUbdlo} \dg{using additionally that $t\leq \epsilon_1^{-1+\delta}$} .

For the second term on the right-hand side of \eqref{Wrenopr1}
we can use again $\Gamma^{n_1} \nabla \partial_t 
= {\mathbf D_t}  \Gamma^{\leq n_1} \nabla + U \cdot  \Gamma^{\leq n_1} \nabla$ to write
\begin{align}\label{Wrenopr3}
\sum_{n_1+n_2\leq n} \Gamma^{n_2} W \cdot \Gamma^{n_1} \nabla \partial_t A & =
   \sum_{n_1+n_2\leq n} \Gamma^{n_2} W \cdot {\mathbf D_t}   \Gamma^{n_1} \nabla A + F^n_4
\end{align}
where
\begin{align}
\label{Wrenopr3R}
F^n_4 := U \cdot \sum_{n_1+n_2\leq n} \Gamma^{n_2} W \cdot  \Gamma^{n_1} \nabla A.
\end{align}
Since $U = V_\omega + \mathcal{O}_1 + V_\omega \cdot \mathcal{O}_1$, see \eqref{lemGU0}, 
it is not hard to verify that $F^n_4$ is an acceptable remainder.
Indeed, by \eqref{OiO1W} in Remark \ref{remOi}, since $\nabla A \in \mathcal{O}_1$,
the quadratic terms in the sum \eqref{Wrenopr3R} are bounded by the right-hand side of \eqref{WdotO1a},
respectively \eqref{WdotO1b},
when $n \leq N_1+12$, respectively $n \leq N_1-10$;
since we also have that  ${\| U \|}_{L^\infty_{x,z}} \lesssim \e_0 \jt^{-1+} + \e_1$
in view of \eqref{lemUdt}, \eqref{O1h} and \eqref{Volo}, we obtain
\begin{align*}
& {\| F^n_4 \|}_{L^2_{x,z} \cap L^{6/5}_{x,z}}
  \lesssim \e_0 \e_1 \jt^{-1+} \cdot ( \e_0 \jt^{-1+} + \e_1), \qquad n_1+n_2 \leq N_1-10,
\\
& {\| F^n_4 \|}_{L^2_{x,z} \cap L^{6/5}_{x,z}}
\lesssim \e_0 \e_1 \jt^{3p_0} \cdot ( \e_0 \jt^{-1+} + \e_1), \qquad n_1+n_2 \leq N_1+12.
\end{align*}
These bounds are enough for \eqref{WrenoF1}-\eqref{WrenoF2} since $3p_0 < \delta$, $\e_0 \leq \e_1$,
and $t \leq T_{\e_1} = \e_1^{-1+\delta}$ gives $\e_1 
\lesssim \jt^{-1-\delta}$.

We now combine \eqref{Wrenopr2}-\eqref{Wrenopr3R},
changing the index in the sums
and then pull the ${\mathbf D_t}$ out to write the first two terms on the right-hand side of \eqref{Wrenopr1}
as
\begin{align}\label{Wrenopr4}
\begin{split}
& \sum_{n_1+n_2\leq n-1} \partial_t \Gamma^{\leq n_1+1} (A - h e_z) 
  \cdot \Gamma^{n_2} \nabla W
   + \sum_{n_1+n_2\leq n} \Gamma^{n_2} W \cdot \Gamma^{n_1} \nabla \partial_t A
   \\
   & = \sum_{n_1+n_2\leq n} {\mathbf D_t} \big( \Gamma^{\leq n_1+1} A +  \Gamma^{\leq n_1+1} h e_z \big)
  \cdot 
  \Gamma^{n_2} W 
  + F^n_3 + F^n_4
\\
& =: {\bf D}_t G^n + F^n_3 + F^n_4 + F^n_5,
\end{split}
\end{align}
upon defining
\begin{align}\label{WrenoprG}
\begin{split}
G^n := \sum_{n_1+n_2\leq n} \big( \Gamma^{\leq n_1+1} A +  \Gamma^{\leq n_1+1} h e_z \big)
  \cdot 
  \Gamma^{n_2} W 
\end{split}
\end{align}
and
\begin{align}\label{F5}
\begin{split}
F^n_5 := \sum_{n_1+n_2\leq n} \big( \Gamma^{\leq n_1+1} A + \Gamma^{\leq n_1+1} h e_z \big)
  \cdot {\mathbf D_t} 
  \Gamma^{n_2} W. 
\end{split}
\end{align}
By letting
\begin{align}
F^n_r = \sum_{i=0}^4 F^n_i,
\end{align}
we have obtained an equation of the form
\begin{align}\label{Wreno1'}
\begin{split}
{\mathbf D_t} \Gamma^n W = {\mathbf D_t}  G^n + Q_1^n + Q_2^n
+ C_1^n + C_2^n + F^n_r + F^n_5
\end{split}
\end{align}
with quadratic and cubic terms as in \eqref{WrenoQuad} and \eqref{WrenoCub}.

{\it Step 2: Estimates for the correction}.
We can directly verify that $G^n$ satisfies \eqref{WrenoG1}-\eqref{WrenoG2}
using the definition \eqref{WrenoprG}, the fact that $\Gamma^{\leq 1} A, \Gamma^{\leq 1} h \in \mathcal{O}_1$,
and \eqref{OiO1W} in Remark \ref{remOi}.

{\it Step 3: Remainder estimates}.
To conclude the proof of the proposition we need to handle the remainders in \eqref{Wreno1'}.
We have already proved that $F^n_r$ is an acceptable remainder satisfying \eqref{WrenoF1}-\eqref{WrenoF2},
so we only need to show that $F^n_5$ contributes an acceptable remainder plus other contributions that are
accounted for in the cubic terms \eqref{WrenoCub}.

First, we use $\Gamma^{\leq 1} A, \Gamma^{\leq 1} h \in \mathcal{O}_1$ to write
\begin{align*}
F^n_5 & = \sum_{n_1+n_2\leq n} \Gamma^{\leq n_1} \mathcal{O}_1 \cdot {\mathbf D_t}  \Gamma^{n_2} W;
\end{align*}
then, using $\Gamma^{\leq 1} \partial_t A, \Gamma^{\leq 1} \partial_t h \in \mathcal{O}_1$,
we express the right-hand side of  \eqref{Wrenopr1} as
\begin{align*}
\sum_{n_1+n_2\leq n} \Gamma^{n_2} W \cdot \Gamma^{n_1}  \mathcal{O}_1.
\end{align*}
Therefore, using the full equation \eqref{Wrenopreq},
and adopting the same notation \eqref{WrenoQuad}-\eqref{WrenoCub} in the statement, we have
\begin{subequations}\label{F5'}
\begin{align}
F^n_5 
  \label{F5q0}
  & =  \sum_{n_1+n_2+n_3\leq n} \Gamma^{\leq n_1} \mathcal{O}_1 \cdot \Gamma^{n_2}  \mathcal{O}_1 \cdot  \Gamma^{n_3} W
  \\
  \label{F5q}
  & + \sum_{n_1+n_2\leq n} \Gamma^{\leq n_1} \mathcal{O}_1
  \cdot \big( Q^{n_2}_1 + Q^{n_2}_2 \big)
  \\
  \label{F5c}
  & + \sum_{n_1+n_2\leq n} \Gamma^{\leq n_1} \mathcal{O}_1
  \cdot \big( C^{n_2}_1 + C^{n_2}_2 \big)
  \\
  \label{F5r}
  & + \sum_{n_1+n_2\leq n} \Gamma^{\leq n_1} \mathcal{O}_1
  \cdot \big( F^{n_2}_0 + F^{n_2}_1 + F^{n_2}_2 \big).
\end{align}
\end{subequations}

The terms \eqref{F5q0} are acceptable remainders satisfying \eqref{WrenoF} in view of \eqref{WO1O1}.

From \dg{\eqref{WrenoQuad}} we see that, adopting the notation \eqref{G<nfg}, the terms in \eqref{F5q} are of the form
\begin{align}
\eqref{F5q} = \sum_{n_1+n_2\leq n} \Gamma^{\leq n_1} \mathcal{O}_1
  \cdot \big( \sum_{n_2+n_3\leq n_2} \Gamma^{n_2} W \cdot \Gamma^{n_3} \nabla V_\omega
  + \sum_{n_2+n_3\leq n_2-1} \Gamma^{\leq n_2+1} V_\omega \cdot \Gamma^{n_3} \nabla W \big);
\end{align}
we then see that these terms are actually cubic terms as in \eqref{WrenoCub} and, therefore, are accounted for in
the main equation \eqref{Wreno1}.

Next, we claim that also the terms \eqref{F5c} are of the same `cubic' form \eqref{WrenoCub}.
Indeed, let us look at the first of the two summands in \eqref{F5c}, that is,
\begin{align}\label{F5c1}
\sum_{n_1+n_2\leq n} \Gamma^{\leq n_1} \mathcal{O}_1 \cdot C^{n_2}_1
	= \sum_{n_1+n_2 + n_3 +n_4 \leq n} \Gamma^{\leq n_1} \mathcal{O}_1 \cdot
	  \Gamma^{n_2} \mathcal{O}_1 \cdot \Gamma^{n_3} W \cdot \Gamma^{n_4} \nabla V_\omega.
\end{align}
Observing that $\Gamma^{\leq n_1} \mathcal{O}_1 \cdot \Gamma^{n_2} \mathcal{O}_1$
satisfies the same bounds of $\Gamma^{n_1+n_2} \mathcal{O}_1$ (better ones,  in fact, of $\mathcal{O}_2$-type),
we see that \eqref{F5c1} is accounted for in \eqref{WrenoCub}. The same reasoning applies to the term involving $C^{n_2}_2$ in \eqref{F5c}.

Finally, we look at \eqref{F5r}. As already shown above the terms $ F^{n_2}_i$, $i=0,1,2$ satisfy the acceptable remainder bounds
in \eqref{WrenoF1}-\eqref{WrenoF2}. Then it is not hard to see that \eqref{F5r} is also an acceptable remainder:
if $n_2 \geq n_1$, so that $n_1 \leq N_1-10$, we use the bound \eqref{defO1b} on $\Gamma^{n_1} \mathcal{O}_1$ and
the bound \eqref{WrenoF2} on $F^{n_2}_i$;
if instead $n_2 \leq n_1$ we use the bound \eqref{defO1a} on $\Gamma^{n_1} \mathcal{O}_1$ and
the bound \eqref{WrenoF1} on $F^{n_2}_i$.
This concludes the proof of the proposition.
\end{proof}

\subsection{Transport estimates and proof of Proposition \ref{mainpropW}}\label{secTransp}

We begin with a general result about propagation of $L^p_{x, z}$ norms for transport equation:

\begin{lemma}[Bounds for the transport equation]\label{lemT}
Let ${\bf D}_t = \partial_t + U \cdot \nabla$ as above, and consider $Z=Z(t,x,z)$ a solution of
\begin{align}\label{lemTeq}
{\bf D}_t Z = N.
\end{align}
Then, for all $t \leq T_{\e_1}$, we have
\begin{align}\label{lemTconc'}
& {\| Z(t) \|}_{L^p_{x,z}} \leq {\| Z(0) \|}_{L^p_{x,z}} + C\int_0^t {\| N(s) \|}_{L^p_{x,z}} \, ds.
\end{align}
\end{lemma}


\begin{proof}
We begin by proving bounds for the Lagrangian flow associated to $U$.
Let $\Phi = \Phi_t$ be such that $\dot{\Phi}(t) = U(t,\Phi)$ with $\Phi(0) = \mathrm{id}$.
We want to show that
\begin{align}\label{prT1}
\sup_{t \leq T_{\e_1}} | \nabla \Phi_t(t) - \mathrm{id} | < 1/2.
\end{align}
We do this by a bootstrap argument. Assume that \eqref{prT1} holds true and denote $J(t) := \nabla \Phi_t(t)$.
Since $\dot{J}(t) = J(t) \nabla U(t,\Phi_t)$, using \dg{\eqref{lemUdt} } to express $U$,
we have
\begin{align*}
J(t) - \mathrm{id} & = \int_0^t J(s) \nabla U(s,\Phi_s) \, ds
\\
& = \int_0^t J(s) \nabla \big[ \partial_s (A - h \, e_z) \big] (s,\Phi_s) \, ds
   + \int_0^t J(s) \nabla V_\omega (s,\Phi_s) \, ds
   + \int_0^t J(s) \nabla R (s,\Phi_s) \, ds
\\
& = J(s) \nabla (A - h \, e_z)(s,\Phi_s) \Big|_{s=0}^{s=t}
  - \int_0^t J(s) \nabla U(s,\Phi_s) \nabla (A - h \, e_z) (s,\Phi_s) \, ds
   \\
   & + \int_0^t J(s) \nabla V_\omega (s,\Phi_s) \, ds
   + \int_0^t J(s) \nabla R (s,\Phi_s) \, ds
   \\ & =: J_1(t) - J_1(0) + J_2(t) + J_3(t) + J_4(t).
\end{align*}
For the first term in the above right-hand side we use $|J(t)| < 3/2$ and that $\nabla A, \nabla h \in\mathcal{O}_1$
(see Lemma \ref{lemO1} and \eqref{defO1b} in Definition \ref{defOi}) to estimate
\begin{align*}
| J_1(s) | \lesssim \e_0 \js^{-1+}
\end{align*}
For the second term we use in addition that $\nabla U \in\mathcal{O}_1$, and bound
\begin{align*}
|J_2(t)| \lesssim \int_0^t |\nabla U(s,\Phi_s)| | \nabla (A - h \, e_z) (s,\Phi_s)| \, ds
  \lesssim \int_0^t \e_0 \js^{-1+} \cdot \e_0 \js^{-1+} \, ds \lesssim \e_0^2
\end{align*}
For the third term we use the bound on $V_\omega$ in \eqref{Voboundlo} to estimate
\begin{align*}
|J_3(t)| \lesssim \int_0^t |\nabla V_\omega(s,\Phi_s)| \, ds \lesssim t \e_1 \lesssim \e_1^{\delta},
\end{align*}
for $t\leq T_{\e_1} = \e_1^{-1+\delta}$.
For the last term, we use that $\nabla R \in \mathcal{O}_2$ (see Lemma \ref{lemO2} and \eqref{defO2b})
and obtain a bound $| J_4(s) | \lesssim \e_0^{1+}$.
Putting all these together shows that, for $\e_0,\e_1$ small enough,
\begin{align*}
\sup_{t \leq T_{\e_1}} | \nabla \Phi_t(t) - \mathrm{id} | < 1/4,
\end{align*}
and therefore we obtain \eqref{prT1}.

We can then use the Lagrangian map to integrate the flow \eqref{lemTeq},
\begin{align*}
    Z(t,\Phi_t(x,z)) = Z(0,x,z) + \int_0^t N(s,\Phi_s(x,z)) \, ds,
\end{align*}
and then deduce \eqref{lemTconc'} by Minkowski's inequality
and changing variables using \eqref{prT1} to control the Jacobian.
\end{proof}

Next, we apply Lemma \ref{lemT} to conclude the proof of the main Proposition \ref{mainpropW}.
The main task left is to obtain suitable bounds on the quadratic (and cubic) terms on the right-hand side of \eqref{Wreno1}.

\begin{proof}[Proof of Proposition \ref{mainpropW}]
For this proof we define
\begin{align}
\delta_n = \left\{
\begin{array}{lcl}
0 &  \quad \mbox{if} & \quad n\leq N_1-10,
\\
\delta & \quad  \mbox{if} & \quad  n \in ( N_1-10, N_1+12] \cap \Z,
\end{array}
\right.
\end{align}
and use the short-hand
\begin{align}\label{notLLpq}
L := L^2_{x,z} \cap L^{6/5}_{x,z}
\end{align}
to denote the relevant Lebesgue space.

We start from \eqref{Wreno1}, apply Lemma \ref{lemT} to $Z = \Gamma^n W - G^n$,
and use the bound \eqref{WrenoG1} for $G^n$ to obtain
\begin{align}\label{propWpr0}
\begin{split}
{\| \Gamma^n W(t) \|}_{L}
& \leq {\| \Gamma^n W(0) \|}_{L} + C \e_0 \e_1 \jt^{\delta_n}
  + C \int_0^t {\| Q_1^n(s)  \|}_{L} + {\| Q_2^n(s)  \|}_{L}  \, ds
  \\
  & + C \int_0^t {\| C_1^n(s)  \|}_{L} + {\| C_2^n(s)  \|}_{L}  \, ds
 + C \int_0^t {\| F^n(s) \|}_{L} \, ds. 
\end{split}
\end{align}
From \eqref{init1'} and \eqref{init0},
we can bound the contribution at the initial time
\begin{align}\label{propWprdata}
{\| \Gamma^n W(0) \|}_{L}  \leq C_0 \e_1;
\end{align}
this is consistent with \eqref{propWconclow}-\eqref{propWconchigh} (for $j=0$) by taking $C$ large enough.
Moreover, using \eqref{WrenoF1}, we can bound
\begin{align*}
\int_0^t {\| F^n(s) \|}_{L} \, ds \lesssim \int_0^t \e_1\e_0^{1+} \js^{-1-\delta} \, ds \lesssim \e_1 \e_0^{1+},
	\qquad n\leq N_1-10,
\end{align*}
and, similarly, using \eqref{WrenoF2},
\begin{align*}
\int_0^t {\| F^n(s) \|}_{L} \, ds \lesssim \int_0^t \e_1\e_0^{1+} \js^{-1+\delta} \, ds \lesssim \e_1 \e_0^{1+} \jt^\delta
	\qquad n\leq N_1+12.
\end{align*}
These last two bounds are also consistent with the desired conclusions \eqref{propWconclow}-\eqref{propWconchigh}.
Therefore, 
we see that the proof of \eqref{propWconclow}-\eqref{propWconchigh} would follow with $j=0$
if we can show that, for all $t \leq T_{\e_1} := \bar{c}\e_1^{-1+\delta}$ with $\bar{c}$ small enough,
\begin{align}
\label{propWprQbd1}
\int_0^t {\| Q^{n}_1 (s) \|}_{L} \, ds  \lesssim \e_1^{1+} \jt^{\delta_n}, 
\\
\label{propWprQbd2}
\int_0^t {\| Q^{n}_2 (s) \|}_{L} \, ds  \lesssim \e_1^{1+} \jt^{\delta_n}, 
\end{align}
and
\begin{align}\label{propWprCbd}
\int_0^t {\| C^{n}_1 (s) \|}_{L} + {\| C^{n}_2 (s) \|}_{L} \, ds  \lesssim \e_1^{1+} \jt^{\delta_n}.
\end{align}

{\it Proof of \eqref{propWprQbd1}.}
Since $Q_1^n$ is given by \eqref{WrenoQuad1},
we observe that in order to obtain \eqref{propWprQbd1} it will suffice to prove the bound
\begin{align}\label{propWpr11'}
{\| \Gamma^{n_1} W(s) \cdot \Gamma^{n_2} \nabla V_\omega(s) \|}_{L}
  \lesssim \e_1^2 \js^\delta, \qquad n_1+n_2 \leq N_1 + 12,
\end{align}
for all $s \leq T_{\e_1}$; indeed, since $T_{\e_1} = \bar{c}\e_1^{-1+\delta}$,
\begin{align*}
\int_0^t \e_1^2 \js^\delta \, ds \leq 2\e_1^2 T_{\e_1}^{1+\delta} \leq 2\bar{c} \e_1^{1+\delta^2}.
\end{align*}


To prove \eqref{propWpr11'} in the case $n_2 \leq n_1$, with $n_1+n_2 \leq N_1+12$, we use the a priori estimate \eqref{propWashigh}
and \eqref{Voboundlo} after Sobolev's embedding (in $x$):
\begin{align}\label{propWpr11'0}
\begin{split}
{\| \Gamma^{n_1} W(s) \cdot \Gamma^{n_2} \nabla V_\omega(s) \|}_{L}
  & \lesssim
  {\| \Gamma^{\leq N_1+12} W(s) \|}_{L}
  {\| \Gamma^{\leq N_1-13} \nabla V_\omega(s) \|}_{L^\infty_{x,z}}
  \\
  & \lesssim \e_1 \js^\delta \cdot {\| \Gamma^{\leq N_1-10} V_\omega(s) \|}_{L^\infty_z L^2_x}
  \\
  & \lesssim \e_1^2 \js^\delta.
\end{split}
\end{align}
For the case $n_1 \leq n_2$, we use instead H\"older's inequality (recall \eqref{notLLpq})
with \eqref{Vohi} and \eqref{propWaslow} after Sobolev's embedding (in $x$):
\begin{align*}
{\| \Gamma^{n_1} W(s) \cdot \Gamma^{n_2} \nabla V_\omega(s) \|}_{L}
  & \lesssim
  {\| \Gamma^{\leq N_1-13} W(s) \|}_{L^\infty_{x,z} \cap L^3_{x,z}}
  {\| \Gamma^{\leq N_1+12} \nabla V_\omega(s) \|}_{L^2_{x,z}}
  \\
  & \lesssim {\| \Gamma^{\leq N_1-10} W(s) \|}_{L^2_{x,z}}
  \cdot \e_1 \js^{\delta}
  \\
  & \lesssim \e_1^2 \js^\delta.
\end{align*}
Note that this is the place where we use the highest order estimate \eqref{Vohi} for $V_\omega$.

{\it Proof of \eqref{propWprQbd2}.}
We now look at the quadratic terms $Q_2^n$ as given by \eqref{WrenoQuad2}; these are linear combinations of terms
of the form $\Gamma^{n_1+1} V_\omega \cdot \Gamma^{n_2} \nabla W$ for $n_1+n_2 \leq n-1$.
As before, we see that for \eqref{propWprQbd2} it suffices to show the stronger bound
\begin{align}\label{propWpr11''}
{\| \Gamma^{n_1+1} V_\omega(s) \cdot \Gamma^{n_2} \nabla W(s) \|}_{L} \lesssim \e_1^2 \js^\delta,
  \qquad n_1+n_2 \leq N_1 + 11.
\end{align}

In the case $n_2 \leq n_1$, with $n_1+n_2 \leq N_1+11$, we use H\"older's inequality
with \eqref{Voboundhi} to estimate $\Gamma^{\leq N_1+12} V_\omega$,
and the a priori estimate \eqref{propWaslow} for $\Gamma^{N_1-10}W$:
\begin{align*}
{\| \Gamma^{n_1+1} V_\omega(s) \cdot \Gamma^{n_2} W(s) \|}_{L}
  & \lesssim
  {\| \Gamma^{\leq N_1+12} V_\omega(s) \|}_{L^\infty_z L^2_x}
  {\| \Gamma^{\leq N_1-12} \nabla W(s) \|}_{ L^2_z L^\infty_x \cap L^{6/5}_z L^3_x }
  \\
  & \lesssim \e_1 \js^\delta \cdot {\| \Gamma^{\leq N_1-10} W(s) \|}_{ L^2_{x,z} \cap L^{6/5}_{x,z} } 
  \\
  & \lesssim \e_1^2 \js^\delta.
\end{align*}

In the case $n_2 \geq n_1$, with $n_1+n_2 \leq N_1+11$, we use instead
the bound \eqref{Volo} on low norms of $V_\omega$, and the a priori estimate \eqref{propWashigh} on high norms of $W$:
\begin{align*}
{\| \Gamma^{n_1+1} V_\omega(s) \cdot \Gamma^{n_2} W(s) \|}_{L}
  & \lesssim
  {\| \Gamma^{\leq N_1-12} V_\omega(s) \|}_{L^\infty_{x,z}} {\| \Gamma^{\leq N_1+10} \nabla W(s) \|}_{L}
  \lesssim \e_1 \cdot \e_1 \js^\delta.
\end{align*}

{\it Proof of \eqref{propWprCbd}.}
Recall the form of the cubic terms from \eqref{WrenoCub};
we only detail how to treat the first one, that is,
\begin{align}\label{WrenoCub'}
\begin{split}
C_1^n
  & := \sum_{n_1+n_2+n_3\leq n} \Gamma^{n_1} W \cdot \Gamma^{n_2} \nabla V_\omega
  \cdot \Gamma^{n_3} \mathcal{O}_1,
\end{split}
\end{align}
as the other term can be dealt with in the same way.

We first observe that if $n_3 \leq n_1+n_2$, then, we can use \eqref{propWpr11'} to estimate
as follows: for all $n \leq N_1+12$
\begin{align*}
{\| C_1^n \|}_{L} & \lesssim \sup_{n_1+n_2 \leq N_1+12} {\| \Gamma^{n_1} W \cdot \Gamma^{n_2} \nabla V_\omega \|}_{L}
  \cdot {\| \Gamma^{\leq N_1-10} \mathcal{O}_1 \|}_{L^\infty_{x,z}}
  \lesssim \e_1^2 \jt^\delta \cdot \e_0 \jt^{-1+},
\end{align*}
which is more than sufficient.
When instead $n_3 \geq n_1+n_2$ we cannot use directly \eqref{propWpr11'}, but using $n_1,n_2 \leq N_1-12$,
together with the bounds on low norms \eqref{propWaslow} and \eqref{Volo}, and  \eqref{defO1a}, we get
\begin{align*}
{\| C_1^n \|}_{L} & \lesssim {\| \Gamma^{\leq N_1-12} W \|}_{L} {\| \Gamma^{\leq N_1-12} \nabla V_\omega \|}_{L^\infty_{x,z}}
  \cdot {\| \Gamma^{\leq N_1+12} \mathcal{O}_1 \|}_{L^\infty_{x,z}}
  \lesssim \e_1 \cdot \e_1 \cdot \e_0 \jt^{3p_0}.
\end{align*}
Putting these together we get
\begin{align}\label{WrenoCubbd}
    {\| C_1^n \|}_{L} + {\| C_2^n \|}_{L} & \lesssim \e_1^2 \e_0 \jt^{3p_0}, \qquad n \leq N_1+12;
\end{align}
upon time integration, we see that the last two bounds above are more than sufficient for \eqref{propWprCbd}.
This concludes the proof of \eqref{propWconclow}-\eqref{propWconchigh} for $j=0$.

\smallskip
{\it Estimates for the time derivative}.
We now prove \eqref{propWconclow}-\eqref{propWconchigh} for $j=1$.
From \eqref{lemDtWeq'} we have, for all $n \leq N_1+11$,
\begin{align}\label{1lemDtWeq'}
\begin{split}
\partial_t \Gamma^n W = - U \cdot \nabla \Gamma^n W
  & + \sum_{n_1+n_2\leq n-1} \Gamma^{\leq n_1+1} U \cdot \Gamma^{n_2} \nabla W
  \\
  & + \sum_{n_1+n_2 \leq n} \Gamma^{n_1} W \cdot \Gamma^{n_2} \nabla X
  + \Gamma^{\leq n} F,
\end{split}
\end{align}
where, recall, $F$ is defined in \eqref{Weq1}.
Using \eqref{lemGU0}, 
and with the notation for quadratic and cubic terms from \eqref{WrenoQuad1}-\eqref{WrenoQuad2}
and \eqref{WrenoCub}, equation \eqref{1lemDtWeq'} is
\begin{align}\label{1lemDtWeq''}
\begin{split}
\partial_t \Gamma^n W & = \mathcal{O}_1 \cdot \nabla \Gamma^n W + V_\omega \cdot \nabla \Gamma^n W
  + V_\omega \cdot \mathcal{O}_1 \cdot \nabla \Gamma^n W
\\
  & + \sum_{n_1+n_2\leq n-1} \Gamma^{\leq n_1+1} \mathcal{O}_1 \cdot \Gamma^{n_2} \nabla W
   + Q_2^n
  + C_2^n
  \\
  & + \sum_{n_1+n_2 \leq n} \Gamma^{n_1} W \cdot \Gamma^{n_2} \mathcal{O}_1
  + Q_1^n
  + \Gamma^{\leq n} F.
\end{split}
\end{align}

To obtain \eqref{propWconclow}, respectively \eqref{propWconchigh}, we need to show that all the terms
on the right-hand side of \eqref{1lemDtWeq''} are bounded by $\e_1$ when $n\leq N_1-11$,
respectively, by $\e_0\e_1 \jt^\delta$ when $n \leq N_1+11$.

Note that the bounds \eqref{propWpr11'} and \eqref{propWpr11''} already give that
\begin{align*}
{\| Q^n_1(t) \|}_{L} + {\| Q^n_2(t) \|}_{L} \lesssim \e_1^2 \jt^\delta, \qquad n \leq N_1+12,
\end{align*}
(since the summation in the definition of $Q^n_2$ goes up to $n-1$, see \eqref{WrenoQuad2}),
which is more than sufficient for the desired bounds.
Similarly, the bounds established before on $\Gamma^{\leq n} F$ are also sufficient;
indeed, $\Gamma^{\leq n} F$ is a combination of $F^n_0$ (see \eqref{F0})
which is an acceptable remainder satisfying \eqref{WrenoF2},
and of a cubic term of the form $C_1^n$ (see the last term in \eqref{Wrenopr1c})
which satisfies \eqref{WrenoCubbd}. The bound \eqref{WrenoCubbd} also handles the term $C_2^n$ in \eqref{1lemDtWeq''}.

\fp{The term $ V_\omega \cdot \nabla \Gamma^n W$, for $n\leq N_1+11$, is similar to one of the terms
appearing in $Q_1^{N_1+12}$ and, in particular, it can be estimated as in \eqref{propWpr11'0} (where the presence of 
$\nabla$ on $V_\omega$ is not used).}

To conclude, we need to estimate the four terms involving the $\mathcal{O}_1$ factors in \eqref{1lemDtWeq''}.
These can be handled directly using \eqref{OiO1W} in Remark \ref{remOi},
with the bounds \eqref{WdotO1b}-\eqref{WdotO1a} giving more than what is needed.
This concludes the proof of Proposition \ref{mainpropW}.
\end{proof}


\bigskip
\section{Proof of Proposition \ref{propvelpot}}\label{secvelpot}
The aim of this section is to prove Proposition \ref{propvelpot}, that is, establish the bound
\begin{align}\label{velpotest'}
& \sum_{r+k \leq N_0-20}
  {\big\| |\nabla|^{1/2} \varphi (t) \big\|}_{Z^r_k(\R^2)} \lesssim \e_0 \jt^{3p_0}
\end{align}
for all $t \in [0,T]$, under the a priori energy bounds on the velocity, vorticity and height \eqref{apriorie0},
and the decay bound on the height \eqref{aprioridecay}
and velocity \eqref{aprioridecayv}.
To prove \eqref{velpotest'} we will in part 
rely on the proof of Proposition \ref{alphaprop} without repeating most of the arguments,
and on some of the material in Section \ref{secVorticity}.
We are going to use the following strategy:

\begin{itemize}

 \item We bootstrap a (weak) bound for the $L^2\cap L^{6/5}$ norm of $W$ with a high number ($N_0-20$) of vector fields,
 just using the high Sobolev energy bound \eqref{apriorie0};
 this bound is just of size $\e_0 \jt^{p_0}$, as opposed to the much better bound $\e_1$ for the low norms;
 see the assumptions of Proposition \eqref{alphaprop} and Proposition \ref{mainpropW}.

 \item We input the above (weak) information into the fixed point argument used to obtain Proposition \ref{alphaprop},
 and obtain corresponding (weak) bounds on $\alpha$, and therefore on the vector potential $V_\omega$.

 \item Finally we obtain bounds for $|\nabla|^{1/2} \varphi$ from trace estimates, thanks to the bounds for
 $\partial^i \Psi = V^i + \partial^i h \partial_z \Psi - V^i_\omega $ that are
 directly implied by the bounds on $V_\omega$ and $V$.

\end{itemize}

\begin{remark}
Note that we get the slightly faster growth rate $3p_0$
in \eqref{velpotest'} as opposed to the more natural $p_0$, as for the high energies,
because in the course of proving the above
bounds we will work in the ``flattened'' variables $(V, \Psi, V_\omega, W)$
instead of the variables $(v, \psi, v_\omega, \omega)$ in the original domain $\D_t$.
When measured in low-order norms (say, less than $N_1$ vector fields and gradients) 
all of the ``flattened'' quantities are equivalent
to the original ones, see \eqref{complow}. 
On the other hand, for higher-order norms, one needs to control
products of high-order norms of $h$ with lower-order norms of the flattened variables in $L^2$. 
Since we do not propagate uniform control on 
all of the flattened variables in $L^2_zL^\infty_x$,
this winds up generating terms which grow slightly faster that $\jt^{p_0}$,
which ultimately leads to the growth rate in \eqref{velpotest'}.
This slightly weaker bound is still sufficient for the rest of our arguments to close, 
in particular those in Section \ref{secVorticity} (see Definition \ref{defOi} for the $\mathcal{O}_1$ class)
and Section \ref{secdecay} (see \eqref{decayas1}).
\end{remark}


\subsection{(Weak) Bounds on the vorticity}
We use the notation of Section \ref{secVorticity}, see in particular Subsection \ref{ssecVorticity1},
and aim to prove the bootstrap Proposition \ref{mainpropW0}, 
which we rewrite here for convenience.

\begin{prop}\label{mainpropW'}
Assume that the a priori bounds \eqref{apriorie0}, \eqref{aprioridecayv},
and \eqref{aprioridecay} hold.
Let $W$ be as defined in \eqref{defW0}, and assume \eqref{aprioriWL} and \eqref{aprioriW0}.
Recall the definition of the space $\mX^n$ from \eqref{omegaflatspace0}:
\begin{align}\label{omegaflatspacebis'}
\begin{split}
{\| f \|}_{\mX^n} := \sum_{|r|+|k| \leq n}
{\big\| \underline{\Gamma}^k \nabla^r_{x,z} \,f \big\|}_{L(\R^2 \times \R_{\leq0})}, 
  \qquad L:= L^2_{x,z} \cap L^{6/5}_{x,z}.
\end{split}
\end{align}
Then, for all $t\in[0,T]$, $T\leq T_{\e_1}$, we have the bound
\begin{align}
\label{propWconchigh'}
& {\| W(t) \|}_{\mX^{N_0-20}} \leq c_{W} \e_0 \jt^{2p_0}.
\end{align}
\end{prop}

In the proof of Proposition \ref{mainpropW'} and in other places, we are going to need a
basic lemma about transfer of norms from $\D_t$ to the flat domain.
As before, for a given $f:[0,T] \times \D_t \rightarrow \R$, we use the corresponding capital letter
to define, for $t\in[0,T]$, $x\in\R^2$ and $z\leq 0$,
\begin{align*}
F(t,x,z) = f(t,x,z+h(t,x)), \qquad  f(t,x,y) = F(t,x,y-h(t,x)).
\end{align*}
In what follows we use the convention about repeated applications of vector field from Subsection \ref{secvfO}
and, for clarity,
we will underline the $3$d vector fields, while reserving $\Gamma$ for the $2$d vector fields.

\begin{lemma}\label{Lemcomp}
With the above definitions, the notation from \ref{secvfO},
and under the a priori bounds on $h$ from \eqref{apriorie0} and \eqref{aprioridecay},
we have the following schematic identity: if $n \leq N_0$,
\begin{align}\label{compid1}
\begin{split}
\underline{\Gamma}^n F(t,x,z) = (\underline{\Gamma}^n f)(t,x,z+h(t,x))
  & + (\underline{\Gamma}^{\leq n} f)(t,x,z+h(t,x)) \cdot O\big(\Gamma^{\leq n/2+1} h(t,x)\big)
  \\
  & + (\underline{\Gamma}^{\leq n/2+1} f)(t,x,z+h(t,x)) \cdot O\big(\Gamma^{\leq n} h(t,x)\big)
\end{split}
\end{align}
where the notation $G = O(\Gamma^{\leq k}h)$ here means that
\begin{align}\label{compidO}
\big| \Gamma^{\ell} G \big| \lesssim \sum_{j=0}^{k+\ell} \big| \Gamma^{j} h(t,x) \big|,
\end{align}
with an absolute implicit constant (depending on $\ell$).
Similarly, we can write
\begin{align}\label{compid2}
\begin{split}
\underline{\Gamma}^n f(t,x,y) = (\underline{\Gamma}^n F)(t,x,y-h(t,x))
  & + (\underline{\Gamma}^{\leq n} F)(t,x,y-h(t,x)) \cdot O\big(\Gamma^{\leq n/2+1} h(t,x)\big)
  \\
  & + (\underline{\Gamma}^{\leq n/2+1} F)(t,x,y-h(t,x)) \cdot O\big(\Gamma^{\leq n} h(t,x)\big).
\end{split}
\end{align}

In particular, for $p \in [2,\infty]$, and $n \leq N_1$, there exists constants $C_1,C_2>0$ such that
\begin{align}\label{complow}
C_1 \sum_{k=0}^n{\big\| \underline{\Gamma}^k f(t) \big\|}_{L^p(\D_t)} \leq
 \sum_{k=0}^n {\big\| \underline{\Gamma}^{k} F(t) \big\|}_{L^p_{x,z}}
 \leq C_2 \sum_{k=0}^n {\big\| \underline{\Gamma}^k f(t) \big\|}_{L^p(\D_t)}.
\end{align}

For $n \leq N_0$ we have instead
\begin{align}\label{comphigh1}
\begin{split}
{\| \underline{\Gamma}^n F(t) \|}_{L^2_{x,z}} & \lesssim
  \sum_{k\leq n} {\| \underline{\Gamma}^k f(t) \|}_{L^2(\D_t)}
  +  \e_0 \jt^{p_0} \sum_{k\leq n/2+3} {\| \underline{\Gamma}^k f(t)\|}_{L^2(\D_t)}
\end{split}
\end{align}
and, similarly,
\begin{align}\label{comphigh2}
\begin{split}
{\| \underline{\Gamma}^n f(t) \|}_{L^2(\D_t)} & \lesssim
  \sum_{k\leq n} {\| \underline{\Gamma}^k F(t) \|}_{L^2_{x,z}}
  +  \e_0 \jt^{p_0} \sum_{k\leq n/2+3} {\| \underline{\Gamma}^k F(t) \|}_{L^2_{x,z}}.
\end{split}
\end{align}

\end{lemma}

\begin{proof}
The identities \eqref{compid1} and \eqref{compid2}
follow from applying repeatedly the composition formulas \eqref{equivpr1}-\eqref{equivpr1'}
and using the uniform bound on the $L^\infty$ norm of $h$ from \eqref{aprioridecay}
to verify the property \eqref{compidO}.

The estimates \eqref{complow} then follow directly since $|\Gamma^k h| \lesssim \e_0$ for all $k \leq N_1$.
For \eqref{comphigh1} we instead apply H\"older's inequality to \eqref{compid1}
by estimating in $L^\infty$ the term $O(\Gamma^{\leq n/2+1} h)$ and in $L^2$ the term $O(\Gamma^{\leq n} h)$,
placing $\underline{\Gamma}^{\leq n/2+1} f$ in $L^2_{z}L^\infty_x$ and then using Sobolev embedding. 
The estimate \eqref{comphigh2} follows similarly from \eqref{compid2}.
\end{proof}


A statement similar to the one in Lemma \ref{Lemcomp} holds for restrictions to the boundary:

\begin{lemma}\label{Lemrestr}
With the same notation, definitions and a priori assumptions in Lemma \ref{Lemcomp},
and denoting $\wt{g}(t,x) := g(t,x,h(t,x))$, we have the following schematic identity:
if $n \leq N_0$,
\begin{align}\label{restrid1}
\begin{split}
\Gamma^n \wt{f} = \wt{\underline{\Gamma}^n f}
  & + \wt{\underline{\Gamma}^{\leq n} f} \cdot O\big(\Gamma^{\leq n/2+1} h \big)
  + \wt{\underline{\Gamma}^{\leq n/2+1} f} \cdot O\big(\Gamma^{\leq n} h \big).
\end{split}
\end{align}
In particular, for $p \in [2,\infty]$, and $n \leq N_1$, there exists constants $C_1,C_2>0$ such that
\begin{align}\label{restrlow}
C_1 \sum_{k=0}^n{\big\| \Gamma^n \wt{f}(t) \big\|}_{L^p(\R^2)} \leq
 \sum_{k=0}^n {\big\| \wt{\underline{\Gamma}^{k}f}(t) \big\|}_{L^p(\R^2)}
 \leq C_2 \sum_{k=0}^n {\big\| \Gamma^n \wt{f}(t) \big\|}_{L^p(\R^2)}.
\end{align}
For $n \leq N_0-2$ we have instead
\begin{align}\label{restrhigh1}
\begin{split}
{\| \Gamma^n \wt{f}(t) \|}_{L^p(\R^2)} & \lesssim
  \sum_{k\leq n} {\| \wt{\Gamma^k f}(t) \|}_{L^p(\R^2)}
  +  \e_0 \jt^{p_0} \sum_{k\leq n/2+1} {\| \wt{\Gamma^k f}(t) \|}_{L^p(\R^2)}
\end{split}
\end{align}

\end{lemma}

\begin{proof}
    The proof is similar to that of Lemma \ref{Lemcomp}.
The identity \eqref{restrid1} follows from applying repeatedly \eqref{restrpr1} 
and using the uniform a priori bound on $h$ in $L^\infty$.
The estimate \eqref{restrlow} then follows immediately, 
while \eqref{restrhigh1} follows using H\"older, estimating $\Gamma^{\leq n} h$ in $L^\infty$,
followed by Sobolev's embedding and \eqref{apriorie0}. 
\end{proof}

We can now give the proof of \eqref{aprioridth2}:


\begin{proof}[Proof of \eqref{aprioridth2}]
Since $\partial_t h = \wt{v} \cdot (-\nabla h,1)$,
distributing vector fields we see that a schematic formula like the one in \eqref{restrid1},
with the notation \eqref{compidO}, holds and, in particular,
\begin{align*}
\begin{split}
|\Gamma^n \partial_t h| \lesssim |\wt{\underline{\Gamma}^n v}|
  & + |\wt{\underline{\Gamma}^{\leq n} v}| \cdot |O\big(\Gamma^{\leq n/2+1} h \big)|
  + |\wt{\underline{\Gamma}^{\leq n/2+1} v}| \cdot |O\big( \fp{\Gamma^{\leq n+1} } h \big)|.
\end{split}
\end{align*}
It follows that 
\begin{align*}
{\| \Gamma^n \partial_t h \|}_{L^2} & \lesssim 
  {\| \underline{\Gamma}^{\leq n+1} v \|}_{L^2(\D_t)} \big( 1 + {\| \Gamma^{\leq n/2+1} h \|}_{L^\infty} \big)
  \\
  & + {\| \underline{\Gamma}^{\leq n/2+1} v \|}_{L^\infty(\D_t)} \cdot {\| \fp{\Gamma^{\leq n+1} } h \|}_{L^2}
  \lesssim \e_0\jt^{p_0},
\end{align*}
where in the last inequality we have used the apriori energy bound \eqref{apriorie0} for both $v$ and $h$,
and the decay bounds \eqref{aprioridecay} and \eqref{aprioridecayv} to control 
uniformly the $L^\infty$ norms.
\end{proof}

\begin{proof}[Proof of Proposition \ref{mainpropW'}]
We use the same notation from Subsection \ref{ssecVorticity1}, 
and the conventions from Subsection \ref{secvfO}.
Recall from \eqref{aprioriWL} and \eqref{aprioriW0} that we are assuming, for all $t\in[0,T]$,
$T\leq T_{\e_1}$, 
\begin{align}
\label{propWaslow'}
& {\| W(t) \|}_{\mX^{N_1-10}} \leq 2c_L \e_1,
\\
\label{propWashigh'}
& {\| W(t) \|}_{\mX^{N_0-20}} \leq 2c_W \e_0 \jt^{2p_0},
\end{align}
for some absolute constant $c_L,c_W>0$ large enough, and that the following assumption
on the initial data hold in view of \eqref{propWprdata0}:
\begin{align}\label{initW'}
 {\|W_0\|}_{\mathcal{X}^{N_1-10}} \leq C \e_1,\qquad
 {\|W_0\|}_{\mathcal{X}^{N_0-20}} \leq C \e_0.
\end{align}
We then aim to show that the improved bound \eqref{propWconchigh'}
holds for all $t \in [0, T]$.

We begin by writing the vorticity equation as in \eqref{Weq0}:
\begin{align}\label{Weq0'}
\begin{split}
& {\bf D}_t W  = W \cdot \nabla V - W^\ell \partial_\ell h \partial_z V^i,
\\
& {\bf D}_t := \partial_t + U \cdot \nabla, \qquad U := V - (\partial_t h + V^\ell \partial_\ell h) e_z.
\end{split}
\end{align}
We then apply vector fields as in Subsection \ref{ssecVorticitycomm}
and obtain the following equation (see Lemma \ref{lemDtW}):
\begin{align}\label{prweak0}
\begin{split}
{\bf D}_t \Gamma^n W & = \sum_{n_1+n_2\leq n-1} \Gamma^{\leq n_1+1} U \cdot \Gamma^{n_2} \nabla W
  + \sum_{n_1+n_2 \leq n} \Gamma^{n_1} W \cdot \Gamma^{n_2} \nabla V
  + \Gamma^{\leq n} F',
\end{split}
\end{align}
with $F' := - W^\ell \partial_\ell h \partial_z V^i$.
Compare this with \eqref{lemDtWeq} and note that the only difference is that in this case we do not
separate the linear and quadratic components in $V$, nor
distinguish the rotational and irrotational components.

To obtain estimates for $W$ based on \eqref{Weq0'} we need energy and decay bounds for $V$ and $U$.
First, using \eqref{complow} and the priori decay assumptions on $v$ in \eqref{aprioridecayv}
    we have
\begin{align}\label{aprioriVLinfty}
{\| \underline{\Gamma}^{\leq n} V \|}_{L^\infty_{x,z}} \lesssim
  {\| \underline{\Gamma}^{\leq n} v \|}_{L^\infty(\D_t)} \lesssim \e_0\jt^{-1} 
    + \e_1 \jt^\delta, \qquad n\leq N_1-5;
\end{align}
using \eqref{comphigh1} and the a priori energy control \eqref{apriorie0} we have
\begin{align}\label{aprioriVL2}
{\| \Gamma^{\leq n} V \|}_{L^2_{x,z}} \lesssim
  {\| \underline{\Gamma}^{\leq n} v \|}_{L^2(\D_t)}
  + \e_0\jt^{p_0} \sum_{k\leq n/2+3} {\| \underline{\Gamma}^k v \|}_{L^{2}(\D_t)}
  \lesssim \e_0 \jt^{2p_0}, \qquad n \leq N_0.
\end{align}
Then, from the definition of $U$ in \eqref{Weq0'}, 
    using $\partial_t h = v \cdot (-\nabla h, 1)$ at the free boundary
with the bounds on $V$ just used above,
and basic product estimates to handle the quadratic term $V\cdot\nabla h$, it follows that
    \begin{alignat}{2}\label{aprioriU}
& {\| \Gamma^{\leq n} U \|}_{L^\infty_zL^2_x} \lesssim \e_0 \jt^{2p_0}, &&\qquad n \leq N_0-10,
\\
\label{aprioriUdecay}
& {\| \Gamma^{\leq n} U \|}_{L^\infty_{x,z}} \lesssim \e_0 \jt^{-1+} + \e_1 \jt^\delta, &&\qquad n \leq N_1-5.
\end{alignat}

    Applying Lemma \ref{lemT} to \eqref{prweak0}, we have
\begin{align}\label{lemTconc}
& {\| \Gamma^n W(t) \|}_{L^p_{x,z}} \leq {\|\Gamma^n W(0) \|}_{L^p_{x,z}}
  + C\int_0^t {\| {\bf D}_s \Gamma^n W(s) \|}_{L^p_{x,z}} \, ds.
\end{align}
Therefore, to obtain \eqref{propWconchigh'} it then suffices to show that
\begin{align}\label{prweak1}
{\| {\bf D}_t \Gamma^n W(t) \|}_{L^2_{x,z} \cap L^{6/5}_{x,z}} \leq 
  c \e_0 \jt^{-1+2p_0}, \qquad n\leq N_0 - 20,
\end{align}
for all $t\leq T \leq T_{\e_1}$ and a sufficiently small absolute constant $c$. 

The estimate \eqref{prweak1} can be verified directly for all the terms on the
right-hand side of \eqref{prweak0} by using elementary product estimates,
the a priori assumptions \eqref{propWaslow'}-\eqref{propWashigh'}, the bounds on $V$ and $U$
in \eqref{aprioriVLinfty}-\eqref{aprioriU}, and the usual bounds on $h$ in \eqref{apriorie0} and \eqref{aprioridecay}.
    In particular, with $L:=L^2_{x,z} \cap L^{6/5}_{x,z}$ we claim that the following bounds hold:
\begin{subequations}
\label{prweakest}
\begin{alignat}{2}
\label{prweakesta}
& {\| \Gamma^{\leq n_1+1} U \cdot \Gamma^{n_2} \nabla W \|}_{L}
  \lesssim \e_0 \e_1 \jt^{\delta + 2p_0} + \e_0^2 \jt^{-1+2p_0},
&&\qquad n_1+n_2\leq N_0-20-1,
\\
\label{prweakestb}
& {\| \Gamma^{n_1} W \cdot \Gamma^{n_2} \nabla V \|}_{L}
  \lesssim \e_0 \e_1\jt^{\delta + 2p_0} + \e_0^2 \jt^{-1+2p_0},
&&\qquad n_1+n_2\leq N_0-20,
\\
\label{prweakestc}
& {\| \Gamma^n (W^\ell \partial_\ell h \partial_z V^i) \|}_{L}
    \lesssim \e_0^2 \e_1 \jt^{2p_0} + \e_0^3 \jt^{-1+2p_0} , &&\qquad n \leq N_0-20.
\end{alignat}
\end{subequations}
Notice that these bounds imply the desired \eqref{prweak1} since $\e_1\jt^{\delta} \ll \jt^{-1}$
for all $t \leq T_{\e_1}$.

Let us prove \eqref{prweakesta}.
When $n_1\geq n_2$ so that, in particular $n_2 \leq N_0/2-10\leq N_1-10$ we can use \eqref{propWaslow'} to estimate
$\Gamma^{n_2} \nabla W$, and use Sobolev embedding and \eqref{aprioriU} to estimate $\Gamma^{\leq n_1+1}U$:
\begin{align}\label{prweakest1}
\begin{split}
& {\| \Gamma^{\leq n_1+1} U \cdot \Gamma^{n_2} \nabla W \|}_{L}
  \lesssim {\| \Gamma^{\leq n_1+1} U \|}_{L^\infty_{x,z}} {\| \Gamma^{n_2} \nabla W \|}_{L}
  \lesssim \e_0\jt^{2p_0} \cdot \e_1.
\end{split}
\end{align}
When instead $n_1 \leq n_2$, so that $n_1 \leq N_0/2-10\leq N_1-10$, we can use
the a priori assumption \eqref{propWashigh'} for $\Gamma^{n_2} W$,
and use the decay estimate \dg{\eqref{aprioriUdecay} } to estimate $\Gamma^{\leq n_1+1}U$:
\begin{align*}
\begin{split}
& {\| \Gamma^{\leq n_1+1} U \cdot \Gamma^{n_2} \nabla W \|}_{L}
  \lesssim  {\| \Gamma^{\leq n_1+1} U \|}_{L^\infty_{x,z}} {\| \Gamma^{n_2} \nabla W  \|}_{L}
  \lesssim (\e_0\jt^{-1} + \e_1 \jt^\delta) \cdot \e_0 \jt^{2p_0}.
\end{split}
\end{align*}
These last two bounds are consistent with the right hand-side of \eqref{prweakesta}.
The other bounds in \eqref{prweakest} can be proven in the same way.
This concludes the proof of \eqref{prweak1} of the proposition.
\end{proof}


\subsection{Bounds for the vector potential}

We now state bounds on $V_\omega$ that follow from the bounds
in the high norm that we just obtained on $W$, see \eqref{propWashigh'}.

\begin{prop}[Bounds for $\alpha$ from bounds on $W$]\label{alphaprop'}
Let $\alpha: [0,T] \times \R^2 \times \R_- \mapsto \R^3$
be defined by $\alpha(t,x,z) := \beta(t,x,z+h(t,x))$
where $\beta$ solves the system \eqref{betaeq} in $\D_t$.
Assume that $h$ satisfies \eqref{apriorih}-\eqref{apriorihp} and \eqref{aprioridth2}-\eqref{aprioridthp},
and let $W$ be given so that, for $t\in[0,T]$,
\begin{align}
\label{alphaaso0'}
& {\| W(t) \|}_{\mX^{N_1-10}} \lesssim \e_1,
\\
\label{alphaaso1'}
& {\| W(t) \|}_{\mX^{N_0-20}} \lesssim \e_0 \jt^{2p_0}.
\end{align}
Then, there exists a unique fixed point $\alpha$ of the map in \eqref{alphafp}
in the space $\dYn^{N_0-20}$, which satisfies
\begin{align}
\label{alphaconcL'}
& {\| \alpha(t) \|}_{\dYn^{N_1-10}} \lesssim \e_1,
\\
\label{alphadtconcH'}
& {\| \alpha(t) \|}_{\dYn^{N_0-20}} \lesssim \e_0 \jt^{2p_0}.
\end{align}
\end{prop}

This proposition is an exact analogue of Proposition \ref{alphaprop}
stated with high norms that contain $N_0-20$ vector fields instead of $N_1+12$;
compare \eqref{alphaaso1} and \eqref{alphaaso1'}.
The conclusion \eqref{alphadtconcH'} is the one that naturally corresponds to \eqref{alphadtconcH}
with the different assumption.
The proof follows verbatim the one is Subsection  \ref{ssecpralphaprop}.
The only thing to observe is that the assumptions used on $h$ in the proof
of Proposition \ref{alphaprop} also suffice when working at a higher level of derivatives.
The only relevant aspect is that half of
the highest number of vector fields, that is $N_0/2-10$ here, needs to be (a couple of units) less than
two numbers:
the number of vector fields for which we have uniform bounds when applied to $h$, that is, $N_1$,
and the number of vector fields in the low norm, that is, $N_1-10$.
These hold in view of 
\eqref{param}.

To conclude the proof of \eqref{velpotest'} 
we want to use the fact that $\varphi = \Psi|_{z = 0}$ where 
$\nabla^i \Psi = V^i + \nabla^i h \pa_z \Psi - V^i_\omega$;
we then need the following bound for $V_\omega$:

\begin{lemma}
Under the same hypotheses of Proposition \ref{mainpropW'} we have
 \begin{equation}
  \sum_{|r| + |k| \leq N_0 - 20} {\| V_\omega^{r, k} \|}_{L^2_zL^2_x}
	\lesssim \e_0 \jt^{2p_0}
  \label{improvedVomegabound20}
 \end{equation}
\end{lemma}

\begin{proof}
By Propositions \ref{mainpropW'} and \ref{mainpropW}, 
the hypotheses of Proposition \ref{alphaprop'} hold, and the bounds \eqref{alphaconcL'}-\eqref{alphadtconcH'}
for $\alpha$ follow. 
Since \eqref{alphadtconcH'} is \eqref{extraonalpha},
the bound \eqref{improvedVomegabound20} now follows from \eqref{boundaVbulkweaker}. 
\end{proof}


\subsection{Conclusion: Proof of \eqref{velpotest'}}
To conclude the proof we first use the trace inequality \eqref{traceineqflat}
to bound the left-hand side of \eqref{velpotest'}
\begin{align}
\sum_{r+|k| \leq N_0-20}
  {\big\| |\nabla|^{1/2} \Gamma^k \varphi (t) \big\|}_{H^r(\R^2)}
  \lesssim & \sum_{r+|k| \leq N_0-20}
  {\big\| \nabla_x^r \underline{\Gamma}^k \nabla_{x,z} \Psi (t) \big\|}_{L^2_{x,z}}.
\end{align}
From the identities $\nabla^i \Psi = V^i + \nabla^i h \pa_z \Psi - V^i_\omega$
for $i = 1,2$ and $\nabla_z \Psi = V^3 - V^3_\omega$,
the bound \eqref{aprioriVL2} for $\|\Gamma^k V\|_{L^2_zL^2_x}$, and the bound \eqref{improvedVomegabound20}
for $V_\omega$, for any $n \leq N_0-20$, we have
\begin{align*}
 \| \underline{\Gamma}^{\leq n} \nabla_{x} \Psi\|_{L^2_zL^2_x}
 \lesssim \sum_{|k| \leq n}\|  \underline{\Gamma}^{k} V\|_{L^2_zL^2_x}
 +\sum_{|k| \leq n}\| \underline{\Gamma}^{k}(\nabla h \pa_z \Psi)\|_{L^2_zL^2_x}
 + \sum_{|k| \leq n}\| \underline{\Gamma}^{k} V_\omega\|_{L^2_zL^2_x}
 \\
 \lesssim \e_0\jt^{2p_0} + \sum_{|k| \leq n}\| \underline{\Gamma}^{k}(\nabla h \pa_z\Psi)\|_{L^2_zL^2_x},
\end{align*}
and
\begin{equation*}
 \| \underline{\Gamma}^{\leq n} \nabla_z \Psi\|_{L^2_xL^2_x}
 \lesssim \sum_{|k| \leq n}\| \underline{\Gamma}^{k} V\|_{L^2_zL^2_x}
 + \sum_{|k| \leq n}\| \underline{\Gamma}^{k} V_\omega\|_{L^2_zL^2_x}
 \lesssim \e_0\jt^{2p_0}.
\end{equation*}
Finally, we can estimate, for any $n \leq N_0-20$, 
\begin{align*}
 \sum_{|k| \leq n} \| \underline{\Gamma}^{k}(\nabla h \partial_z \Psi)\|_{L^2_zL^2_x}
 & \lesssim
 \sum_{|k| \leq n}
 \|\Gamma^k \nabla h\|_{L^\infty} 
 \sum_{k \leq n} \| \underline{\Gamma}^k \pa_z\Psi\|_{L^2_zL^2_x}
 \lesssim \e_0^2 \jt^{3p_0}.
\end{align*}
This concludes the proof of \eqref{velpotest'}.
\hfill $\Box$

\subsection{Proofs of Lemma \ref{lemdecayirrot} and \ref{lemdecayrot}}\label{ssecdecay}
We conclude this section by showing how to recover the a priori decay assumption \eqref{aprioridecayv}
through \eqref{decayirrotconc} and \eqref{decayrotconc}.

{\it Proof of \eqref{decayirrotconc}}
First, recall that, in view of \eqref{decayirrconcl} we have
\begin{align}\label{decayirrconclphi}
\sum_{\ell \in \Z} \sum_{|k| \leq N_1 
} {\| \Gamma^{k} P_\ell |\nabla|^{1/2} \varphi (t) \|}_{L^\infty(\R^2)} \leq c_B \e_0 \jt^{-1}
\end{align}
and, therefore, in view of Remark \ref{remPsiinfty}, we get
\begin{align}\label{prdecayirrotconc1}
\sum_{|k| \leq N_1-1} \sum_{\ell \in \Z} {\big\| P_\ell \underline{\Gamma}^{k} \nabla_{x,z} \Psi(t) 
  \big\|}_{L^2_z L^\infty_x} \leq C c_B\e_0 \jt^{-1}
\end{align}
for some generic $C>0$.
We then use the composition estimate \eqref{complow}, sum over dyadic Littlewood-Paley pieces 
and use \eqref{prdecayirrotconc1} after Sobolev's embedding in $z$ to get, with $n=N_1-5$,
\begin{align*}
\sum_{r+k \leq n} {\| \nabla \psi(t)  \|}_{X^{r,\infty}_k(\D_t)}
 & \leq C \sum_{|k| \leq n} {\big\| \underline{\Gamma}^{k} \nabla_{x,z} \Psi(t) \big\|}_{L^\infty_{x,z}}
 \\
 & \leq C \sum_{|k| \leq n} \sum_{\ell \in \Z} {\big\| P_\ell \underline{\Gamma}^{k} \nabla_{x,z} \Psi(t) 
  \big\|}_{L^\infty_{x,z}} \leq C c_B \e_0 \jt^{-1}.
\end{align*}
This gives us Lemma \ref{lemdecayirrot}. 

{\it Proof of \eqref{decayrotconc}}
We use the composition estimate \eqref{complow}, followed by Sobolev's embedding in $x$,
and then apply directly the second estimate on $V_\omega$ from \eqref{boundaVbulk0} with $j=0$
to see that
\begin{align}\label{prdecayrotconc}
\sum_{r+k \leq N_1-5} {\| v_\omega(t) \|}_{X^{r,\infty}_k(\D_t)} 
& 
  \leq C \sum_{|k| \leq  N_1-3} {\| \underline{\Gamma}^k V_\omega(t) \|}_{L^\infty_z L^2_x}
  \leq C c_H'\e_1 \jt^\delta.
\end{align}
This gives Lemma \ref{lemdecayrot} and closes the bootstrap for the norm in \eqref{aprioridecayv}.


\bigskip
\section{Decay of the boundary variables} 
\label{secdecay}

In this section we use the restriction of the free boundary Euler equations to the boundary surface
to establish time decay for the dispersive variable $u$ and prove Proposition \ref{propdecayirr},
and, in fact, the better estimate \eqref{decayirrconcl}.


\subsection{Set-up and equations at the boundary}
Recall that
\begin{align}\label{secdecu}
u = h + i \Lambda^{1/2}\varphi, \qquad u_+ =u, \quad u_- = \bar{u}.
\end{align}
With
$P_\omega^i = \wt{v_\omega}^i = v_\omega^i|_{\pa \D_t}$, $i = 1,2$,
one can show, see \eqref{der30}-\eqref{derB_0} in Appendix \ref{appder}, that $u$ solves
\begin{align}\label{decayequ}
(\partial_t + i \Lambda^{1/2}) u = B_0(u,u) + B_{0,1}(u+\bar{u}, P_\omega) + B_1(P_\omega, P_\omega)
  + L(P_\omega) + N_3
\end{align}
where:

\setlength{\leftmargini}{1.5em}
\begin{itemize}

\item The quadratic terms involving only $u$ are given by
\begin{align}\label{decayB0}
\begin{split}
B_0(u,u) & = \sum_{\eps_1,\eps_2 \in \{+,-\}} B_{\eps_1\eps_2}(u_{\eps_1},u_{\eps_2}),
\end{split}
\end{align}
with the definitions \eqref{derb+-} for the symbols and the notation \eqref{OpSym} for the associated operators;

\item
The quadratic terms involving at least one copy of $P_\omega$ are
\begin{align}\label{decayB1}
B_{0,1}(f,g) := (i\nabla\Lambda^{-1/2} f) \cdot g,  \qquad B_1(f,g) := - \frac{1}{2} f \cdot g;
\end{align}

\item
The `linear forcing term' due to the vorticity is
\begin{align}\label{decayuL}
L(P_\omega) := - i\Lambda^{-1/2} R \cdot \pa_t P_\omega;
\end{align}

\item $N_3$ are the cubic and higher order terms in \eqref{der33}.

\end{itemize}

\subsubsection{Vectorfields and Duhamel's formula}
We start by applying vector fields 
to \eqref{decayequ} and deriving an equation for
\begin{align}\label{decayun}
u^n := \Gamma^n u, \qquad |n| < N_0, \qquad \Gamma \in \{S,\Omega\}.
\end{align}
Note that we are not including regular derivatives in this notation.
Using \dg{Lemma \ref{lemsymcomm} } and \dg{the formula \eqref{qGamma} } to commute vector fields and quadratic symbols,
we see that there exist real constants $a_{n_1}$, $c_{n_1n_2}$ and $d_{n_1n_2}$ such that
\begin{align}\label{decayequn}
\begin{split}
(\partial_t + i \Lambda^{1/2}) u^n = \sum_{|n_1|+|n_2| \leq |n|} c_{n_1n_2} B_0(u^{n_1},u^{n_2})
  + d_{n_1n_2} B_{0,1}(u^{n_1}+\bar{u}^{n_1}, P_\omega^{n_2})
  \\ + \sum_{n_1+n_2 = n}B_1(P_\omega^{n_1}, P_\omega^{n_2})
  + \sum_{|n_1| \leq |n|} a_{n_1} L^{n_1}(P_\omega^{n_1}) + \Gamma^n N_3,
\end{split}
\end{align}
where the linear term is given by
\begin{align}\label{decayuLk}
L^{n_1}(P_\omega^{n_1}) := - i\Lambda^{-1/2} R^{(n_1)} \cdot \pa_t P_\omega^{n_1},
  \qquad \mbox{with} \qquad R^{(n_1)} \in \mathrm{span}\{ R, R^\perp\}.
\end{align}
Duhamel's formula for \eqref{decayequn} gives
\begin{align}\label{decaydu1}
u^n(t) = e^{-it\Lambda^{1/2}} u_0^n
 & + \sum_{|n_1|+|n_2| \leq |n|} c_{n_1n_2} \int_0^t e^{i(s-t)\Lambda^{1/2}} B_{0}(u^{n_1}, u^{n_2})\, ds
 \\
 & + \sum_{|n_1|+|n_2| \leq |n|} d_{n_1n_2}
 \int_0^t e^{i(s-t)\Lambda^{1/2}} B_{0,1}(u^{n_1}+\bar{u}^{n_1}, P_\omega^{n_2})\, ds
 \\
 & + \sum_{|n_1|+|n_2| = n} \int_0^t e^{i(s-t)\Lambda^{1/2}} B_{1}(P_\omega^{n_1}, P_\omega^{n_2})\, ds
 \\
 & + \sum_{|n_1| \leq |n|} \int_0^t e^{i(s-t)\Lambda^{1/2}} a_{n_1}L^{n_1}(P_\omega^{n_1})
 \, ds
 + \int_0^t e^{i(s-t)\Lambda^{1/2}} \Gamma^n N_3(s) \, ds.
\end{align}
The linear flow $e^{-it\Lambda^{1/2}} u_0$ is directly handled using Lemma \ref{lemlinear}.

\subsubsection{Reduction of the proof of Proposition \ref{propdecayirr}}
To prove Proposition \ref{propdecayirr} it then suffices to prove the following bounds:
\begin{align}\label{decest0}
& 
  {\Big\| \int_{0}^t e^{i(s-t)\Lambda^{1/2}} B_{0}(u^{n_1},u^{n_2})\, ds \Big \|}_{W^{r,\infty}} \lesssim
  \e_0^{1+} \jt^{-1}, \qquad r + (|n_1|+|n_2|) \leq N_1;
\end{align}
for sufficiently small $c' = c_{r,n}'$
\begin{align}\label{decest1}
\begin{split}
& 
{\Big\| \int_0^t e^{i(s-t)\Lambda^{1/2}} B(s) \, ds \Big\|}_{W^{r,\infty}} \leq c' \e_0 \jt^{-1},
  \\
  & \mbox{with} \quad B \in \big\{ B_{0,1} (u_\pm^{n_1}, P_\omega^{n_2}), \,
  \, B_{1} (P_\omega^{n_1}, P_\omega^{n_2}) \big\}, \qquad r + (|n_1|+|n_2|) \leq N_1;
\end{split}
\end{align}
for sufficiently small $c'' = c''_{n_1}$
\begin{align}\label{decest1'}
\begin{split}
{\Big\| \int_0^t e^{i(s-t)\Lambda^{1/2}} L^{n_1}(P_\omega^{n_1}) \, ds \Big\|}_{W^{r,\infty}} \leq
  c''\e_0 \jt^{-1}, \qquad r + |n_1| \leq N_1;
\end{split}
\end{align}
and, finally,
\begin{align}\label{decest2}
& 
  {\Big\| \int_0^t e^{i(s-t)\Lambda^{1/2}} \Gamma^n N_3(s) \, ds \Big\|}_{W^{r,\infty}} \lesssim \e_0^2 \jt^{-1},
  \qquad r + |n| \leq N_1.
\end{align}

\begin{remark}\label{Remdecay}
As mentioned in Remark \ref{Rempropdecayirr},
we are actually going to show stronger estimates than \eqref{decest0}-\eqref{decest2},
with $\ell^1$ sums over frequencies, that is, we will prove all of the estimates
for the Besov $B^{r}_{\infty,1}$ instead of the $W^{r,\infty}$ norm.
These bounds are essentially automatic in view of the following:
(a) the estimates for bilinear forms \eqref{prdecM1} which we are going to use  
to establish the bounds needed for \eqref{decest0} and
(b) the linear bound \eqref{linest1} which we are going to use (as in Lemma \ref{lemdecayDu}, for example)
and which already has the $\ell^1$ sum on the right-hand side.
\end{remark}
%

For convenience we recall that, in what follows, we will be working under the assumptions \eqref{decayirrasu}
and \eqref{aprioridecay}, that is,
for all $t\leq T_{\e_1}$,
\begin{align}
\label{decayas1}
& \sum_{r+|k| \leq N_0-20} {\| u^k(t) \|}_{H^r(\R^2)}
  \lesssim \e_0 \jt^{3p_0}, 
\\
\label{decayas2}
& \sum_{r+|k| \leq N_1} {\big\| u^k(t) \big\|}_{W^{r,\infty}(\R^2)}
  \leq c_0 \e_0 \jt^{-1}.
\end{align}
We will often use these assumptions without referring to them explicitly.
We will also use the bound on the rotational component in \eqref{decayirrasVP}:
\begin{align}\label{decayVP1}
\sum_{r + k \leq N_1+12-j
  } {\big\| \partial_t^j P_\omega(t) 
  \big\|}_{Z^r_k(\R^2)} \lesssim \e_1 \e_0^j \jt^\delta,  \quad j=0,1.
\end{align}

\subsection{Normal form transformation}\label{ssecNF}
For the quadratic terms which only depend on the dispersive variable we need normal form transformations.
Define the profile
\begin{align}\label{decprof}
f^n(t) =  e^{it\Lambda^{1/2}} u^n(t)
\end{align}
and, in accordance with \eqref{decayB0}, write
\begin{align}\label{prdec1}
\begin{split}
& \int_{0}^t e^{is\Lambda^{1/2}} B_{0}(u^{n_1},u^{n_2})\, ds = \frac{1}{(2\pi)^2}
  \sum_{\eps_1,\eps_2 \in \{+,-\}} \mathcal{F}^{-1} I_{\eps_1\eps_2}(t),
\\
& I_{\eps_1\eps_2}(t) :=
  \int_{0}^t \!\! \int_{\R^3} e^{is \Phi_{\eps_1\eps_2}(\xi,\eta)}
  b_{\eps_1\eps_2}(\xi,\eta) \widehat{f_{\eps_1}^{n_1}}(\xi-\eta) \widehat{f_{\eps_2}^{n_2}}(\eta) \, d\eta ds,
\end{split}
\end{align}
where
\begin{align}\label{prdecphase}
\Phi_{\eps_1\eps_2}(\xi,\eta) = |\xi|^{1/2} - \eps_2 |\xi-\eta|^{1/2} - \eps_1 |\eta|^{1/2},
\end{align}
and we omitted the dependence on $n_1,n_2$ of $I_{\eps_1\eps_2}$.
In what follows we will often use the short-hand
\begin{align}\label{prdenot}
f_j := f_{\eps_j}^{n_j}, \quad u_j := u^{n_j}_{\eps_j} = e^{-\eps_j it\Lambda^{1/2}}f_j, \quad j=1,2.
\end{align}

Define
\begin{align}\label{prdecm}
m_{\eps_1\eps_2}(\xi,\eta) = \frac{b_{\eps_1\eps_2}(\xi,\eta)}{i\Phi_{\eps_1\eps_2}(\xi,\eta)}
\end{align}
where the symbols $b_{\eps_1\eps_2}$ are given by \eqref{derb+-}.
Integrating by parts in time we can write
\begin{subequations}\label{prdec2}
\begin{align}\label{prdec2.1}
I_{\eps_1\eps_2}(t) & = \left. \int_{\R^3} e^{is \Phi_{\eps_1\eps_2}(\xi,\eta)}
  m_{\eps_1\eps_2}(\xi,\eta) \what{f_1}(s,\xi-\eta) \what{f_2}(s,\eta) \, d\eta  \right|^{s=t}_{s=0}
\\
\label{prdec2.2}
& - \int_{0}^t \!\! \int_{\R^3} e^{is \Phi_{\eps_1\eps_2}(\xi,\eta)}
  m_{\eps_1\eps_2}(\xi,\eta) 
  \pa_s \Big[ \what{f_1}(s,\xi-\eta) \what{f_2}(s,\eta) \Big] \, d\eta ds.
\end{align}
\end{subequations}

To estimate the above expressions we recall the definition in \eqref{symclass} and observe that
\begin{align}\label{prdecphasebd}
\fp{ {\| \Phi^{k,k_1,k_2}(\xi,\eta) \|}_{\S^\infty} \gtrsim 2^{(1/2)\min(k,k_1,k_2)} }
\end{align}
and, therefore, in view of \eqref{derb+-est} and Lemma \ref{lemsym},
\begin{align}\label{prdec3}
{\| m_{\eps_1\eps_2}^{k,k_1,k_2}(\xi,\eta) \|}_{\S^\infty} \lesssim
  2^{(1/2)k} \cdot 2^{(1/2)\max(k_1,k_2)}.
\end{align}

In particular, the symbols appearing in \eqref{prdec2} are not singular
and, using the estimate \eqref{lemsymbil} of Lemma \ref{lemsym} with the bound \eqref{prdec3},
we get, for all $1/p = 1/p_1 + 1/p_2$,
\begin{align}\label{prdecMbound}
{\| P_k M_{\eps_1\eps_2} (P_{k_1}g(t), P_{k_2}h(t)) \|}_{W^{r,p}}
  \lesssim 2^{r k_+} 2^{k/2} 2^{(1/2)\max(k_1,k_2)}
  {\| P_{k_1}g(t) \|}_{L^{p_2}} {\| P_{k_2}h(t) \|}_{L^{p_1}}.
\end{align}
With $p=p_1=p_2 = \infty$, using Bernstein's inequality, we can deduce that
\begin{align}\label{prdecM1}
\begin{split}
{\| M_{\eps_1\eps_2} (g,h) \|}_{L^\infty}
  \lesssim \sum_k {\| P_k M_{\eps_1\eps_2} (g,h) \|}_{L^\infty}
  \lesssim {\| g \|}_{W^{2,\infty-}} {\| h \|}_{W^{2,\infty-}}.
\end{split}
\end{align}
Using \eqref{prdecMbound} with $p=2$ and Bernstein's inequality 
we can obtain
\begin{align}\label{prdecM2}
\begin{split}
{\| M_{\eps_1\eps_2} (g, h) \|}_{H^r} \lesssim
  \min \big( {\| g \|}_{H^{r+2}} {\| h \|}_{W^{2,\infty-}}, {\| g \|}_{W^{2,\infty-}} {\| h \|}_{H^{r+2}} \big).
\end{split}
\end{align}

\subsection{Proof of \eqref{decest0}}\label{secdecest0} 
From \eqref{prdec1} and \eqref{prdec2.1}-\eqref{prdec2.2} we see that the desired bound \eqref{decest0}
follows if we show
\begin{align}
\label{decest0a}
& {\big\|  e^{-it\Lambda^{1/2}} \mathcal{F}^{-1} \eqref{prdec2.1} \big\|}_{W^{r,\infty}} \lesssim
  \e_0^{1+} \jt^{-1},
 \\
\label{decest0b}
& {\big\|  e^{-it\Lambda^{1/2}}\mathcal{F}^{-1} \eqref{prdec2.2} \big\|}_{W^{r,\infty}} \lesssim \e_0^{1+} \jt^{-1},
\end{align}
for all $r + |n_1|+|n_2| \leq N_1$ (recall the notation in \eqref{prdenot}).

\subsubsection{Proof of \eqref{decest0a}}
We only look at the terms in \eqref{prdec2.1} with $s=t$ since the $s=0$ contribution is easier to estimate.
We write these as
\begin{align*}
\int_{\R^3} e^{it \Phi_{\eps_1\eps_2}(\xi,\eta)}
  \frac{b_{\eps_1\eps_2}(\xi,\eta)}{i\Phi_{\eps_1\eps_2}(\xi,\eta)}
  \widehat{f_1}(t,\xi-\eta) \widehat{f_2}(t,\eta) \, d\eta
  = \mathcal{F} \big[ e^{it \Lambda^{1/2}} M_{\eps_1\eps_2} (u_1(t), u_2(t)) \big],
\end{align*}
and, according to \eqref{decest0a}, 
aim to show that
\begin{align}\label{prdec20}
{\| M_{\eps_1\eps_2} (\nabla^{r_1}u_1(t), \nabla^{r_2}u_2(t)) \|}_{L^\infty}
  \lesssim \e_0^{1+} \jt^{-1}, \qquad r_1+r_2+(|n_1|+|n_2|) \leq N_1.
\end{align}

Using \eqref{prdecM1} we can estimate
\begin{align}\label{prdec22}
\begin{split}
& {\| M_{\eps_1\eps_2} (\nabla^{r_1}u_1(t), \nabla^{r_2}u_2(t)) \|}_{L^\infty}
\lesssim {\| \nabla^{r_1}u_1(t) \|}_{W^{2,\infty-}} {\| \nabla^{r_2}u_2(t)  \|}_{W^{2,\infty-}}.
\end{split}
\end{align}
Then, recall from \eqref{decayas1}-\eqref{decayas2} that since $u_j = u^{n_j}_{\eps_j}$,
we have in particular
\begin{align}\label{prdecapriori}
\jt {\| u_j \|}_{W^{r_j,\infty}} + \jt^{-3p_0} {\| u_j \|}_{H^{r_j+(N_0-N_1-20)}} \lesssim \e_0,
  \qquad r_j+|n_j| \leq N_1,
\end{align}
so that Sobolev-Gagliardo-Nirenberg interpolation gives
\begin{align}\label{prdecapriori'}
{\| u_j \|}_{W^{r_j+2,\infty-}} \lesssim \e_0 \jt^{-2/3},  \qquad r_j+|n_j| \leq N_1.
\end{align}
Using this inequality in \eqref{prdec22} gives \eqref{prdec20}.

\subsubsection{Proof of \eqref{decest0b}}
To estimate bulk terms \eqref{prdec2.2}
we need estimates for the time derivative of the profile $f_j$.
First, from the definition \eqref{prdenot}
we have $\partial_t f_j = e^{\eps_j it\Lambda^{1/2}}(\partial_t + \eps_j i \Lambda^{1/2})u_j$.
Assuming $\eps_j = +$ (the other case is obtained by conjugation), with the notation \eqref{decayun}
and using the equation \eqref{decayequn} we have
\begin{align}\label{prdtf0.1}
& \partial_t f^n = e^{it\Lambda^{1/2}} \sum_{|n_1| \leq |n|} a_{n_1} L^{n_1}(P_\omega^{n_1})
  + e^{it\Lambda^{1/2}} Q^n
\end{align}
with
\begin{align}\label{prdtf0.2}
\begin{split}
Q^n := \sum_{|n_1|+|n_2| \leq |n|} c_{n_1n_2} B_0(u^{n_1},u^{n_2})
  + d_{n_1n_2} B_{0,1}(u^{n_1} + \bar{u}^{n_1}, P_\omega^{n_2})
  \\ + \sum_{|n_1|+|n_2| \leq |n|}B_1(P_\omega^{n_1}, P_\omega^{n_2}) + \Gamma^n N_3.
\end{split}
\end{align}
We first establish some estimates for $Q^n$
and will then rely on \eqref{decayVP1} to estimate the contribution 
from the operator $L(P_\omega)$.

\begin{lemma}\label{lemdtf}
Under the a priori assumptions, for any $t \leq T$, we have
\begin{align}\label{dtfest}
{\big\| Q^n \big\|}_{H^r} \lesssim \e_0 \jt^{-2/3}, \qquad  r + |n| \leq N_1 + 11,
\end{align}
and
\begin{align}\label{dtfest'}
{\big\| Q^n \big\|}_{W^{r,\infty}} \lesssim \e_0 \jt^{-5/4}, \qquad  r + |n| \leq N_1-5.
\end{align}
\end{lemma}

Notice that the estimates in the above Lemma are not optimal in terms of decay rates,
since $Q^n$ is effectively quadratic in $(u,P_\omega)$, but they will suffice for our purposes.

\begin{proof}[Proof of Lemma \ref{lemdtf}]
We first estimate separately all the terms on the right-hand side of \eqref{prdtf0.2}
in $H^r$ with the claimed number of vector fields. 
The $L^\infty$-type estimate will follow similarly.

\smallskip
{\it Proof of \eqref{dtfest}}.
In view of \eqref{derb+-est}, $B_0(u,u)$ satisfies standard product estimates up to a small loss
of derivatives: 
\begin{align}\label{prdtf1}
{\| B_0(g,h) \|}_{L^p} \lesssim \min\big( {\| g \|}_{W^{2,p}} {\| h \|}_{W^{2,\infty}},
  {\| g \|}_{W^{2,\infty}} {\| h \|}_{W^{2,p}} \big).
\end{align}
The desired bound on the $B_0$ terms is implied by
\begin{align}\label{prdtf2}
{\| B_0(\nabla^{r_1}u^{n_1}, \nabla^{r_2} u^{n_2}) \|}_{L^2} \lesssim \e_0 \jt^{-2/3},
  \qquad r_1+r_2 + (|n_1|+|n_2|) \leq \fp{N_0 - 22}.
\end{align}
To show \eqref{prdtf2}, without loss of generality, let us assume that $r_1+|n_1| \leq \fp{(N_0-22)}/2 \leq N_1-2$,
see \eqref{param}, and estimate using \eqref{prdtf1},
\begin{align*}
{\| B_0(\nabla^{r_1}u^{n_1},\nabla^{r_2}u^{n_2}) \|}_{L^2}
  &  \lesssim {\| \nabla^{r_1}u^{n_1} \|}_{W^{2,\infty}} {\| \nabla^{r_2}u^{n_2} \|}_{H^2}
  \lesssim \e_0 \jt^{-1} \cdot \e_0 \jt^{3p_0},
\end{align*}
which is more than sufficient; we have used \eqref{decayas1} (since $r_2+2+|n_2| \leq \fp{N_0 - 20}$)
and \eqref{decayas2} for the last inequality .


The terms $B_{0,1}(u_\pm^{n_1}, P_\omega^{n_2})$ and $B_1(P_\omega^{n_1}, P_\omega^{n_2})$
are easier to treat, using the estimates for $P_\omega^n$ in \eqref{decayVP1}.
From the definition in \eqref{decayB1}, H\"older's inequality and Sobolev's embedding we have,
for $r_1+r_2 + (|n_1|+|n_2|) \leq N_1+11$,
\begin{align*}
{\| B_{0,1}(u_\pm^{n_1}, P_\omega^{n_2}) \|}_{L^2} \lesssim
  {\| u_\pm^{n_1} \|}_{H^3} {\| P_\omega^{n_2} \|}_{L^2} \lesssim \e_0 \jt^{3p_0} \cdot \e_1 \jt^\delta,
\end{align*}
having used again \eqref{decayas1};
this is sufficient since $\e_1 \lesssim \jt^{-1}$.
Similarly, assuming without loss of generality that $r_1+ |n_1| \leq (N_1+11)/2$,
we can estimate, using again \eqref{decayVP1},
\begin{align*}
{\| B_1(P_\omega^{n_1}, P_\omega^{n_2}) \|}_{L^2} \lesssim
  {\| P_\omega^{n_1} \|}_{H^2} {\| P_\omega^{n_2} \|}_{L^2} \lesssim \e_1^2 \jt^{2\delta}
\end{align*}
which clearly suffices.
Since the bound for $\Gamma^n N_3$ follows directly from the stronger estimate
\eqref{lemN3conc}, the proof of \eqref{dtfest} is concluded.

{\it Proof of \eqref{dtfest'}}.
The $L^\infty$ type bound  \eqref{dtfest'} can be obtained similarly, by estimating in $W^{r,\infty}$
all the terms on the right-hand side of \eqref{decayequn} by means of the product estimate \eqref{prdtf1}
with $p=\infty$,
and using \eqref{lemN3conc} for $N_3$.
\end{proof}

We now go back to the proof of \eqref{decest0b}. 
First observe that by symmetry 
it suffices to consider the case when $\partial_s$
hits the first profile in the formulas for \eqref{prdec2.2}, and show
\begin{align}\label{prdec35}
\begin{split}
{\Big\| e^{-it \Lambda^{1/2}} \int_0^t e^{is\Lambda^{1/2}} M_{\eps_1\eps_2}
  \big(e^{-\eps_1is\Lambda^{1/2}} \partial_s f^{n_1}_{\eps_1}(s), u^{n_2}_{\eps_2}(s) \big) \, ds \Big\|}_{W^{r,\infty}}
  \lesssim \e_0^{1+} \jt^{-1},
  \\
  \qquad r+(|n_1|+|n_2|) \leq N_1.
\end{split}
\end{align}
In what follows we drop the $\eps_1,\eps_2$ signs since they do not play any role.
Using \eqref{prdtf0.1}-\eqref{prdtf0.2} we see that \eqref{prdec35} reduces to showing
the two following estimates:
\begin{align}\label{prdec35a}
\begin{split}
{\Big\| e^{-it \Lambda^{1/2}} \int_0^t e^{is\Lambda^{1/2}} M 
  \big( Q^{n_1} 
  (s), u^{n_2}
  (s) \big) \, ds \Big\|}_{W^{r,\infty}}
  \lesssim \e_0^{1+} \jt^{-1},
  \qquad r+(|n_1|+|n_2|) \leq N_1
\end{split}
\end{align}
and
\begin{align}\label{prdec35b}
\begin{split}
{\Big\| e^{-it \Lambda^{1/2}} \int_0^t e^{is\Lambda^{1/2}} M 
  \big(L^{n_1}(P_\omega^{n_1}), u^{n_2}
  (s) \big) \, ds \Big\|}_{W^{r,\infty}} \lesssim \e_0^{1+} \jt^{-1},
  \qquad r+(|n_1|+|n_2|) \leq N_1.
\end{split}
\end{align}

In what follows we are going to use the following lemma,
which is a consequence of the linear decay estimate in Lemma \ref{lemlinear}.

\begin{lemma}\label{lemdecayDu}
Let $F=F(t,x)$ be such that, for all $k=0,\dots,3$ and $n=0,1$,
\begin{align}\label{decayDuas}
\sum_{\ell \in \Z} 2^{\ell/2} {\| S^n \Omega^k P_\ell F(t) \|}_{L^2} \leq A_n(t).
\end{align}
Then we have the non-homogeneous decay bound
\begin{align}\label{decayDuconc}
{\Big\| \int_0^t e^{i(s-t)\Lambda^{1/2}} F(s) \, ds \Big\|}_{L^\infty} \lesssim
A_0(t) + \jt^{-1} \int_0^t ( A_0(s) + A_1(s)) \, ds.
\end{align}
\end{lemma}

\begin{proof}
We first apply the linear estimate \ref{linest1} to obtain
\begin{align*}
\jt \, {\Big\| e^{-it\Lambda^{1/2}} \int_0^t e^{is\Lambda^{1/2}} F(s) \, ds \Big\|}_{L^\infty}
  & \lesssim \sup_{k=0,\dots,3} \sum_{\ell\in\Z} 2^{\ell/2}
  {\Big\| \Sigma \Omega^k \int_0^t e^{is\Lambda^{1/2}} P_\ell F(s) \, ds \Big\|}_{L^2}
  \\
  & \lesssim \sup_{k=0,\dots,3, \, n=0,1} \sum_{\ell\in\Z} 2^{\ell/2}
  {\Big\| \int_0^t e^{is\Lambda^{1/2}} S^n \Omega^k P_\ell F(s) \, ds \Big\|}_{L^2}
  \\
  & + \sup_{k=0,\dots,3} \sum_{\ell\in\Z} 2^{\ell/2}
  {\Big\| \int_0^t s \partial_s \Big[ e^{is\Lambda^{1/2}} \Omega^k P_\ell F(s) \Big] \, ds \Big\|}_{L^2}.
\end{align*}
We have used that $\Sigma := x \cdot \nabla = S - (1/2)s\partial_s$,
and $[S,e^{is\Lambda^{1/2}}] = -1$ and $[\Omega,e^{is\Lambda^{1/2}}]=0$.
The first of the two terms on the above right-hand side is already accounted for in the bound \eqref{decayDuconc}.
For the last term we integrate by parts in $s$ and use the assumption \eqref{decayDuas} to conclude.
\end{proof}


\begin{proof}[Proof of \eqref{prdec35a}]
Using Lemma \ref{lemdecayDu} above,
and commuting the scaling and rotation vector fields and derivatives, we see that
in order to obtain \eqref{prdec35a} it suffices to show
\begin{align}\label{prdec36}
\begin{split}
{\big\| M \big( 
  \nabla^{r_1} Q^{n_1}(t), \nabla^{r_2}u^{n_2}(t) \big) \big\|}_{L^2}
  \lesssim \e_0^{1+} \jt^{-1-}, \\ \qquad |r_1|+|r_2|+|n_1|+|n_2| \leq N_1 + 5.
\end{split}
\end{align}
The number $N_1 + 5$ is coming from the presence of four vector fields in the definition
of $A_1$ in \eqref{decayDuas} and taking the $H^1$ norm instead of the
Besov norm $\dot{B}^{1/2}_{2,1}$.

%

%
%
%

\noindent
{\it Case $|r_1|+|n_1|\geq N_1/2+6$}.
Using \eqref{prdecM2} we can estimate the left-hand side of \eqref{prdec36} by
\begin{align*}
\begin{split}
C{\| Q^{n_1}(t) 
  \|}_{H^{|r_1|+2}} {\| u^{n_2}(t) \|}_{W^{|r_2|+2,\infty-}}.
\end{split}
\end{align*}
We can then use \eqref{dtfest} (since $|r_1|+2 + |n_1| \leq N_1+11$) to estimate the first term,
and \eqref{prdecapriori} (since $|r_2|+|n_2| \leq N_1/2+3 < N_1$) to estimate the second;
this gives and upper bound of $C\e_0 \jt^{-2/3} \cdot \e_0 \jt^{-1+}$ which suffices.

\noindent
{\it Case $|r_1|+|n_1| \leq N_1/2 + 5$}.
In this case we use \eqref{prdecM2} to estimate the left-hand side of \eqref{prdec36} by
\begin{align*}
\begin{split}
& C{\| Q^{n_1}(t) 
  \|}_{W^{|r_1|+2,\infty-}} {\| u^{n_2}(t) \|}_{H^{|r_2|+2}}
  \lesssim \e_0 \jt^{-11/10} \cdot \e_0 \jt^{3p_0},
\end{split}
\end{align*}
having used \eqref{dtfest'} 
and Sobolev-Gagliardo-Nirenberg interpolation with \eqref{dtfest},
and the a priori bound \eqref{decayas1}. 
This concludes the proof of \eqref{prdec36}, hence of \eqref{prdec35a}.
\end{proof}


\begin{proof}[Proof of \eqref{prdec35b}]
To prove \eqref{prdec35b} we are going to use the following lemma,
which gives a (non-optimal) interpolation-type estimate for $u$
when more than $N_1$ vector fields and derivatives are applied to it.

\begin{lemma}\label{lemGu}
Under the assumptions \eqref{decayas1}-\eqref{decayas2} and \eqref{decayVP1}, we have
\begin{align}\label{lemGuconc}
\sup_{2^k \geq \jt^{-2/3}}
  {\| P_k \nabla^r u^n(t) \|}_{L^\infty} \lesssim \e_0 \jt^{-2/3}, \qquad |r|+|n| \leq N_1+7 \quad (= N-5).
\end{align}
\end{lemma}

A lower bound on the frequencies in \eqref{lemGuconc} is needed to avoid dealing with very small frequencies
in the arguments below, but the exact restriction $2^k \geq \jt^{-2/3}$ is rather arbitrary
and it is unrelated to the decay rate on the right-hand side.

\begin{proof}[Proof of Lemma \ref{lemGu}]
We begin by using \eqref{linest2} to see that, for all $|r|+|n| \leq N_1+7$,
\begin{align}
\label{lemGupr1}
{\| P_k \nabla^r u^n \|}_{L^\infty} & \lesssim
  \jt^{-1} \sup_{|r|+|n| \leq N_1+11} \sum_{\ell \in \Z} 2^{\ell/2} {\| \nabla^r P_\ell u^n \|}_{L^2(\R^2)}
  \\
\label{lemGupr2}
& + \sup_{|r|+|n| \leq N_1+10} \sum_{2^\ell \gtrsim \jt^{-2/3}} 2^{\ell/2}
  {\| 
  \nabla^r P_\ell \partial_t f^n \|}_{L^2(\R^2)}.
\end{align}
Notice how we kept the restriction on not-too-low frequencies in \eqref{lemGupr2}.

The term on the right-hand side of \eqref{lemGupr1}
is estimated directly using the $L^2$ based a priori assumption \eqref{decayas1},
which gives a bound for \eqref{lemGupr1} by $C\jt^{-1} \cdot \e_0 \jt^{3p_0}$;
this is more than enough for the desired conclusion \eqref{lemGuconc}.

To handle the terms in \eqref{lemGupr2} we use the equation \eqref{prdtf0.1}. 
The contribution to $\partial_t f^n$ from the terms $Q^n$ can be estimated directly using \eqref{dtfest},
which is consistent with the bound \eqref{lemGuconc}.
For the contribution of the linear forcing term instead, we recall the definition
\eqref{decayuL} and estimate for any $|r|+|n| \leq N_1+10$,
\begin{align*}
& \sum_{2^\ell \gtrsim \jt^{-2/3}} 2^{\ell/2}
  {\| \nabla^r P_\ell L^n (P_\omega^n) \|}_{L^2(\R^2)}
  \lesssim \sum_{2^\ell \gtrsim \jt^{-2/3}} {\| \nabla^r P_\ell \, \partial_t P_\omega^n \|}_{L^2(\R^2)}
  \\
  & \lesssim
  \log(2+t) \, {\| \nabla^r \partial_t P_\omega^n \|}_{L^2(\R^2)}
  +
  \sum_{2^\ell \gtrsim 1}
  {\| \nabla^r P_\ell \, \partial_t P_\omega^n \|}_{L^2(\R^2)}
  \\
  &
  \lesssim \log(2+t) {\| \nabla^r \partial_t P_\omega^n \|}_{H^1(\R^2)}
  \lesssim \log(2+t) \e_0 \e_1 \jt^\delta,
\end{align*}
having used \eqref{decayVP1};
since $\e_1 \lesssim \jt^{-1}$ this concludes the proof of the lemma.
\end{proof}

We now proceed with the proof of \eqref{prdec35b}.
    Using again Lemma \ref{lemdecayDu} as in the proof of \eqref{prdec35a} above,
we reduce matters to an estimate analogous to \eqref{prdec36}, that is,
\begin{align}\label{prdec36'}
\begin{split}
{\big\| M \big( \nabla^{r_1} L^{n_1}(P_\omega^{n_1})(t), \nabla^{r_2}u^{n_2}(t) \big) \big\|}_{L^2}
  \lesssim \e_0^{1+} \jt^{-1-}, \\ \qquad |r_1|+|r_2|+(|n_1|+|n_2|) \leq N_1 + 5.
\end{split}
\end{align}

We first prove that
\begin{align}\label{dtVoconc'}
\sum_{k \in \Z} {\| P_k \nabla^{r} L^{n}(P_\omega^{n}) \|}_{L^{4+}} \lesssim \e_0\e_1 \jt^\delta,
  \qquad r+|n| \leq N_1+10,
\end{align}
and
\begin{align}\label{prdec36'k}
\sum_{k \in \Z} {\big\| P_k \nabla^{r}u^{n}(t) \big\|}_{L^{4-}}
  & \lesssim \e_0 \jt^{-1/4}, \qquad r+|n| \leq N_1 + 7.
\end{align}

From the definition \eqref{decayuLk}, the estimate \eqref{decayVP1}
and Bernstein's inequality, we have, for all $r+|n| \leq N_1+10$,
\begin{align*}
\sum_{k \in \Z} {\| P_k \nabla^{r} L^{n}(P_\omega^{n}) \|}_{L^{4+}}
  & \lesssim \sum_{k \leq 0} 2^{-k/2} {\| P_k \nabla^{r} \partial_t P_\omega^{n} \|}_{L^{4+}}
  + \sum_{k \geq 1} 2^{-k/2} {\| P_k \nabla^{r} \partial_t P_\omega^{n} \|}_{L^{4+}}
  \\
  & \lesssim \sum_{k \leq 0} 2^{(0+)k} {\| \nabla^{r} \partial_t P_\omega^{n} \|}_{L^2}
  + {\| \nabla^{r} \partial_t P_\omega^{n} \|}_{H^1}
  \lesssim {\| \partial_t P_\omega^{n} \|}_{H^{|r|+1}} \lesssim \e_0\e_1 \jt^\delta.
\end{align*}

Using Bernstein's inequality and interpolation with the bounds \eqref{lemGuconc} and \eqref{decayas1}
we have, for all $|r|+|n| \leq N_1 + 7$,
\begin{align*}
& \sum_{k \in \Z} {\big\| P_k \nabla^{r}u^{n}(t) \big\|}_{L^{4-}}
  \lesssim
 \sum_{2^k \leq \jt^{-2/3}} {\big\| P_k \nabla^{r} u^{n}(t) \big\|}_{L^{4-}}
 + \sum_{2^k \geq \jt^{-2/3}} {\big\| P_k \nabla^{r} u^{n}(t) \big\|}_{L^{4-}}
\\
& \lesssim \jt^{(-1/3)+} {\big\| \nabla^{r} u^{n}(t) \big\|}_{L^{2}}
 + \sum_{2^k \geq \jt^{-2/3}}
 {\big\| P_k \nabla^{r}u^{n}(t) \big\|}_{L^{2}}^{(1/2)+}
  {\big\| P_k \nabla^{r}u^{n}(t) \big\|}_{L^{\infty}}^{(1/2)-}
\\
& \lesssim \jt^{(-1/3)+} \e_0 \jt^{3p_0} + (\e_0 \jt^{-2/3})^{(1/2)-}
  \cdot \sum_{2^k \geq \jt^{-2/3}} {\big\| P_k \nabla^{r}u^{n}(t) \big\|}_{L^{2}}^{(1/2)+}
\\
& \lesssim \jt^{(-1/3)+} \e_0 \jt^{3p_0} + (\e_0 \jt^{-2/3})^{(1/2)-} \log(2+t)
  \sup_{2^k \geq \jt^{-2/3}} \big( 2^{k_+} {\big\| P_k \nabla^{r}u^{n}(t) \big\|}_{L^{2}} \big)^{(1/2)+}
\\
& \lesssim \e_0 \jt^{-1/4}.
\end{align*}
Note how we included a small loss in the second to last
inequality coming from the summation of $k$ with $\jt^{-2/3} \lesssim 2^k \lesssim 1$,
and how we used the validity of \eqref{lemGuconc} for frequencies $2^k \gtrsim \jt^{-2/3}$ in the third inequality.

Using the H\"older-type bound \eqref{prdecMbound} 
with the fact that $k \leq \max(k_1,k_2) + 5$,
and the estimates \eqref{dtVoconc'} and \eqref{prdec36'k} above,
we can bound the left-hand side of \eqref{prdec36'}
for all $|r_1|+|r_2|+|n_1|+|n_2| \leq N_1 + 5$ as follows:
\begin{align*}
\begin{split}
& {\big\| M \big( \nabla^{r_1} L^{n_1}(P_\omega^{n_1})(t), \nabla^{r_2}u^{n_2}(t) \big) \big\|}_{L^2}
\\
& \lesssim
\sum_{k_1,k_2} 2^{\max(k_1,k_2)} {\big\| P_{k_1} \nabla^{r_1} L^{n_1}(P_\omega^{n_1})(t) \|}_{L^{4+}}
  {\big\| P_{k_2} \nabla^{r_2}u^{n_2}(t) \big\|}_{L^{4-}}
\\
& \lesssim  \e_0\e_1 \jt^\delta \cdot \e_0 \jt^{-1/4};
\end{split}
\end{align*}
since $\e_1 \leq \jt^{-1}$, this concludes the proof of \eqref{prdec35b}.

\end{proof}

With \eqref{prdec35a}-\eqref{prdec35b} we have obtained
the desired estimates for the cubic bulk terms, 
and the proof of \eqref{decest0b} is concluded.
The bound \eqref{decest0} follows.

\subsection{Proof of \eqref{decest1}}\label{secprdecest1}
The proof of \eqref{decest1} in the case of the $B_1$ terms is easier than for the $B_{0,1}$ terms
so we can just focus on these latter.
From Lemma \ref{lemdecayDu} we see that it suffices to show (we drop the $\pm$)
\begin{align}\label{prdecest11}
\begin{split}
{\big\| B_{0,1} (\nabla^{r_1}u^{n_1}(t), \nabla^{r_2}P_\omega^{n_2}(t)) \big\|}_{L^2}
  \leq c' \e_0 \jt^{-1-},
  \\ |r_1|+|r_2| + |n_1|+|n_2| \leq N_1 + 5.
\end{split}
\end{align}
Let us consider the case $|r_1|+|n_1| \geq N_1/2$,
and disregard the complementary case which is easier since we can estimate
$\nabla^{r_1}u^{n_1}$ in $L^\infty$ and obtain a bound of the form $C\e_0 \jt^{-1} \cdot \e_1 \jt^{\delta}$.
    We then recall the definition \eqref{decayB1} of $B_{0, 1}$
and estimate using the bounds \eqref{decayVP1} for $P_\omega$, Bernstein,
the estimate \eqref{lemGuconc} for $P_k \nabla^r u^n$ and the a priori assumption 
\eqref{decayas1} on the energy:
for all $|r_1|+|n_1| \geq N_1/2$, $|r_2|+|n_2| \leq N_1/2+5$ we have
\begin{align*}
& {\big\| B_{0,1} (\nabla^{r_1}u^{n_1}, \nabla^{r_2}P_\omega^{n_2}) \big\|}_{L^2}
  \lesssim {\big\| \Lambda^{-1/2} \nabla \nabla^{r_1}u^{n_1} \|}_{L^\infty} {\| \nabla^{r_2}P_\omega^{n_2} \big\|}_{L^2}
  \\
  & \lesssim \Big( \sum_{2^{k_1} \leq \jt^{-2/3}} 2^{k_1/2} {\| P_{k_1} \nabla^{r_1}u^{n_1} \|}_{L^\infty}
  + \sum_{2^{k_1} \geq \jt^{-2/3}} 2^{k_1/2} {\| P_{k_1} \nabla^{r_1}u^{n_1} \|}_{L^\infty} \Big)
  \cdot \e_1 \jt^\delta
  \\
  & \lesssim \Big( \jt^{-1} {\| \nabla^{r_1}u^{n_1} \|}_{L^2} + \e_0 \jt^{-2/3} \Big) \cdot \e_1 \jt^\delta
  \\
  & \lesssim \e_0 \jt^{-2/3} \cdot \e_1 \jt^{\delta}.
\end{align*}
This is a more than sufficient bound for \eqref{prdecest11} since $\e_1 \lesssim \jt^{-1}$.

\subsection{Proof of \eqref{decest1'}}
Because of the $\Lambda^{-1/2}$ factor in \eqref{decayuLk}
we need to be careful once again about the handling of low frequencies and summations over
dyadic indexes. 
With the notation in \eqref{decayuLk}, we
denote the quantity on the left-hand side of \eqref{decest1'} 
as
\begin{align}\label{decest'pr1}
I(t) := \int_0^t e^{i(s-t)\Lambda^{1/2}} 
  \Lambda^{-1/2} R^{(n_1)} \cdot \pa_s P_\omega^{n_1}(s) \, ds;
\end{align}
we drop the dependence on $r,n_1$ with $|r| + |n_1| \leq N_1$, and recall that
$R^{(n_1)} \in \mathrm{span}\{ R, R^\perp\}$.
We then write
\begin{align}\label{decest'pr2}
\begin{split}
I & = I_{l} + I_{m} + I_{h},
\\
I_l & := P_{< L} I, \qquad I_m := P_{[L,H]} I, \qquad I_h := P_{> H} I,
\\
& L := \log_2 \e_0^2, \quad H:= \log_2 \e_0^{-2}, 
\end{split}
\end{align}
with the notation \eqref{cut0}-\eqref{LP0}; note that here $\e_0$ is the same quantity in \eqref{decest1'}.
$I_l$ is a low-frequency contribution with frequencies of size less than $\e_0^2$;
$I_h$ is a high-frequency contribution with frequencies of size larger than $\e_0^{-2}$;
and $I_m$ is the remaining `medium-frequencies' contribution.
We estimate separately the three contributions in \eqref{decest'pr2}.

We first look at the medium frequencies contribution and begin by estimating
\begin{align}\label{Imed}
{\big\| I_m \big\|}_{W^{r,\infty}} \leq C | \log(\e_0) |
  \sup_{\ell' \in [\e_0^2,\e_0^{-2}]} {\| P_{\ell'} I \|}_{W^{r,\infty}}.
\end{align}
We then want to apply Lemma \ref{lemdecayDu} with $F = P_{\ell'} \Lambda^{-1/2} R^{(n_1)} \cdot \pa_s P_\omega^{n_1}$.
We note that, for fixed $\ell'$, and $k = 0,\dots 3$, $n=0,1$, we have
\begin{align}\label{prdecV1'}
\begin{split}
& \sum_{\ell\in\Z} 2^{\ell/2} {\Big\| S^n \Omega^k P_\ell P_{\ell'} \Lambda^{-1/2} R^{(n_1)}
  \cdot \pa_s P_\omega^{n_1}(s) \Big\|}_{H^r}
  \\
  & \leq C \sum_{n_1' \leq n_1 + 4} {\big\| \pa_s P_\omega^{n_1'}(s) \big\|}_{H^r}
  \leq C \e_1 \e_0 \js^\delta,
\end{split}
\end{align}
having used \eqref{decayVP1} for the last inequality. 
Then \eqref{Imed} and the conclusion of Lemma \ref{lemdecayDu} give
\begin{align*}
{\big\| P_{[\e_0^2,\e_0^{-2}]} I \big\|}_{W^{r,\infty}}
 & \leq C | \log(\e_0) |
 \Big( \e_1 \e_0 \jt^\delta + \jt^{-1} \int_0^t \e_1 \e_0 \js^\delta \, ds \Big)
 \\
 & \leq C | \log(\e_0) | \e_1 \e_0 \jt^{\delta};
\end{align*}
the last quantity above is bounded by the right-hand side of \eqref{decest1'} as desired,
since $|\log \e_0| \leq |\log\e_1|$ and we have $t \leq T := c \e_1^{-1+\delta}$
for $c$ sufficiently small, so that
$C | \log(\e_0) | \e_1 \e_0 \jt^{\delta} \leq c'' \e_0 \jt^{-1}$ for all $t\in[0,T]$.

To handle the low-frequency contributions from $I_l$, and the high-frequency
contributions from $I_h$, we first rewrite \eqref{decest'pr1} in a different way integrating by parts in $s$:
\begin{align}\label{decest'ibp2}
I(t) & = \Lambda^{-1/2} R^{(n_1)} \cdot P_\omega^{n_1}(t)
\\
\label{decest'ibp1}
& - e^{-it\Lambda^{1/2}} \Lambda^{-1/2} R^{(n_1)} \cdot P_\omega^{n_1}(0)
\\
\label{decest'ibp3}
& - i\int_0^t e^{i(s-t)\Lambda^{1/2}} R^{(n_1)} \cdot P_\omega^{n_1}(s) \, ds.
\end{align}

Let us first look at the contribution from \eqref{decest'ibp2}.
For the low frequencies, using Bernstein's inequality and the bound \eqref{decayVP1}, we get
\begin{align}
{\| P_{< L} \Lambda^{-1/2} R^{(n_1)} \cdot P_\omega^{n_1}(t) \|}_{W^{r,\infty}}
  \lesssim 2^{L/2} {\| P_\omega^{n_1}(t) \|}_{L^2} \lesssim \e_0 \cdot \e_1 \jt^{\delta}.
\end{align}
As before this is sufficient for the desired bound by the right-hand side of \eqref{decest1'}.
For the high frequencies we instead estimate 
using Sobolev's embedding and \eqref{decayVP1},
\begin{align}
{\| P_{> H} \Lambda^{-1/2} R^{(n_1)} \cdot P_\omega^{n_1}(t) \|}_{W^{r,\infty}}
  \lesssim 2^{-H/2} {\| P_\omega^{n_1}(t) \|}_{W^{r,\infty}} \lesssim \e_0 \cdot \e_1 \jt^{\delta}.
\end{align}

The term \eqref{decest'ibp1} can be handled in the same way, relying on the bounds at the initial time:
\begin{align*}
& {\| P_{< L} e^{-it\Lambda^{1/2}} \Lambda^{-1/2} R^{(n_1)} \cdot P_\omega^{n_1}(0) \|}_{W^{r,\infty}}
  \leq 2^{L/2} {\| P_\omega^{n_1}(0) \|}_{L^2} \leq C \e_0 \e_1,
\\
& {\| P_{> H} e^{-it\Lambda^{1/2}}\Lambda^{-1/2} R^{(n_1)} \cdot P_\omega^{n_1}(0) \|}_{W^{r,\infty}}
  \leq 2^{-H/2} {\| P_\omega^{n_1}(0) \|}_{H^{r+1+}} \leq C \e_0 \e_1.
\end{align*}

Finally, we estimate the small and high frequencies contributions from the term \eqref{decest'ibp3}.
We want to apply again Lemma \ref{lemdecayDu} in a suitable way.
First, we look at $P_{< L} \eqref{decest'ibp3}$,
let $F = P_{< L} R^{(n_1)} \cdot P_\omega^{n_1}(s)$
and, for all $k = 0,\dots 3$, $n=0,1$, estimate
\begin{align*}
\begin{split}
& \sum_{\ell\in\Z} 2^{\ell/2} {\Big\| S^n \Omega^k P_{<L} P_\ell R^{(n_1)} \cdot P_\omega^{n_1}(s) \Big\|}_{H^r}
  \leq C 2^{L/2} \sum_{n_1' \leq n_1 + 4} {\big\| P_\omega^{n_1'}(s) \big\|}_{L^2}
  \leq C \e_0 \e_1 \js^\delta,
\end{split}
\end{align*}
having used \eqref{decayVP1}.
Then, applying the conclusion \eqref{decayDuconc} from Lemma \ref{lemdecayDu} we get
\begin{align}
\begin{split}
{\Big\| \int_0^t e^{i(s-t)\Lambda^{1/2}} R^{(n_1)} \cdot P_\omega^{n_1}(s) \, ds \Big\|}
  & \leq C \e_0 \e_1 \jt^\delta,
\end{split}
\end{align}
which again is consistent with \eqref{decest1'}.
For the high frequency contribution we can proceed similarly using again Lemma \ref{lemdecayDu}
with the bound
\begin{align*}
\begin{split}
& \sum_{\ell\in\Z} 2^{\ell/2} {\Big\| S^n \Omega^k P_{>H} P_\ell R^{(n_1)} \cdot P_\omega^{n_1}(s) \Big\|}_{H^r}
  \leq C 2^{-H/2} \sum_{n_1' \leq n_1 + 4} {\big\| P_\omega^{n_1'}(s) \big\|}_{H^{r+1}}
  \leq C \e_0 \e_1 \js^\delta.
\end{split}
\end{align*}
The proof of \eqref{decest1'} is completed.

\subsection{Proof of \eqref{decest2}}
The last estimate \eqref{decest2} follows similarly using again the linear decay estimate
in Lemma \ref{lemdecayDu}
and then the bound \eqref{lemN3conc} for $N_3$.

The proof of Proposition \ref{propdecayirr} is concluded. $\hfill \Box$

\appendix

\bigskip
\section{Supporting material}\label{appsupp}

\subsection{Linear decay estimate}
Here is the linear estimate which we use to prove decay in Section \ref{secdecay}.

\begin{lemma}[Linear estimate]\label{lemlinear}
With the definitions \eqref{2dvf} and $\Sigma:= x\cdot \nabla_x$, $x\in\R^2$, we have
\begin{align}\label{linest1}
{\big\| e^{-it\Lambda^\frac{1}{2}} f\big\|}_{L^\infty_x(\R^2)} \lesssim
  |t|^{-1} \sum_{k \leq 3} \sum_{\ell \in \Z}
  2^{\ell/2} \fp{\big( {\| \Sigma \Omega^k P_\ell f \|}_{L^2} + {\| \Omega^k P_\ell f \|}_{L^2} \big) }.
\end{align}
%
%
As a consequence
\begin{align}\label{linest2}
\begin{split}
{\| u \|}_{L^\infty_x(\R^2)} & \lesssim
  \jt^{-1} \sum_{|I| \leq 3} \sum_{\ell \in \Z} 2^{\ell/2} \fp{ \big( {\| S (\Omega,\nabla)^I P_\ell u \|}_{L^2(\R^2)}
  + {\|(\Omega,\nabla)^I P_\ell u \|}_{L^2(\R^2)} + {\| u \|}_{H^2(\R^2)} \big) }
  \\ & + \sum_{|I|\leq 3} \sum_{\ell \in \Z} 2^{\ell/2}
  {\| (\Omega,\nabla)^{I} (\partial_t + i \Lambda^\frac{1}{2}) P_\ell u \|}_{L^2(\R^2)}.
\end{split}
\end{align}

\end{lemma}

\fp{We remark the importance of the appearance of at most one scaling vector field}
in \eqref{linest1} and \eqref{linest2}. We will often just use the less precise estimate
\begin{align}\label{linest2-}
\begin{split}
{\| u \|}_{L^\infty_x(\R^2)} & \lesssim
  \jt^{-1} \fp{ \sum_{k\leq 1,|I| \leq 4} {\| S^k (\Omega,\nabla)^I  u \|}_{L^2(\R^2)} }
  + \sum_{|I|\leq 4}
  {\| (\Omega,\nabla)^{I} (\partial_t + i \Lambda^\frac{1}{2}) u \|}_{L^2(\R^2)},
\end{split}
\end{align}
which dispenses of the summation over frequencies.
The presence of the $2^{\ell/2}$ factor at small frequencies is important when estimating the contribution
from the vector potential \eqref{decest1'} with \eqref{decayuLk}.


Similar estimates were proved  in \cite{GMSC} for the propagator $e^{it\Lambda^{3/2}}$,
and used in \cite{Wu3DWW} as well.

\begin{proof}[Proof of Lemma \ref{lemlinear}]
We begin by writing
\begin{align*}
(e^{-it|\nabla|^{1/2}} f)(x)
  = \sum_{\ell\in\Z}(e^{-it|\nabla|^{1/2}} P_\ell f) (x)
  = \sum_{\ell\in\Z} \int_{\R^2} e^{i (x\cdot\xi - t|\xi|^{1/2})} \varphi_\ell(\xi)\what{f}(\xi) \, d\xi,
\end{align*}
and aim to prove that for all $g = P_{[\ell-2,\ell+2]}g$ with
$\sum_{|I| \leq 3} {\| \Sigma (\Omega,\nabla)^I g \|}_{L^2(\R^2)} = 1$, we have
\begin{align}\label{ll1}
    {\| e^{it|\nabla|^{1/2}} g \|}_{L^\infty_x} \lesssim 2^{\ell/2} \dg{|t|}^{-1}.
\end{align}
We use polar coordinates $\xi = \rho\theta$, $\rho \geq 0$ and $\theta \in \mathbb{S}^1$
and expand $\what{g}(\xi)$ in Fourier series in the angular variable:
\begin{align}\label{ll2}
& (e^{it|\nabla|^{1/2}} g)(x)
= \sum_{m\in \Z} \int_{0}^{\infty}\dg{\int_0^{2\pi}} e^{i(|x| \rho \cos \theta - t\rho^{1/2})}
    e^{i\theta m} \what{g}_m(\rho) \, \varphi_\ell(\rho)\, d\theta \rho d\rho,
    \\
\nonumber
& \what{g}_m(\rho) := \frac{1}{2\pi} \int_{0}^{2\pi}
    e^{-i\theta m} \what{g}(\rho(\cos\theta,\sin\theta)) d\theta,
\end{align}
having assumed without loss of generality that $x=(|x|,0)$.
Then we can rewrite \eqref{ll2} as
\begin{align}\label{ll3}
(e^{it|\nabla|^{1/2}} g)(x)
  = \sum_{m\in \Z} \int_{0}^{\infty} e^{-it\rho^{1/2}} J_m(|x|\rho) \what{g}_m(\rho) \, \varphi_\ell(\rho)\, \rho d\rho,
\end{align}
where $J_m$ is the Bessel function of order $m$ and satisfies (see Stein \cite{steinbook})
\begin{align}\label{llbess}
\begin{split}
& J_m(y) := \int_0^{2\pi} \, e^{i(y\cos\theta + m\theta )} \, d\theta = e^{iy} J_{m,+}(y) + e^{-iy} J_{m,-}(y),
\\
& \mbox{with} \qquad \langle y \rangle^{1/2} \big| J_{m,\pm}(y) \big|
  + \langle y \rangle^{3/2}\big| J_{m,\pm}'(y) \big| \lesssim m^2,
\end{split}
\end{align}
where the implicit constant is independent of $m$.
In what follows we only consider the contribution from $J_{m,+}$ since the one involving $J_{m,-}$ is similar.
The term to bound in the sum in \eqref{ll3} is
\begin{align}\label{ll4}
I_m(t,x) := \int_{0}^{\infty} e^{i(|x|\rho-t\rho^{1/2})} J_{m,+}(|x|\rho) \what{g}_m(\rho) \, \varphi_\ell(\rho)\, \rho d\rho.
\end{align}
Notice how this term resembles a $1$ dimensional evolution of the form
$(e^{-it|\partial_x|^{1/2}} G)(|x|)$, with $\what{G}(\rho) \sim \what{g}_m(\rho) \, \varphi_\ell(\rho)\, \rho$.
For such a $1$d evolution one can use \fp{standard stationary phase arguments
to obtain an $L^1 \rightarrow L^\infty$ estimate and then a
(scaling-invariant) interpolation inequality} to get,
\begin{align}\label{ll1d}
\begin{split}
{\| e^{-it|\partial_x|^{1/2}} P_\ell G \|}_{L^\infty} 
& \fp{\lesssim |t|^{-1/2} 2^{3\ell/4} {\| P_{\sim \ell} G \|}_{L^1} }
\\
& \lesssim 2^{\ell/4} \fp{|t|^{-1/2}}
  \fp{\big( {\| x\partial_x P_{\sim \ell} G \|}_{L^2} + {\| P_{\sim \ell} G \|}_{L^2} \big)^{1/2} } 
  {\| P_{\sim \ell} G \|}_{L^2}^{1/2}.
\end{split}
\end{align}
See for example \cite{IoPu4} in the case of $e^{-it|\partial_x|^{3/2}}$.
Applying a similar argument to \eqref{ll4},
and using

(1) the presence of the factor $J_{m,+}$ which decays (see \eqref{llbess})
in the quantity $|x|\rho \approx t 2^{\ell/2}$ (at the stationary point of the phase $|x|\rho-t\rho^{1/2}$),
and

(2) the extra factor of $\rho \approx 2^\ell$,
we deduce
\begin{align}\label{llest}
{\| I_m(t) \|}_{L^\infty} \lesssim \fp{|t|^{-1}} 2^{\ell/2} \fp{m^2}
  \Big[ {\| \rho\partial_\rho \what{g_m} \|}_{L^2(\rho d\rho)} +  {\|\what{g_m} \|}_{L^2(\rho d\rho)} \Big].
\end{align}
\dg{
    The desired result now follows after using the bound
\begin{equation}
  \label{}
  \sum_{m \in \mathbb{Z}} m^2 \| \what{q}_m\|_{L^2(\rho d \rho)}
  \lesssim
  \left(\sum_{m \in \mathbb{Z}} m^6 \|\what{q}_m\|^2_{L^2(\rho d\rho)} \right)^{1/2}
  \lesssim
  \sum_{j \leq 3}
    \|\Omega^j q \|_{L^2(\mathbb{R}^2)},
\end{equation}
for both $q = g$ and $\rho\pa_\rho g$, where we have used Plancherel's theorem.}

The estimate \eqref{linest2} follows from \eqref{linest1} by letting
$u:=e^{-it\Lambda^{1/2}} f$, writing $\Sigma = S - \frac{1}{2}t\partial_t$
and using $\partial_t f = (\partial_t + i\Lambda^\frac{1}{2})u$. 
\end{proof}

\subsection{Bilinear operators and estimates}\label{appsym}
We consider the class of symbols
\begin{align}\label{symclass}
\S^\infty := \big\{
  m: \R^2 \times \R^2 \rightarrow \C \, : \, {\| m \|}_{\S^\infty} := {\| \mathcal{F}^{-1}(m) \|}_{L^1} < \infty \big\}.
\end{align}
Given a symbol $b :\R^2\times \R^2 \rightarrow \C$ we define the corresponding bilinear operator
\begin{align}\label{OpSym}
B(f,g) = \frac{1}{(2\pi)^2} \mathcal{F}^{-1} \Big( \int_{\R^2} b(\xi,\eta) 
  \widehat{f}(\xi-\eta) \widehat{g}(\eta) \, d\eta \Big).
\end{align}
We use the following notation for the localized symbols/operator:
\begin{align}\label{symloc}
b^{k,k_1,k_2}(\xi,\eta) := b(\xi,\eta) \varphi_k(\xi) \varphi_{k_1}(\xi-\eta) \varphi_{k_2}(\eta).
\end{align}
We have the following basic lemma:

\begin{lemma}\label{lemsym}
(i) We have $\S^\infty \hookrightarrow L^\infty(\R \times \R)$.
If $m,m'\in \S^\infty$ then $m\cdot m'\in \S^\infty$ and
\begin{equation}\label{lemsymmm'}
{\| m\cdot m' \|}_{S^\infty} \lesssim {\|m\|}_{S^\infty} {\|m'\|}_{S^\infty}.
\end{equation}
Moreover, if $m \in \S^\infty$, $A: \R^2 \to \R^2$
is a linear transformation, $v \in \R^2$, and $m_{A,v}(\xi,\eta) := m(A(\xi,\eta)+v)$, then
\begin{equation}\label{lemsymAv}
{\| m_{A,v} \|}_{\S^\infty} = {\|m\|}_{\S^\infty}.
\end{equation}

(ii) Assume $p,q,r\in[1,\infty]$ satisfy $1/p+1/q=1/r$, and $m\in \S^\infty$.
Then, for any $f,g\in L^2(\R)$,
\begin{equation}\label{lemsymbil}
{\| B(f,g) \|}_{L^r} \lesssim {\|m\|}_{\S^\infty} {\|f\|}_{L^p} {\|g\|}_{L^q}.
\end{equation}

\end{lemma}

We also need a lemma which describes commutation with our vector fields:

\begin{lemma}\label{lemsymcomm}
Given a bilinear operator as in \eqref{OpSym}, define the bilinear commutator with $\Gamma = \Omega$ or $S$ as
\begin{align}
[\Gamma,B(f,g)] := \Gamma B(f,g) - B(\Gamma f,g) - B(f,\Gamma g).
\end{align}
We have
\begin{align}
[\Gamma,B(f,g)] = B^\Gamma(f,g)
\end{align}
with symbols
\begin{align}\label{symcomm}
\begin{split}
b^S(f,g) &:= -(\xi \cdot \nabla_\xi + \eta \cdot \nabla_\eta) b(\xi,\eta),
\\
b^\Omega(f,g) &:= (\xi \wedge \nabla_\xi + \eta \wedge \nabla_\eta) b(\xi,\eta),
\end{split}
\end{align}


\end{lemma}

\begin{proof}
These formulas can be checked by a direct calculation.
\end{proof}


Here is basic lemma to handle product and pseudo-product in our spaces.

\begin{lemma}\label{lemprod}
Recall the definition \eqref{defZ} with \eqref{2dvf}. For all $N \geq 0$
\begin{align}\label{prodest0}
\sum_{r+k \leq N} {\| f g \|}_{Z^r_k}
  & \lesssim \sum_{\dg{r+k} \leq N/2} {\| f \|}_{Z^{r,\infty}_k} \sum_{r+k \leq N} {\| g \|}_{Z^r_k}
  + \sum_{r+k \leq N} {\| f \|}_{Z^r_k} \sum_{r+k \leq N/2} {\| g \|}_{Z^{r,\infty}_k},
\end{align}
and, more in general, for $1/p = 1/p_1+1/p_2 = 1/q_1+1/q_2$
\begin{align}\label{prodest0'}
\sum_{r+k \leq N} {\| f g \|}_{Z^{r,p}_k}
  & \lesssim \sum_{r+k \leq N/2} {\| f \|}_{Z^{r,p_1}_k} \sum_{r+k \leq N} {\| g \|}_{Z^{r,p_2}_k}
  + \sum_{r+k \leq N} {\| f \|}_{Z^{r,q_1}_k} \sum_{r+k \leq N/2} {\| g \|}_{Z^{r,q_2}_k}.
\end{align}

The same bounds hold if we replace the product $fg$ by a pseudo-product $B(f,g)$ as in \eqref{OpSym} with
a symbol $b$ satisfying
\begin{align}\label{prodestsym}
{\big\| b^{\Gamma^k} \big\|}_{S^\infty} \lesssim 1, \quad k=0,\dots, N,
\end{align}
where $b^\Gamma$, $\Gamma \in \{S,\Omega\}$ is defined as in \eqref{symcomm} and
$b^{\Gamma^k}$ is defined inductively by $b^{\Gamma^{k}}=( b^{\Gamma^{k-1}} )^\Gamma$ with $b_0 = b$.
\end{lemma}

\begin{proof}
We use the notation from \ref{vfnotation} for repeated applications of vector fields
(just $2$d ones, here). 
For any $|r|+|k| \leq N$ with $r = r_1 + r_2$ and $k=k_1+k_2$ we have 
\begin{align}\label{prodestpr1}
\begin{split}
\nabla_x^r \Gamma^k (f g) & = \big( \nabla_x^{r_1} \Gamma^{k_1} f \big)
  \, \big( \nabla_x^{r_2} \Gamma^{k_2} g \big).
\end{split}
\end{align}
Without loss of generality, by the symmetry of the right-hand side of \eqref{prodest0},
we may assume that $|r_1|+|k_1| \leq N/2$ and estimate
\begin{align*}
{\big\| \big( \nabla_x^{r_1} \Gamma^{k_1} f \big)
  \, \big( \nabla_x^{r_2} \Gamma^{k_2} g \big) \big\|}_{L^2}
  & \lesssim {\big\| \nabla_x^{r_1} \Gamma^{k_1} f \big\|}_{L^\infty}
  {\big\| \nabla_x^{r_2} \Gamma^{k_2} g \big\|}_{L^2}
  \\
  & \lesssim \sum_{r_1+k_1 \leq N/2} {\big\| f \big\|}_{Z^{r_1,\infty}_{k_1}}
  \sum_{r_2+k_2 \leq N} {\big\| g \big\|}_{Z^{r_2}_{k_2}}.
\end{align*}
The estimate \eqref{prodest0'} follows identically.
For the same estimate with $f\,g$ replaced by $B(f,g)$ it suffices to use Lemma \ref{lemsymcomm}
to commute vector fields, followed by an application \eqref{lemsymbil} using the assumption \eqref{prodestsym}.
\end{proof}

We also use the following standard product estimate: 
\begin{lemma}\label{lemprodestuse0}
for $f,g:\R^2 \rightarrow \C$, the following estimate holds:
\begin{align}\label{prodestuse00}
{\big\| |\nabla|^{1/2} (f g) \big\|}_{L^2} \lesssim
  {\| f \|}_{W^{1,3}} {\big\| |\nabla|^{1/2} g \big\|}_{L^2},
\end{align}
\end{lemma}

\subsection{A basic trace inequality}
We use the following basic trace estimate:

\begin{lemma}[Trace inequalities]\label{tracelem}
Let $f:\mathcal{D}_t \to \mathbb{R}$, define
$F: \mathbb{R}^2 \times (-\infty, 0]$ by $F(x, z) = f(x, z+h(x))$,
and assume that $\lim_{z\rightarrow -\infty} 
\| |\nabla|  F(\cdot, z)\|_{L^2(\mathbb{R}^2)} = 
\lim_{z\rightarrow -\infty} \| F(\cdot, z)\|_{L^2(\mathbb{R}^2)} = 0$.
Then,
\begin{align}\label{traceineqflat}
    \| |\nabla|^{1/2} F \big|_{\{z = 0\}} \|_{L^2(\mathbb{R}^2)} 
      &\lesssim \| \nabla_{x,z} F\|_{L^2_zL^2_z},
  \\
  \label{traceineqflat0}
      \|F \big|_{\{z = 0\}} \|_{L^2(\mathbb{R}^2)}
      &\lesssim \|F\|_{L^2_zL^2_x} + \| \nabla_{x,z} F\|_{L^2_zL^2_z}.
    \end{align}
In particular,
	if $\nabla h \in L^\infty$,
	\begin{align}
	 \| |\nabla|^{1/2} F \big|_{\{z = 0\}}\|_{L^2(\mathbb{R}^2)}
	 &\lesssim \| \nabla_{x,y} f\|_{L^2(\mathcal{D}_t)},
	 \label{traceineq}
	 \\
	 \|F \big|_{\{z = 0\}} \|_{L^2(\mathbb{R}^2)}
	 &\lesssim \|f\|_{L^2(\mathcal{D}_t)} + \|\nabla_{x, y} f\|_{L^2(\mathcal{D}_t)}.
	 \label{traceineq2}
	\end{align}
\end{lemma}
\begin{proof}
	We have
 \begin{align*}
  \int_{\mathbb{R}^2} ||\nabla|^{1/2} F(x, 0)|^2\, dx
	&= 2 \int_{-\infty}^0\int_{\mathbb{R}^2} \pa_z |\nabla|^{1/2} F(x, z) |\nabla|^{1/2} F(x, z)\, dxdz\\
	&= 2 \int_{-\infty}^0\int_{\mathbb{R}^2} \pa_z  F(x, z) |\nabla| F(x, z)\, dxdz\\
	&\lesssim \|\nabla_{x, z}F\|_{L^2_zL^2_x}^2
	\lesssim (1 + \|\nabla h\|_{L^\infty(\mathbb{R}^2)})^2 
	\|\nabla_{x, y} f\|_{L^2(\mathcal{D}_t)}^2,
\end{align*}
\dg{using \eqref{equivpr1}},
which gives \eqref{traceineqflat} and \eqref{traceineq}. The bounds
\eqref{traceineqflat0} and \eqref{traceineq2} are proved
in exactly the same way without including the factor $|\nabla|^{1/2}$.
\end{proof}


\bigskip
\section{The boundary equations}\label{appder}
Recall the definitions
\begin{align}\label{vsplit}
v = \nabla \psi + v_\omega, \qquad \varphi = \psi |_{\Dt}
\end{align}
with $\Delta\psi =0$, and 
$v_\omega \cdot n =0$.
Define the restrictions of the horizontal and vertical components to the boundary as follows:
\begin{align}\label{vrest}
v |_{\Dt} = (P,B), \qquad \nabla \psi |_{\Dt} = (P_{ir},B_{ir}), \qquad
  v_\omega |_{\Dt} = \wt{v_\omega} = (P_\omega,B_\omega).
\end{align}
We have
\begin{align}\label{VBir}
P_{ir} = \nabla \varphi - B_{ir} \nabla h,
  \qquad B_{ir} = \frac{G(h)\varphi + \nabla h \cdot \nabla \varphi}{1+|\nabla h|^2},
\end{align}
Recall that
\begin{align}\label{GBV}
G(h)\varphi = (P_{ir},B_{ir})\cdot (-\nabla h, 1) =  (P,B)\cdot (-\nabla h, 1),
\end{align}
where the last identity follows since $v_\omega \cdot n = 0$ and, therefore, we have
\begin{align}\label{VBoid}
B_\omega = \nabla h \cdot P_\omega,
\end{align}
that is, the rotational part of the velocity does not move the boundary.
Moreover, one can show that there exists $a_\omega$ such that
\begin{align}\label{nabao}
\nabla a_\omega = U_\omega := P_\omega + \nabla h B_\omega,
\end{align}
since $\partial_2 U_\omega^1 - \partial_1 U_\omega^2 = \omega |_{\Dt} \cdot (-\nabla h ,1 ) = 0$ 
in view of our assumption that $\omega$ vanishes on the boundary.

Notice that we can also express $a_\omega$ just in terms of $h$ and $P_\omega$:
\begin{align}\label{deraV}
a_\omega = \Lambda^{-1} R \cdot \nabla a_\omega = \Lambda^{-1} R \cdot \big(P_\omega
  + \nabla h (\nabla h \cdot P_\omega) \big).
\end{align}

\subsection{Evolution equations at the boundary}
The equation for the motion of the interface is the standard $\partial_t h = G(h)\varphi$.
The equation for the evolution of the potential is more involved and given by the following:

\begin{lemma}[Boundary evolution equations]\label{lemeqo}
Assume that $\omega\cdot n = 0$. With the notation \eqref{vsplit}-\eqref{nabao} we have
\begin{align}\label{eqo}
\left\{
\begin{array}{rl}
 \partial_t h & = G(h) \varphi
\\
\partial_t \varphi & = - h -\dfrac{1}{2}|\nabla \varphi|^2 + \dfrac{1}{2} \dfrac{(G(h)\varphi + \nabla h
  \cdot \nabla \varphi)^2}{1+|\nabla h|^2} - \Lambda^{-1} R \cdot \partial_t P_\omega
   - \nabla \varphi \cdot P_\omega - \dfrac{1}{2}|P_\omega|^2 + R_\omega
\end{array}
\right.
\end{align}
where
\begin{align}\label{eqoR}
\begin{split}
R_\omega = - \Lambda^{-1} R \cdot  \partial_t \big(\nabla h (\nabla h \cdot P_\omega) \big)
	- \frac{1}{2} (P_\omega \cdot \nabla h)^2 + (G(h)\varphi) P_\omega \cdot \nabla h.
\end{split}
\end{align}
\end{lemma}

\begin{remark}
In the irrotational case the Zakharov formulation in terms of $(P,B)$ reads
\begin{align}\label{ZakVB}
\begin{split}
\partial_t \varphi & = -h -\frac{1}{2}|\nabla \varphi|^2 + \frac{1}{2} (1+|\nabla h|^2) B^2
  = -h - \frac{1}{2}P^2 + \frac{1}{2}B^2 - B \nabla h \cdot P.
\end{split}
\end{align}
One can check (and we will do this in the proof below)
that the same equation extends to the rotational case, in the sense that
\begin{align}\label{ZakVBrot}
\begin{split}
\partial_t (P + \nabla h B)  & = \nabla\big( -h - \frac{1}{2}P^2 + \frac{1}{2}B^2 - B \nabla h \cdot P \big),
\end{split}
\end{align}
where $(P,B)$ are now defined as in \eqref{vrest}.
Then, by expanding $P = P_{ir} + P_\omega$ and $B = B_{ir} + B_\omega$, using \eqref{VBir} and \eqref{VBoid},
\eqref{nabao} and \eqref{deraV}, one can arrive at \eqref{eqo}-\eqref{eqoR}.

Equation \eqref{ZakVBrot} is also equivalent to the formulation used by Castro-Lannes \cite{CaLa}
which has the form
\begin{align}\label{eqaomega1'CL}
\begin{split}
\partial_t U_{\parallel} & =
  -\nabla h - \frac{1}{2} \nabla |U_\parallel|^2 + \frac{1}{2}\nabla \big( (1+|\nabla h|^2) B^2 \big)
\end{split}
\end{align}
where $ U_{\parallel} = P + B \nabla h = \nabla (\varphi + a_\omega)$.
\end{remark}

For completeness we give here a derivation of the equation for $\partial_t \varphi$ in \eqref{eqo},
which slightly differs from the one in \cite{Gin1}.

\begin{proof}[Proof of Lemma \ref{lemeqo}]
Restricting Euler's equations \eqref{freebdy} to the boundary $\partial \mathcal{D}_t$,
using that $p=0$ on the boundary, we have ($g=1$)
\begin{align}\label{eqBV}
& (\partial_t + P \cdot \nabla)P = \fp{- a \nabla h}, \qquad a := -\partial_3 p,
\\
& (\partial_t + P \cdot \nabla)B = a - 1.
\end{align}
From these we get an evolution equations for $\nabla \varphi + \nabla a_\omega = P + \nabla h B$,
that is, the tangential component of the velocity field restricted to the boundary:
\begin{align}\label{eqaomega1}
\begin{split}
\partial_t (\nabla \varphi + \nabla a_\omega ) & =
  - P \cdot \nabla P - a \nabla h
  + \nabla h \big(- P \cdot \nabla B + a - 1 \big)
  + (\nabla\partial_t h) B
  \\
  & = -\nabla h - P\cdot \nabla P - \nabla h (P \cdot \nabla B) +
  B \nabla (-\nabla h \cdot P + B),
\end{split}
\end{align}
having used the kinematic boundary condition $\partial_t h = -\nabla h \cdot P + B$.
We then rewrite \eqref{eqaomega1} as
\begin{align}\label{eqaomega1'}
\begin{split}
\partial_t (\nabla \varphi + \nabla a_\omega ) & =
  \nabla \big( - h + (1/2)B^2 - B \nabla h \cdot P  - (1/2) |P|^2 \big)
  \\
  & + (1/2) \nabla |P|^2 - P\cdot \nabla P - \nabla h (P \cdot \nabla B)
  + \nabla B (\nabla h \cdot P).
\end{split}
\end{align}
To conclude, we observe that the last line above vanishes in view of the
condition  $\omega \cdot n = 0$.
In fact, using that, for $i=1,2$, $(\partial_i F) |_{\Dt} = \partial_i (F|_{\Dt}) - (\partial_3 F)|_{\Dt}
  \partial_i h$,
we can calculate
\begin{align}
    \label{curldotn}
\begin{split}
    \dg{(\nabla \times v)} \cdot n & = - \partial_1 h \big(\partial_2 v_3 - \partial_3 v_2  \big) |_{\Dt}
  - \partial_2 h \big(\partial_3 v_1 - \partial_1 v_3  \big) |_{\Dt}
  + (\partial_1 v_2 - \partial_2 v_1) |_{\Dt}
\\
& = - \partial_1 h \big(\partial_2 B - \partial_3 v_3|_{\Dt} \partial_2 h - \partial_3 v_2|_{\Dt}  \big)
  - \partial_2 h \big(\partial_3 v_1|_{\Dt} - \partial_1 B + \partial_3 v_3|_{\Dt} \partial_1 h  \big)
  \\
  & + (\partial_1 P_2 - \partial_3 v_2|_{\Dt} \partial_1 h
  - \partial_2 P_1 + \partial_3 v_1|_{\Dt} \partial_2 h)
\\
& = - \partial_1 h \partial_2 B
  + \partial_2 h \partial_1 B + \partial_1 P_2 - \partial_2 P_1
  = \curl (P + B \nabla h).
\end{split}
\end{align}
The $\curl$ above naturally denotes the scalar operator in $2$d.
Note how the same calculation with $v_\omega,P_\omega,B_\omega$ instead of $v,P,B$,
shows \eqref{nabao} since $\nabla \times v_\omega \cdot n
= 
\curl (P_\omega + \nabla h B_\omega)$.

\dg{Returning to the last line in \eqref{eqaomega1'},
    by a direct calculation we find $-\nabla h  P\cdot \nabla B + \nabla B (P\cdot \nabla h)
= (\nabla B\cdot \nabla^\perp h) P^\perp$,
and writing $P \cdot \nabla P = (1/2) \nabla |P|^2 + (\curl P) P^\perp$,
we arrive at}
\begin{align}
\begin{split}
& (1/2) \nabla |P|^2 - P\cdot \nabla P - \nabla h (P \cdot \nabla B)
  + \nabla B (\nabla h \cdot P)
  \\
& = \dg{\big(-\curl P +  \nabla B \cdot \nabla^\perp h \big) P^\perp}
  = 0,
\end{split}
\end{align}
\dg{where we used the identity $-\nabla B\cdot \nabla^\perp h = \curl (B\nabla h)$,
the identity \eqref{curldotn} we just proved, and the assumption that
$\curl \omega = 0$ on $\pa \D_t$.}

%

From \eqref{eqaomega1'} we have thus arrived at
\begin{align}\label{eqaomega3}
\begin{split}
\partial_t (\varphi + a_\omega ) & = - h + \frac{1}{2}B^2 - B \nabla h \cdot P  - \frac{1}{2} |P|^2.
\end{split}
\end{align}
To finally obtain \eqref{eqo} we rewrite this as
\begin{align}\label{eqaomega5}
\begin{split}
\partial_t \varphi & = - h - \frac{1}{2} |P_{ir}|^2  + \frac{1}{2}B_{ir}^2 - B_{ir} \nabla h \cdot P_{ir}
  - \partial_t a_\omega
  \\
  & - \frac{1}{2} |P_{\omega}|^2  + \frac{1}{2}B_{\omega}^2 - B_{\omega} \nabla h \cdot P_{\omega}
  - P_{ir} P_\omega  + B_{ir}B_\omega - B_{ir} \nabla h \cdot P_{\omega}
  - B_{\omega} \nabla h \cdot P_{ir}.
\end{split}
\end{align}
The first line of \eqref{eqaomega5} matches the first three terms on the right-hand side of
the equation \eqref{eqo} for $\partial_t \varphi$ (see \eqref{VBir})
plus the terms that contain a $\partial_t$ (see \eqref{deraV}).

The desired claim then follows provide we verify that all the terms
in the second line of \eqref{eqaomega5} match the remaining terms in \eqref{eqo}, that is,
the expression
\begin{align*}
- \nabla \varphi \cdot P_\omega - \dfrac{1}{2}|P_\omega|^2
- \frac{1}{2} (P_\omega \cdot \nabla h)^2 + (G(h)\varphi) P_\omega \cdot \nabla h.
\end{align*}
This can be done by direct inspection using \eqref{VBoid} and \eqref{VBir}-\eqref{GBV}.
\end{proof}


\subsection{Evolution equations for $u$}
Here we diagonalize \eqref{eqo} and derive the main boundary equations  \eqref{der30}
in terms of the single complex valued unknown $h + i \Lambda^{1/2} \varphi$.
\eqref{der30} are the main equations at the boundary, and are used to show
decay for $u$ in Section \ref{secdecay}. 

Let us introduce some notation for the Dirichlet-Neumann map: we let
\begin{align}\label{derG}
& G(h)\varphi = |\nabla|\varphi + G_{\geq 2}(h)\varphi  = |\nabla|\varphi + G_2(h)\varphi + G_{\geq 3}(h,\varphi)
\\
\label{derG2}
& G_2(h)\varphi := -\nabla \cdot (h\nabla\varphi) - |\nabla|(h |\nabla| \varphi),
\end{align}
and $G_{\geq 3}$, defined by \eqref{derG},
contains cubic and higher order terms in $(h,\varphi)$; see Proposition \ref{propDN} below.
The following is a direct consequence of \eqref{eqo}:

\begin{lemma}\label{lemequ}
Let
\begin{align}\label{eqodefu}
u = h + i \Lambda^{1/2} \varphi,
  \qquad h = \frac{1}{2}(u+\bar{u}), \quad \varphi = \frac{1}{2i\Lambda^{1/2}}(u-\bar{u}).
\end{align}
Then
\begin{align}\label{der30}
\begin{split}
\partial_t u + i \Lambda^{1/2} u & = B_0+ N_2 + N_3,
\end{split}
\end{align}
where:

\begin{itemize}

\item $B_0$ are the quadratic terms in $(h,\varphi)$ given by
\begin{align}\label{der31}
B_0 := -\nabla \cdot (h\nabla\varphi) - |\nabla|(h |\nabla| \varphi)
	+ i\Lambda^{1/2} \Big(- \frac{1}{2}|\nabla \varphi|^2 + \frac{1}{2} (|\nabla| \varphi)^2 \Big).
\end{align}

\item $N_2$ gathers quadratic terms with at least one rotational term: 
\begin{align}\label{der32}
N_2 := - \nabla \varphi \cdot P_\omega - \frac{1}{2}|P_\omega|^2
- i \Lambda^{-1/2} R \cdot \partial_t P_\omega.
\end{align}

\item The remainders of cubic and higher homogeneity are given by
\begin{align}\label{der33}
\begin{split}
 N_3 := G_{\geq 3}(h,h,\varphi) + i \Lambda^{1/2} \left[ \frac{(G(h)\varphi + \nabla h
  \cdot \nabla \varphi)^2}{2(1+|\nabla h|^2)}
  - \frac{1}{2}(|\nabla| \varphi)^2 + R_\omega \right]
\end{split}
\end{align}
where $R_\omega$ given as in \eqref{eqoR}.
\end{itemize}
\end{lemma}

\subsection{Symbols of quadratic operators}
By defining the symbols 
\begin{align}\label{derq}
q(\xi,\eta) := \frac{1}{4i |\eta|^{1/2}} \big( \xi \cdot \eta - |\xi| |\eta| \big)
	- \frac{i |\xi|^{1/2} }{8 |\eta|^{1/2} |\xi-\eta|^{1/2}}
	\big( \eta \cdot (\xi-\eta) + |\eta| |\xi-\eta|  \big)
\end{align}
and
\begin{align}\label{derb+-}
\begin{split}
b_{++}(\xi,\eta) & := q(\xi,\eta),
\\
b_{--}(\xi,\eta) & := - q(\xi,\eta),
\\
b_{+-}(\xi,\eta) & := - q(\xi,\eta) + q(\xi,\xi-\eta),\qquad
b_{-+}(\xi,\eta) := 0,
\end{split}
\end{align}
we write the quadratic terms in \eqref{der31} as
\begin{align}\label{derB_0}
\begin{split}
B_0 = B_0 (u,u) & = \sum_{\eps_1,\eps_2 \in \{+,-\}} B_{\eps_1\eps_2}(u_{\eps_1},u_{\eps_2}),
\end{split}
\end{align}
consistently with the notation symbols/operator \eqref{OpSym}.
Note how we have expressed only the quadratic terms in $(h,\varphi)$ as functions of $(u,\bar{u})$
since these are the only terms on which we will need to use Fourier analysis. 
In particular, 
\eqref{der30}-\eqref{der33} with \eqref{derb+-}-\eqref{derB_0} give us \eqref{decayequ}-\eqref{decayuL}.

It is not hard to verify that the symbols satisfy, see \eqref{symloc} and Lemma \ref{lemsym},
\begin{align}\label{derb+-est}
\begin{split}
& {\big\|b_{\eps_1\eps_2}^{k,k_1,k_2} \big\|}_{\mathcal{S}^\infty} \lesssim 2^{k} 2^{\min(k_1,k_2)/2}.
\end{split}
\end{align}
Note that we chose to write this bound by putting in evidence the vanishing in the output frequency $|\xi|$
since this will be helpful in the nonlinear analysis.
Moreover, one can also directly check that the same bounds hold true after
commuting with vector fields since, see \eqref{symcomm},
\begin{align}\label{qGamma}
q^S(\xi,\eta) = -\frac{3}{2} q(\xi,\eta), \qquad q^\Omega(\xi,\eta) = 0,
\end{align}
and therefore, according to the definition \eqref{symclass} we have, for all $j$,
\begin{align}\label{derb+-estk}
\begin{split}
& {\big\| (b_{\eps_1\eps_2})^{\Gamma^j} 
  \varphi_k(\xi) \varphi_{k_1}(\xi-\eta) \varphi_{k_2}(\eta)\big\|}_{\mathcal{S}^\infty} \lesssim_{|j|} 2^{k} 2^{\min(k_1,k_2)/2}.
\end{split}
\end{align}

\subsection{Dirichlet-Neumann map and estimates of remainder terms}\label{ssecDN}
To conclude this section we give estimates for the cubic and higher order remainder terms $N_3$
defined in \eqref{der33}.

First, let us recall that, in view of the a priori bounds on $h$
in \eqref{aprioridecay} 
and \eqref{apriorie0}, together with the (elliptic) bound \eqref{velpotest} on $\varphi$,
we have, for all $t\in[0,T]$, a priori, that 
\begin{align}\label{aprioriapp}
& 
  \sum_{r+k \leq N_1} {\big\| u(t) \big\|}_{Z^{r,\infty}_{k}(\R^2)} \lesssim \e_0 \jt^{-1},
\quad
\sum_{r+k \leq N_0-20} {\| u(t) \|}_{Z^{r}_k}^2 \lesssim \e_0 \jt^{3p_0}. 
\end{align}
Moreover, in view of the (elliptic) result of Proposition \ref{propalphaintro}
we have the following bound on $\tilde{v}_\omega$, for all $t\in[0,T]$:
\begin{align}\label{lemVoconcapp}
\sum_{r+k \leq N} {\| \wt{v_\omega} \|}_{Z^{r}_k} \lesssim \e_1 \jt^\delta,
\quad \sum_{r+k \leq N-1} {\| \partial_t \wt{v_\omega} \|}_{Z^{r}_k} \lesssim \e_1 \e_0 \jt^{\delta}.
\end{align}
where $N=N_1+12$, see \eqref{paramN}; see \eqref{lembounda-Voconc0}.

We can then show the following estimate on the cubic remainder in the equation \eqref{der30}:

\begin{lemma}\label{lemN3}
Under the a priori assumptions \eqref{apriorie0}-\eqref{aprioridecay}, which in particular
imply \eqref{aprioriapp} and \eqref{lemVoconcapp}, we have 
\begin{align}
\label{lemN3conc}
& \sum_{r+k \leq N_1 + 11} {\| N_3 \|}_{Z^r_k} \lesssim \e_0^2 \jt^{-5/4}.
\end{align}
\end{lemma}

To prove the above Lemma we need some bounds on the Dirichlet-Neumann operator,
which we state in the next proposition.
Recall the notation \eqref{derG}-\eqref{derG2}, that is,
$G(h)\varphi = |\nabla|\varphi + G_2(h)\varphi + G_{\geq 3}(h,\varphi)$,
and let $G_{\geq 2}(h,\varphi) := G_2(h)\varphi + G_{\geq 3}(h,\varphi)$.

\begin{prop}\label{propDN}
    Under the assumptions \eqref{aprioriapp} we have the linear bounds 
\begin{align}
\label{G1est}
& \sum_{r+k \leq N_0-22} {\| G(h)\varphi \|}_{Z^r_k} \lesssim \e_0 \jt^{3p_0},
\quad
\sum_{r+k \leq N_1-2} {\| G(h)\varphi \|}_{Z^{r,\infty}_k} \lesssim \e_0 \jt^{-1+},
\end{align}
the quadratic bounds
\begin{align}
\label{G2est}
& \sum_{r+k \leq N_0-24} {\| G_{\geq 2}(h,\varphi) \|}_{Z^r_k} \lesssim \e_0^2 \jt^{-1+3p_0},
\quad
& \sum_{r+k \leq N_1-4} {\| G_{\geq 2}(h,\varphi) \|}_{Z^{r,\infty}_k} \lesssim \e_0^2 \jt^{-4/3},
\end{align}
and the cubic bounds
\begin{align}
\label{G3est}
& \sum_{r+k \leq N_0-26} {\| G_{\geq 3}(h,\varphi) \|}_{Z^r_k} \lesssim \e_0^3 \jt^{-5/4}.
\end{align}
\end{prop}


The above estimate are rather standard and essentially based on a Taylor expansion for small $h$ of
the Dirichlet-Neumann map. In particular, they do not require any paralinearization argument,
and losses of derivatives are allowed, as one can see from the number of vector fields that we use.
However, to our knowledge, they cannot
be found in one single reference in the exact way that they are stated above.
Without the vector fields $S$ and $\Omega$ these are proven, for example, in \cite{GMS2};
Proposition F.1 there 
gives explicit bounds for the quartic and
higher remainder terms, while the term of homogeneity up to three can be handled explicitly.
In the same reference the authors also give estimates involving a weight $x$,
which resemble those for the scaling vector field $S$.
Estimates with rotation vector fields are included in the work of Deng-Ionescu-Pausader-Pusateri \cite{DIPP};
see Proposition B.1 there.
Analogous estimates with the scaling vector fields can also be derived in the same exact way\footnote{The
scaling vector field is not included in the estimates
for the DN map in \cite{DIPP} since that work deals with the gravity-capillary waves system 
which is not scale invariant.}

In what follows, we first use Proposition \ref{propDN}
to obtain the estimate \eqref{lemN3conc} on the cubic remainder $N_3$ in the main evolution equation \eqref{der30}.
We will then sketch the proof of Proposition \ref{propDN} at the end of this section, 
relying on Lemma \ref{LemmaPsi}.

\begin{proof}[Proof of Lemma \ref{lemN3}]
Looking at the definition \eqref{der33} we see that the term $G_{\geq 3}$
is already estimated as desired using \eqref{G3est}.
The remaining terms are 
\begin{align}
\label{N31}
N_{3,1} & := i \Lambda^{1/2} \frac{1}{2} \left[ (G(h)\varphi)^2 - (|\nabla| \varphi)^2 \right],
\\
\label{N32}
N_{3,2} & := i \Lambda^{1/2} \frac{1}{2}\big( (G(h)\varphi) \big)^2
  \left[ \frac{1}{1+|\nabla h|^2} - 1 \right],
\\
\label{N33}
N_{3,3} & := i \Lambda^{1/2} \frac{(G(h)\varphi) \, \fp{(\nabla h \cdot \nabla \varphi)}}{(1+|\nabla h|^2)},
\\
\label{N34}
N_{3,4} & := -i \Lambda^{-1/2} R \cdot  \partial_t \big(\nabla h (\nabla h \cdot P_\omega) \big),
\\
\label{N35}
N_{3,5} & := - \frac{i}{2} \Lambda^{1/2} (P_\omega \cdot \nabla h)^2,
\\
\label{N36} 
N_{3,6} & := i \Lambda^{1/2} \big[ (G(h)\varphi) P_\omega \cdot \nabla h \big].
\end{align}
The first term can be written as $2N_{3,1} = i \Lambda^{1/2} G_{\geq 2}(h,\varphi)
  (G(h)\varphi + |\nabla| \varphi)$,
and we can use \eqref{prodest0}
followed by \eqref{G2est} and \eqref{G1est} and \eqref{aprioriapp} to bound
(recall from \eqref{param} that $N_1 \geq N_0/2 + 5$):
\begin{align*}
\sum_{r+k \leq N_1+11} & {\| N_{3,1} \|}_{Z^r_k}
  \\
  & \lesssim \sum_{r+k \leq N_1+11} {\| G_{\geq 2}(h,\varphi) \|}_{Z^r_k}
  \sum_{r+k \leq N_1-10} \big( {\| G(h)\varphi \|}_{Z^{r,\infty}_k}
    + {\| |\nabla|\varphi \|}_{Z^{r,\infty}_k} \big)
  \\
  & + \sum_{r+k \leq N_1-10} {\| G_{\geq 2}(h,\varphi) \|}_{Z^{r,\infty}_k}
  \sum_{r+k \leq N_1+11} \big( {\| G(h)\varphi \|}_{Z^r_k} + {\| |\nabla|\varphi \|}_{Z^r_k} \big)
  \\
  & \lesssim \e_0^2 \jt^{-3/4} \cdot \e_0 \jt^{-3/4} + \e_0 \jt^{-4/3} \cdot \e_0 \jt^{3p_0}
  \\
  & \lesssim \e_0^3 \jt^{-5/4},
\end{align*}
consistently with \eqref{lemN3conc}.

The terms \eqref{N32} and \eqref{N33} are easily
estimated using \eqref{prodest0} and the linear bounds \eqref{aprioriapp} and \eqref{G1est}.

For the term \eqref{N34} we first use fractional integration and the
standard commutation rules 
to estimate
\begin{align*}
\sum_{r+k \leq N_1+11} {\| N_{3,4} \|}_{Z^r_k} & \lesssim
  \sum_{r+k \leq N_1+11} {\| \partial_t \big(\nabla h (\nabla h \cdot P_\omega) \big) \|}_{Z^{r,4/3}_k}
\end{align*}
Let us look at the term where $\partial_t$ hits the first $h$ factor;
when it hits the second $h$ the argument is identical, and when it hits $P_\omega$ the estimates are even simpler.
Using product estimates, and Sobolev's embedding, we can bound
\begin{align*}
 & \sum_{r+k \leq N_1+11} {\| (\partial_t \nabla h) (\nabla h \cdot P_\omega) \|}_{Z^{r,4/3}_k}
 \\
 & \lesssim
 \sum_{r+k \leq N_1-10} \big( {\| h \|}_{Z^{r,\infty}_k} + {\| \partial_t h \|}_{Z^{r,\infty}_k} \big)
 \big( \sum_{r+k \leq N_1+11} {\| h \|}_{Z^r_k} + {\| \partial_t h \|}_{Z^r_k} \big)
 \sum_{r+k \leq N_1+11} {\| P_\omega \|}_{Z^r_k}
 \\
 & \lesssim \e_0 \jt^{-3/4} \cdot \e_0 \jt^{3p_0} \cdot \e_1 \jt^\delta
 \lesssim \e_0^2 \jt^{-3/2},
\end{align*}
having used 
\eqref{aprioriapp} to estimate $h$,
\eqref{G1est} for $\partial_t h = G(h)\varphi$, \eqref{lemVoconcapp} for $P_\omega$,
and $\e_1 \lesssim \jt^{-1}$. 

The remaining terms \eqref{N35} and \eqref{N36} can be estimated similarly to the ones above
using again \eqref{G1est}, \eqref{lemVoconcapp} and \eqref{aprioriapp}.
\end{proof}

The next lemma constructs and bounds the velocity potential given its value at the surface.
This result is then used to obtain Proposition \ref{propDN}.

\begin{lemma}\label{LemmaPsi}
Fix an integer $N \in (N_1+15, N_0) \cap \Z$,
and let $N_1$ be as above (in particular $N_1 \geq N/2+5$).
Assume that $h$ and $|\nabla|^{1/2}\varphi$ satisfy 
\begin{align}\label{LemmaPsih}
    \sum_{\dg{|r|+|k|} \leq N+1} {\| \nabla^r \Gamma^{k} h \|}_{L^2} \lesssim \e_0 \jt^{p_0},
    \qquad \sum_{\dg{|r|+|k|} \leq N_1} {\| \nabla^r \Gamma^{k} h \|}_{L^\infty} \lesssim \e_0 \jt^{-1},
\end{align}
and
\begin{align}\label{LemmaPsivarphi}
    \sum_{\dg{|r|+|k|} \leq N} {\| \nabla^r \Gamma^{k} |\nabla|^{1/2} \varphi \|}_{L^2} \lesssim \e_0 \jt^{3p_0},
    \qquad \sum_{\dg{|r|+|k|} \leq N_1} {\| \nabla^r \Gamma^{k} |\nabla|^{1/2} \varphi \|}_{L^\infty} \lesssim \e_0 \jt^{-1},
\end{align}

Then, there exists a unique solution $\psi$ to the elliptic problem $\Delta \psi = 0$
in $\D_t$, with $\psi = \varphi$ on $\partial \D_t$,
with $\nabla_{y} \psi \rightarrow 0$ as $y \rightarrow -\infty$.
If $\Psi(x,z) = \psi(x,z+h(t,x))$, adopting the notation from \ref{secvfO}, we have the (linear) $L^2$ bounds
\begin{align}\label{PsiL2}
& {\big\| \underline{\Gamma}^n \, \nabla_{x,z} \Psi \big\|}_{L^2_z L^2_x}
  + {\big\|  \underline{\Gamma}^n |\nabla|^{1/2}\, \Psi \big\|}_{L^\infty_z L^2_x} \lesssim \e_0 \jt^{3p_0},
  \qquad n \leq N,
\end{align}
the (linear) $L^\infty_x$-type bounds, for all $\ell\in\Z$,
\begin{align}\label{Psiinfty}
& {\big\| \underline{\Gamma}^n \nabla_{x,z} P_\ell \Psi \big\|}_{L^2_z L^\infty_x}
  + {\big\| \underline{\Gamma}^n |\nabla|^{1/2} P_\ell \Psi \big\|}_{L^\infty_z L^\infty_x} \lesssim \e_0 \jt^{-1},
  \qquad n \leq N_1-1,
\end{align}
where $P_\ell$ is the standard Littlewood-Paley projection in the $x$ variable.

Moreover, we have the quadratic $L^2_x$-bounds
\begin{align}\label{PsiL2quad}
& {\big\| \underline{\Gamma}^n \nabla_{x,z}(\Psi - e^{z|\nabla|}\varphi) \big\|}_{L^2_zL^2_x}
  + {\big\| \underline{\Gamma}^n |\nabla|^{1/2}(\Psi - e^{z|\nabla|}\varphi) \big\|}_{L^\infty_zL^2_x}
 \lesssim \e_0^2 \jt^{-1+3p_0}, \qquad n\leq N,
\end{align}
and the quadratic $L^\infty_x$-bounds
\begin{align}
\label{Psiinftyquad}
& {\big\| \underline{\Gamma}^n \nabla_{x,z}(\Psi - e^{z|\nabla|}\varphi) \big\|}_{L^2_zL^\infty_x}
  + {\big\| \underline{\Gamma}^n |\nabla|^{1/2} (\Psi - e^{z|\nabla|}\varphi) \big\|}_{L^\infty_zL^\infty_x}
 \lesssim \e_0^2 \jt^{-2+}, \qquad n\leq N_1-1.
\end{align}
\end{lemma}


Note that the assumptions \eqref{LemmaPsih}-\eqref{LemmaPsivarphi} are all
consistent with the bounds \eqref{aprioriapp} for $N < N_0-20$. 


A similar version of Lemma \ref{LemmaPsi} is essentially contained in Appendix B of \cite{DIPP}
(see in particular Lemma B.4).
The scaling vector field and multiple $\partial_z$ are not included in that Lemma, but can be added
with minor changes to the proofs.
We give some details of the proof for completeness. 

\begin{proof}[Proof of Lemma \ref{LemmaPsi}]
Transforming the elliptic equation $\Delta_{x,y}\psi = 0$ to the flat domain as in \eqref{alphaeq} with \eqref{Eadef}-\eqref{Ebdef}
(with the roles of $(\alpha,\beta)$ played by $(\Psi,\psi)$ here)
and then applying the formula \eqref{dirichletformula} with $F=0$, 
we see that $\psi$ is harmonic with (we omit the time variable) $\psi(x,h(x))=\varphi(x)$, 
if and only if $\psi(x,z+h(t,x)) =: \Psi(x,z)$ is a fixed point of the map
\begin{align}\label{dirichletformulaPsi}
\begin{split}
(T\Psi)(x,z) := e^{z|\nabla|} \varphi(x) &
+ \frac{1}{2} \int_{-\infty}^0 e^{-|z-s||\nabla|} ( \mathrm{sign}(z-s)\dg{E^a - E^b}  
  ) \, ds
  \\
                                         & - \frac{1}{2} \int_{-\infty}^0  e^{(z+s)|\nabla|} (\dg{E^a - E^b} 
  ) \, ds,
\end{split}
\end{align}
with
\begin{align}\label{EadefPsi}
    \dg{E^a} (\pa \Psi) & = \frac{\nabla}{|\nabla|}\cdot \left( \nabla h \, \pa_z \Psi \right), 
\qquad
\dg{E^b}(\pa \Psi) = -|\nabla h|^2 \pa_z \Psi + \nabla h \cdot \nabla \Psi.
\end{align}
Based on \eqref{dirichletformulaPsi},
one can perform a fixed point argument in a small $C\e_0$ ball in an apposite space (that is, the space $\mathcal{L}_0$
in \eqref{prPsi10} below)
that will then imply the main conclusions \eqref{PsiL2}-\eqref{Psiinfty},
and, as a byproduct also \eqref{PsiL2quad}-\eqref{Psiinftyquad}.
We define the following spaces, which will be used just within this proof:
for $g \in \R^2 \rightarrow \mathbb{C}$, let $\mathcal{F}_p$, for $p \in[-10,0]$, be defined by the norm
\begin{align}\label{prPsi0}
\begin{split}
{\| g \|}_{\mathcal{F}_p}
  & := \jt^{-3p_0} \sup_{|n| \leq N+p} {\big\| \Gamma^n |\nabla|^{1/2} g(t) \big\|}_{L^2_x}
  +  
   \jt \sup_{|n| \leq N_1-1+p}  \sup_{\ell \in \Z} {\big\| \Gamma^n \, |\nabla|^{1/2} P_\ell  g(t) \big\|}_{L^\infty_x};
\end{split}
\end{align}
for $G \in \R^2\times(-\infty,0] \rightarrow \mathbb{C}$, let $\mathcal{L}_p$
be defined by the norm
\begin{align}\label{prPsi10}
\begin{split}
{\| G \|}_{\mathcal{L}_{p}} 
  & :=  \jt^{-3p_0} \sup_{|n| \leq N+p} \big( {\big\| \underline{\Gamma}^n \, \nabla_{x,z} G(t) \big\|}_{L^2_z L^2_x}
  + {\big\| \underline{\Gamma}^n |\nabla|^{1/2}\, G(t) \big\|}_{L^\infty_z L^2_x} \big)
  \\
  & +  \jt \sup_{\ell \in \Z} \sup_{|n| \leq N_1 -1 + p}
  \big( {\big\| \underline{\Gamma}^n \, \nabla_{x,z} P_\ell  G(t) \big\|}_{L^2_z L^\infty_x}
  + {\big\|  \underline{\Gamma}^n |\nabla|^{1/2} \, P_\ell  G(t) \big\|}_{L^\infty_z L^\infty_x} \big).
\end{split}
\end{align}
Note that in the $L^\infty_x$ based spaces we only take the $\sup$ over Littlewood-Paley projections.
These spaces are natural ones to estimate the Poisson kernel in. 
Indeed, we have
\begin{align}\label{prPsiPoi1}
{\| e^{z|\nabla|} \varphi \|}_{\mathcal{L}_p} \lesssim {\| \varphi \|}_{\mathcal{F}_p};
\end{align}
the estimate for the $L^2_x$ components follows from the bounds \eqref{expbounds1}-\eqref{expbounds2};
the estimate for the $L^\infty_x$ components follow from the standard
estimates for each fixed Littlewood-Paley piece
\begin{align}\label{prPsiPoiinfy}
\begin{split}
& {\| |\nabla| e^{z|\nabla|} P_\ell \varphi \|}_{L^2_z L^\infty_x} 
  \lesssim {\| |\nabla|^{1/2} P_\ell \varphi \|}_{L^\infty},
\\
& {\| |\nabla|^{1/2} e^{z|\nabla|} P_\ell \varphi \|}_{L^\infty_z L^\infty_x} 
  \lesssim {\| |\nabla|^{1/2} P_\ell \varphi \|}_{L^\infty}, \qquad \ell \in \Z.
\end{split}
\end{align}
We also have bounds for bulk integrals like those appearing in \eqref{dirichletformulaPsi}:
\begin{align}\label{prPsiPoiint}
\begin{split}
{\Big\|  \int_{-\infty}^0 e^{-|z-s||\nabla|} \mathbf{1}_\pm(z-s) F(\cdot,s) \, ds \Big\|}_{\mathcal{L}_p}
	& \lesssim 
	\jt^{-3p_0} \sup_{|n| \leq N+p} {\big\| \underline{\Gamma}^n F(t) \big\|}_{L^2_z L^2_x}
	\\
	& +  \jt \sup_{\ell \in \Z} \sup_{|n| \leq N_1 - 1 + p}
	{\big\| \underline{\Gamma}^n  P_\ell  F(t) \big\|}_{L^2_z L^\infty_x}.
\end{split}
\end{align}
The bound \eqref{prPsiPoiint} for the $L^2_x$-based components of the norm are implied by \eqref{expbounds3}.
The bounds for the $L^\infty_x$-based components are instead obtained
from the following $L^\infty_x$-based estimates at fixed dyadic frequency 
\begin{align}\label{prPsiPoi2infty}
\begin{split}
& {\Big\| \nabla_{x,z} P_\ell \int_{-\infty}^0 e^{-|z-s||\nabla|} \mathbf{1}_\pm(z-s) F(\cdot,s) \, ds \Big\|}_{L^2_zL^\infty_x}
\\
& +  {\Big\| |\nabla|^{1/2} P_\ell \int_{-\infty}^0 e^{-|z-s||\nabla|} \mathbf{1}_\pm(z-s) F(\cdot,s) \, ds \Big\|}_{L^\infty_z L^\infty_x}
	\lesssim {\| F \|}_{L^2_z L^\infty_x},
\end{split}
\end{align}
and then using the same argument that gives \eqref{expbounds3} 
by applying vector fields and using the commutation identities \eqref{Tpmcomm}.
See \ref{ssecPoisson} for the details. 

Applying \eqref{prPsiPoi1} and \eqref{prPsiPoiint} to \eqref{dirichletformulaPsi} we have
\begin{align}
\label{prPsi11}
{\| T \Psi \|}_{\mathcal{L}_0} \lesssim {\| \varphi \|}_{\mathcal{F}_0}
    & + \jt^{-3p_0} \sup_{|n| \leq N} \big( {\big\| \underline{\Gamma}^n \dg{E^a}(t) \big\|}_{L^2_z L^2_x}
	+  {\big\| \underline{\Gamma}^n \dg{E^b}(t) \big\|}_{L^2_z L^2_x} \big)
	\\
\label{prPsi12}
	& +  \jt \sup_{\ell \in \Z} \sup_{|n| \leq N_1 - 1}
    \big( {\big\| \underline{\Gamma}^n  P_\ell  \dg{E^a}(t) \big\|}_{L^2_z L^\infty_x}
    + {\big\| \underline{\Gamma}^n  P_\ell  \dg{E^b}(t) \big\|}_{L^2_z L^\infty_x} \big).
\end{align}
In view of the definition \eqref{prPsi0} and the assumption \eqref{LemmaPsivarphi}
we have ${\| \varphi \|}_{\mathcal{F}_0} \lesssim \e_0$.

From the definitions of the nonlinear terms in \eqref{EadefPsi},
distributing vector fields as usual, and using the assumptions on $h$ in \eqref{LemmaPsih},
we see that
\begin{align}\label{prPsi21}
\begin{split}
& \sup_{|n| \leq N} {\big\| \underline{\Gamma}^n \dg{E^a}(t) \big\|}_{L^2_z L^2_x}  
\\
& \lesssim \sup_{|n| \leq N} {\big\| \underline{\Gamma}^n \nabla_{x,z} \Psi \big\|}_{L^2_z L^2_x}
  \sup_{|n| \leq N/2} {\big\| \Gamma^n \nabla h \big\|}_{L^\infty}
  +  \sup_{|n| \leq N/2} {\big\| \underline{\Gamma}^n \nabla_{x,z} \Psi \big\|}_{L^2_z L^\infty_x}
  \sup_{|n| \leq N} {\big\| \Gamma^n \nabla h \big\|}_{L^2}
\\
& \lesssim \jt^{3p_0} {\| \Psi \|}_{\mathcal{L}_0} \cdot \e_0 \jt^{-1}
  + \jt^{-1+} {\| \Psi \|}_{\mathcal{L}_0} \cdot \e_0 \jt^{p_0}
  \lesssim \e_0 \jt^{-1+3p_0} {\| \Psi \|}_{\mathcal{L}_0};
\end{split}
\end{align}
note that we have used Bernstein's inequality to deduce the inequality for
the $L^2_zL^\infty_x$ norm of $\nabla_{x,z} \Psi$ as follows: for $|n|\leq N/2$
\begin{align*}
& {\big\| \underline{\Gamma}^n \nabla_{x,z} \Psi \big\|}_{L^2_z L^\infty_x}
	\lesssim \sum_\ell {\big\| \underline{\Gamma}^n \nabla_{x,z} P_\ell \Psi \big\|}_{L^2_z L^\infty_x}
	\\
	& \lesssim \log (2+t) \sup_\ell {\big\| \underline{\Gamma}^n \nabla_{x,z} P_\ell \Psi \big\|}_{L^2_z L^\infty_x}
	+ \sum_{2^\ell \geq \jt^{5}} {\big\| \underline{\Gamma}^n \nabla_{x,z} P_\ell \Psi \big\|}_{L^2_z L^\infty_x}
	+ \sum_{2^\ell \leq \jt^{-5}}  {\big\| \underline{\Gamma}^n \nabla_{x,z} \Psi \big\|}_{L^2_z L^\infty_x}
	\\
	& \lesssim \log (2+t) \jt^{-1} {\| \Psi \|}_{\mathcal{L}_0}
	+ \jt^{-5} \sup_{n \leq N/2+3} {\big\| \underline{\Gamma}^n \nabla_{x,z} \Psi \big\|}_{L^2_z L^2_x}
	+ \jt^{-5} \sup_{n \leq N/2} {\big\| \underline{\Gamma}^n \nabla_{x,z} \Psi \big\|}_{L^2_z L^2_x}
	\\
	& \lesssim \jt^{-1+} {\| \Psi \|}_{\mathcal{L}_0}
\end{align*}
A bound as in \eqref{prPsi21} also holds for $\dg{E^b}$, so that, in particular,
the nonlinear terms in \eqref{prPsi11} are bounded by $\e_0 {\| \Psi \|}_{\mathcal{L}_0}$.

We can use similar argument to estimate the $L^\infty_x$ components of the norm appearing in \eqref{prPsi12}:
for any $|n| \leq N_1 -1$ and $\ell \in \Z$
\begin{align}\label{prPsi22}
\begin{split}
    {\big\| \underline{\Gamma}^n  P_\ell  \dg{E^a}(t) \big\|}_{L^2_z L^\infty_x} 
	 & \lesssim \sup_{n \leq N_1 -1} {\big\| \underline{\Gamma}^n \nabla_{x,z} \Psi \big\|}_{L^2_z L^\infty_x}
	\sup_{n \leq N_1 - 1} {\big\| \Gamma^n \nabla h \big\|}_{L^\infty}
	\\
	& \lesssim \jt^{-1+} {\| \Psi \|}_{\mathcal{L}_0} \cdot \e_0 \jt^{-1}
	\lesssim \e_0 \jt^{-2+} {\| \Psi \|}_{\mathcal{L}_0},
\end{split}
\end{align}
having once again used Bernstein to deduce the bound on the $L^2_zL^\infty_x$ norm of $\nabla_{x,z} \Psi$
for very large and very small frequencies from the stronger $L^2_zL^2_x$ norm.
The same bound holds for $\dg{E^b}$.

We have thus obtained
${\| T \Psi \|}_{\mathcal{L}_0} \lesssim {\| \varphi \|}_{\mathcal{F}_0} + \e_0 {\| \Psi \|}_{\mathcal{L}_0}$,
and in the same way we can estimate differences and obtain
${\| T (\Psi_1-\Psi_2) \|}_{\mathcal{L}_0} \lesssim \e_0 {\| \Psi_1 - \Psi_2 \|}_{\mathcal{L}_0}.$
We therefore have a unique fixed point for the map $T$, hence a unique solution to the given elliptic problem
that satisfies ${\|\Psi \|}_{\mathcal{L}_0} \lesssim \e_0$; in view of the definition \eqref{prPsi0},
this gives the desired \eqref{PsiL2}-\eqref{Psiinfty}.

To conclude, we show how \eqref{PsiL2quad} and \eqref{Psiinftyquad} follow from the bounds just proven above.
Indeed, since 
\begin{align}\label{dirichletformulaPsi'}
\begin{split}
\Psi - e^{z|\nabla|} \varphi & =
\frac{1}{2} \int_{-\infty}^0 e^{-|z-s||\nabla|} ( \mathrm{sign}(z-s)\dg{E^a - E^b}  
  ) \, ds
  \\
  & - \frac{1}{2} \int_{-\infty}^0  e^{(z+s)|\nabla|} (\dg{E^a - E^b} 
  ) \, ds,
\end{split}
\end{align}
see \eqref{dirichletformulaPsi}, the bounds \eqref{expbounds3} 
together with the estimate \eqref{prPsi21} (and the analogous one with $\dg{E^b}$ instead of $\dg{E^a}$) imply \eqref{PsiL2quad},
while \eqref{prPsiPoi2infty} (more precisely, its version with vector fields) 
together with the estimate \eqref{prPsi22} (and the analogous one with $\dg{E^b}$ instead of $\dg{E^a}$) give \eqref{Psiinftyquad}.
\end{proof}

\begin{remark}\label{remPsiinfty}
The proof of \eqref{Psiinfty} shows that if we replace the $L^\infty$ bound in \eqref{LemmaPsivarphi} 
by a slightly stronger assumption with an $\ell^1$ sum over frequencies, that is,
\begin{align}\label{LemmaPsivarphi'}
\sum_{\ell\in\Z} \sum_{|r|+|k| \leq N_1} {\| \nabla^r \Gamma^{k} |\nabla|^{1/2} P_\ell \varphi \|}_{L^\infty} \lesssim \e_0 \jt^{-1},
\end{align}
then we can obtain the stronger conclusion
\begin{align}\label{Psiinfty'}
&  \sum_{\ell\in\Z} {\big\| \underline{\Gamma}^n \nabla_{x,z} P_\ell \Psi \big\|}_{L^2_z L^\infty_x} + 
 \sum_{\ell\in\Z} {\big\| \underline{\Gamma}^n |\nabla|^{1/2} P_\ell \Psi \big\|}_{L^\infty_z L^\infty_x} \lesssim \e_0 \jt^{-1}, \qquad n \leq N_1-1,
\end{align}
instead of \eqref{Psiinfty}. 
Indeed, it suffices to sum over the index $\ell$ in
the bounds \eqref{prPsiPoiinfy} and use the stronger assumption \eqref{LemmaPsivarphi'},
and sum over $\ell$ in the inhomogeneous bounds \eqref{prPsiPoi2infty}
and estimate the sum over $\ell$ of the right-hand side of \eqref{prPsi21}
using the stronger $L^2_x$ bounds for very large and very small frequencies.
The estimate \eqref{LemmaPsivarphi'} is obtained in Section \ref{secdecay}, 
see Remark \ref{Remdecay}.
The estimate \eqref{Psiinfty'} is used to deduce decay for the irrotational component
of the velocity in the interior; see Lemma \ref{lemdecayirrot}. 
\end{remark}

\begin{proof}[Proof of Proposition \ref{propDN}]
In Lemma \ref{LemmaPsi} we gave bounds on $\Psi$ following from the assumptions \eqref{aprioriapp}.
Since
\begin{align}
G(h)\varphi = (1+|\nabla h|^2)\partial_z \Psi|_{z=0} - \nabla h \cdot \nabla_x \Psi|_{z=0},
\end{align}
the bounds in \eqref{PsiL2} imply the first bound in \eqref{G1est} provided $N$ is chosen large enough,
and the second bound in \eqref{Psiinfty} implies the second bound in \eqref{G1est},
where the small $\jt^{0+}$ loss is coming from estimating the $\ell^1$ sum over dyadic indexes by
the $\ell^\infty$ norm for an $O(\log \jt)$ set of frequencies $2^\ell \in [\jt^{-5}, \jt^{5}]$, and using
the bound on the $L^2_x$-norm in \eqref{PsiL2} for the remaining very small and very high frequencies.

To obtain the quadratic bounds \eqref{G2est} it suffices to observe that
\begin{align*}
G_{\geq 2}(h)\varphi & =  G(h)\varphi - |\nabla| \varphi
\\
& =
\partial_z \big(\Psi - e^{z|\nabla|}\varphi\big) |_{z=0}
  + |\nabla h|^2 \partial_z \Psi|_{z=0} - \nabla h \cdot \nabla_x \Psi|_{z=0}, 
\end{align*}
and use \eqref{PsiL2quad} and \eqref{Psiinftyquad} in addition to \eqref{PsiL2}-\eqref{Psiinfty}. 

The last estimate \eqref{G3est} can be obtained from similar arguments
and the fixed point formulation in the proof of Lemma \ref{LemmaPsi};
one needs to expand to one more order in the Taylor series for $G(h)\varphi$,
and use \eqref{PsiL2}-\eqref{Psiinftyquad}, 
along the lines of the arguments in \cite{DIPP}.
\end{proof}


\bigskip
\section{The elliptic system for the vector potential}\label{Appalpha}

This appendix contains the details of the proofs of the supporting results in Section \ref{secVP}.

\subsection{The elliptic equation}
We first give the proof of Lemma \ref{flatlem} that derives the elliptic system in the half-space.

\begin{proof}[Proof of Lemma \ref{flatlem}]
This is a somewhat standard calculation but we include some details for the convenience of the reader.
The main point is to translate the boundary conditions \eqref{betapi}-\eqref{betan}
to the flat domain and obtain boundary conditions for $\alpha$.


\noindent
{\it The Poisson equation \eqref{eqflat0}.}
The first step is to express the Laplacian $\Delta_{x,y}$ in terms of the new coordinates.
For this we compute the inverse metric
$g^{-1}$ in this coordinate system, whose components $g^{ab}$ are given by
\begin{equation}
 g^{ab} = \delta^{ij} \pa_i X^a \pa_j X^b,
 \label{}
\end{equation}
with $X^1(x,y) = x^1, X^2(x,y) = x^2, X^3(x,y) = y-h(x)$. We compute
\begin{align*}
 & g^{11} = g^{22} = 1, \qquad
 g^{33} = 1 + |\nabla h|^2, \\
 & g^{i3} = g^{3i} = -\pa_i h, \quad i = 1,2,
\end{align*}
and the remaining entries vanish.
Since $\det g = 1$, the Laplacian in these coordinates takes the form
\begin{equation*}
\Delta = g^{ab}\pa_a\pa_b + \pa_a(g^{ab})\pa_b.
\end{equation*}
We have
\begin{align*}
& g^{ab}\pa_a\pa_b = \pa_1^2 + \pa_2^2 + (1+|\nabla h|^2)\pa_z^2,
- 2\pa_1h \pa_1\pa_z - 2\pa_2 h \pa_2 \dg{\pa_z}
\\
& \pa_a(g^{ab})\pa_b = -\pa_1^2h\pa_z - \pa_2^2 h \pa_z,
\end{align*}
and adding these together we find
\begin{equation*}
\Delta q = (\pa_1^2 + \pa_2^2 + \pa_z^2)q
 + \pa_z ( |\nabla h|^2 \pa_z q) -\pa_1 (\pa_1h \pa_z q )
 - \pa_2 ( \pa_2 h \pa_z q) -\pa_z \left( \nabla h \cdot \nabla q\right).
\end{equation*}
In terms of $\alpha(X,z) = \beta(X, z+h(X))$ and $W(X,z) = \omega(X, z+h(X))$,
the Poisson
equation \eqref{betaell} reads
\begin{equation}\label{alphaeq}
    \pa_z^2\alpha + (\pa_1^2 + \pa_2^2)\alpha = |\nabla| \dg{E^a} + \pa_z \dg{E^b},
\end{equation}
where, writing $\nabla = \nabla_X$,
\begin{align}
\label{Eadef}
\dg{E^a}(\pa \alpha) & = \frac{\nabla}{|\nabla|}\cdot
 \left( \nabla h \, \pa_z \alpha \right) + \frac{1}{|\nabla|} W,
\\
\label{Ebdef}
\dg{E^b}(\pa \alpha) & = -|\nabla h|^2 \pa_z \alpha + \nabla h \cdot \nabla \alpha.
\end{align}

\noindent
{\it The boundary conditions \eqref{eqflat1}-\eqref{eqflat3}.}
We now write the boundary conditions \eqref{betapi}-\eqref{betan} explicitly.
The normal vector to the boundary is
\begin{equation}
 n = (1 + |\nabla h|^2)^{-1/2} (\pa_y - \nabla h\cdot \nabla),
 \label{normal}
\end{equation}
which is defined for all $(x,y)$. Recalling that $\Pi_i^j = \delta_i^j - n_i n^j$,
we compute
\begin{align}\label{Pis}
\begin{split}
\Pi_1^1 &= 1 - n_1n^1 = 1 - (1 + |\nabla h|^2)^{-1}(\pa_1h)^2,
  \\
 \Pi_1^2 &= -n_1n^2 = -(1+|\nabla h|^2)^{-1} \pa_1h\pa_2 h,\\
 \Pi_1^3 &= -n_1 n^3 = (1+|\nabla h|^2)^{-1} \pa_1 h,\\
 \Pi_2^2 &= 1 - n_2n^2 = 1 - (1 + |\nabla h|^2)^{-1} (\pa_2h)^2,\\
 \Pi_2^3 &= -n_2 n^3 = (1+|\nabla h|^2)^{-1} \pa_2 h,\\
 \Pi_3^3 &= 1 - n_3n^3 = 1 - (1 + |\nabla h|^2)^{-1},
\end{split}
\end{align}
which determine the remaining components since $\Pi$ is symmetric.
The boundary conditions \eqref{betapi} then give us, for $i=1,2$,
\begin{align}
\beta_i - (1+|\nabla h|^2)^{-1}\partial_i h \big( \partial_1 h \, \beta_1 + \partial_2 h \beta_2 - \beta_3) = 0
\end{align}
and, therefore, in terms of $\alpha$ they read
\begin{align}
\alpha_i & = B_i,\qquad
 B_i := - (1+ |\nabla h|^2)^{-1} \pa_i h(\alpha_3 - \nabla h\cdot \alpha)|_{z = 0},
 \quad i =1,2
\end{align}
as claimed in \eqref{eqflat1}-\eqref{eqflat2}.

We now write \eqref{betan} explicitly.
We start from \eqref{divS} which we rewrite as
\begin{equation}\label{divS'}
\pa_n \beta_n + (\Pi_i^j\pa_j n^i) \beta_n = 0,
\end{equation}
where we recall that $\beta_n = n \cdot \beta = (1+|\nabla h|^2)^{-1/2})(\beta_3 - \nabla h \cdot (\beta_1,\beta_2))$.
We then pass to the new coordinates using the expression \eqref{normal} for the normal vector,
and calculate the first term in \eqref{divS'}:
\begin{align}
\label{divS'1}
\begin{split}
& \pa_n \beta_n =
 \\
 & = (1+ |\nabla h|^2)^{-1/2}
 \left( (1+ |\nabla h|^2) \pa_z - \nabla h \cdot \nabla \right)
 \left( (1+ |\nabla h|^2)^{-1/2} (\alpha_3 -\nabla h \cdot (\alpha_1,\alpha_2))
 \right)\\
 & = \partial_z \alpha_3 - \nabla h\cdot \partial_z (\alpha_1,\alpha_2)
 - (1+|\nabla h|^2)^{-1} \nabla h\cdot \nabla (\alpha_3 - \nabla h\cdot (\alpha_1,\alpha_2))
 \\
 & - \big( (1+|\nabla h|^2)^{-1/2} \nabla h \cdot \nabla (1+|\nabla h|^2)^{-1/2} \big)
 (\alpha_3 -\nabla h\cdot (\alpha_1,\alpha_2)), 
\end{split}
\end{align}
where the expressions above are evaluated at $z=0$.
We then write out explicitly the curvature terms $\Pi_i^j\pa_in^j$;
we first record that
\begin{align*}
\pa_i n^j &= - \pa_i \left( (1+|\nabla h|^2)^{-1/2} \pa_j h\right), \quad i,j=1,2
\\
\pa_i n^3 &= \pa_i \left( (1+|\nabla h|^2)^{-1/2}\right), \quad i=1,2,
\end{align*}
and then, using the expressions \eqref{Pis} for the projection $\Pi$, we find
\begin{align}\label{divS'2}
\begin{split}
\Pi_i^j\pa_in^j
 & = 
 -\nabla \cdot \big( (1 + |\nabla h|^2)^{-1/2} \nabla h \big).
\end{split}
\end{align}
In view of \eqref{divS'1} and \eqref{divS'2}, the boundary condition \eqref{divS'} 
then reads
\begin{equation}\label{dzaplha3}
\pa_z \alpha_3 = B_3,
\end{equation}
where
\begin{align}\label{dnformulaB3app}
\begin{split}
B_3 & = \nabla h \cdot \partial_z (\alpha_1,\alpha_2)
 + (1+|\nabla h|^2)^{-1} \nabla h \cdot \nabla (\alpha_3 - \nabla h\cdot (\alpha_1,\alpha_2))
 \\
 & + \big( (1+|\nabla h|^2)^{-1/2} \nabla h \cdot \nabla (1+|\nabla h|^2)^{-1/2} \big)
 (\alpha_3 -\nabla h\cdot (\alpha_1,\alpha_2))
 \\
 & + \big[ \nabla \cdot \big( (1 + |\nabla h|^2)^{-1/2} \nabla h \big) \big]
 (1+ |\nabla h|^2)^{-1/2} \big( \alpha_3 -\nabla h \cdot (\alpha_1,\alpha_2) \big)
 \\
 & = \nabla h \cdot \partial_z (\alpha_1,\alpha_2)
 + (1+|\nabla h|^2)^{-1} \nabla h \cdot \nabla (\alpha_3 - \nabla h\cdot (\alpha_1,\alpha_2))
 \\
 & + A(\nabla h, \nabla^2 h)(\alpha_3 -\nabla h\cdot (\alpha_1,\alpha_2))
\end{split}
\end{align}
with
\begin{align}\label{dnformulaA3app}
\begin{split}
A(\nabla h, \nabla^2 h) := 
  \nabla \cdot \big( (1 + |\nabla h|^2)^{-1} \nabla h \big) .
\end{split}
\end{align}
This concludes the proof of Lemma \ref{flatlem}.
\end{proof}


Next, we give the formulas for the solution of Poisson's equation in the half-space
that are used to obtain the fixed point formulation of Lemma \ref{lemmafp}.

\begin{lemma}\label{lemmaAfp}
Let $u : \R^2_x \times \{z \leq 0\}$ be the solution of
\begin{align}\label{eqflatA}
    (\partial_z^2 + \Delta_x) u = \partial_z \dg{E^a} + |\nabla|
    \dg{E^b} + F, \quad \text{ in} \quad \R^2_x \times \{z<0\}.
\end{align}
Then:
\begin{itemize}

\item[(i)] If we assign Dirichlet boundary conditions
\begin{align}\label{bcA}
u(x,0) & = B(x)
\end{align}
$u$ is formally given by
\begin{align}\label{dirichletformula}
\begin{split}
u(x,z) = e^{z|\nabla|} B(x)
 & + \frac{1}{2} \int_{-\infty}^0 e^{-|z-s||\nabla|} \Big( \mathrm{sign}(z-s)\dg{E^a - E^b}  - \frac{1}{|\nabla|}F \Big) \, ds
 \\
 & - \frac{1}{2} \int_{-\infty}^0  e^{(z+s)|\nabla|} \Big(\dg{E^a - E^b} - \frac{1}{|\nabla|} F \Big) \, ds.
\end{split}
\end{align}

\item[(ii)] If we assign Neumann boundary conditions
\begin{align}\label{bcA'}
\partial_z u(x,0) & = B'(x)
\end{align}
$u$ is formally given by
\begin{align}\label{neumformula}
\begin{split}
u(x,z) & = \frac{1}{|\nabla|}  e^{z|\nabla|} B'(x)
- \frac{1}{|\nabla|} e^{z|\nabla|} \dg{E^a}(z=0)
 \\
       & - \frac{1}{2} \int_{-\infty}^0 e^{(z+s)|\nabla|} \big(- \dg{E^a + E^b} + \frac{1}{|\nabla|} F  \big) \,ds
 \\
 & + \frac{1}{2} \int_{-\infty}^0 e^{-|z-s||\nabla|}
 \big( \mathrm{sign}(z-s)\dg{E^a  - E^b} - \frac{1}{|\nabla|} F  \big)\, ds.
\end{split}
\end{align}
\end{itemize}

\end{lemma}

\begin{proof}
Taking the Fourier transform in $x$
one obtains the general solution of $(\partial_z^2 + \Delta_x) u = f$ in the lower half-space,
that decays to zero as $z\rightarrow -\infty$, in the form
\begin{align}\label{general}
\what{u}(\xi,z) = c_1 e^{z|\xi|} + \int_{-\infty}^z \frac{1}{2|\xi|}
 \big( e^{(z-s)|\xi|} - e^{(s-z)|\xi|} \big) \what{f}(s,\xi) \, ds.
\end{align}
Imposing the boundary condition \eqref{bcA} gives
\begin{equation*}
\what{u}(z, \xi)
 = \what{u}(0, \xi) e^{z|\xi|}
 +\int_{-\infty}^0 \frac{1}{2|\xi|} e^{(z+s)|\xi|}\what{f}\,ds
 -\int_{-\infty}^0 \frac{1}{2|\xi|} e^{-|z-s||\xi|} \what{f}\, ds.
\end{equation*}
When $f$ is the right-hand of \eqref{eqflatA}, an integration by parts in $s$ on the terms $\partial_s\dg{E^a}$
gives \eqref{dirichletformula}.

When instead we impose Neumann boundary conditions \eqref{bcA'}, from \eqref{general} we compute
\begin{align*}
\what{B'}(\xi) = c_1 |\xi| + \int_{-\infty}^0 \frac{1}{2}
 \big( e^{-s|\xi|} + e^{s|\xi|} \big) \hat{f}(s,\xi) \, ds,
\end{align*}
and therefore
\begin{align*}
\what{u}(\xi,z) & = \frac{1}{|\xi|} e^{z|\xi|} \Big(\what{B'}(\xi)
   - \int_{-\infty}^0 \frac{1}{2}
   \big( e^{-s|\xi|} + e^{s|\xi|} \big) \hat{f}(s,\xi) \, ds \Big)
\\
& + \int_{-\infty}^z \frac{1}{2|\xi|}
 \big( e^{(z-s)|\xi|} - e^{(s-z)|\xi|} \big) \hat{f}(s,\xi) \, ds
 \\
& = \frac{1}{|\xi|} e^{z|\xi|} \what{B'}(\xi)
  - \int_{-\infty}^0 \frac{1}{2|\xi|} e^{(z+s)|\xi|} \hat{f}(s,\xi) \, ds
  - \int_{-\infty}^0 \frac{1}{2|\xi|} e^{-|z-s||\xi|} \hat{f}(s,\xi) \, ds.
\end{align*}
Plugging-in for $f$ the right-hand side of \eqref{eqflatA} and integrating by parts on the $\partial_s\dg{E^a}$ term
gives \eqref{neumformula}.
\end{proof}

\subsection{Bounds for the Poisson kernel in weighted spaces}\label{ssecPoisson}
In this subsection we give the proof of Lemma \ref{explem}.

\begin{proof}[Proof of Lemma \ref{explem}]
We prove the estimates \eqref{expbounds1}-\eqref{expbounds3} as well as $L^\infty_x$
based estimates that are used in the proof of Lemma \ref{LemmaPsi}.
 
{\it Proof of  \eqref{expbounds1} and \eqref{expbounds2}}.
We being by recalling the following standard bounds for the Poisson Kernel,
\begin{align}
\label{expbounds01}
{\big\| e^{z|\nabla|} f \big\|}_{L^\infty_zL^p_x} & \lesssim {\| f \|}_{L^p_x},\qquad 1 < p <\infty,
\\
\label{expbounds02}
{\big\| |\nabla|^{1/2} e^{z|\nabla|} f \big\|}_{L^2_zL^2_x} & \lesssim {\| f\|}_{L^2_x}.
\end{align}
The estimates \eqref{expbounds01}-\eqref{expbounds02} give \eqref{expbounds1} and \eqref{expbounds2} with $r=0=k$.
The bounds with $k=0$ and any $r$ follow immediately.
To prove the bound with vector fields we first compute the commutator of $\underline{\Gamma} = \Omega$ or
$\underline{S} = z\partial_z + x \cdot \nabla + (1/2)t\partial_t$ with the Poisson kernel:
\begin{align}\label{comm0}
\begin{split}
[\underline{\Gamma}, e^{z|\nabla|}] &= c_\Gamma e^{z|\nabla|},
 \qquad c_{\underline{S}} = -2, \quad c_{\Omega} = 0,
 \\
[\underline{\Gamma}, |\nabla|^\ell] &= d_{\Gamma,\ell} |\nabla|^\ell,
 \qquad d_{\underline{S},\ell} = \ell, \quad d_{\Omega,\ell} = 0,
\end{split}
\end{align}
which are easy to see, for example taking the Fourier transform in $x$.

Then, using \eqref{comm0} and \eqref{expbounds01},
and recalling the definition of the $Z^{r,p}_k$ spaces, we have
\begin{align*}
{\big\| \underline{\Gamma}^k e^{z|\nabla|} f \big\|}_{L^\infty_z W^{r,p}} &
  \lesssim \sum_{k' \leq k} {\big\| e^{z|\nabla|} \underline{\Gamma}^{k'} f \big\|}_{L^\infty_z W^{r,p}}
  \\
  & \lesssim \sum_{k' \leq k} {\big\| \Gamma^{k'} f \big\|}_{W^{r,p}} = {\| f \|}_{Z^{r,p}_k}.
\end{align*}
This is \eqref{expbounds1}.
The second estimate \eqref{expbounds2} can be obtained in the same way.

\noindent
{\it Proof of  \eqref{expbounds3}}.
We adopt the short-hand
\begin{align*}
T_{\pm} f (x,z) := \int_{-\infty}^0 e^{-|z-s||\nabla|}  \mathbf{1}_\pm(s-z) f(x,s) \, ds.
\end{align*}
Consider first the case without vector fields, $k=0$.
In what follows, we let $p \in [1,2]$ and $q\in[2,\infty]$.
Taking the Fourier transform in $x$ 
we have
\begin{align}\label{explempr1}
\mathcal{F}\big( T_{\pm} f (x,z) \big)
  = e^{-|\cdot||\xi|}  \mathbf{1}_\mp(\cdot) \ast_z \widehat{f}(\xi,\cdot) \mathbf{1}_-(\cdot).
\end{align}
Using the Littlewood-Paley projectors $P_l$, $l\in\Z$ (see \eqref{LP0}),
and orthogonality we see that
\begin{align*}
{\| T_{\pm} f(x,z) \|}_{L^q_z L^2_x} \approx
  {\big\| {\| \mathcal{F}\big( T_{\pm} P_l f (\cdot,z) \big) \|}_{\ell^2(\Z)L^2_\xi(\R^2) } \big\|}_{L^q_z}.
\end{align*}
Applying Minkowski's inequality, followed by \eqref{explempr1} and Youngs's inequality with $1+1/q=1/p+1/\rho$,
gives
\begin{align*}
\begin{split}
{\| T_{\pm} f(x,z) \|}_{L^q_z L^2_x} & \lesssim
  {\big\| {\| \varphi_{[l-2,l+2]}(\xi)e^{-|\cdot||\xi|}  \mathbf{1}_\mp(\cdot)
  \ast_z \varphi_l(\xi)\widehat{f}(\xi,\cdot) \|}_{L^q_z} \big\|}_{\ell^2(\Z)L^2_\xi(\R^2) }
  \\
& \lesssim
  {\big\| {\| \varphi_{[l-2,l+2]}(\xi)e^{-|\cdot||\xi|} \|}_{L^\rho_z}
  \, {\| \varphi_l(\xi)\widehat{f}(\xi,\cdot) \|}_{L^p_z} \big\|}_{\ell^2(\Z)L^2_\xi(\R^2)}
\\
& \lesssim {\big\| {\|2^{-l/\rho} \varphi_l(\xi)\widehat{f}(\xi,\cdot) \|}_{L^p_z}
  \big\|}_{\ell^2(\Z)L^2_\xi(\R^2)}.
\end{split}
\end{align*}
Applying again Minkowski and using orthogonality we get
\begin{align}\label{explemprqp}
{\big\| |\nabla|^{(1+1/q-1/p)} T_{\pm} f(x,z) \big\|}_{L^q_z L^2_x}
  & \lesssim {\| f \|}_{L^p_z L^2_x}.
\end{align}
Using $(q,p) = (\infty,2)$ and $(2,2)$ we obtain the bounds
\begin{align}\label{explempr2}
{\big\| |\nabla|^{1/2} T_{\pm} f(x,z) \big\|}_{L^\infty_z L^2_x}
  + {\big\| |\nabla| T_{\pm} f(x,z) \big\|}_{L^2_z L^2_x}
  & \lesssim {\| f \|}_{L^2_z L^2_x}.
\end{align}
Similar estimates hold if we replace the kernel $e^{-|z-s||\nabla|}$ with $e^{(z+s)|\nabla|}$; 
see also Remark \ref{remexp}.
Using instead \eqref{explemprqp} with $(q,p) = (\infty,6/5)$ and $(2,6/5)$ we obtain the bounds
\begin{align}\label{explempr2'}
{\big\| |\nabla|^{1/2} T_{\pm} f(x,z) \big\|}_{L^\infty_z L^2_x}
  + {\big\| |\nabla| T_{\pm} f(x,z) \big\|}_{L^2_z L^2_x}
  & \lesssim {\| |\nabla|^{1/3}f \|}_{L^{6/5}_z L^2_x}.
\end{align}
The bounds \eqref{explempr2} and \eqref{explempr2'} give us \eqref{expbounds3} with $k=0$ and any $r$.

To obtain the estimates with vector fields it suffices to use the following
identities:
\begin{align}\label{Tpmcomm}
\underline{\Gamma} T_\pm & = T_\pm \underline{\Gamma} f + c^\pm_{\underline{\Gamma}} T_\pm f,
\end{align}
where $c^\pm_{\underline{S}} = -1$ and $c^{\pm}_\Omega = 0$.
The identity for $\underline{\Gamma} = \Omega$ is obvious.
To obtain the one for $\underline{\Gamma} = \underline{S}$, recall \eqref{3dvf},
it suffices to show the same identity just for the spatial part $\underline{\Sigma} := z\partial_z + x \cdot \nabla_x$.
Observe that for any $\tau$
\begin{equation*}
[x\cdot \nabla_x , e^{\tau |\nabla|}] = -(2 + \tau |\nabla|) e^{\tau|\nabla|}
\end{equation*}
and, therefore,
\begin{equation}\label{horizontalscaling'}
[z\partial_z + x\cdot \nabla_x , e^{(z-s) |\nabla|}] = (s|\nabla| -2)e^{(z-s)|\nabla|}
  = (-s\partial_s -2)e^{(z-s)|\nabla|}.
\end{equation}
It follows that
\begin{align*}
\begin{split}
\underline{\Sigma} T_+ f (x,z) & = \underline{\Sigma} \int_{z}^0 e^{(z-s)|\nabla|} f(x,s) \, ds,
   \\
   & = \int_{z}^0 e^{z|\nabla|} (-s\partial_s e^{-s|\nabla|}) f (x,s) \, ds
   + \int_{z}^0 e^{(z-s)|\nabla|} (x\cdot \nabla - 2) f (x,s) \, ds
   - z f(x,z).
\end{split}
\end{align*}
Integrating by parts in $s$ in the first integral above, we see that all the boundary terms cancel out 
and we obtain
\begin{align*}
\begin{split}
\underline{\Sigma} T_+ f (x,z)
   & = \int_{z}^0 e^{(z-s)|\nabla|} \partial_s \big(s f (x,s) \big)\, ds
   + \int_{z}^0 e^{(z-s)|\nabla|} (x\cdot \nabla - 2) f (x,s) \, ds
   \\
   & = T_+ \big( (\underline{\Sigma} - 1) f \big) (x,z),
\end{split}
\end{align*}
which implies \eqref{Tpmcomm} for $T_+$. 
For the operator $T_-$ we use the same argument:
from \eqref{horizontalscaling'}
\begin{align*}
\begin{split}
\underline{\Sigma} T_- f (x,z) & = \underline{\Sigma} \int_{-\infty}^z e^{(s-z)|\nabla|} f(x,s) \, ds,
   \\
   & = \int_{-\infty}^z e^{-z|\nabla|} (-s\partial_s e^{s|\nabla|}) f (x,s) \, ds
   + \int_{-\infty}^z e^{(s-z)|\nabla|} (x\cdot \nabla - 2) f (x,s) \, ds
   + z f(x,z)
   \\
   & = T_- \big( (\underline{\Sigma} -1) f \big)(x,z)
\end{split}
\end{align*}
having used again integration by parts in $s$ in the last step.

To conclude we 
use the above commutation identities \eqref{Tpmcomm} and \eqref{explempr2}-\eqref{explempr2'} 
to obtain
\begin{align*}
\begin{split}
& \sum_{|k'|\leq k} {\Big\| \underline{\Gamma}^{k'}
  |\nabla|^{1/2} \int_{-\infty}^0 e^{-|z-s||\nabla|}  \mathbf{1}_\pm(s-z) f(x,s) \, ds \Big\|}_{L^\infty_z H^r}
  \\
  & \lesssim \sum_{|k'|\leq k} {\Big\| |\nabla|^{1/2} \int_{-\infty}^0 e^{-|z-s||\nabla|}
  \mathbf{1}_\pm(s-z) \, \underline{\Gamma}^{k'} f(x,s) \, ds \Big\|}_{L^\infty_z H^r}
  \\
  & \lesssim
  \sum_{|k'|\leq k} \min \big( {\| \Gamma^{k'} f \|}_{L^2_z H^r},
  {\big\| |\nabla|^{1/3} \Gamma^{k'} f \big\|}_{L^{6/5}_z H^r} \big)
\end{split}
\end{align*}
This gives us the desired bound on the first term on the left-hand side of \eqref{expbounds3}.
The same argument can be applied to the second term on the left-hand side of \eqref{expbounds3}
using the $L^2_z$ bounds in \eqref{explempr2} and \eqref{explempr2'}.
This concludes the proof of Lemma \ref{explem}.
\end{proof}


%
%


\bigskip
\section{Energy estimates: proof of Proposition \ref{propEv}}\label{secEv}
\fp{Our goal in this appendix is to show how to obtain energy estimate involving} 
vector fields for solutions of \eqref{freebdy}, 
as opposed to estimates that only involve derivatives $\nabla_{x, y}$, as those that 
can be found in \cite{CL,CS2,ShZ2} for example; \fp{see also \cite{GMSC} where weighted energy estimates
are obtained for the irrotational problem with surface tension.}

To prove the main energy estimate \eqref{propEvEE} one must exploit the invariances and structure of the equations.
This structure is more transparent in the irrotational problem, 
once it has been properly rewritten on the boundary,
because it is straightforward to commute the linearized operator $\pa_t + i \Lambda^{1/2}$ with the fields $(S, \Omega)$. 
For the rotational problem, we must instead commute suitable vector fields with the full system \eqref{freebdy}. 
\fp{In what follows we will give some details on how to do this 
and, in particular, on how to derive the higher-order system \eqref{freebdyvfs},
which is the main step for the proof of \eqref{propEvEE}.}
We can then verify that the commuted system \eqref{freebdyvfs} has essentially the same structure as the 
original problem up to acceptable lower order error terms. 
\fp{This naturally leads to the definition
of the weighted energy functionals \eqref{En1n2} 
which control $n_1$ scaling fields and $n_2$ rotation fields
applied to the velocity $v$ and the height $h$, see \eqref{vn}.
These are in turn related to the functionals $\mathcal{E}_{r,k}$ appearing in the statement of Proposition \ref{propEv}.}

\medskip
\fp{{\it Set-up and the higher-order system}.}
To obtain \eqref{freebdyvfs} we commute
\eqref{freebdy} with the fields $\uS$ and $\Omega$ in \eqref{3dvf}.
More precisely, we apply the scaling field $\uS$ componentwise but 
use Lie derivatives $\mathcal{L}_{\uO}$ ($\mathcal{L}_{\uO} X =[\uO, X]$
for vector fields $X$ and $\mathcal{L}_{\uO} q = \uO q$ for functions $q$)
with respect to the rotation fields. 
This is because these operators preserve the divergence-free condition
\begin{align}\label{divSLie}
\div \uS^{n_1} \mathcal{L}_{\uO}^{n_2} v = 0. 
\end{align}
Moreover, the Lie derivatives $\mathcal{L}_{\uO}$ commute with gradients,\dg{
\begin{equation}
  \label{liecomm}
  \mathcal{L}_{\uO}^n \nabla_{x, y} q = \nabla_{x, y} \uO^n q,
\end{equation}
while the scaling field nearly commutes with the gradient, in the sense that
    \begin{equation}
\uS^n \nabla_{x, y} q = \nabla_{x, y} (\uS - 1)^n q.\end{equation}}
As a result, we have the \dg{following identity, which we will use to commute gradients
with our operators,}
\begin{equation}
  \label{SLie}
  \uS^{n_1} \mathcal{L}_{\uO}^{n_2} \nabla_{x, y} q = 
  \nabla_{x, y} (\uS - 1)^{n_1} \mathcal{L}_{\uO}^{n_2} q.
\end{equation}
In light of the above, and the fact that $\uS^{n} \pa_t  = \pa_t (\uS - \tfrac{1}{2})^n$, 
it is natural to work with \dg{the commuted quantities}
\begin{equation}
  \label{vn}
  v^{n_1, n_2} := \uS_{1/2}^{n_1} \mathcal{L}_{\uO}^{n_2} v,
    \qquad
    h^{n_1, n_2} := S_1^{n_1} \Omega^{n_2} h,
    \qquad
    P^{n_1, n_2} := \uS_1^{n_1}\uO^{n_2} (p + y)
    = \uS_1^{n_1}\uO^{n_2}p ,
\end{equation}
where we are abbreviating $\uS_a := \uS - a$, and similarly with $S_a$,
and where the last identity in \eqref{vn} holds if $n_1 + n_2 \geq 1$. 
Notice that Sobolev norms of $(v^{n_1, n_2}, h^{n_1, n_2})$ 
are equivalent to norms of 
$(\uS^{n_1} \Omega^{n_2} v, S^{n_1} \Omega^{n_2}  h)$, for $n_1 + n_2 \leq k$ with a fixed $k$.

Our main claim is that the above variables satisfy the system
\begin{subequations}\label{freebdyvfs}
\begin{alignat}{2}
  (\pa_t + v\cdot \nabla) v^{n_1, n_2} + \nabla_{x,y} P^{n_1, n_2} &= F_{n_1, n_2},
                                                        &&\quad \text{ in } \mathcal{D}_t,
                                                        \label{vn1n2eq}
  \\
  \div v^{n_1, n_2} &= 0, &&\quad \text{ in } \mathcal{D}_t,
  \label{divvn1n2eq}
  \\
  P^{n_1, n_2} &= (-\pa_y p) h^{n_1, n_2} + G_{n_1, n_2},
               &&\quad \text{ on } \pa \mathcal{D}_t, \label{pn1n2eqn}\\
  (\pa_t + v\cdot \nabla) h^{n_1, n_2} &= (v^{n_1, n_2})\cdot (1, -\nabla h) + H_{n_1, n_2},
                                       &&\quad \text{ on } \pa \mathcal{D}_t,
                                       \label{dtheqn}
\end{alignat}
\end{subequations}
where the terms $F_{n_1, n_2}, G_{n_1, n_2}, H_{n_1, n_2}$ consist of nonlinear
acceptable error terms, \fp{in the sense that they 
involve less (or equal) than $n_1$ scaling or $n_2$ rotation vector fields;}
\dg{in other words, these will satisfy estimates of the form
\begin{align}\label{errorbound}
\begin{split}
\|F_{n_1, n_2}\|_{L^2(\mathcal{D}_t)}
  + \|G_{n_1, n_2}\|_{L^2(\pa \mathcal{D}_t)}
  + \|H_{n_1, n_2}\|_{L^2(\pa \mathcal{D}_t)}
  \\ \lesssim Z_0(t) \sum_{r+k \leq n_1+n_2}\left( 
  {\| v(t) \|}_{X^r_k(\D_t)} 
  + {\| h(t) \|}_{Z^r_k(\R^2)}\right),
\end{split} 
\end{align}
where $Z_0$ is defined as in \eqref{propEvdec};
notice that the norms on the right-hand side of \eqref{errorbound} 
are included in the right-hand side of \eqref{propEveq}.}

\medskip
\fp{{\it Proof of \eqref{freebdyvfs}-\eqref{errorbound}}}
The equations \eqref{freebdyvfs} with the bounds \eqref{errorbound} follow
after applying $\uS^{n_1} \mathcal{L}_{\uO}^{n_2}$
to the original system \eqref{freebdyeul}-\eqref{freebdyinc} using the identities \eqref{SLie}, \eqref{divSLie}
and distributing vector fields, as we now show.


First, to derive the boundary conditions \eqref{pn1n2eqn}-\eqref{dtheqn}
we split the fields $\uS, \uO$ into tangential (to the boundary)
and transverse components, by defining
\begin{equation}
    \label{tangdecomposition}
    \uS_T := \uS + \uS(h - y)\pa_y, 
    \qquad
    \uO_T := \uO + \uO(h-y) \pa_y,
\end{equation}
which are tangent to the boundary since they annihilate the boundary-defining
function $y - h$. \dg{We also note that by definition $\uS_T h = \uS h = Sh$
and $\uO_T h = \Omega h$.}
\dg{To get \eqref{pn1n2eqn}, we use that $\uS (h-y) = (S - 1) h$ and $p = 0$
at the boundary, and we find }
\begin{align}\label{Pn1n2}
\begin{split}
P^{n_1, n_2} = (\uS_T-1)^{n_1}\uO_T^{n_2} p -(S-1)^{n_1} \Omega^{n_2} h \pa_y p
  + G_{n_1, n_2}
  \\
  = ( - \pa_y p) h^{n_1, n_2} + G_{n_1, n_2}, 
\end{split}
\end{align}
where $G_{n_1, n_2}$ collects nonlinear error terms
\dg{generated by using the expressions in \eqref{tangdecomposition} to express
vector fields in terms of tangential vector fields. These terms have
strictly fewer vector fields falling on $h$ and $p$ than in the other quantities
in the above expression, and can be bounded as in \eqref{errorbound}.}

\dg{We now show the validity of \eqref{dtheqn} 
which is the higher order version of $\partial_t h = v \cdot N$ with $N := (-\nabla h ,1)$. 
We start by re-writing
the quantity $v^{n_1, n_2}$ appearing on the right-hand side of \eqref{dtheqn} as
\begin{align}\label{dtheqnpr1}
v^{n_1,n_2} = (\uS_T - \tfrac{1}{2})^{n_1} \mathcal{L}_{\uO_T}^{n_2} v + H_{n_1,n_2}^1
\end{align}
where $H_{n_1,n_2}^1$ are acceptable nonlinear error terms with fewer vector fields, which can be bounded
as in \eqref{errorbound}. Then (slightly abusing notation)}
\begin{equation}
  \label{vcdotN}
  \big( (\uS_T - \tfrac{1}{2})^{n_1}\mathcal{L}_{\uO_T}^{n_2} v \big)\cdot N 
  = (\uS_T - \tfrac{1}{2})^{n_1}\uO_T^{n_2} \left(v \cdot N 
 \right) 
  + v\cdot  (\uS_T - \tfrac{1}{2})^{n_1} \mathcal{L}_{\uO}^{n_2} \nabla h
  + H_{n_1, n_2}^2,
\end{equation}
\dg{where $H_{n_1, n_2}^2$ can also be bounded by the right-hand side of \eqref{errorbound}}.
\dg{To deal with the first term on the right-hand side of \eqref{vcdotN},
    we recall that \eqref{freebdybc} gives $v\cdot N 
  = \pa_t h$, and since the operators
$\uS_T, \uO_T$ are tangent to the boundary, at the boundary we have
\begin{align}
\label{vNdth}
\begin{split}
  (\uS_T - \tfrac{1}{2})^{n_1}\uO_T^{n_2} \left(v \cdot N 
  \right) & = 
    (\uS_T - \tfrac{1}{2})^{n_1}\uO_T^{n_2} \pa_t h
    = \pa_t \left((S - 1)^{n_1} \uO^{n_2} h \right) = \pa_t h_{n_1,n_2},
\end{split}
\end{align}
where we used $(S - \tfrac{1}{2})\pa_t = \pa_t(S - 1)$.
}

\dg{To handle the second term on the right-hand side of \eqref{vcdotN}, we 
    use that $\uS_T h = S h$ and $\uO_T h = \Omega h$ and the commutator identity
    \eqref{SLie} to write
\begin{align}\label{dtheqnpr2}
v\cdot(\uS_T - \tfrac{1}{2})^{n_1} \mathcal{L}_{\uO}^{n_2} \nabla h = v\cdot \nabla h^{n_1, n_2}
 + H_{n_1, n_2}^3,
\end{align}
where $H_{n_1, n_2}^3$are terms that can be bounded by the right-hand side of \eqref{errorbound}
 after using the trace inequality \eqref{traceineq2}.
Combining \eqref{dtheqnpr1}-\eqref{dtheqnpr2}, 
completes the derivation of \eqref{dtheqn}.
}

\fp{The equation \eqref{vn1n2eq} can be derived in a more standard fashion, using again \eqref{SLie}
so we skip the details.}

\medskip
\fp{{\it Energy functionals and conclusion}.}
Starting from \eqref{freebdyvfs}, one can begin to carry out energy estimates by 
(applying derivatives as in the standard case and) multiplying 
the equation \eqref{vn1n2eq} with (derivatives of) $v^{n_1,n_2}$ and integrating over $\D_t$.
Integrating by parts the pressure term one sees that bulk terms vanish 
in view of \eqref{divvn1n2eq} and \eqref{divSLie}.
The remaining boundary integral of $P^{n_1,n_2} (v^{n_1,n_2} \cdot n)$,
where $n$ is the unit normal, 
can be manipulated using the formulas \eqref{pn1n2eqn} and \eqref{vcdotN}-\eqref{vNdth}.
This motivates the definition of high order energies of the form
\begin{align}\label{En1n2}
E_{n_1,n_2}(t) := \frac{1}{2} \int_{\D_t} \big| v^{n_1,n_2} \big|^2 \,dxdy
  + \frac{1}{2} \int_{\mathbb{R}^2} \big(- \partial_y p) \big| h^{n_1,n_2} \big|^2 \, dx.
\end{align}
Note that the Taylor sign condition $-\pa_y p \geq c_0 > 0$ holds automatically in our setting
of small solutions,
since \eqref{freebdyeul} gives $-\pa_y p = g + \partial_t v_3 + v \cdot \nabla v_3 \geq g - C\e_0$.
\fp{One can then consider the functionals \eqref{En1n2} for $n_1+n_2 \leq k$ and include $r$ regular derivatives,
by defining
\begin{align*}
\mathcal{E}_{r,k} := \sum_{n_1+n_2\leq k} \frac{1}{2} \int_{\D_t} \big| \nabla^r_{x,y}v^{n_1,n_2} \big|^2 \,dxdy
  + \frac{1}{2} \int_{\mathbb{R}^2} \big(- \partial_y p) \big| \nabla^r h^{n_1,n_2} \big|^2 \, dx.
\end{align*}
These can be chosen as the functionals appearing in Proposition \ref{propEv}, 
and it is easy to verify that \eqref{propEveq} holds.
The claimed a priori estimates \eqref{propEvEE} can 
then be obtained based on the system \eqref{freebdyvfs}, following standard arguments as those in \cite{CL,ShZ2}.}




\bigskip
\begin{thebibliography}{100}

\normalsize

\bibitem{ABZ2}
T. Alazard, N. Burq and C. Zuily.
\newblock On the Cauchy problem for gravity water waves.
\newblock {\em Invent. Math.} 198 (2014), no. 1, 71-163.


\bibitem{ADa}
T. Alazard and J.-M. Delort.
\newblock Global solutions and asymptotic behavior for two dimensional gravity water waves.
\newblock {\em Ann. Sci. \'Ec. Norm. Sup\'er.} 48 (2015), 1149-1238.

\bibitem{ADb}
T. Alazard and J.-M. Delort.
\newblock Sobolev estimates for two dimensional gravity water waves
\newblock {\em  Ast\'erisque} 374 (2015) viii+241 pages.




\bibitem{BeFrMa} 
M. Berti, L. Franzoi and A. Maspero.
\newblock Pure gravity traveling quasi-periodic water waves with constant vorticity.
\newblock 
{\em Comm. Pure Appl. Math}. Vol. 77 (2024), no. 2, 990-1064.


\bibitem{BeMaMu} 
M. Berti, A. Maspero and F. Murgante.
\newblock Hamiltonian Birkhoff normal form for gravity-capillary water waves with constant vorticity: 
almost global existence.
\newblock Preprint {\em arXiv:2212.12255}.


 
\bibitem{BeFePu} 
M. Berti, R. Feola, and F. Pusateri.
\newblock Birkhoff normal form and long time existence for periodic gravity water waves.
\newblock {\em Comm. Pure Appl. Math.} 76 (2023), no. 7, 1416-1494.



\bibitem{CaLa}
A. Castro and D. Lannes.
\newblock Well-posedness and shallow-water stability for
a new Hamiltonian formulation of the water waves equations with vorticity.
\newblock {\em Indiana Univ. Math. J.} 64 (2015), no. 4, 1169-1270.

\bibitem{CHS}
H. Christianson, V. Hur, and G. Staffilani.
\newblock Strichartz estimates for the water-wave problem with surface tension.
\newblock {\em Comm. Partial Differential Equations} 35 (2010), no. 12, 2195-2252.


\bibitem{CL}
D. Christodoulou and H. Lindblad.
\newblock On the motion of the free surface of a liquid.
\newblock {\em Comm. Pure Appl. Math.} 53 (2000), no. 12, 1536-1602.

\bibitem{LL}
H. Lindblad, and C. Luo. 
\newblock A priori estimates for the compressible Euler equations for a liquid with 
free surface boundary and the incompressible limit. 
\newblock {\em Communications on Pure and Applied Mathematics.} 71.7 (2018): 1273-1333.

\bibitem{CS2}
D. Coutand and S. Shkoller.
\newblock Well-posedness of the free-surface incompressible Euler equations with or without surface tension.
\newblock {\em  J. Amer. Math. Soc.} 20 (2007), no. 3, 829-930.

\bibitem{CraigLim}
W. Craig.
\newblock An existence theory for water waves and the Boussinesq and Korteweg-de Vries scaling limits.
\newblock {\em  Comm. Partial Differential Equations}, 10 (1985), no. 8, 787-1003.


\bibitem{CraigSS}
W. Craig, U. Schanz and C. Sulem.
\newblock The modulational regime of three-dimensional water waves and the Davey-Stewartson system.
\newblock {\em Ann. Inst. H. Poincar\'e Anal. Non Lin\'eaire}, 14 (1997), no. 5, 615-667.


\bibitem{Craik}
A.D.D. Craik.
\newblock The origins of water wave theory.
\newblock {\em Annual review of fluid mechanics}. Vol. 36, 1-28.
\newblock Annu. Rev. Fluid Mech., 36, Annual Reviews, Palo Alto, CA, 2004.


\bibitem{DelortICM}
J.-M. Delort.
\newblock Long time existence results for solutions of water waves equations. 
\newblock {\em Proceedings of the ICM}, Rio de Janeiro 2018. Vol. III. 
Invited lectures, 2241-2260, World Sci. Publ., Hackensack, NJ, 2018.


\bibitem{DeIoPu1}
Y. Deng, A. D. Ionescu, and F. Pusateri.
\newblock On the wave turbulence theory of 2D gravity waves, I: deterministic energy estimates.
\newblock Preprint {\em arXiv:2211.10826}.




\bibitem{DIPP}
Y. Deng, A. D. Ionescu, B. Pausader, and F. Pusateri.
\newblock Global solutions for the $3D$ gravity-capillary water waves system.
\newblock {\em Acta Math.} 219 (2017), no. 2, 213-402.



\bibitem{Zhengetal}
M. Ehrnstrom, S. Walsh and C. Zeng.
\newblock Smooth stationary water waves with exponentially localized vorticity.
\newblock {\em J. Eur. Math. Soc.} (JEMS), 25 (2023), no. 3, 1045-1090. 



\bibitem{GMS2}
P. Germain, N. Masmoudi and J. Shatah.
\newblock Global solutions for the gravity surface water waves equation in dimension 3.
\newblock {\em Ann. of Math.} 175 (2012), 691-754.



\bibitem{GMSC}
P. Germain, N. Masmoudi and J. Shatah.
\newblock Global solutions for capillary waves equation in dimension 3.
\newblock {\em Comm. Pure Appl. Math.}, 68 (2015), no. 4, 625-687.


\bibitem{Gin1}
D. Ginsberg.
\newblock On the lifespan of three-dimensional gravity water waves with vorticity.
\newblock {\em arXiv:1812.01583}. 

\bibitem{Sun}
C. 
Sun.
\newblock Large time existence of Euler–Korteweg equations and two-fluid Euler–Maxwell equations with vorticity.
\newblock {\em Nonlinear Analysis}, Volume 207, (2021), 112273.


\bibitem{StraussONEPAS}
S. Haziot, V. Hur, W. A. Strauss, J. F. Toland, 
E. Wahl\'{e}n, S. Walsh, M. H. Wheeler. 
\newblock Traveling water waves -- the ebb and flow of two centuries. 
\newblock {\em Quart. Appl. Math.} 80 (2022), no. 2, 317-401.



\bibitem{IT}
M. Ifrim and D. Tataru.
\newblock Two dimensional water waves in holomorphic coordinates II: global solutions.
\newblock {\em Bull. Soc. Math. France} 144 (2016), 369-394.

\bibitem{IT2}
M. Ifrim and D. Tataru.
\newblock The lifespan of small data solutions in two dimensional capillary water waves.
\newblock  {\em Arch. Ration. Mech. Anal.} 225 (2017), no. 3, 1279-1346. 


\bibitem{ITv}
M. Ifrim and D. Tataru.
\newblock Two dimensional gravity water waves with constant vorticity: I. Cubic lifespan.
\newblock {\em Anal. PDE} 12 (2019), no. 4, 903-967.



\bibitem{IPTT}
M. Ifrim, B. Pineau, D. Tataru and M. A. Taylor.
\newblock Sharp Hadamard local well-posedness, enhanced uniqueness and pointwise continuation 
criterion for the incompressible free boundary Euler equations.
\newblock Preprint {\em arXiv:2309.05625}.




\bibitem{IoPu2}
A. D. Ionescu and F. Pusateri.
\newblock Global solutions for the gravity water waves system in 2D.
\newblock  {\em Invent. Math.} 199 (2015), no. 3, 653-804.

\bibitem{IoPu3}
A. D. Ionescu and F. Pusateri.
\newblock Global analysis of a model for capillary water waves in 2D.
\newblock {\em Comm. Pure Appl. Math.} 69 (2016), no. 11, 2015-2071.

\bibitem{IoPu4}
A. D. Ionescu and F. Pusateri.
\newblock Global regularity for 2d water waves with surface tension.
\newblock {\em Mem. Amer. Math. Soc.} 256 (2018), Memo 1227.

\bibitem{IoPuRev}
A. D. Ionescu and F. Pusateri.
\newblock Recent advances on the global regularity for irrotational water waves.
\newblock {\em  Philos. Trans. Roy. Soc. A} 376 (2018), no. 2111, 20170089, 28 pp.

\bibitem{IoPu5}
A. D. Ionescu and F. Pusateri.
\newblock Long-time existence for multi-dimensional periodic water waves. 
\newblock {\it Geom. Funct. Anal.} 29 (2019), 811-870.


\bibitem{IoLie}
A. D. Ionescu and V. Lie.
\newblock Long term regularity of the one-fluid Euler-Maxwell system in 3D with vorticity.
\newblock {\em Advances in Mathematics} 325 (2018), 719-769.


\bibitem{Lannes}
D. Lannes.
\newblock Well-posedness of the water waves equations.
\newblock {\em J. Amer. Math. Soc.} 18 (2005), 605-654.


\bibitem{LannesBook}
D. Lannes.
\newblock The water waves problem. Mathematical analysis and asymptotics.
\newblock Mathematical Surveys and Monographs, Vol. 188. American Mathematical Society, Providence, RI, 2013. xx+321 pp.

\bibitem{HL}
H. Lindblad.
\newblock Well-posedness for the motion of an incompressible liquid with free surface boundary.
\newblock {\em Ann. of Math.} 162 (2005), 109-194.


\bibitem{ShZ2}
J. Shatah and C. Zeng.
\newblock A priori estimates for fluid interface problems. 
\newblock {\em Comm. Pure Appl. Math.} 61 (2008), no. 6, 848-876.


\bibitem{ShZ3}
J. Shatah and C. Zeng.
\newblock Local well-posedness for the fluid interface problem.
\newblock {\em Arch. Ration. Mech. Anal.} 199 (2011), no. 2, 653-705.



\bibitem{steinbook}
E. M. Stein.
\newblock Harmonic analysis: real-variable methods, orthogonality, and oscillatory integrals.
\newblock 
Princeton University Press, Princeton, NJ, 1993, xiv+695 pp.




\bibitem{SulemBook}
C. Sulem and P.L. Sulem.
\newblock The nonlinear Schr\"{o}dinger equation. Self-focussing and wave collapse.
\newblock {\em Applied Mathematical Sciences}, 139.
\newblock Springer-Verlag, New York, 1999.


\bibitem{Su1}
Q. Su.
\newblock Long time behavior of 2D water waves with point vortices. 
\newblock {\em Comm. Math. Phys.} 380 (2020), no. 3, 1173-1266.


\bibitem{Su2}
Q. Su.
\newblock
On the Transition of the Rayleigh-Taylor Instability in 2d Water Waves with Point Vortices. 
\newblock {\em Annals of PDE} 9.2 (2023) 19.

\bibitem{Wanglow}
L. Wang.
\newblock Low regularity well-posedness for two dimensional deep gravity water waves with constant vorticity.
\newblock Preprint {\em arXiv:2312.09347}.



\bibitem{Wa1}
X. Wang.
\newblock Global infinite energy solutions for the 2D gravity water waves system.
\newblock {\em  Comm. Pure Appl. Math.} 71 (2018), no. 1, 90-162.


\bibitem{Wa2}
X. Wang.
\newblock Global regularity for the 3D finite depth capillary water waves. 
\newblock {\em Ann. Sci. \'Ec. Norm. Sup\'er.} (4) 53 (2020), no. 4, 847-943.



\bibitem{WZZZ}
C. Wang, Z. Zhang, W. Zhao and Y. Zheng.
\newblock Local well-posedness and break-down criterion of the incompressible Euler equations with free boundary. 
\newblock {\em Mem. Amer. Math. Soc.} 270 (2021), no. 1318, v + 119 pp.



\bibitem{Wu1}
S. Wu.
\newblock Well-posedness in Sobolev spaces of the full water wave problem in 2-D.
\newblock {\em Invent. Math.} 130 (1997), 39-72.


\bibitem{Wu2}
S. Wu.
\newblock Well-posedness in Sobolev spaces of the full water wave problem in 3-D.
\newblock {\em J. Amer. Math. Soc.}, 12 (1999), 445-495.


\bibitem{WuAG}
S. Wu.
\newblock Almost global wellposedness of the 2-D full water wave problem.
\newblock {\em Invent. Math.}, 177 (2009), 45-135.


\bibitem{Wu3DWW}
S. Wu.
\newblock Global wellposedness of the 3-D full water wave problem.
\newblock {\em Invent. Math.}, 184 (2011), 125-220.



\bibitem{WunonC1}
S. Wu.
\newblock Wellposedness of the 2D full water wave equation in a regime that allows for non-$C^1$ interfaces.
\newblock {\em Invent. Math.}, 217 (2019), 241-375.



\bibitem{Zak0} 
V. E. Zakharov.
\newblock Stability of periodic waves of finite amplitude on the surface of a deep fluid.
\newblock {\em Zhurnal Prikladnoi Mekhaniki i Teckhnicheskoi Fiziki} 9 (1968), no.2,  86-94.
\newblock {\em J. Appl. Mech. Tech. Phys.}, 9, 1990-1994.


\end{thebibliography}
\end{document}